\documentclass[12pt]{amsart}
\usepackage{amssymb}
\usepackage{amsmath}
\usepackage{amsfonts}
\usepackage{amsthm}
\usepackage{graphicx}
\usepackage{color}
\usepackage{setspace}
\usepackage{html}
\usepackage{harvard}
\usepackage[foot]{amsaddr}
\usepackage{xr}

\numberwithin{equation}{section}
\theoremstyle{plain}
\newtheorem{thm}{Theorem}
\newtheorem{cor}{Corollary}

\newtheorem{TheoremAUC}{Theorem AUC\hspace{-0.5ex}}
\newtheorem{LemmaAUC}{Lemma AUC\hspace{-0.5ex}}
\newtheorem{LemmaA}{Lemma A\hspace{-1ex}}

\newtheorem{LemmaC}{Lemma B\hspace{-1ex}}
\newtheorem{LemmaD}{Lemma C\hspace{-1ex}}
\newtheorem{lem}{Lemma}
\theoremstyle{assumption}
\newtheorem{AssumptionA}{Assumption A\hspace{-1ex}}
\newtheorem{AssumptionB}{Assumption B\hspace{-1ex}}
\newtheorem{AssumptionC}{Assumption C\hspace{-1ex}}
\newtheorem{AssumptionAUC}{Assumption AUC\hspace{-1ex}}

\theoremstyle{definition}
\newtheorem{defn}{Definition}
\theoremstyle{remark}
\newtheorem{rem}{Remark}
\input{tcilatex}

\begin{document}
\title[General Functional Inequalities]{Testing for a General Class of
Functional Inequalities}
\thanks{We would like to thank Emmanuel Guerre and participants at numerous
seminars and conferences for their helpful comments. We also thank Kyeongbae
Kim, Koohyun Kwon and Jaewon Lee for capable research assistance. Lee's work was supported by the European Research Council
(ERC-2009-StG-240910-ROMETA) and by the National Research  Foundation of Korea Grant funded by the Korean 
Government (NRF-2012S1A5A8023573).
Song acknowledges the financial support of
Social Sciences and Humanities Research Council of Canada. Whang's work was
supported by the SNU Creative Leading Researcher Grant.}
\date{16 June 2015}
\author[Lee]{Sokbae Lee$^{1,2}$}
\address{$^1$Department of Economics, Seoul National University, 1
Gwanak-ro, Gwanak-gu, Seoul, 151-742, Republic of Korea.}
\address{$^2$Centre for Microdata Methods and Practice, Institute for Fiscal
Studies, 7 Ridgmount Street, London, WC1E 7AE, UK.}
\email{sokbae@sokbae.ac.kr}
\author[Song]{Kyungchul Song$^3$}
\address{$^3$Vancouver School of Economics, University of British Columbia,
997 - 1873 East Mall, Vancouver, BC, V6T 1Z1, Canada}
\email{kysong@mail.ubc.ca}
\author[Whang]{Yoon-Jae Whang$^1$}
\email{whang@snu.ac.kr}

\begin{abstract}
In this paper, we propose a general method for testing inequality
restrictions on nonparametric functions. Our framework includes many
nonparametric testing problems in a unified framework, with a number of
possible applications in auction models, game theoretic models, wage
inequality, and revealed preferences. Our test involves a one-sided version
of $L_{p}$ functionals of kernel-type estimators $(1\leq p <\infty )$ and is
easy to implement in general, mainly due to its recourse to the bootstrap
method. The bootstrap procedure is based on nonparametric bootstrap applied
to kernel-based test statistics, with an option of estimating \textquotedblleft contact
sets.\textquotedblright\ We provide regularity conditions under which the
bootstrap test is asymptotically valid uniformly over a large class of
distributions, including the cases that the limiting distribution of the
test statistic is degenerate. Our bootstrap test is shown to exhibit good
power properties in Monte Carlo experiments, and we provide a general form
of the local power function.
As an illustration, we consider testing implications from auction theory,
provide primitive conditions for our test, and demonstrate its usefulness by applying our
test to real data. We supplement this example with the second empirical illustration in the
context of wage inequality. \bigskip

 \bigskip

\end{abstract}

\maketitle

\onehalfspacing

\clearpage

\section{Introduction}

In this paper, we propose a general method for testing inequality
restrictions on nonparametric functions. To describe our testing problem,
let $v_{\tau ,1},\ldots ,v_{\tau ,J}$ denote nonparametric real-valued
functions on $\mathbf{R}^{d}\ $ for each index $\tau \in \mathcal{T}$, where 
$\mathcal{T}$ is a subset of a finite dimensional space. We focus on testing 
\begin{equation}
\begin{split}
H_{0}& :\max \{v_{\tau ,1}(x),\cdot \cdot \cdot ,v_{\tau ,J}(x)\}\leq 0\text{
for all }(x,\tau )\in \mathcal{X}\times \mathcal{T}\text{, against} \\
H_{1}& :\max \{v_{\tau ,1}(x),\cdot \cdot \cdot ,v_{\tau ,J}(x)\}>0\text{
for some }(x,\tau )\in \mathcal{X\times \mathcal{T}},
\end{split}
\label{null}
\end{equation}%
where $\mathcal{X}\times \mathcal{T}$ is a domain of interest. We propose a
one-sided $L_{p}$ integrated test statistic based on nonparametric
estimators of $v_{\tau ,1},\ldots ,v_{\tau ,J}$. We provide general
asymptotic theory for the test statistic and suggest a bootstrap procedure
to compute critical values. We establish that our test has correct uniform
asymptotic size and is not conservative. We also determine the asymptotic
power of our test under fixed alternatives and some local alternatives.

We allow for a general class of nonparametric functions, including, as
special cases, conditional mean, quantile, hazard, and distribution
functions and their derivatives. For example, $v_{\tau ,j}(x)=P(Y_{j}\leq
\tau |X=x)$ can be the conditional distribution function of $Y_{j}$ given $%
X=x$, or $v_{\tau ,j}(x)$ can be the $\tau $-th quantile of $Y_{j}$
conditional on $X=x$. We can also allow for transformations of these
functions satisfying some regularity conditions. The nonparametric
estimators we consider are mainly kernel-type estimators but can be allowed
to be more general, provided that they satisfy certain Bahadur-type linear
expansions.

Inequality restrictions on nonparametric functions arise often as testable
implications from economic theory. For example, in first-price auctions, %
\citeasnoun{GPV} show that the quantiles of the observed equilibrium bid
distributions with different numbers of bidders should satisfy a set of
inequality restrictions (Equation (5) of \citeasnoun{GPV}). If the auctions
are heterogeneous so that the private values are affected by observed
characteristics, we may consider conditionally exogenous participation with
a conditional version of the restrictions (see Section 3.2 of %
\citeasnoun{GPV}). Such restrictions are in the form of multiple
inequalities for linear combinations of nonparametric conditional quantile
functions. Our test then can be used to test whether the restrictions hold
jointly uniformly over quantiles and observed characteristics in a certain
range.
In this paper, we use this auction example
to illustrate the usefulness of our general framework. To the best of our knowledge, there
does not exist an alternative test available in the literature for this kind of examples.

In addition to \citeasnoun[GPV hereafter]{GPV}, a large number of auction models are
associated with some forms of functional inequalities. See, for example, %
\citeasnoun{Haile/Tamer:03}, \citeasnoun{Haile/Hong/Shum:03}, %
\citeasnoun{AGQ:13a}, \citeasnoun{AGQ:13c}, and %
\citeasnoun{Krasnokutskaya/Song/Tang:13}, among others. Our method can be
used to make inference in their setups, while allowing for continuous
covariates.

Econometric models of games belong to a related but distinct branch of the
literature, compared to the auction models. In this literature, inference on
many game theoretic models are recently based on partial identification or
functional inequalities. For example, see \citeasnoun{Tamer:03}, %
\citeasnoun{Andrews/Berry/Jia:04}, \citeasnoun{Berry/Tamer:07}, %
\citeasnoun{Aradillas-Lopez/Tamer:08}, \citeasnoun{Ciliberto/Tamer:09}, %
\citeasnoun{Beresteanu/Molchanov/Molinari:11}, \citeasnoun{Galichon/Henry:11}%
, \citeasnoun{Cheser/Rosen:12}, and \citeasnoun{Aradillas-Lopez/Rosen:13},
among others. See \citeasnoun{dePaula:13} and references therein for a broad
recent development in this literature. Our general method provides
researchers in this field a new inference tool when they have continuous
covariates.

Inequality restrictions also arise in testing revealed preferences. %
\citeasnoun{Blundell/Browning/Crawford:08} used revealed preference
inequalities to provide the nonparametric bounds on average consumer
responses to price changes. In addition, %
\citeasnoun{Blundell/Kristensen/Matzkin:11} used the same inequalities to
bound quantile demand functions. It would be possible to use our framework
to test revealed preference inequalities either in average demand functions
or in quantile demand functions. See also \citeasnoun{Hoderlein:Stoye:13}
and \citeasnoun{Kitamura:Stoye:13} for related issues of testing revealed
preference inequalities.

In addition to the literature mentioned above, many results on partial
identification can be written as functional inequalities. See, e.g., %
\citeasnoun{Imbens/Manski:04}, \citeasnoun{Manski:03}, \citeasnoun{Manski:07}%
, \citeasnoun{Manski/Pepper:00}, \citeasnoun{Tamer:10}, %
\citeasnoun{Cheser/Rosen:13}, and references therein.

Our framework has several distinctive merits. First, our proposal is easy to
implement in general, mainly due to its recourse to the bootstrap method.
The bootstrap procedure is based on nonparametric bootstrap applied to
kernel-based test statistics. We establish the general asymptotic (uniform)
validity of the bootstrap procedure.

Second, our proposed test is shown to exhibit good power properties both in
finite and large samples. Good power properties can be achieved by the use
of critical values that adapt to the binding restrictions of functional
inequalities. This could be done in various ways; in this paper, we follow
the ``contact set'' approach of \citeasnoun{Linton/Song/Whang:10} and
propose bootstrap critical values. As is shown in this paper, the bootstrap
critical values yield significant power improvements. Furthermore, we find
through our local power analysis that this class of tests exhibits dual
convergence rates depending on Pitman directions, and in many cases, the
faster of the two rates achieves a parametric rate of $\sqrt{n}$, despite
the use of kernel-type test statistics.

Third, we establish the asymptotic validity of the proposed test uniformly
over a large class of distributions, without imposing restrictions on the
covariance structure among nonparametric estimates of $v_{\tau,j}(\cdot)$,
thereby allowing for degenerate cases. Such a uniformity result is crucial
for ensuring good finite sample properties for tests whose (pointwise)
limiting distribution under the null hypothesis exhibits various forms of
discontinuity. The discontinuity in the context of this paper is highly
complex, as the null hypothesis involves inequality restrictions on a
multiple number of (or even a continuum of) nonparametric functions. We
establish the uniform validity of the test in a way that covers these
various incidences of discontinuity. Our new uniform asymptotics may be of
independent interest in many other contexts.

Much of the recent literature on testing inequality restrictions focuses on
conditional moment inequalities.\footnote{%
There exists a large literature on inference on models with a finite number of
unconditional moment inequality restrictions. Some examples include %
\citeasnoun{Andrews/Jia:08}, \citeasnoun{Andrews/Guggenberger:09}, %
\citeasnoun{Andrews/Soares:10}, \citeasnoun{Beresteanu/Molinari:08}, %
\citeasnoun{Bugni:07}, \citeasnoun{Canay:07}, %
\citeasnoun{Chernozhukov/Hong/Tamer:07}, \citeasnoun{Galichon/Henry:06a}, %
\citeasnoun{Romano/Shaikh:06b}, \citeasnoun{Romano/Shaikh:06a}, and %
\citeasnoun{Rosen:05}, among others.} Researches on conditional moment
inequalities include \citeasnoun{Andrews/Shi:13a}, %
\citeasnoun{Andrews/Shi:JoE13}, \citeasnoun{Armstrong:11b}, %
\citeasnoun{Armstrong:11a}, \citeasnoun{Armstrong/Chan:13}, \citeasnoun{CLR}%
, \citeasnoun{Chetverikov:11}, \citeasnoun{Fan/Park:13}, %
\citeasnoun{Khan/Tamer:09}, \citeasnoun{Kim:08}, \citeasnoun{LSW}, %
\citeasnoun{Menzel:08}, and \citeasnoun{Ponomareva:10}, among others. In
contrast, this paper's approach naturally covers a wide class of inequality
restrictions among nonparametric functions that the moment inequality
framework does not (or at least is cumbersome to) apply. Such examples
include testing multiple inequalities that are defined by differences in
conditional quantile functions uniformly over covariates and quantiles.%
\footnote{%
A working paper version \cite{Andrews/Shi:09} of \citeasnoun{Andrews/Shi:13a}
covers testing moment inequalities indexed by $\tau \in \mathcal{T}$, but
their framework does not appear to be easily extendable to deal with
functions of multiple conditional quantiles such as differences in
conditional quantiles.} If we restrict our attention to the conditional
moment inequalities, then our approach is mostly comparable to the moment
selection approach of \citeasnoun{Andrews/Shi:13a}. Our general framework is
also related to testing qualitative nonparametric hypotheses such as
monotonicity in mean regression. See, for example, \citeasnoun{BSL:05}, %
\citeasnoun{Chetverikov:12}, \citeasnoun{DS:01}, and %
\citeasnoun{Ghosal/Sen/vanderVaart:00} among many others. See also %
\citeasnoun{Lee/Linton/Whang:06} and \citeasnoun{Delgado:Escanciano:12} for
testing stochastic monotonicity.

Among aforementioned papers, \citeasnoun{CLR} developed a sup-norm approach
in testing inequality restrictions on nonparametric functions using
pointwise asymptotics, and in principle, could be extended to test general
functional inequalities as in \eqref{null}.\footnote{%
Our test involves a one-sided version of $L_{p}$-type functionals of
nonparametric estimators $(1\leq p<\infty )$. We regard the sup-norm and $%
L_p $ norm approaches complementary, each with its own strength and
weakness. For example, our test and also the test of %
\citeasnoun{Andrews/Shi:13a} have higher power against relatively flat
alternatives, whereas the test of \citeasnoun{CLR} has higher power against
sharply-peaked alternatives. See the results of Monte Carlo experiments
reported in Appendix \ref{sec:mc:AS}. See also \citeasnoun{Andrews/Shi:13a}, %
\citeasnoun{Andrews/Shi:JoE13}, and \citeasnoun{CLR} for related discussions
and further Monte Carlo evidence.} Example 4 of \citeasnoun{CLR} considered
the case of one inequality with a conditional quantile function at a
particular quantile, but it is far from trivial to extend this example to
multiple inequalities of differences in conditional quantile functions
uniformly over a range of quantiles. As this paper demonstrates through an
empirical application, such testing problem can arise in empirical research
(see Section \ref{emp-example auction}).

The uniformity result in this paper is non-standard since our test is based
on asymptotically non-tight processes, in contrast to %
\citeasnoun{Andrews/Shi:13a} who convert conditional moment inequalities
into an infinite number of unconditional moment inequalities. This paper's
development of asymptotic theory draws on the method of Poissonization (see,
e.g., \citeasnoun{Horvath:91} and \citeasnoun{GMZ}). For applications of
this method, see \citeasnoun{Anderson/Linton/Whang:12} for inference on a
polarization measure, \citeasnoun{Chang/Lee/Whang:14} 
% \citeasnoun{Lee/Whang:09} 
for testing for conditional treatment effects, and \citeasnoun{LSW} for
testing inequalities for nonparametric regression functions using the
numerator of the Nadaraya-Watson estimator (based on pointwise asymptotics).
Also, see \citeasnoun{Mason/Polonik:09} and \citeasnoun{BCMP:09} for support
estimation.

The remainder of the paper is as follows. Section \ref{sec:overview} gives
an informal description of our general framework by introducing test
statistics and critical values and by providing intuitions behind our
approach. In Section \ref{sec:theory}, we establish the uniform asymptotic
validity of our bootstrap test. We also provide a class of distributions for
which the asymptotic size is exact. In Section \ref{sec:power}, we establish
consistency of our test and its local power properties. In Section \ref%
{sec:mc:AS}, we report results of some Monte Carlo experiments. 
In Sections \ref{emp-example auction} and \ref{emp-example wage}, we give two empirical examples. The first empirical example in Section \ref{emp-example auction} is on testing auction models following GPV, and the second one in Section %
\ref{emp-example wage} is about testing functional inequalities via differences-in-differences in conditional quantiles, inspired by \citeasnoun{Acemoglu:Autor:11}. The empirical examples given in this section are not covered easily by existing inference methods; however, they are all special cases of our general framework.
Section \ref{sec:conclusion} concludes.
Appendices provide all the proofs of theorems with a roadmap of the
proofs to help readers.

\section{General Overview}

\label{sec:overview}

\subsection{Test Statistics}

We present a general overview of this paper's framework by introducing test
statistics and critical values. To ease the exposition, we confine our
attention to the case of $J=2$ here. The definitions and formal results for
general $J$ are given later in Section \ref{sec:theory}.

Throughout this paper, we assume that $\mathcal{T}$ is a compact subset of a
Euclidean space. This does not lose much generality because when $\mathcal{T}
$ is a finite set, we can redefine our test statistic by taking $\mathcal{T}$
as part of the finite index $j$ indexing the nonparametric functions.

For $j=1,2$, let $\hat{v}_{\tau ,j}(x)$ be a kernel-based nonparametric
estimator of $v_{\tau ,j}(x)$ and let its appropriately scaled version be 
\begin{equation*}
\hat{u}_{\tau ,j}(x)\equiv \frac{r_{n,j}\hat{v}_{\tau ,j}(x)}{\hat{\sigma}%
_{\tau ,j}(x)},
\end{equation*}%
where $r_{n,j}$ is an appropriate normalizing sequence that diverges to
infinity,\footnote{%
Permitting the convergence rate $r_{n,j}$ to differ across $j\in \mathbb{N}%
_{J}$ can be convenient, when the nonparametric estimators have different
convergence rates. For example, this accommodates a situation where one
jointly tests the non-negativity and monotonicity of a nonparametric
function.} and $\hat{\sigma}_{\tau ,j}(x)$ is an appropriate (possibly
data-dependent) scale normalization.\footnote{%
While our framework permits the case where $\hat{\sigma}_{\tau ,j}(x)$ is
simply chosen to be $1$, we allow for a more general case where $\hat{\sigma}%
_{\tau ,j}(x)$ is a consistent estimator for some nonparametric quantity.}
Then the inference is based on the following statistic:%
\begin{eqnarray}
\hat{\theta} &\equiv &\int_{\mathcal{T}}\int_{\mathcal{X}}\max \left\{ \hat{u%
}_{\tau ,1}(x),\hat{u}_{\tau ,2}(x),0\right\} ^{p}dxd\tau  \label{thetah} \\
&\equiv&\int_{\mathcal{X}\times \mathcal{T}}\max \left\{ \hat{u}_{\tau
,1}(x),\hat{u}_{\tau ,2}(x),0\right\} ^{p}dQ(x,\tau ),  \notag
\end{eqnarray}
where $Q$ is Lebesgue measure on $\mathcal{X}\times \mathcal{T}$. In this
overview section, we focus on the case of using the max function under the
integral in (\ref{thetah}). In addition, we consider the sum $\sum_{j=1}^2
\max \left\{ \hat{u}_{\tau ,j}(x),0\right\} ^{p}$ in general theory (see %
\eqref{lambda_p}).

\subsection{Bootstrap Critical Values}

\label{subsec:bootstrap-cv} As we shall see later, the asymptotic
distribution of the test statistic exhibits complex ways of discontinuities
as one perturbs the data generating processes. This suggests that the finite
sample properties of the asymptotic critical values may not be stable.
Furthermore, the location-scale normalization requires nonparametric
estimation and thus a further choice of tuning parameters. This can worsen
the finite sample properties of the critical values further. To address
these issues, this paper develops a bootstrap procedure.

In the following, we let $\hat{v}_{\tau ,j}^{\ast }(x)$ and $\hat{\sigma}_{\tau ,j}^{\ast }(x)$, $%
j=1,2$, denote the bootstrap counterparts of $\hat{v}_{\tau ,j}(x)$ and $%
\hat{\sigma}_{\tau ,j}(x),\ j=1,2.$ Let the bootstrap counterparts be
constructed in the same way as the nonparametric estimators $\hat{v}_{\tau
,j}(x)$ and $\hat{\sigma}_{\tau ,j}(x),\ j=1,2,$ with the bootstrap sample
independently drawn with replacement from the empirical distribution of the
original sample. We let%
\begin{equation}
\hat{s}_{\tau ,j}^{\ast }(x)\equiv \frac{r_{n,j}\{\hat{v}_{\tau ,j}^{\ast
}(x)-\hat{v}_{\tau ,j}(x)\}}{\hat{\sigma}_{\tau ,j}^{\ast }(x)},\text{ }%
j=1,2.  \label{shat*}
\end{equation}%
Note that $\hat{s}_{\tau ,j}^{\ast }(x)$ is a centered and scale normalized
version of the bootstrap quantity $\hat{v}_{\tau ,j}^{\ast }(x)$.
Using these bootstrap quantities, we consider two versions of bootstrap critical values: one based on the least favorable case and the other based on estimating a contact set.

\subsubsection{The Least Favorable Case}

Under the least favorable configuration (LFC),  we
construct a bootstrap version of the right hand side of (\ref{thetah}) as%
\begin{align*}
\hat{\theta}^{\ast }_{\mathrm{LFC}} &\equiv 
\int \max \left\{ \hat{s}_{\tau ,1}^{\ast
}(x),\hat{s}_{\tau ,2}^{\ast }(x),0\right\} ^{p}dQ(x,\tau ).
\end{align*}%
Under regularity conditions, bootstrap critical values based on the LFC can be shown to yield tests that are asymptotically valid \textit{uniformly}
in $P$. However, they are often too conservative in practice. 
As an alternative to the LFC-based bootstrap critical value, we propose a bootstrap critical value
that can be less conservative but at the expense of introducing an additional tuning parameter.

\subsubsection{Estimating a Contact Set}

As we shall show formally in a more general form in Lemma 1 in Section \ref%
{sec:theory} below, it is satisfied that under $H_0$, for each sequence $%
c_{n}\rightarrow \infty $ such that $\sqrt{\log n}/c_{n}\rightarrow 0$ as $%
n\rightarrow \infty $,%
\begin{eqnarray}
\hat{\theta} &=&\int_{B_{n,\{1\}}(c_{n})}\max \left\{ \hat{u}_{\tau
,1}(x),0\right\} ^{p}dQ(x,\tau )  \label{theta2} \\
&&+\int_{B_{n,\{2\}}(c_{n})}\max \left\{ \hat{u}_{\tau ,2}(x),0\right\}
^{p}dQ(x,\tau )  \notag \\
&&+\int_{B_{n,\{1,2\}}(c_{n})}\max \left\{ \hat{u}_{\tau ,1}(x),\hat{u}%
_{\tau ,2}(x),0\right\} ^{p}dQ(x,\tau ),  \notag
\end{eqnarray}%
with probability approaching one, where, letting $u_{n,\tau ,j}(x)\equiv
r_{n,j}v_{n,\tau ,j}(x)/\sigma _{n,\tau ,j}(x)$, i.e., a population version
of $\hat{u}_{\tau ,j}(x)$,\footnote{%
It is convenient for general development to let the population quantities $%
v_{n,\tau ,j}(x)$ and $\sigma _{n,\tau ,j}(x)$ depend on $n$.} we define%
\begin{eqnarray*}
B_{n,\{1\}}(c_{n}) &\equiv &\left\{ (x,\tau )\in \mathcal{X}\times \mathcal{T%
}:|u_{n,\tau ,1}(x)|\ \leq c_{n}\text{ and }u_{n,\tau ,2}(x)<-c_{n}\right\} 
\text{,} \\
B_{n,\{2\}}(c_{n}) &\equiv &\left\{ (x,\tau )\in \mathcal{X}\times \mathcal{T%
}:|u_{n,\tau ,2}(x)|\ \leq c_{n}\text{ and }u_{n,\tau ,1}(x)<-c_{n}\right\} 
\text{ and} \\
B_{n,\{1,2\}}(c_{n}) &\equiv &\left\{ (x,\tau )\in \mathcal{X}\times 
\mathcal{T}:|u_{n,\tau ,1}(x)|\ \leq c_{n}\text{ and }|u_{n,\tau ,2}(x)|\
\leq c_{n}\right\} .
\end{eqnarray*}%
For example, the set $B_{n,\{1\}}(c_{n})$ is a set of points $(x,\tau )$
such that $|v_{n,\tau ,1}(x)/\sigma _{n,\tau ,1}(x)|$ is close to zero, and $%
v_{n,\tau ,2}(x)/\sigma _{n,\tau ,2}(x)$ is negative and away from zero. We
call \textit{contact sets} such sets as $B_{n,\{1\}}(c_{n})$,$\
B_{n,\{2\}}(c_{n})$, and $B_{n,\{1,2\}}(c_{n})$.

Now, comparing (\ref{theta2}) with (\ref{thetah}) reveals that the limiting
distribution of $\hat{\theta}$ under the null hypothesis will not depend on
points outside the union of the contact sets. Thus it is natural to base the
bootstrap critical values on the quantity on the right hand side of (\ref%
{theta2}) instead of that on the last integral in (\ref{thetah}). As we will
explain shortly in the next subsection, this leads to a test that is
uniformly valid and exhibits substantial improvement in power.

%\clearpage
\begin{figure}[tbph]
\caption{Contact Set Estimation}
\label{figure-cset}
\begin{center}
\makebox{
\includegraphics[origin=bl,scale=.40,angle=0]{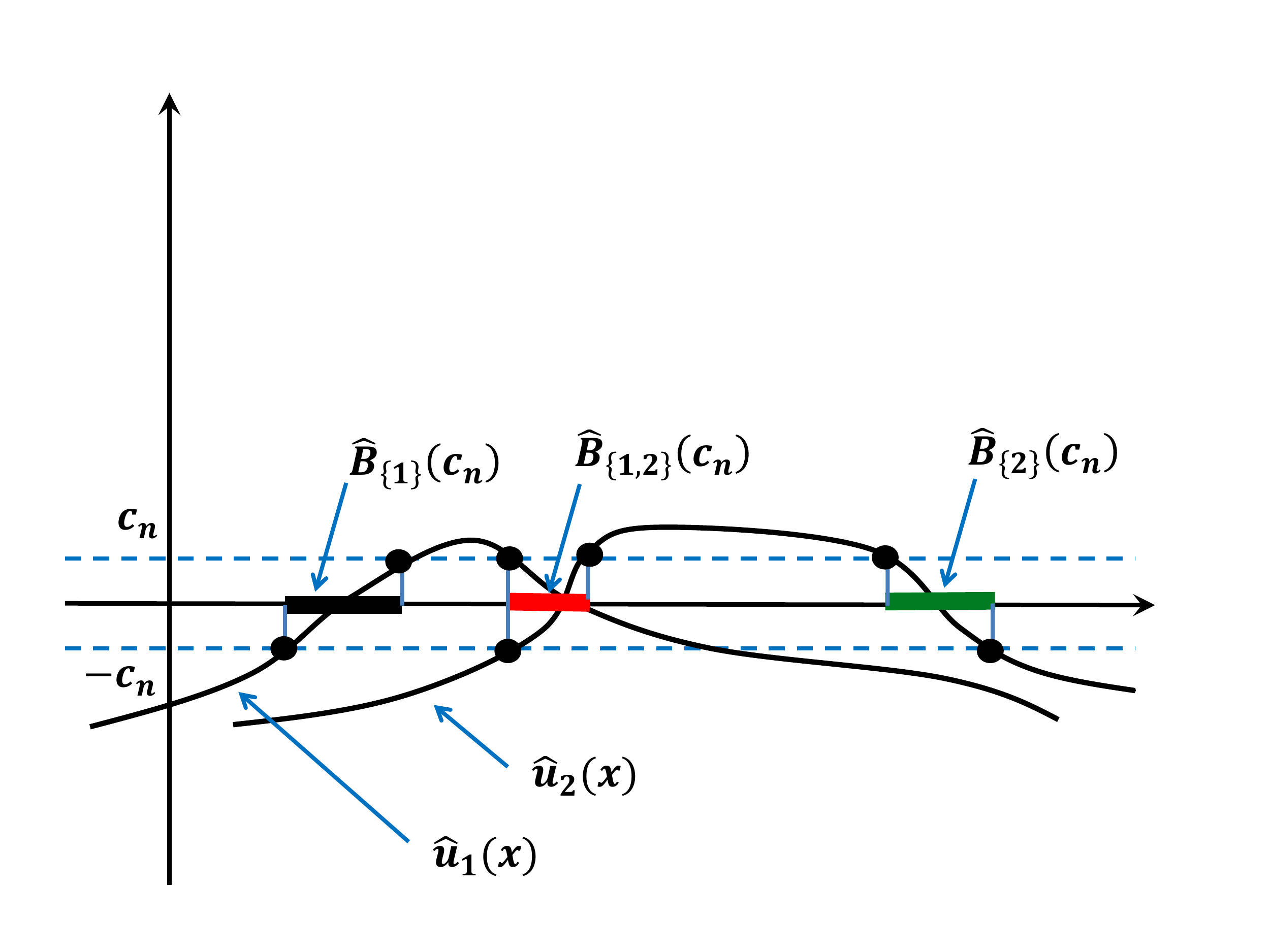}
}
\end{center}
\par
\parbox{5in}{
Note: This figure illustrates estimated contact sets when $J = 2$. The black, red, and green line segments on the x-axis represent estimated contact sets. }
\end{figure}

%\textbf{[The figure will be updated once a research assistant puts hats in the figure. (SL)]}

To construct bootstrap critical values, we introduce sample versions of the
contact sets:%
\begin{eqnarray*}
\hat{B}_{\{1\}}(c_{n}) &\equiv &\left\{ (x,\tau )\in \mathcal{X}\times 
\mathcal{T}:|\hat{u}_{\tau ,1}(x)|\ \leq c_{n}\text{ and }\hat{u}_{\tau
,2}(x)<-c_{n}\right\} \text{,} \\
\hat{B}_{\{2\}}(c_{n}) &\equiv &\left\{ (x,\tau )\in \mathcal{X}\times 
\mathcal{T}:|\hat{u}_{\tau ,2}(x)|\ \leq c_{n}\text{ and }\hat{u}_{\tau
,1}(x)<-c_{n}\right\} \text{ and} \\
\hat{B}_{\{1,2\}}(c_{n}) &\equiv &\left\{ (x,\tau )\in \mathcal{X}\times 
\mathcal{T}:|\hat{u}_{\tau ,1}(x)|\ \leq c_{n}\text{ and }|\hat{u}_{\tau
,2}(x)|\ \leq c_{n}\right\} .
\end{eqnarray*}%
See Figure \ref{figure-cset} for illustration of estimation of contact sets
when $J=2$.

Given the contact sets, we construct  a bootstrap version of the right hand side of (\ref{theta2}) as%
\begin{eqnarray}
\hat{\theta}^{\ast } &\equiv &\int_{\hat{B}_{\{1\}}(\hat{c}_{n})}\max
\left\{ \hat{s}_{\tau ,1}^{\ast }(x),0\right\} ^{p}dQ(x,\tau )  \label{crit}
\\
&&+\int_{\hat{B}_{\{2\}}(\hat{c}_{n})}\max \left\{ \hat{s}_{\tau ,2}^{\ast
}(x),0\right\} ^{p}dQ(x,\tau )  \notag \\
&&+\int_{\hat{B}_{\{1,2\}}(\hat{c}_{n})}\max \left\{ \hat{s}_{\tau ,1}^{\ast
}(x),\hat{s}_{\tau ,2}^{\ast }(x),0\right\} ^{p}dQ(x,\tau ),  \notag
\end{eqnarray}%
where $\hat{c}_{n}$ is a data dependent version of $c_{n}$. We will discuss
a way to construct $\hat{c}_{n}$ shortly. We also define 
\begin{equation*}
\hat{a}^{\ast }\equiv \mathbf{E}^{\ast }\hat{\theta}^{\ast },
\end{equation*}%
where $\mathbf{E}^{\ast }$ denotes the expectation under the bootstrap
distribution. Let $c_{\alpha }^{\ast }$ be the $(1-\alpha )$-th quantile
from the bootstrap distribution of $\hat{\theta}^{\ast }$. 
In practice, both quantities $\hat{a}^{\ast }$ and $c_{\alpha }^{\ast }$ are approximated by the sample mean and sample $(1-\alpha )$-th quantile, respectively,
from a large number of bootstrap repetitions.
Then for a small constant $%
\eta \equiv 10^{-3}$, we take $c_{\alpha ,\eta }^{\ast }\equiv
\max \{c_{\alpha }^{\ast },h^{d/2}\eta +\hat{a}^{\ast }\}$ as the critical
value to form the following test:%
\begin{equation}
\text{Reject }H_{0}\text{ if and only if }\hat{\theta}>c_{\alpha ,\eta
}^{\ast }.  \label{Test}
\end{equation}%
Then it is shown later that the test has asymptotically correct size, i.e.,%
\begin{equation}
\underset{n\rightarrow \infty }{\text{limsup}}\sup_{P\in \mathcal{P}_{0}}P\{%
\hat{\theta}>c_{\alpha ,\eta }^{\ast }\}\leq \alpha ,  \label{size}
\end{equation}%
where $\mathcal{P}_{0}$ is the collection of potential distributions that
satisfy the null hypothesis.
To implement the test, there are two important tuning parameters, namely the bandwidth $h$ used for nonparametric estimation and  the constant $\hat{c}_{n}$ for contact set estimation.\footnote{We fix the value of $\eta$ for the precise definition of the test statistic; however, its value does not matter in terms of the first-order asymptotic theory.}
We discuss how to obtain the latter in the context of  our Monte Carlo experiments in Section \ref{mc-tuning-parameters}.

\subsection{Discontinuity, Uniformity, and Power}

\label{subsec:discontinuity}

Many tests of inequality restrictions exhibit discontinuity in its limiting
distribution under the null hypothesis. When the inequality restrictions
involve nonparametric functions, this discontinuity takes a complex form, as
emphasized in Section 5 of \citeasnoun{Andrews/Shi:13a}.

To see the discontinuity problem in our context, let $\{(Y_{i},X_{i})^{\top
}\}_{i=1}^{n}$ be i.i.d. copies from an observable bivariate random vector, $%
(Y,X)^{\top }\in \mathbf{R}\times \mathbf{R}$, where $X_{i}$ is a continuous
random variable with density $f$. We consider a simple testing example: 
\begin{align}  \label{simple-example}
H_{0}:\mathbf{E}[Y|X=x]\leq 0\text{ for all $x\in \mathcal{X}$ \ vs. }H_{1}:%
\mathbf{E}[Y|X=x]>0\text{ for some $x\in \mathcal{X}$}.
\end{align}%
Here, with the subscript $\tau $ suppressed, we set $J=1$, $r_{n,1}=\sqrt{nh}
$, $p=d=1$, and define $[v]_{+}\equiv \max \{v,0\}$. Let 
\begin{align}  \label{v&s}
\hat{v}_{1}(x)& =\frac{1}{nh}\sum_{i=1}^{n}Y_{i}K\left( \frac{X_{i}-x}{h}%
\right)\ \text{and} \ \hat{\sigma}_{1}^{2}(x) =\frac{1}{nh}%
\sum_{i=1}^{n}Y_{i}^{2}K^{2}\left( \frac{X_{i}-x}{h}\right),
\end{align}%
where $K$ is a nonnegative, univariate kernel function with compact support
and $h$ is a bandwidth.

Assume that the density of $X$ is strictly positive on $\mathcal{X}$. Then,
in this example, $v_{n,1}(x)\equiv \mathbf{E}\hat{v}_{1}(x)\leq 0$ for
almost every $x$ in $\mathcal{X}$ whenever the null hypothesis is true.
Define 
\begin{equation*}
Z_{n,1}(x)=\sqrt{nh}\left\{ \frac{\hat{v}_{1}(x)-v_{n,1}(x)}{\hat{\sigma}%
_{1}(x)}\right\} \text{ and }B_{n,1}(0)=\left\{ x\in \mathcal{X}:\left\vert 
\sqrt{nh}v_{n,1}(x)\right\vert =0\right\} .
\end{equation*}%
We analyze the asymptotic properties of $\hat{\theta}$ as follows. We first
write 
\begin{eqnarray}
h^{-1/2}(\hat{\theta}-a_{n,1}) &=&h^{-1/2}\left\{ \int_{B_{n,1}(0)}\left[
Z_{n,1}(x)\right] _{+}dx-a_{n,1}\right\}  \label{dec78} \\
&&+h^{-1/2}\int_{\mathcal{X}\backslash B_{n,1}(0)}\left[ Z_{n,1}(x)+\frac{%
\sqrt{nh}v_{n,1}(x)}{\hat{\sigma}_{1}(x)}\right] _{+}dx,  \notag
\end{eqnarray}%
where 
\begin{equation*}
a_{n,1}=\mathbf{E}\left[ \int_{B_{n,1}(0)}\left[ Z_{n,1}(x)\right] _{+}dx%
\right] .
\end{equation*}%
When liminf$_{n\rightarrow \infty }Q\left( B_{n,1}(0)\right) >0$ with $%
Q(B_{n,1}(0))$ denoting Lebesgue measure of $B_{n,1}(0)$, we can show that
the leading term on the right hand side in (\ref{dec78}) becomes
asymptotically $N(0,\sigma _{0}^{2})$ for some $\sigma _{0}^{2}>0$. On the
other hand, the second term vanishes in probability as $n\rightarrow \infty $
under $H_{0}$ because for each $x\in \mathcal{X}\backslash B_{n,1}(0)$, 
\begin{equation*}
0>\sqrt{nh}v_{n,1}(x)\rightarrow -\infty
\end{equation*}%
as $n\rightarrow \infty $ under $H_{0}$. Thus we conclude that when liminf$%
_{n\rightarrow \infty }Q\left( B_{n,1}(0)\right) >0\ $under $H_{0},$ 
\begin{equation}
h^{-1/2}(\hat{\theta}-a_{n,1})\approx h^{-1/2}\left\{ \int_{B_{n,1}(0)}\left[
Z_{n,1}(x)\right] _{+}dx-a_{n,1}\right\} \rightarrow _{d}N(0,\sigma
_{0}^{2}).  \label{as}
\end{equation}

This asymptotic theory is pointwise in $P$ (with $P$ fixed and letting $%
n\rightarrow \infty $), and may not be adequate for finite sample
approximation. There are two sources of discontinuity. First, the pointwise
asymptotic theory essentially regards the drift component $\sqrt{nh}%
v_{n,1}(x)$ as $-\infty $, whereas in finite samples, the component can be
very negative, but not $-\infty $. Second, even if the nonparametric
function $\sqrt{nh}v_{n,1}(x)$ changes continuously, the contact set $%
B_{n,1}(0)$ may change discontinuously in response.\footnote{%
%In general,  $A_n \equiv \{x \in \mathcal{X}: f_n(x) = 0 \}$ does not converge
%to $A_0 \equiv \{x \in \mathcal{X}: f_0(x) = 0 \}$ even if 
%$f_n(x)$ converges to $f_0(x)$ uniformly over $x$. 
For example, take $\sqrt{nh}v_{n,1}(x) = -x^2/n$ on $\mathcal{X} = [-1,1]$.
Let $v_0(x) \equiv 0$. Then $\sqrt{nh}v_{n,1}(x)$ goes to $v_0(x)$ uniformly
in $x \in \mathcal{X}$ as $n \rightarrow \infty$. However, for each $n$, $%
B_{n,1}(0) = \{x \in \mathcal{X}: \sqrt{nh}v_{n,1}(x) = 0\} = \{ 0 \}$,
which does not converge in Hausdorff distance to $B_{1}(0) \equiv \{x \in 
\mathcal{X}: v_0(x) = 0\} = \mathcal{X}$.} While there is no discontinuity
in the finite sample distribution of the test statistic, there may arise
discontinuity in its pointwise asymptotic distribution. Furthermore, the
complexity of the discontinuity makes it harder to trace its source, when we
have $J>2$. As a result, the asymptotic validity of the test that is
established pointwise in $P$ is not a good justification of the test. We
need to establish the asymptotic validity that is \textit{uniform} in $P$
over a reasonable class of probabilities.

Recall that bootstrap critical values based on the least
favorable configuration use a bootstrap quantity such as 
\begin{equation}
\hat{\theta}_{\mathrm{LFC}}^{\ast }\equiv \int_{\mathcal{X}}\left[ \mathbf{%
\hat{s}}^{\ast }(x)\right] _{+}dx,\ \text{where }\mathbf{\hat{s}}^{\ast }(x)=%
\sqrt{nh}\left\{ \frac{\hat{v}_{1}^{\ast }(x)-\hat{v}_{1}(x)}{\hat{\sigma}%
_{1}^{\ast }(x)}\right\} ,  \label{LFC}
\end{equation}%
which can yield  tests that are asymptotically valid \textit{uniformly}
in $P$. However, using a
critical value based on 
\begin{equation*}
\hat{\theta}_{1}^{\ast }\equiv \int_{\hat{B}_{\{1\}}(c_{n})}\left[ \mathbf{%
\hat{s}}^{\ast }(x)\right] _{+}dx
\end{equation*}%
also yields an asymptotically valid test, and yet $\hat{\theta}_{\mathrm{LFC}%
}^{\ast }>\hat{\theta}_{1}^{\ast }$ in general. Thus the bootstrap tests
that use the contact set have better power properties than those that do
not. The power improvement is substantial in many simulation designs and can
be important in real-data applications.

%\footnote{%
%There may exist an alternative approach to improve the power of our test. %
%\citeasnoun{Romano/Shaikh/Wolf:13} proposed a computationally attractive
%two-step method for testing a finite number of unconditional moment
%inequalities. It is an interesting topic to extend their two-step approach
%to our setup, but it is beyond the scope of this paper.}

Now, let us see how the choice of $c_{\alpha ,\eta }^{\ast }\equiv \max
\{c_{\alpha }^{\ast },h^{1/2}\eta +\hat{a}^{\ast }\}$ (with $d=1$ here)
leads to bootstrap inference that is valid even when the test statistic
becomes degenerate under the null hypothesis. The degeneracy arises when the
inequality restrictions hold with large slackness, so that the convergence
in (\ref{as}) holds with $\sigma _{0}^{2}=0$, and hence 
\begin{equation*}
h^{-1/2}(\hat{\theta}-a_{n,1})=o_{P}(1).
\end{equation*}%
For the bootstrap counterpart, note that%
\begin{eqnarray*}
h^{-1/2}(c_{\alpha ,\eta }^{\ast }-a_{n,1}) &=&h^{-1/2}\max \{c_{\alpha
}^{\ast }-a_{n,1},h^{1/2}\eta +\hat{a}^{\ast }-a_{n,1}\} \\
&\geq &\eta +h^{-1/2}(\hat{a}^{\ast }-a_{n,1}),
\end{eqnarray*}%
where it can be shown that $h^{-1/2}(\hat{a}^{\ast }-a_{n,1})=o_{P}(1)$.
Therefore, the bootstrap inference is designed to be asymptotically valid
even when the test statistic becomes degenerate.

Note that for the sake of validity only, one may replace $h^{1/2}\eta $ by a
fixed constant, say $\bar{\eta}>0$. However, this choice would render the
test asymptotically too conservative. The choice of $h^{1/2}\eta $ in this
paper makes the test asymptotically exact for a wide class of probabilities,
while preserving the uniform validity in both the cases of degeneracy and
nondegeneracy.\footnote{%
Our fixed positive constant $\eta $ plays a role similar to a fixed constant
in \citeasnoun{Andrews/Shi:13a}'s modification of the sample
variance-covariance matrix of unconditional moment conditions, transformed
by instruments ($\varepsilon $ in their notation in equation (3.5) of %
\citeasnoun{Andrews/Shi:13a}).} The precise class of probabilities under
which the test becomes asymptotically exact is presented in Section \ref%
{sec:theory}.

There are two remarkable aspects of the local power behavior of our
bootstrap test. First, the test exhibits two different kinds of convergence
rates along different directions of Pitman local alternatives. Second,
despite the fact that the test uses the approach of local smoothing by
kernel as in \citeasnoun{Haerdle/Mammen:93}, the faster of the two
convergence rates achieves a parametric rate of $\sqrt{n}$. To see this more
closely, let us return to the simple example in \eqref{simple-example}, and
consider the following local alternatives:%
\begin{equation}
v_{n}(x)=v_{0}(x)+\frac{\delta (x)}{b_{n}},  \label{loc1}
\end{equation}%
where $v_{0}(x)\leq 0$ for all $x\in \mathcal{X}$ and $\delta (x)>0$ for
some $x\in \mathcal{X}$, and $b_{n}\rightarrow \infty $ as $n\rightarrow
\infty $ such that $v_{n}(x)>0$ for some $x\in \mathcal{X}$. The function $%
\delta (\cdot )$ represents a Pitman direction of the local alternatives.
Suppose that the test has nontrivial local power against local alternatives
of the form in (\ref{loc1}), but trivial power whenever $b_{n}$ in (\ref%
{loc1}) is replaced by $b_{n}^{\prime }$ that diverges faster than $b_{n}$.
In this case, we say that the test has convergence rate equal to $b_{n}$
against the Pitman direction $\delta $.

As we show later, there exist two types of convergence rates of our test,
depending on the choice of $\delta (x)$. Let $B^{0}(0)\equiv \left\{ x\in 
\mathcal{X}:v_{0}(x)=0\right\} $ and $\sigma _{1}^{2}(x)\equiv \mathbf{E}%
[Y_{i}^{2}|X_{i}=x]f(x)\int K^{2}(u)du$. When $\delta (\cdot )$ is such that%
\begin{equation*}
\int_{B^{0}(0)}\frac{\delta (x)}{\sigma _{1}(x)}dx>0,
\end{equation*}%
the test achieves a parametric rate $b_{n}=\sqrt{n}$. On the other hand,
when $\delta (\cdot )$ is such that 
\begin{equation*}
\int_{B^{0}(0)}\frac{\delta (x)}{\sigma _{1}(x)}dx=0\text{ and }%
\int_{B^{0}(0)}\frac{\delta ^{2}(x)}{\sigma _{1}^{2}(x)}dx>0\text{,}
\end{equation*}%
the test achieves a slower rate $b_{n}=\sqrt{n}h^{1/4}$. See Section \ref%
{sec:lpa1} for heuristics behind the results. In Section \ref{sec:lpa2}, the
general form of local power functions is derived.

\section{Uniform Asymptotics under General Conditions}

\label{sec:theory}

In this section, we establish uniform asymptotic validity of our bootstrap
test. We also provide a class of distributions for which the asymptotic size
is exact. We first define the set of distributions we consider.

\begin{defn}
Let $\mathcal{P}$ denote the collection of the potential joint distributions
of the observed random vectors that satisfy Assumptions A\ref{assumption-A1}%
-A\ref{assumption-A6}, and B\ref{assumption-B1}-B\ref{assumption-B4} given
below. Let $\mathcal{P}_{0}\subset \mathcal{P}$ be the sub-collection of
potential distributions that satisfy the null hypothesis.
\end{defn}

Let $||\cdot ||$ denote the Euclidean norm throughout the paper. For any
given sequence of subcollections $\mathcal{P}_{n}\subset \mathcal{P}$, any
sequence of real numbers $b_{n}>0$, and any sequence of random vectors $%
Z_{n} $, we say that $Z_{n}/b_{n}\rightarrow _{P}0,$ $\mathcal{P}_{n}$%
-uniformly, or $Z_{n}=o_{P}(b_{n}),$ $\mathcal{P}_{n}$-uniformly, if for any 
$a>0$, 
\begin{equation*}
\underset{n\rightarrow \infty }{\text{limsup}}\sup_{P\in \mathcal{P}%
_{n}}P\left\{ ||Z_{n}||>ab_{n}\right\} =0\text{.}
\end{equation*}%
Similarly, we say that $Z_{n}=O_{P}(b_{n})$, $\mathcal{P}_{n}$-uniformly, if
for any $a>0$, there exists $M>0$ such that 
\begin{equation*}
\underset{n\rightarrow \infty }{\text{limsup}}\sup_{P\in \mathcal{P}%
_{n}}P\left\{ ||Z_{n}||>Mb_{n}\right\} <a\text{.}
\end{equation*}%
We also define their bootstrap counterparts. Let $P^\ast$ denote the
probability under the bootstrap distribution. For any given sequence of
subcollections $\mathcal{P}_{n}\subset \mathcal{P}$, any sequence of real
numbers $b_{n}>0$, and any sequence of random vectors $Z_{n}^{\ast }$, we
say that $Z_{n}^{\ast }/b_{n}\rightarrow _{P^{\ast }}0,$ $\mathcal{P}_{n}$%
-uniformly, or $Z_{n}^{\ast }=o_{P^{\ast }}(b_{n}),$ $\mathcal{P}_{n}$%
-uniformly, if for any $a>0$,%
\begin{equation*}
\underset{n\rightarrow \infty }{\text{limsup}}\sup_{P\in \mathcal{P}%
_{n}}P\left\{ P^{\ast }\left\{ ||Z_{n}^{\ast }||>ab_{n}\right\} >a\right\} =0%
\text{.}
\end{equation*}%
Similarly, we say that $Z_{n}^{\ast }=O_{P^{\ast }}(b_{n})$, $\mathcal{P}%
_{n} $-uniformly, if for any $a>0$, there exists $M>0$ such that 
\begin{equation*}
\underset{n\rightarrow \infty }{\text{limsup}}\sup_{P\in \mathcal{P}%
_{n}}P\left\{ P^{\ast }\left\{ ||Z_{n}^{\ast }||>Mb_{n}\right\} >a\right\} <a%
\text{.}
\end{equation*}
In particular, when we say $Z_{n}=o_{P}(b_{n})$ or $O_{P}(b_{n}),$ $\mathcal{%
P}$-uniformly, it means that the convergence holds uniformly over $P\in 
\mathcal{P}$, and when we say $Z_{n}=o_{P}(b_{n})$ or $O_{P}(b_{n}),$ $%
\mathcal{P}_{0}$-uniformly, it means that the convergence holds uniformly
over all the probabilities in $\mathcal{P}$ that satisfy the null hypothesis.

\subsection{Test Statistics and Critical Values in General Form}

\label{test-stat-general-form}

First, let us extend the test statistics and the bootstrap procedure to the
general case of $J\geq 1$. Let $\Lambda _{p}:\mathbf{R}^{J}\rightarrow
\lbrack 0,\infty )$ be a nonnegative, increasing function indexed by $p$ such that $1 \leq p < \infty$. 
While the theory of this paper can be extended to various general forms
of map $\Lambda _{p}$, we focus on the following type:%
\begin{equation}
\Lambda _{p}(v_{1},\cdot \cdot \cdot ,v_{J})=\left( \max \{[v_{1}]_{+},\cdot
\cdot \cdot ,[v_{J}]_{+}\}\right) ^{p}\text{ or }\Lambda _{p}(v_{1},\cdot
\cdot \cdot ,v_{J})=\sum_{j=1}^{J}[v_{j}]_{+}^{p},  \label{lambda_p}
\end{equation}%
where for $a\in \mathbf{R}$, $[a]_{+}=\max \{a,0\}$. The test statistic is
defined as%
\begin{equation*}
\hat{\theta}=\int_{\mathcal{X}\times \mathcal{T}}\Lambda _{p}\left( \hat{u}%
_{\tau ,1}(x),\cdot \cdot \cdot ,\hat{u}_{\tau ,J}(x)\right) dQ(x,\tau ).
\end{equation*}%
To motivate our bootstrap procedure, it is convenient to begin with the
following lemma. Let us introduce some notation. Define $\mathcal{N}%
_{J}\equiv 2^{\mathbb{N}_{J}}\backslash \{\varnothing \}$, i.e., the
collection of all the nonempty subsets of $\mathbb{N}_{J}\equiv \{1,2,\cdot
\cdot \cdot ,J\}$. For any $A\in \mathcal{N}_{J}$ and $\mathbf{v}%
=(v_{1},\cdot \cdot \cdot ,v_{J})^{\top }\in \mathbf{R}^{J}$, we define $%
\mathbf{v}_{A}$ to be $\mathbf{v}$ except that for each $j\in \mathbb{N}%
_{J}\backslash A$, the $j$-th entry of $\mathbf{v}_{A}$ is zero, and let%
\begin{equation}
\Lambda _{A,p}(\mathbf{v})\equiv \Lambda _{p}(\mathbf{v}_{A}).
\label{lamb_A}
\end{equation}%
That is, $\Lambda _{A,p}(\mathbf{v})$ is a \textquotedblleft
censoring\textquotedblright\ of $\Lambda _{p}(\mathbf{v})$ outside the index
set $A$. Now, we define a general version of contact sets: for $A\in 
\mathcal{N}_{J}$ and for $c_{n,1},c_{n,2}>0,$%
\begin{equation}
B_{n,A}(c_{n,1},c_{n,2})\equiv \left\{ (x,\tau )\in \mathcal{X}\times 
\mathcal{T}:%
\begin{array}{ll}
|r_{n,j}v_{n,\tau ,j}(x)/\sigma _{n,\tau ,j}(x)|\ \leq c_{n,1}, & \text{ for
all }j\in A \\ 
r_{n,j}v_{n,\tau ,j}(x)/\sigma _{n,\tau ,j}(x)\ <-c_{n,2}, & \text{ for all }%
j\in \mathbb{N}_{J}/A%
\end{array}%
\right\} ,  \label{BA0}
\end{equation}%
where $\sigma _{n,\tau ,j}(x)$ is a ``population'' version of $\hat{\sigma}%
_{\tau ,j}(x)$ (see e.g. Assumption A5 below.) When $c_{n,1}=c_{n,2}=c_{n}$
for some $c_{n}>0$, we write $B_{n,A}(c_{n})=B_{n,A}(c_{n,1},c_{n,2})$.

\begin{lem}
\label{nice-lemma} Suppose that Assumptions A1-A3 and A4(i) in Section \ref%
{subsec:high-level} hold. Suppose further that $c_{n,1}>0$ and $c_{n,2}>0$
are sequences such that 
\begin{equation*}
\sqrt{\log n}\{c_{n,1}^{-1}+c_{n,2}^{-1}\}\rightarrow 0,
\end{equation*}%
as $n\rightarrow \infty $. Then as $n\rightarrow \infty $,%
\begin{equation*}
\inf_{P\in \mathcal{P}_{0}}P\left\{ \hat{\theta}=\dsum\limits_{A\in \mathcal{%
N}_{J}}\int_{B_{n,A}(c_{n,1},c_{n,2})}\Lambda _{A,p}(\hat{u}_{\tau
,1}(x),\cdot \cdot \cdot ,\hat{u}_{\tau ,J}(x))dQ(x,\tau )\right\}
\rightarrow 1,
\end{equation*}%
where $\mathcal{P}_{0}$ is the set of potential distributions of the
observed random vector under the null hypothesis.
\end{lem}

The lemma above shows that the test statistic $\hat{\theta}$ is uniformly
approximated by the integral with domain restricted to the contact sets $%
B_{n,A}(c_{n,1},c_{n,2})$ in large samples. Note that the result of Lemma 1
implies that the approximation error between $\hat{\theta}$ and the
expression on the right-hand side is $o_{P}(\varepsilon _{n})$ for any $%
\varepsilon _{n}\rightarrow 0$, thereby suggesting that one may consider a
bootstrap procedure that mimics the representation of $\hat{\theta}$ in
Lemma 1.

We begin by introducing a sample version of the contact sets. For $A\in 
\mathcal{N}_{J},$%
\begin{equation*}
\hat{B}_{A}(\hat{c}_{n})\equiv \left\{ (x,\tau )\in \mathcal{X}\times 
\mathcal{T}:%
\begin{array}{ll}
|r_{n,j}\hat{v}_{\tau ,j}(x)/\hat{\sigma}_{\tau ,j}(x)|\ \leq \hat{c}_{n}, & 
\text{ for all }j\in A \\ 
r_{n,j}\hat{v}_{\tau ,j}(x)/\hat{\sigma}_{\tau ,j}(x)\ <-\hat{c}_{n}, & 
\text{ for all }j\in \mathbb{N}_{J}\backslash A%
\end{array}%
\right\} .
\end{equation*}%
The explicit condition for $\hat{c}_{n}$ is found in Assumption A4 below.
Given the bootstrap counterparts, $\{[\hat{v}_{\tau ,j}^{\ast }(x),\hat{%
\sigma}_{\tau ,j}^{\ast }(x)]:j\in \mathbb{N}_{J}\},$ of $\{[\hat{v}_{\tau
,j}(x),\hat{\sigma}_{\tau ,j}(x)]:j\in \mathbb{N}_{J}\}$, we define our
bootstrap test statistic as follows:%
\begin{equation*}
\hat{\theta}^{\ast }\equiv \sum_{A\in \mathcal{N}_{J}}\int_{\hat{B}_{A}(\hat{%
c}_{n})}\Lambda _{A,p}(\hat{s}_{\tau ,1}^{\ast }(x),\cdot \cdot \cdot ,\hat{s%
}_{\tau ,J}^{\ast }(x))dQ(x,\tau ),
\end{equation*}%
where for $j\in \mathbb{N}_{J}$, $\hat{s}_{\tau ,j}^{\ast }(x)\equiv r_{n,j}(%
\hat{v}_{\tau ,j}^{\ast }(x)-\hat{v}_{\tau ,j}(x))/\hat{\sigma}_{\tau
,j}^{\ast }(x)$. We also define 
\begin{equation*}
\hat{a}^{\ast }\equiv \sum_{A\in \mathcal{N}_{J}}\int_{\hat{B}_{A}(\hat{c}%
_{n})}\mathbf{E}^{\ast }\Lambda _{A,p}(\hat{s}_{\tau ,1}^{\ast }(x),\cdot
\cdot \cdot ,\hat{s}_{\tau ,J}^{\ast }(x))dQ(x,\tau ).
\end{equation*}%
Let $c_{\alpha }^{\ast }$ be the $(1-\alpha )$-th quantile from the
bootstrap distribution of $\hat{\theta}^{\ast }$ and take 
\begin{equation*}
c_{\alpha ,\eta }^{\ast }=\max \{c_{\alpha }^{\ast },h^{d/2}\eta +\hat{a}%
^{\ast }\}
\end{equation*}%
as our critical value, where $\eta >0$ is a small fixed number.

One of the main technical contributions of this paper is to present precise
conditions under which this proposal of bootstrap test works. We present and
discuss them in subsequent sections.

To see the intuition for the bootstrap validity, first note that the uniform
convergence of $r_{n,j}\{\hat{v}_{\tau ,j}(x)-v_{n,\tau ,j}(x)\}$ over $%
(x,\tau )$ implies that 
\begin{equation}
B_{n,A}(c_{n,L},c_{n,U})\subset \hat{B}_{A}(\hat{c}_{n})\subset
B_{n,A}(c_{n,U},c_{n,L})  \label{contact}
\end{equation}%
with probability approaching one, whenever $P\left\{ c_{n,L}\leq \hat{c}%
_{n}\leq c_{n,U}\right\} \rightarrow 1$. Therefore, if $\sqrt{\log n}%
/c_{n,L}\rightarrow 0$, then, (letting $\hat{s}_{\tau ,j}\equiv r_{n,j}(\hat{%
v}_{\tau ,j}(x)-v_{n,\tau ,j}(x))/\hat{\sigma}_{\tau ,j}(x)$), we have%
\begin{equation}
\hat{\theta}\leq \sum_{A\in \mathcal{N}_{J}}\int_{B_{n,A}(c_{n,L},c_{n,U})}%
\Lambda _{A,p}\left( \hat{s}_{\tau ,1}(x),\cdot \cdot \cdot ,\hat{s}_{\tau
,J}(x)\right) dQ(x,\tau )\text{,}  \label{heuristics}
\end{equation}%
with probability approaching one, by Lemma 1 and the null hypothesis. When
the last sum has a nondegenerate limit, we can approximate its distribution
by the bootstrap distribution%
\begin{eqnarray*}
&&\sum_{A\in \mathcal{N}_{J}}\int_{B_{n,A}(c_{n,L},c_{n,U})}\Lambda
_{A,p}\left( \hat{s}_{\tau ,1}^{\ast }(x),\cdot \cdot \cdot ,\hat{s}_{\tau
,J}^{\ast }(x)\right) dQ(x,\tau ) \\
&\leq &\sum_{A\in \mathcal{N}_{J}}\int_{\hat{B}_{A}(\hat{c}_{n})}\Lambda
_{A,p}\left( \hat{s}_{\tau ,1}^{\ast }(x),\cdot \cdot \cdot ,\hat{s}_{\tau
,J}^{\ast }(x)\right) dQ(x,\tau )\equiv \hat{\theta}^{\ast },
\end{eqnarray*}%
where the inequality follows from (\ref{contact}).\footnote{%
In fact, the main challenge here is to prove the bootstrap approximation
using the method of Poissonization that is uniform in $P\in \mathcal{P}_{0}$.%
} Thus the critical value is read from the bootstrap distribution of $\hat{%
\theta}^{\ast }$. On the other hand, if the last sum in (\ref{heuristics})
has limiting distribution degenerate at zero, we simply take a small
positive number $\eta $ to control the size of the test. This results in our
choice of $c_{\alpha ,\eta }^{\ast }=\max \{c_{\alpha }^{\ast },h^{d/2}\eta +%
\hat{a}^{\ast }\}$.

\subsection{Assumptions}

\label{subsec:high-level}

In this section, we provide assumptions needed to develop general results.
We assume that $\mathcal{S}\equiv \mathcal{X}\times \mathcal{T}$ is a
compact subset of a Euclidean space. We begin with the following assumption.

\begin{AssumptionA}
\label{assumption-A1} (Asymptotic Linear Representation) For each $j\in 
\mathbb{N}_{J}\equiv \{1,\cdot \cdot \cdot ,J\},$ there exists a
nonstochastic function $v_{n,\tau ,j}(\cdot ):\mathbf{R}^{d}\rightarrow 
\mathbf{R}$ such that (a) $v_{n,\tau ,j}(x)\leq 0$ for all $(x,\tau )\in 
\mathcal{S}$ under the null hypothesis, and (b) as $n\rightarrow \infty $, 
\begin{equation}
\sup_{(x,\tau )\in \mathcal{S}}\left\vert r_{n,j}\left\{ \frac{\hat{v}_{\tau
,j}(x)-v_{n,\tau ,j}(x)}{\hat{\sigma}_{\tau ,j}(x)}\right\} -\sqrt{nh^{d}}\{%
\hat{g}_{\tau ,j}(x)-\mathbf{E}\hat{g}_{\tau ,j}(x)\}\right\vert =o_{P}(%
\sqrt{h^{d}}),\text{ }\mathcal{P}\text{-uniformly,}  \label{AL2}
\end{equation}%
where, with $\{(Y_{i}^{\top },X_{i}^{\top })\}_{i=1}^{n}$ being a random
sample such that $Y_{i}=(Y_{i1}^{\top },\ldots ,Y_{iJ}^{\top })^{\top }\in 
\mathbf{R}^{J\bar{L}}$, $Y_{ij}\in \mathbf{R}^{\bar{L}}$, $X_{i}\in \mathbf{R%
}^{d}$, and the distribution of $X_{i}$ is absolutely continuous with
respect to Lebesgue measure,\footnote{%
Throughout the paper, we assume that $X_{i}\in \mathbf{R}^{d}$ is a
continuous random vector. It is straightforward to extend the analysis to
the case where $X_{i}$ has a subvector of discrete random variables.} we
define%
\begin{equation*}
\hat{g}_{\tau ,j}(x)\equiv \frac{1}{nh^{d}}\sum_{i=1}^{n}\beta _{n,x,\tau
,j}\left( Y_{ij},\frac{X_{i}-x}{h}\right) ,
\end{equation*}%
and $\beta _{n,x,\tau ,j}:\mathbf{R}^{\bar{L}}\times \mathbf{R}%
^{d}\rightarrow \mathbf{R}$ is a function which may depend on $n\geq 1$.
\end{AssumptionA}

Assumption A\ref{assumption-A1} requires that there exist a nonparametric
function $v_{n,\tau ,j}(x)$ around which the asymptotic linear
representation holds uniformly in $P\in \mathcal{P}$, and $v_{n,\tau
,j}(x)\leq 0$ under the null hypothesis. The required rate of convergence in %
\eqref{AL2} is $o_{P}(h^{d/2})$ instead of $o_{P}(1)$. We need this stronger
convergence rate primarily because $\hat{\theta}-a_{n}$ is $O_{P}(h^{d/2})$
for some nonstochastic sequence $a_{n}$.\footnote{%
To see this more clearly, we assume that $\mathcal{T}=\{\tau \}$, $p=1$, and 
$J=1$, and suppress the subscripts $\tau $ and $j$ from the notation, and
take $\hat{\sigma}(x)=1$ for simplicity. We write (in the case where $%
v_{n}(x)=0$) 
\begin{eqnarray*}
h^{-d/2}\hat{\theta} &=&h^{-d/2}\int_{\mathcal{X}}\max \left\{ r_{n}\{\hat{v}%
(x)-v_{n}(x)\},0\right\} dx \\
&=&h^{-d/2}\int_{\mathcal{X}}\max \left\{ \sqrt{nh^{d}}\{\hat{g}(x)-\mathbf{E%
}\hat{g}(x)\},0\right\} dx+h^{-d/2}R_{n},
\end{eqnarray*}%
where $R_{n}$ is an error term that has at least the same convergence rate
as the convergence rate of the remainder term in the asymptotic linear
representation for $\hat{v}(x)$. Now we let 
\begin{equation*}
a_{n}=\mathbf{E}\left[ \int_{\mathcal{X}}\max \left\{ \sqrt{nh^{d}}\{\hat{g}%
(x)-\mathbf{E}\hat{g}(x)\},0\right\} dx\right]
\end{equation*}%
and write $h^{-d/2}\hat{\theta}-h^{-d/2}a_{n}$ as%
\begin{equation*}
h^{-d/2}\left( \int_{\mathcal{X}}\max \left\{ \sqrt{nh^{d}}\{\hat{g}(x)-%
\mathbf{E}\hat{g}(x)\},0\right\} dx-a_{n}\right) +h^{-d/2}R_{n}.
\end{equation*}%
It can be shown that the leading term is asymptotically normal using the
method of Poissonization. Hence $h^{-d/2}\hat{\theta}-h^{-d/2}a_{n}$ becomes
asymptotically normal, if $R_{n}=o_{P}(h^{d/2})$. This is where the faster
error rate in the asymptotic linear representation in Assumption A1(i) plays
a role.} 
%As we shall see in a later section devoted to examples, the asymptotic
%linear representation with the required convergence rate can be established
%using standard arguments.

When $\hat{v}_{\tau ,j}(x)$ is a sample mean of i.i.d. random quantities
involving nonnegative kernels and $\hat{\sigma}_{n,\tau }(x)=1$, we may take 
$v_{n,\tau ,j}(x)=\mathbf{E}\hat{v}_{\tau ,j}(x)$, and then $o_{P}(\sqrt{%
h^{d}})$ is in fact precisely equal to $0$. If the original nonparametric
function $v_{\tau ,j}(\cdot )$ satisfies some smoothness conditions, we may
take $v_{n,\tau ,j}(x)=v_{\tau ,j}(x)$, and handle the bias part $\mathbf{E}%
\hat{v}_{\tau ,j}(x)-v_{\tau ,j}(x)$ using the standard arguments to deduce
the error rate $o_{P}(\sqrt{h^{d}})$. Assumption A1 admits both set-ups. For
instance, consider the simple example in Section \ref{subsec:discontinuity}.
The asymptotic linear representation in Assumption 1 can be shown to hold
with 
\begin{equation*}
\beta _{n,x,1}\left( Y_{i},(X_{i}-x)/h\right) =Y_{i}K((X_{i}-x)/h)/\sigma
_{n,1}(x),
\end{equation*}%
where $\sigma _{n,1}^{2}(x)=\mathbf{E}[Y_{i}^{2}K^{2}((X_{i}-x)/h)]/h$, if $%
\hat{\sigma}_{n,1}(x)$ is chosen as in (\ref{v&s}).

The following assumption for $\beta _{n,x,\tau ,j}$ essentially defines the
scope of this paper's framework.

\begin{AssumptionA}
\label{assumption-A2}(Kernel-Type Condition) For some compact $\mathcal{K}%
_{0}\subset \mathbf{R}^{d}$ that does not depend on $P\in \mathcal{P}$ or $n$%
, it is satisfied that $\beta _{n,x,\tau ,j}(y,u)=0$ for all $u\in \mathbf{R}%
^{d}\backslash \mathcal{K}_{0}$ and all $(x,\tau ,y)\in \mathcal{X}\times 
\mathcal{T}\times \mathcal{Y}_{j}$ and all $j\in \mathbb{N}_{J}$, where $%
\mathcal{Y}_{j}$ denotes the support of $Y_{ij}$.
\end{AssumptionA}

Assumption A\ref{assumption-A2} can be immediately verified when the
asymptotic linear representation in (\ref{AL2}) is established. This
condition is satisfied in particular when the asymptotic linear
representation involves a multivariate kernel function with bounded support
in a multiplicative form. In such a case, the set $\mathcal{K}_{0}$ depends
only on the choice of the kernel function, not on any model primitives.

\begin{AssumptionA}
\label{assumption-A3}(Uniform Convergence Rate for Nonparametric Estimators)
For all $j\in \mathbb{N}_{J}$,%
\begin{equation*}
\sup_{(x,\tau )\in \mathcal{S}}r_{n,j}\left\vert \frac{\hat{v}_{\tau
,j}(x)-v_{n,\tau ,j}(x)}{\hat{\sigma}_{\tau ,j}(x)}\right\vert =O_{P}\left( 
\sqrt{\log n}\right) \text{, }\mathcal{P}\text{-uniformly.}
\end{equation*}
\end{AssumptionA}

Assumption A\ref{assumption-A3} (in combination with A\ref{assumption-A5}
below) requires that $\hat{v}_{\tau ,j}(x)-v_{n,\tau ,j}(x)$ have the
uniform convergence rate of $O_{P}(r_{n,j}^{-1}\sqrt{\log n})$ uniformly
over $P\in \mathcal{P}$. Lemma 2 in Section \ref{sec:suff-A3-B2} provides
some sufficient conditions for this convergence.

We now introduce conditions for the bandwidth $h$ and the tuning parameter $%
c_{n}$ for the contact sets.

\begin{AssumptionA}
\label{assumption-A4}(Rate Conditions for Tuning Parameters) (i) As $%
n\rightarrow \infty $, $h\rightarrow 0$, $\sqrt{\log n}/r_{n}\rightarrow 0$,
and $n^{-1/2}h^{-d-\nu _{1}}\rightarrow 0$ for some arbitrarily small $\nu
_{1}>0$, where $r_{n}\equiv \min_{j\in \mathbb{N}_{J}}r_{n,j}$.\newline
(ii) For each $n\geq 1,$ there exist nonstochastic sequences $c_{n,L}>0$ and 
$c_{n,U}>0 $ such that $c_{n,L}\leq c_{n,U}$, and%
\begin{equation*}
\inf_{P\in \mathcal{P}}P\left\{ c_{n,L}\leq \hat{c}_{n}\leq c_{n,U}\right\}
\rightarrow 1,\text{ and\ }\sqrt{\log n}/c_{n,L}+c_{n,U}/r_{n}\rightarrow 0,
\end{equation*}%
as $n\rightarrow \infty $.
\end{AssumptionA}

The requirement that $\sqrt{\log n}/r_{n}\rightarrow 0$ is satisfied easily
for most cases where $r_{n}$ increases at a polynomial order in $n$.
Assumption A4(ii) requires that $\hat{c}_{n}$ increase faster than $\sqrt{%
\log n}$ but slower than $r_{n}$ with probability approaching one.

\begin{AssumptionA}
\label{assumption-A5}(Regularity Conditions for $\hat{\sigma}_{\tau ,j}(x)$)
For each $(\tau ,j)\in \mathcal{T}\times \mathbb{N}_{J}$, there exists $%
\sigma _{n,\tau ,j}(\cdot ):\mathcal{X}\rightarrow (0,\infty )$ such that $%
\liminf_{n\rightarrow \infty }\inf_{(x,\tau )\in \mathcal{S}}\inf_{P\in 
\mathcal{P}}\sigma _{n,\tau ,j}(x)>0$, and%
\begin{equation*}
\sup_{(x,\tau )\in \mathcal{S}}\left\vert \hat{\sigma}_{\tau ,j}(x)-\sigma
_{n,\tau ,j}(x)\right\vert =o_{P}(1)\text{, }\mathcal{P}\text{-uniformly.}
\end{equation*}
\end{AssumptionA}

Assumption A\ref{assumption-A5} requires that the scale normalization $\hat{%
\sigma}_{\tau ,j}(x)$ should be asymptotically well defined. The condition
precludes the case where estimator $\hat{\sigma}_{\tau ,j}(x)$ converges to
a map that becomes zero at some point $(x,\tau )$ in $\mathcal{S}$.
Assumption A\ref{assumption-A5} is usually satisfied by an appropriate
choice of $\hat{\sigma}_{\tau ,j}(x)$. When one chooses $\hat{\sigma}_{\tau
,j}(x)=1$, which is permitted in our framework, Assumption A5 is immediately
satisfied with $\sigma _{n,\tau ,j}(x)=1$. Again, if we go back to the
simple example considered in Section \ref{subsec:discontinuity}, it is
straightforward to see that under regularity conditions, with the subscript $%
\tau $ suppressed, $\hat{\sigma}_{1}^{2}(x)=\sigma _{n,1}^{2}(x)+o_{P}(1)$
and $\sigma _{n,1}^{2}(x)=\sigma _{1}^{2}(x)+o(1)$, where $\sigma
_{1}^{2}(x)\equiv \mathbf{E}(Y^{2}|X=x)f(x)\int K^{2}(u)du,\ $as $%
n\rightarrow \infty $. The convergence can be strengthened to a uniform
convergence when $\sigma _{1}^{2}(x)$ is bounded away from zero uniformly
over $x\in \mathcal{X}$ and $P\in \mathcal{P}$, so that Assumption A\ref%
{assumption-A5} holds.

We introduce assumptions about the moment conditions for $\beta _{n,x,\tau
,j}(\cdot ,\cdot )$ and other regularity conditions. For $\tau \in \mathcal{T%
}$ and $\varepsilon _{1}>0,$ let $\mathcal{S}_{\tau }(\varepsilon
_{1})\equiv \{x+a:x\in \mathcal{S}_{\tau }$, $a\in \lbrack -\varepsilon
_{1},\varepsilon _{1}]^{d}\}$, where $\mathcal{S}_{\tau }\equiv \{x\in 
\mathcal{X}:(x,\tau )\in \mathcal{S}\}$ for each $\tau \in \mathcal{T}$. Let 
$\mathcal{U}\equiv \mathcal{K}_{0}+\mathcal{K}_{0}$ such that $\mathcal{U}$
contains $\{0\}$ in its interior and $\mathcal{K}_{0}$ is the same as
Assumption \ref{assumption-A2}. Here, $+$ denotes the Minkowski sum of sets.

\begin{AssumptionA}
\label{assumption-A6} (i) There exist $M\geq 2(p+2),\ C>0,$ and $\varepsilon
_{1}>0$ such that%
\begin{equation*}
\mathbf{E}[|\beta _{n,x,\tau ,j}\left( Y_{ij},u\right)
|^{M}|X_{i}=x]f(x)\leq C,
\end{equation*}%
for all $(x,u)\in \mathcal{S}_{\tau }(\varepsilon _{1})\times \mathcal{U}$, $%
\tau \in \mathcal{T}$, $j\in \mathbb{N}_{J},$ $n\geq 1$, and $P\in \mathcal{P%
}$, where $f(\cdot )$ is the density of $X_{i}$.\footnote{%
The conditional expectation $\mathbf{E}_{P}\left[ |\beta _{n,x,\tau
,j}\left( Y_{ij},u\right) |^{M}|X_{i}=x\right] $ is of type $\mathbf{E}\left[
f(Y,x)|X=x\right] $, which is not well defined according to Kolmogorov's
definition of conditional expectations. See, e.g. %
\citeasnoun{Proschan/Presnell:98} for this problem. Here we define the
conditional expectation in an elementary way by using conditional densities
or conditional probability mass functions of $(Y_{ij},Y_{ik})$ given $%
X_{i}=x $, depending on whether $(Y_{ij},Y_{ik})$ is continuous or discrete.}

\noindent (ii) For each $a\in (0,1/2)$, there exists a compact set $\mathcal{%
C}_{a}\subset \mathbf{R}^{d}$ such that 
\begin{equation*}
0<\inf_{P\in \mathcal{P}}P\{X_{i}\in \mathbf{R}^{d}\backslash \mathcal{C}%
_{a}\}\leq \sup_{P\in \mathcal{P}}P\{X_{i}\in \mathbf{R}^{d}\backslash 
\mathcal{C}_{a}\}<a.
\end{equation*}
\end{AssumptionA}

Assumption A\ref{assumption-A6}(i) requires that conditional moments of $%
\beta _{n,x,\tau ,j}\left( Y_{ij},z\right) $ be bounded. Assumption A\ref%
{assumption-A6}(ii) is a technical condition for the distribution of $X_{i}$%
. The third inequality in Assumption A\ref{assumption-A6}(ii) is satisfied
if the distribution of $X_{i}$ is uniformly tight in $\mathcal{P}$, and
follows, for example, if sup$_{P\in \mathcal{P}}\mathbf{E}||X_{i}||<\infty $%
. The first inequality in Assumption A\ref{assumption-A6}(ii) requires that
there be a common compact set outside which the distribution of $X_{i}$
still has positive probability mass uniformly over $P\in \mathcal{P}$. The
main thrust of Assumption A\ref{assumption-A6}(ii) lies in the requirement
that such a compact set be independent of $P\in \mathcal{P}$. While it is
necessary to make this technical condition explicit as stated here, the
condition itself appears very weak.

%The condition, for example, even allows
%for a situation where some supports of $X_{i}$ with different $P$'s in $%
%\mathcal{P}$ do not overlap at all.

This paper's asymptotic analysis adopts the approach of Poissonization (see,
e.g., \citeasnoun{Horvath:91} and \citeasnoun{GMZ}). However, existing
methods of Poissonization are not readily applicable to our testing problem,
mainly due to the possibility of local or global redundancy among the
nonparametric functions. In particular, the conditional covariance matrix of 
$\beta _{n,x,\tau ,j}(Y_{ij},u)$'s across different $(x,\tau ,j)$'s given $%
X_{i}$ can be singular in the limit. Since the empirical researcher rarely
knows \textit{a priori} the local relations among nonparametric functions,
it is important that the validity of the test is not sensitive to the local
relations among them, i.e., the validity should be uniform in $P$.

This paper deals with this challenge in three steps. First, we introduce a
Poissonized version of the test statistic and apply a certain form of
regularization to facilitate the derivation of its limiting distribution
uniformly in $P\in \mathcal{P}$, i.e., regardless of singularity or
degeneracy in the original test statistic. Second, we use a
Berry-Esseen-type bound to compute the finite sample influence of the
regularization bias and let the regularization parameter go to zero
carefully, so that the bias disappears in the limit. Third, we translate
thus computed limiting distribution into that of the original test
statistic, using so-called de-Poissonization lemma. This is how the
uniformity issue in this complex situation is covered through the
Poissonization method combined with the method of regularization.

\subsection{Asymptotic Validity of Bootstrap Procedure}

Recall that $\mathbf{E}^{\ast }$ and $P^{\ast }$ denote the expectation and
the probability under the bootstrap distribution. We make the following
assumptions for $\hat{v}_{\tau ,j}^{\ast }(x)$.

%Again, this paper provides later low level conditions that ensure
%Assumption B1.

\begin{AssumptionB}
\label{assumption-B1}(Bootstrap Asymptotic Linear Representation) For each $%
j\in \mathbb{N}_{J}$,%
\begin{equation*}
\sup_{(x,\tau )\in \mathcal{S}}\left\vert r_{n,j}\left\{ \frac{\hat{v}_{\tau
,j}^{\ast }(x)-\hat{v}_{\tau ,j}(x)}{\hat{\sigma}_{\tau ,j}^{\ast }(x)}%
\right\} -\sqrt{nh^{d}}\{\hat{g}_{\tau ,j}^{\ast }(x)-\mathbf{E}^{\ast }\hat{%
g}_{\tau ,j}^{\ast }(x)\}\right\vert =o_{P^{\ast }}(\sqrt{h^{d}}),\ \mathcal{%
P}\text{-uniformly,}
\end{equation*}%
where 
\begin{equation*}
\hat{g}_{\tau ,j}^{\ast }(x)\equiv \frac{1}{nh^{d}}\sum_{i=1}^{n}\beta
_{n,x,\tau ,j}\left( Y_{ij}^{\ast },\frac{X_{i}^{\ast }-x}{h}\right) ,
\end{equation*}%
and $\beta _{n,x,\tau ,j}$ is a real valued function introduced in
Assumption A1. \newline
\end{AssumptionB}

\begin{AssumptionB}
\label{assumption-B2} For all $j\in \mathbb{N}_{J},$%
\begin{equation*}
\sup_{(x,\tau )\in \mathcal{S}}r_{n,j}\left\vert \frac{\hat{v}_{\tau
,j}^{\ast }(x)-\hat{v}_{\tau ,j}(x)}{\hat{\sigma}_{\tau ,j}^{\ast }(x)}%
\right\vert =O_{P^{\ast }}(\sqrt{\log n}),\ \mathcal{P}\text{-uniformly}.
\end{equation*}
\end{AssumptionB}

\begin{AssumptionB}
\label{assumption-B3} For all $j\in \mathbb{N}_{J},$%
\begin{equation*}
\sup_{(x,\tau )\in \mathcal{S}}\left\vert \hat{\sigma}_{\tau ,j}^{\ast }(x)-%
\hat{\sigma}_{\tau ,j}(x)\right\vert =o_{P^{\ast }}(1),\ \mathcal{P}\text{%
-uniformly.}
\end{equation*}
\end{AssumptionB}

Assumption B\ref{assumption-B1} is the asymptotic linear representation of
the bootstrap estimator $\hat{v}_{\tau ,j}^{\ast }(x)$. The proof of the
asymptotic linear representation can be typically proceeded in a similar way
that one obtains the original asymptotic linear representation in Assumption
A1. Assumptions B2 and B3 are the bootstrap versions of Assumptions A\ref%
{assumption-A3} and A\ref{assumption-A5}.

\begin{AssumptionB}
\label{assumption-B4} (Bandwidth Condition) $n^{-1/2}h^{-\left( \frac{3M-4}{%
2M-4}\right) d-\nu _{2}}\rightarrow 0$ as $n\rightarrow \infty $, for some
small $\nu _{2}>0$ and for $M>0$ that appears in Assumption A\ref%
{assumption-A6}(i).
\end{AssumptionB}

When $\beta _{n,x,\tau ,j}\left( Y_{ij},u\right) $ is bounded uniformly over 
$(n,x,\tau ,j),$ the bandwidth condition in Assumption B\ref{assumption-B4}
can be reduced to $n^{-1/2}h^{-3d/2-\nu _{2}}\rightarrow 0$. If Assumption A%
\ref{assumption-A6}(i) holds with $M = 6$ and $p=1$, the bandwidth condition
in Assumption B\ref{assumption-B4} is reduced to $n^{-1/2}h^{-7d/4-\nu
_{2}}\rightarrow 0$.

Note that Assumption B\ref{assumption-B4} is stronger than the bandwidth
condition in Assumption A\ref{assumption-A4}(i). The main reason is that we
need to prove that for some $a_{\infty }>0$, we have $a_{n}=a_{\infty
}+o(h^{d/2})$ and $a_{n}^{\ast }=a_{\infty }+o_{P}(h^{d/2}), \; \mathcal{P}%
\text{-uniformly}$, where $a_{n}$ is an appropriate location normalizer of
the test statistic, and $a_{n}^{\ast }$ is a bootstrap counterpart of $a_{n}$%
. To show these, we utilize a Berry-Esseen-type bound for a nonlinear
transform of independent sum of random variables. Since the approximation
error depends on the moment bounds for the sum, the bandwidth condition in
Assumption B\ref{assumption-B4} takes a form that involves $M>0$ in
Assumption A\ref{assumption-A6}.

%This latter
%condition is slightly stronger than that of \citeasnoun{LSW} who assumed $%
%n^{-1/2}h^{-3d/2}\rightarrow 0$ in testing inequality restrictions for
%nonparametric regression functions.

We now present the result of the uniform validity of our bootstrap test.

\begin{thm}
\label{Thm1} Suppose that Assumptions A\ref{assumption-A1}-A\ref%
{assumption-A6} and B\ref{assumption-B1}-B\ref{assumption-B4} hold. Then 
\begin{equation*}
\underset{n\rightarrow \infty }{\mathrm{limsup}}\sup_{P\in \mathcal{P}%
_{0}}P\{\hat{\theta}>c_{\alpha, \eta }^{\ast }\}\leq \alpha .
\end{equation*}
\end{thm}

One might ask whether the bootstrap test $1\{\hat{\theta}>c_{\alpha, \eta
}^{\ast }\}$ can be asymptotically exact, i.e., whether the inequality in
Theorem \ref{Thm1} can hold as an equality. As we show below, the answer is
affirmative. The remaining issue is a precise formulation of a subset of $%
\mathcal{P}_{0}$ such that the rejection probability of the bootstrap test
achieves the level $\alpha $ asymptotically, uniformly over the subset.

To see when the test will have asymptotically exact size, we apply Lemma 1
to find that with probability approaching one,%
\begin{equation*}
\hat{\theta}=\sum_{A\in \mathcal{N}_{J}}\int_{B_{n,A}(c_{n,U},c_{n,L})}%
\Lambda _{A,p}\left( \mathbf{\hat{s}}_{\tau }(x)+\mathbf{u}_{n,\tau }(x;\hat{%
\sigma})\right) dQ(x,\tau ),
\end{equation*}%
where $\mathbf{\hat{s}}_{\tau }(x)\equiv \left[ r_{n,j}\{\hat{v}_{n,\tau
,j}(x)-v_{n,\tau ,j}(x)\}/\hat{\sigma}_{\tau ,j}(x)\right] _{j=1}^{J}$, and $%
\mathbf{u}_{n,\tau }(x;\hat{\sigma})\equiv \left[ r_{n,j}v_{n,\tau ,j}(x)/%
\hat{\sigma}_{\tau ,j}(x)\right] _{j=1}^{J},$ and $c_{n,U}>0$ and $c_{n,L}>0$
are nonstochastic sequences that satisfy Assumption A4(ii).
%\footnote{%
%Note that we use Lemma 1 with $B_{n,A}(c_{n,U},c_{n,L})$ here, differently
%from \eqref{heuristics}. This is because for asymptotic exactness, we need
%to use different arguments. See the roadmap of Appendix A for detailed
%explanations.} 
 We fix a positive sequence $q_{n}\rightarrow 0$, and write
the right hand side as 
\begin{eqnarray}
&&\sum_{A\in \mathcal{N}_{J}}\int_{B_{n,A}(q_{n})}\Lambda _{A,p}\left( 
\mathbf{\hat{s}}_{\tau }(x)+\mathbf{u}_{n,\tau }(x;\hat{\sigma})\right)
dQ(x,\tau )  \label{dec} \\
&&+\sum_{A\in \mathcal{N}_{J}}\int_{B_{n,A}(c_{n,U},c_{n,L})\backslash
B_{n,A}(q_{n})}\Lambda _{A,p}\left( \mathbf{\hat{s}}_{\tau }(x)+\mathbf{u}%
_{n,\tau }(x;\hat{\sigma})\right) dQ(x,\tau ).  \notag
\end{eqnarray}%
Under the null hypothesis, we have $v_{n,\tau ,j}(x)\leq 0$, and hence the
last sum is bounded by 
\begin{equation*}
\sum_{A\in \mathcal{N}_{J}}\int_{B_{n,A}(c_{n,U},c_{n,L})\backslash
B_{n,A}(q_{n})}\Lambda _{A,p}\left( \mathbf{\hat{s}}_{\tau }(x)\right)
dQ(x,\tau ),
\end{equation*}%
with probability approaching one. Using the uniform convergence rate in
Assumption A3, we find that as long as 
\begin{equation*}
Q(B_{n,A}(c_{n,U},c_{n,L})\backslash B_{n,A}(q_{n}))\rightarrow 0,
\end{equation*}%
fast enough, the second term in (\ref{dec}) vanishes in probability. As for
the first integral, since for all $x\in B_{n,A}(q_{n})$, we have $%
|r_{n,j}v_{n,\tau ,j}(x)/\hat{\sigma}_{\tau ,j}(x)|\leq q_{n}$ for all $j\in
A$, we use the Lipschitz continuity of the map $\Lambda _{A,p}$ on a compact
set, to approximate the leading sum in (\ref{dec}) by 
\begin{equation*}
\bar{\theta}_{1,n}(q_{n})\equiv \sum_{A\in \mathcal{N}_{J}}%
\int_{B_{n,A}(q_{n})}\Lambda _{A,p}\left( \mathbf{\hat{s}}_{\tau }(x)\right)
dQ(x,\tau ).
\end{equation*}%
Thus we let%
\begin{equation}
\mathcal{\tilde{P}}_{n}(\lambda _{n},q_{n})\equiv \left\{ P\in \mathcal{P}%
:Q\left( \dbigcup\limits_{A\in \mathcal{N}_{J}}B_{A,n}(c_{n,U},c_{n,L})%
\backslash B_{A,n}(q_{n})\right) \leq \lambda _{n}\right\} ,  \label{P}
\end{equation}%
and find that%
\begin{equation*}
\hat{\theta}=\bar{\theta}_{1,n}(q_{n})+o_{P}(h^{d/2}),\text{ }\mathcal{%
\tilde{P}}_{n}(\lambda _{n},q_{n})\cap \mathcal{P}_{0}\text{-uniformly,}
\end{equation*}%
as long as $\lambda _{n}$ and $q_{n}$ converge to zero fast enough. We will
specify the conditions in Theorem 2 below.

Let us deal with $\bar{\theta}_{1,n}(q_{n})$. First, it can be shown that
there are sequences of nonstochastic numbers $a_{n}(q_{n})\in \mathbf{R}$
and $\sigma _{n}(q_{n})>0$ that depend on $q_{n}$ such that%
\begin{equation}
h^{-d/2}\{\bar{\theta}_{1,n}(q_{n})-a_{n}(q_{n})\}/\sigma _{n}(q_{n})\overset%
{d}{\rightarrow }N(0,1),  \label{an}
\end{equation}%
if liminf$_{n\rightarrow \infty }\sigma _{n}(q_{n})>0$. We provide the
precise formulae for $\sigma _{n}(q_{n})$ and $a_{n}(q_{n})$ in Section \ref%
{sec:lpa2}. Since the distribution of $h^{-d/2}\{\bar{\theta}%
_{1,n}(q_{n})-a_{n}(q_{n})\}/\sigma _{n}(q_{n})$ is approximated by the
bootstrap distribution of $h^{-d/2}\{\hat{\theta}^{\ast
}-a_{n}(q_{n})\}/\sigma _{n}(q_{n})$ in large samples, we find that%
\begin{equation*}
\frac{h^{-d/2}\{c_{\alpha }^{\ast }-a_{n}(q_{n})\}}{\sigma _{n}(q_{n})}=\Phi
^{-1}(1-\alpha )+o_{P}(1).
\end{equation*}%
Hence the bootstrap critical value $c_{\alpha }^{\ast }$ will dominate $%
h^{-d/2}\eta +\hat{a}^{\ast }>0$, if for all $n\geq 1,$%
\begin{eqnarray*}
\Phi ^{-1}(1-\alpha ) &\geq &\frac{h^{-d/2}\{h^{d/2}\eta +\hat{a}^{\ast
}-a_{n}(q_{n})\}}{\sigma _{n}(q_{n})} \\
&=&\frac{\eta +h^{-d/2}\{\hat{a}^{\ast }-a_{n}(q_{n})\}}{\sigma _{n}(q_{n})}%
\text{.}
\end{eqnarray*}%
We can show that $\hat{a}^{\ast }-a_{n}(q_{n})=o_{P}(h^{d/2})$, which
follows if $\lambda _{n}$ in (\ref{P}) vanishes to zero sufficiently fast.
Hence if 
\begin{equation*}
\sigma _{n}(q_{n})\geq \eta /\Phi ^{-1}(1-\alpha ),
\end{equation*}%
we have $c_{\alpha }^{\ast }$ becomes approximately equal to our bootstrap
critical value $c_{\alpha ,\eta }^{\ast }.$ This leads to the following
formulation of probabilities.

\begin{defn}
Define%
\begin{equation*}
\mathcal{P}_{n}(\lambda _{n},q_{n})\equiv \left\{ P\in \mathcal{\tilde{P}}%
_{n}(\lambda _{n},q_{n}):\sigma _{n}(q_{n})\geq \eta /\Phi ^{-1}(1-\alpha
)\right\} ,
\end{equation*}%
where $\mathcal{\tilde{P}}_{n}(\lambda _{n},q_{n})$ is as defined in (\ref{P}%
).
\end{defn}

The following theorem establishes the asymptotic exactness of the size of
the bootstrap test over $P\in \mathcal{P}_{n}(\lambda _{n},q_{n})\cap 
\mathcal{P}_{0}$.

\begin{thm}
\label{Thm2} Suppose that Assumptions A\ref{assumption-A1}-A\ref%
{assumption-A6} and B\ref{assumption-B1}-B\ref{assumption-B4} hold. Let $%
\lambda _{n}\rightarrow 0$ and $q_{n}\rightarrow 0$ be positive sequences
such that%
\begin{eqnarray}
h^{-d/2}\left( \log n\right) ^{p/2}\lambda _{n} &\rightarrow &0\text{ and}
\label{lamq} \\
h^{-d/2}q_{n}\{(\log n)^{(p-1)/2}+q_{n}^{p-1}\} &\rightarrow &0\text{.} 
\notag
\end{eqnarray}%
Then%
\begin{equation*}
\underset{n\rightarrow \infty }{\mathrm{limsup}}\sup_{P\in \mathcal{P}%
_{n}(\lambda _{n},q_{n})\cap \mathcal{P}_{0}}\left\vert P\{\hat{\theta}%
>c_{\alpha ,\eta }^{\ast }\}-\alpha \right\vert =0.
\end{equation*}
\end{thm}

Theorem \ref{Thm2} shows that the rejection probability of our bootstrap
test achieves exactly the level $\alpha $ uniformly over the set of
probabilities in $\mathcal{P}_{n}(\lambda _{n},q_{n})\cap \mathcal{P}_{0}$.
If $v_{n,\tau ,j}(x)\equiv 0$ for each $(x,\tau )$ and for each $j$ (the
least favorable case, say $P_{\mathrm{LFC}}$), then it is obvious that the
distribution $P_{\mathrm{LFC}}$ belongs to $\mathcal{P}_{n}(\lambda
_{n},q_{n})$ for any positive sequences $\lambda _{n}\rightarrow 0$ and $%
q_{n}\rightarrow 0$. This would be the only case of asymptotically exact
coverage if bootstrap critical values were obtained as in \eqref{LFC},
without contact set estimation. By estimating the contact sets and obtaining
a critical value based on them, Theorem \ref{Thm2} establishes the
asymptotically uniform exactness of the bootstrap test for distributions
such that they may not satisfy $v_{n,\tau ,j}(x)\equiv 0$ everywhere.

\subsection{Sufficient Conditions for Uniform Convergences in Assumptions A%
\protect\ref{assumption-A3} and B\protect\ref{assumption-B2}}

\label{sec:suff-A3-B2}

This subsection gives sufficient conditions that yield Assumptions A\ref%
{assumption-A3} and B\ref{assumption-B2}. The result is formalized in the
following lemma.

\begin{lem}
(i) Suppose that Assumptions A1-A2 hold and that for each $j\in \mathbb{N}%
_{J}$, there exist finite constants $C,\gamma _{j}>0,$ and a positive
sequence $\delta _{n,j}>0$ such that for all $n\geq 1,$ and all $(x_{1},\tau
_{1})\in \mathcal{S}$,%
\begin{equation}
\mathbf{E}\left[ \sup_{(x_{2},\tau _{2})\in \mathcal{S:}||x_{1}-x_{2}||+||%
\tau _{1}-\tau _{2}||\leq \lambda }\left( b_{n,ij}(x_{1},\tau
_{1})-b_{n,ij}(x_{2},\tau _{2})\right) ^{2}\right] \leq C\delta
_{n,j}^{2}\lambda ^{\gamma _{j}},\text{ for all }\lambda >0,  \label{Lp-Cont}
\end{equation}%
where $b_{n,ij}(x_{1},\tau _{1})\equiv \beta _{n,x_{1},\tau _{1},j}\left(
Y_{ij},(X_{i}-x_{1})/h\right) $ and $\limsup_{n\rightarrow \infty }\mathbf{E}%
[\sup_{(x,\tau )\in \mathcal{S}}b_{n,ij}^{4}(x,\tau )]\leq C$ and $\delta
_{n,j}=n^{s_{1,j}}$ and $h=n^{s_{2}}$ for some $s_{1,j},s_{2}\in \mathbf{R}$%
. Furthermore, assume that%
\begin{equation*}
n^{-1/2}h^{-d-\nu }\rightarrow 0,
\end{equation*}%
for some small $\nu >0$. Then, Assumption A\ref{assumption-A3} holds.\newline
(ii) Suppose further that Assumptions B\ref{assumption-B1} and B\ref%
{assumption-B3} hold. Then, Assumption B\ref{assumption-B2} holds.\bigskip
\end{lem}

The condition (\ref{Lp-Cont}) is the local $L_{2}$-continuity condition for $%
\beta _{n,x,\tau ,j}\left( Y_{ij},(X_{i}-x)/h\right) $ in $(x,\tau )$. The
condition corresponds to what \citeasnoun{Andrews:04} called
\textquotedblleft Type IV class\textquotedblright . The condition is
satisfied by numerous maps that are continuous or discontinuous, as long as
regularity conditions for the random vector $(Y_{i},X_{i})$ are satisfied.%
\footnote{\citeasnoun[Theorem
3]{Chen/Linton/VanKeilegom:03} introduced its extension to functions indexed
partly by infinite dimensional parameters, and called it local uniform $%
L_{2} $-continuity. For further discussions, see \citeasnoun{Andrews:04} and %
\citeasnoun{Chen/Linton/VanKeilegom:03}.} Typically, $\delta _{n,j}$
diverges to infinity at a polynomial rate in $h^{-1}$. The constant $\gamma
_{j}$ is 2 or can be smaller than 2, depending on the smoothness of the
underlying function $b_{n,ij}(x,\tau )$. The value of $\gamma _{j}$ does not
affect the asymptotic theory of this paper, as long as it is strictly
positive.

\section{Power Properties}

\label{sec:power}

In this section, we consider the power properties of the bootstrap test. In
Section \ref{sec:consistency}, we establish the consistency of our test.
Section \ref{sec:lpa1} provides heuristic arguments behind local power
properties of our tests, and Section \ref{sec:lpa2} presents the local power
function in a general form.\footnote{%
The local power results in this section are more general than those of %
\citeasnoun{LSW}. In particular, the results accommodate a wider class of
local alternatives that may not converge to the least favorable case.}

\subsection{Consistency}

\label{sec:consistency}

First, to show consistency of our test, we make the following assumption.

\begin{AssumptionC}
For each $j\in \mathbb{N}_{J}$ and $(x,\tau )\in \mathcal{S}$, $v_{n,\tau
,j}(x)=v_{\tau ,j}(x)+o(1),$ and\textit{\ }%
\begin{equation}
\limsup_{n\rightarrow \infty }\sup_{(x,\tau )\in \mathcal{S}}|v_{n,\tau
,j}(x)|<\infty .  \label{dominance}
\end{equation}
\end{AssumptionC}

The pointwise convergence $v_{n,\tau ,j}(x)=v_{\tau ,j}(x)+o(1)$ holds
typically by an appropriate choice of $v_{n,\tau ,j}(x)$. In many examples,
condition (\ref{dominance}) is often implied by Assumptions A1-A6. If we
revisit the simple example considered in Section \ref{subsec:discontinuity},
it is straightforward to see that under Assumptions A1-A6, with the
subscript $\tau $ suppressed, $v_{n,1}(x)=v_{1}(x)+o(1)$, where $%
v_{n,1}(x)\equiv \mathbf{E}\hat{v}_{n,1}(x)$ and $v_{1}(x)\equiv \mathbf{E}%
(Y|X=x)f(x)$, and \eqref{dominance} holds easily.

We now establish the consistency of our proposed test as follows.

\begin{thm}
\label{Thm3} Suppose that Assumptions A1-A6, B\ref{assumption-B1}-B4, and C1
hold and that we are under a fixed alternative hypothesis such that 
\begin{equation*}
\int \Lambda _{p}\left( v_{\tau ,1}(x),\cdot \cdot \cdot ,v_{\tau
,J}(x)\right) dQ(x,\tau )>0.
\end{equation*}%
Then as $n\rightarrow \infty ,$ 
\begin{equation*}
P\{\hat{\theta}>c_{\alpha, \eta }^{\ast }\}\rightarrow 1.
\end{equation*}
\end{thm}

\subsection{Local Power Analysis: Definitions and Heuristics}

\label{sec:lpa1}

In this section, we investigate the local power properties of our test. For
local power analysis, we formally define the space of Pitman directions. Let 
$\mathcal{D}$ be the collection of $\mathbf{R}^{J}$-valued bounded functions
on $\mathcal{X}\times \mathcal{T}$ such that for each $\delta =(\delta
_{1},\cdot \cdot \cdot ,\delta _{J})\in \mathcal{D}$, $Q\{(x,\tau) \in 
\mathcal{S}: \delta_j(x,\tau) \neq 0 \} > 0$ for some $j=1,\dots,J$. That
is, at least one of the components of any $\delta \in \mathcal{D}$ is a
non-zero function a.e. For each $\delta =(\delta _{1},\cdot \cdot \cdot
,\delta _{J})\in \mathcal{D}$, we write $\delta _{\tau ,j}(x)=\delta
_{j}(x,\tau )$, $j=1,\cdot \cdot \cdot ,J$.

For a given vector of sequences $b_{n}=(b_{n,1},\cdot \cdot \cdot ,b_{n,J})$%
, such that $b_{n,j}\rightarrow \infty $, and $\delta \in \mathcal{D},$ we
consider the following type of local alternatives:

\begin{equation}
H_{\delta }:v_{\tau ,j}(x)=v_{\tau ,j}^{0}(x)+\frac{\delta _{\tau ,j}(x)}{%
b_{n,j}},\ \text{for all }j\in \mathbb{N}_{J},  \label{la}
\end{equation}%
where $v_{\tau ,j}^{0}(x)\leq 0$ for all $(x,\tau ,j)\in \mathcal{X}\times 
\mathcal{T}\times \mathbb{N}_{J}$, $\delta _{\tau ,j}(x)>0$ for some $%
(x,\tau ,j)\in \mathcal{X}\times \mathcal{T}\times \mathbb{N}_{J}$ such that 
$v_{\tau ,j}(x)>0$ for some $(x,\tau ,j)\in \mathcal{X}\times \mathcal{T}%
\times \mathbb{N}_{J}$. Note that in \eqref{la}, $v_{\tau ,j}(x)$ is a
sequence of Pitman local alternatives that consist of three components: $%
v_{\tau ,j}^{0}(x)$, ${b}_{n}$, and ${\delta _{\tau ,j}(x)}$.

The first component $v_{\tau ,j}^{0}(x)$ determines where the sequence of
local alternatives converges to. For example, if $v_{\tau ,j}^{0}(x)\equiv 0$
for all $(x,\tau ,j)$, then we have a sequence of local alternatives that
converges to the least favorable case. We allow for negative values for $%
v_{\tau ,j}^{0}(x)$, so that we include the local alternatives that do not
converge to the least favorable case as well.

From here on, we assume the local alternative hypotheses of the form in (\ref%
{la}). We fix $v_{\tau ,j}^{0}(x)$ and identify each local alternative with
a pair $(b_{n},\delta )$ for each Pitman direction $\delta \in \mathcal{D}$.
The following definitions are useful to explain our local power results.

\begin{defn}
(i) Given a Pitman direction $\delta \in \mathcal{D}$, we say that an $%
\alpha $-level test, $1\{T>c_{\alpha }\},$ \textit{has nontrivial local
power against} $(b_{n},\delta )$, if under the local alternatives $%
(b_{n},\delta )$,%
\begin{equation*}
\text{liminf}_{n\rightarrow \infty }\ P\left\{ T>c_{\alpha }\right\} >\alpha,
\end{equation*}%
and say that the test has \textit{trivial local power against} $%
(b_{n},\delta )$, if under the local alternatives $(b_{n},\delta )$,%
\begin{equation*}
\text{limsup}_{n\rightarrow \infty }\ P\left\{ T>c_{\alpha }\right\} \leq
\alpha.
\end{equation*}
(ii) Given a collection $\mathcal{D}$, we say that a test \textit{has
convergence rate} $b_{n}$ \textit{against} $\mathcal{D}$, if the test has
nontrivial local power against $(b_{n},\delta )$ for some $\delta \in 
\mathcal{D}$, and has trivial local power against $(b_{n}^{\prime },\delta )$
for all $\delta \in \mathcal{D}$ and all $b_{n}^{\prime }$ such that $%
b_{n,j}^{\prime }/b_{n,j}\rightarrow \infty $ as $n\rightarrow \infty $, for
all $j=1,\ldots,J$.
\end{defn}

One of the remarkable aspects of the local power properties is that our test
has \textit{two types of convergence rates}. More specifically, there exists
a partition $(\mathcal{D}_{1},\mathcal{D}_{2})$ of $\mathcal{D}$, where our
test has a rate $b_{n}$ against $\mathcal{D}_{1}$ and another rate $%
b_{n}^{\prime }$ against $\mathcal{D}_{2}$. Furthermore, in many
nonparametric inequality testing environments, the faster of the two rates $%
b_{n}$ and $b_{n}^{\prime }$ achieves the parametric rate of $\sqrt{n}$.

To see this closely, let us assume the set-up of testing inequality
restrictions on a mean regression function in Section 2.4, and consider the
following local alternatives:%
\begin{equation}
v_{n,1}(x)=v_{0}(x)+\frac{\delta (x)}{b_{n}},  \label{la-e}
\end{equation}%
where $v_{0}(x)\leq 0$ for all $x\in \mathcal{X}$, and $\delta \in \mathcal{D%
}$.

First, we set $b_{n}=\sqrt{n}$. Then under this local alternative hypothesis 
$(b_{n},\delta )$, we can verify that with probability approaching one,%
\begin{equation}
h^{-1/2}(\hat{\theta}-a_{n,0})=h^{-1/2}\left\{ \int_{B_{n}^{0}(c_{n})}\left[
Z_{n,1}(x)+\frac{\sqrt{nh}v_{0}(x)}{\hat{\sigma}_{1}(x)}+\frac{h^{1/2}\delta
(x)}{\hat{\sigma}_{1}(x)}\right] _{+}dx-a_{n,0}\right\} ,  \label{dec59}
\end{equation}%
where $Z_{n,1}(x)=\sqrt{nh}\left\{ {\hat{v}_{1}(x)-v_{n,1}(x)}\right\}/{\hat{%
\sigma}_{1}(x)} $, $B_{n}^{0}(c_{n})=\left\{ x\in \mathcal{X}:\left\vert 
\sqrt{nh}v_{0}(x)\right\vert \leq c_{n}\right\} ,$ $c_{n}\rightarrow \infty $%
, $\sqrt{\log n}/c_{n}\rightarrow 0$, and 
\begin{equation*}
a_{n,0}=\mathbf{E}\left[ \int_{B_{n}^{0}(c_{n})}\left[ Z_{n,1}(x)\right]
_{+}dx\right] .
\end{equation*}%
Under regularity conditions, the right-hand side of (\ref{dec59}) is
approximated by 
\begin{equation}
h^{-1/2}\left\{ \int_{B^{0}(0)}\left[ Z_{n,1}(x)+\frac{h^{1/2}\delta (x)}{%
\sigma _{1}(x)}\right] _{+}dx-a_{n,\delta }\right\} +h^{-1/2}\left\{
a_{n,\delta }-a_{n,0}\right\} ,  \label{decomp22}
\end{equation}%
where $B^{0}(0)=\{x\in \mathcal{X}:v_{0}(x)=0\}$ and 
\begin{equation*}
a_{n,\delta }=\mathbf{E}\left[ \int_{B^{0}(0)}\left[ Z_{n,1}(x)+\frac{%
h^{1/2}\delta (x)}{\sigma _{1}(x)}\right] _{+}dx\right] .
\end{equation*}%
The leading term in (\ref{decomp22}) converges in distribution to $\mathbb{Z}%
_{1}\sim N(0,\sigma _{0}^{2})$ precisely as in (\ref{as}). Furthermore, we
can show that 
\begin{eqnarray*}
a_{n,\delta } &=&\int_{B^{0}(0)}\mathbf{E}\left[ \mathbb{Z}_{1}+\frac{%
h^{1/2}\delta (x)}{\sigma _{1}(x)}\right] _{+}dx+o(h^{1/2})\text{ and\ } \\
a_{n,0} &=&\int_{B^{0}(0)}\mathbf{E}\left[ \mathbb{Z}_{1}\right]
_{+}dx+o(h^{1/2}).
\end{eqnarray*}%
Therefore, as for the last term in (\ref{decomp22}), we find that%
\begin{eqnarray*}
h^{-1/2}\left\{ a_{n,\delta }-a_{n,0}\right\}
&=&\int_{B^{0}(0)}h^{-1/2}\left( \mathbf{E}\left[ \mathbb{Z}_{1}+\frac{%
h^{1/2}\delta (x)}{\sigma _{1}(x)}\right] _{+}-\mathbf{E}\left[ \mathbb{Z}%
_{1}\right] _{+}\right) dx \\
&=&2\phi (0)\int_{B^{0}(0)}\frac{\delta (x)}{\sigma _{1}(x)}dx+o(1),
\end{eqnarray*}%
where the last equality follows from expanding $h^{-1/2}\left\{ \mathbf{E}%
\left[ \mathbb{Z}_{1}+h^{1/2}\delta (x)/\sigma _{1}(x)\right] _{+}-\mathbf{E}%
\left[ \mathbb{Z}_{1}\right] _{+}\right\} $. We conclude that under the
local alternatives, we have%
\begin{equation*}
h^{-1/2}(\hat{\theta}-a_{n,0}) \ \ \rightarrow _{d}\ \ \mathbb{Z}_{1}+2\phi
(0)\int_{B^{0}(0)}\frac{\delta (x)}{\sigma _{1}(x)}dx.
\end{equation*}%
The magnitude of the last term in the limit determines the local power of
the test. Thus under Pitman local alternatives such that 
\begin{equation}
\int_{B^{0}(0)}\frac{\delta (x)}{\sigma _{1}(x)}dx>0,  \label{delta_int}
\end{equation}%
the test has nontrivial power against $\sqrt{n}$-converging Pitman local
alternatives. Note that the integral in \eqref{delta_int} is defined on the
population contact set $B^{0}(0)$. Thus, the test has nontrivial power,
unless the contact set has Lebesgue measure zero or $\delta (\cdot )$ is
\textquotedblleft too often negative\textquotedblright\ on the contact set.

When the integral in (\ref{delta_int}) is zero, we consider the local
alternatives $(b_{n},\delta )\ $with a slower convergence rate $%
b_{n}=n^{1/2}h^{1/4}$. Following similar arguments as before, we now have 
\begin{equation*}
h^{-1/2}(\hat{\theta}-a_{n,0}) \rightarrow _{d}\mathbb{Z}_{1}+\text{lim}%
_{n\rightarrow \infty }h^{-1/2}\left\{ \bar{a}_{n,\delta }-a_{n,0}\right\} ,
\end{equation*}%
where%
\begin{equation*}
\bar{a}_{n,\delta }=\int_{B^{0}(0)}\mathbf{E}\left[ Z_{n,1}(x)+\frac{%
h^{1/4}\delta (x)}{\sigma _{1}(x)}\right] _{+}dx,
\end{equation*}%
which can be shown again to be equal to 
\begin{align*}
\int_{B^{0}(0)} \mathbf{E}\left[ \mathbb{Z}_{1}+\frac{h^{1/4}\delta (x)}{%
\sigma _{1}(x)}\right] _{+}\ dx + o( h^{1/2}).
\end{align*}
However, observe that%
\begin{eqnarray*}
&&h^{-1/2}\int_{B^{0}(0)}\left\{ \mathbf{E}\left[ \mathbb{Z}_{1}+\frac{%
h^{1/4}\delta (x)}{\sigma _{1}(x)}\right] _{+}-\mathbf{E}\left[ \mathbb{Z}%
_{1}\right] _{+}\right\} dx \\
&=&h^{-1/4}2\phi (0)\int_{B^{0}(0)}\frac{\delta (x)}{\sigma _{1}(x)}dx+\frac{%
1}{2}\int_{B^{0}(0)}\frac{\delta ^{2}(x)}{\sigma _{1}^{2}(x)}dx+o(1) \\
&=&\frac{1}{2}\int_{B^{0}(0)}\frac{\delta ^{2}(x)}{\sigma _{1}^{2}(x)}%
dx+o(1),
\end{eqnarray*}%
because $\int_{B^{0}(0)}\{\delta (x)/\sigma _{1}(x)\}dx=0$. We find that
under the local alternative hypothesis in (\ref{la-e}) with $%
b_{n}=n^{1/2}h^{1/4}$,%
\begin{equation*}
h^{-1/2}(\hat{\theta}-a_{n,0}) \ \ \ \rightarrow _{d}\ \ \ \mathbb{Z}_{1}+%
\frac{1}{2}\int_{B^{0}(0)}\frac{\delta ^{2}(x)}{\sigma _{1}^{2}(x)}dx.
\end{equation*}%
Therefore, even when $\int_{B^{0}(0)}\{\delta (x)/\sigma _{1}(x)\}dx=0$, the
test still has nontrivial power against $n^{1/2}h^{1/4}$-converging Pitman
local alternatives, if the Pitman directions are such that 
\begin{equation*}
\int_{B^{0}(0)}\{\delta ^{2}(x)/\sigma _{1}^{2}(x)\}dx>0.
\end{equation*}%
Now let us consider the partition $(\mathcal{D}_{1},\mathcal{D}_{2})$ of $%
\mathcal{D}$, where%
\begin{eqnarray*}
\mathcal{D}_{1} &=&\left\{ \delta \in \mathcal{D}:\int_{B^{0}(0)}\delta
(x)/\sigma _{1}(x)dx\neq 0\right\} \text{ and} \\
\mathcal{D}_{2} &=&\left\{ \delta \in \mathcal{D}:\int_{B^{0}(0)}\delta
(x)/\sigma _{1}(x)dx=0\text{ and }\int_{B^{0}(0)}\{\delta ^{2}(x)/\sigma
_{1}^{2}(x)\}dx>0\right\} .
\end{eqnarray*}%
When inf$_{x\in \mathcal{X}}\sigma _{1}^{2}(x)>c>0$ for some $c>0$ (recall
Assumption A5) and $Q(B^{0}(0))>0$, we have $\int_{B^{0}(0)}\{\delta
^{2}(x)/\sigma _{1}^{2}(x)\}dx>0$ and the set $\{\mathcal{D}_{1},\mathcal{D}%
_{2}\}$ becomes a partition of $\mathcal{D}$. Thus the bootstrap test has a
convergence rate of $\sqrt{n}$ against $\mathcal{D}_{1}$ and $n^{1/2}h^{1/4}$%
-rate against $\mathcal{D}_{2}$. In the next section, Corollary \ref%
{partition-cor} provides a general result of this phenomenon of dual
convergence rates of our bootstrap test.

\subsection{Local Power Analysis: Results}

\label{sec:lpa2}

We now provide general local power functions explicitly. We first present
explicit forms of location and scale normalizers, $a_{n}(q_{n})$ and $\sigma
_{n}(q_{n})$ in (\ref{an}). Let for $j,k\in \mathbb{N}_{J},$ and $\tau
_{1},\tau _{2}\in \mathcal{T}$,%
\begin{equation}
\rho _{n,\tau _{1},\tau _{2},j,k}(x,u)\equiv \frac{1}{h^{d}}\mathbf{E}\left[
\beta _{n,x,\tau _{1},j}\left( Y_{ij},\frac{X_{i}-x}{h}\right) \beta
_{n,x,\tau _{2},k}\left( Y_{ik},\frac{X_{i}-x}{h}+u\right) \right] .
\label{rho}
\end{equation}%
This function approximates the asymptotic covariance between $\sqrt{n}(\hat{v%
}_{\tau ,j}(x)-v_{n,\tau ,j}(x))/\hat{\sigma}_{\tau ,j}(x)$ and $\sqrt{n}(%
\hat{v}_{\tau ,j}(x+uh)-v_{n,\tau ,j}(x+uh))/\hat{\sigma}_{\tau ,j}(x)$. We
define $\Sigma _{n,\tau _{1},\tau _{2}}(x,u)$ to be the $J$-dimensional
square matrix with $(j,k)$-th entry given by $\rho _{n,\tau _{1},\tau
_{2},j,k}(x,u)$.

Define for $\mathbf{v}\in \mathbf{R}^{J}$,%
\begin{equation*}
\bar{\Lambda}_{x,\tau }(\mathbf{v})\equiv \sum_{A\in \mathcal{N}_{J}}\Lambda
_{A,p}(\mathbf{v})1\left\{ (x,\tau )\in B_{n,A}(q_{n})\right\} .
\end{equation*}%
Then we define%
\begin{equation*}
a_{n}(q_{n})\equiv \int_{\mathcal{X}}\int_{\mathcal{T}}\mathbf{E}\left[ \bar{%
\Lambda}_{x,\tau _{1}}(\mathbb{W}_{n,\tau ,\tau }^{(1)}(x,0))\right] d\tau
dx,
\end{equation*}%
and%
\begin{equation}
\sigma _{n}^{2}(q_{n})\equiv \int_{\mathcal{U}}\int_{\mathcal{X}}\int_{%
\mathcal{T}}\int_{\mathcal{T}}C_{n,\tau _{1},\tau _{2}}(x,u)d\tau _{1}d\tau
_{2}dxdu,  \label{sig_A}
\end{equation}%
where%
\begin{equation*}
C_{n,\tau _{1},\tau _{2}}(x,u)\equiv Cov\left( \bar{\Lambda}_{x,\tau _{1}}(%
\mathbb{W}_{n,\tau _{1},\tau _{2}}^{(1)}(x,u)),\bar{\Lambda}_{x,\tau _{2}}(%
\mathbb{W}_{n,\tau _{1},\tau _{2}}^{(2)}(x,u))\right) ,
\end{equation*}%
and $[\mathbb{W}_{n,\tau _{1},\tau _{2}}^{(1)}(x,u)^{\top },\mathbb{W}%
_{n,\tau _{1},\tau _{2}}^{(2)}(x,u)^{\top }]^{\top }$ is a mean zero $%
\mathbf{R}^{2J}$-valued Gaussian random vector whose covariance matrix is
given by%
\begin{equation}
\left[ 
\begin{array}{c}
\Sigma _{n,\tau _{1},\tau _{1}}(x,0) \\ 
\Sigma _{n,\tau _{1},\tau _{2}}(x,u)^{\top }%
\end{array}%
\begin{array}{c}
\Sigma _{n,\tau _{1},\tau _{2}}(x,u) \\ 
\Sigma _{n,\tau _{2},\tau _{2}}(x+uh,0)%
\end{array}%
\right] .  \label{cov}
\end{equation}%
The multiple integral in (\ref{sig_A}) is nonnegative.

The limit of the quantity $\sigma _{n}^{2}(q_{n})$ as $n\rightarrow \infty ,$
if it is positive, is nothing but the asymptotic variance of the test
statistic $\hat{\theta}$ (after location-scale normalization). Not
surprisingly the asymptotic variance does not depend on points $(x,\tau )$
of $\mathcal{X}\times \mathcal{T}$ such that $v_{n,\tau ,j}(x)/\sigma
_{n,\tau ,j}(x)$ is away below zero, as is expressed through its dependence
on the contact sets $B_{n,A}(q_{n})$ and the \textquotedblleft truncated
map\textquotedblright\ $\bar{\Lambda}_{x,\tau }(\cdot )$ involving $A$'s
restricted to $\mathcal{N}_{J}$.

We first make the following assumptions.

\begin{AssumptionC}
(i) For each $(\tau ,j)\in \mathcal{T}\times \mathbb{N}_{J}$, there exists a
map $v_{n,\tau ,j}^{0}:\mathbf{R}^{d}\rightarrow \mathbf{R}$ such that for
each $x\in \mathcal{S}_{\tau }(\varepsilon _{1}),$ $v_{n,\tau ,j}^{0}(x)\leq
0$, and%
\begin{equation}
v_{n,\tau ,j}(x)=v_{n,\tau ,j}^{0}(x)+\frac{\delta _{\tau ,j}(x)}{b_{n,j}}%
\left( 1+o(1)\right) ,  \label{la4}
\end{equation}%
where $o(1)$ is uniform in $x\in \mathcal{S}_{\tau }$ and in $\tau \in 
\mathcal{T}$, as $n\rightarrow \infty $ and $b_{n,j}\rightarrow \infty $ is
the positive sequence in (\ref{la}).\newline
(ii) $\sup_{(x,\tau )\in \mathcal{S}}|\sigma _{n,\tau ,j}(x)-\sigma _{\tau
,j}(x)|=o(1)$, as $n\rightarrow \infty $, for some function $\sigma _{\tau
,j}(x)$ such that $\inf_{(x,\tau )\in \mathcal{S}}\sigma _{\tau ,j}(x)>0$.
\end{AssumptionC}

Assumption C2 can also be shown to hold in many examples. When appropriate
smoothness conditions for $v_{\tau ,j}(x)$ hold and a suitable (possibly
higher-order) kernel function is used, we can take $v_{n,\tau ,j}(x)$ in
Assumption A1 to be identical to $v_{\tau ,j}(x)$, and hence Assumption C2
is implied by (\ref{la}). For the simple example in Section 2.4, if we take $%
v_{n,j}(x)=\mathbf{E}\hat{v}_{j}(x)$, it follows that $%
v_{n,j}(x)=v_{n,j}^{0}(x)+b_{n,j}^{-1}\int \delta _{j}(x+zh)K(z)dz$, with $%
v_{n,j}^{0}(x)=\int v_{j}^{0}(x+zh)K(z)dz$. Hence when $\delta _{j}(x)$ is
uniformly continuous in $x$, we obtain Assumption C2.

The local asymptotic power function is based on the asymptotic normal
approximation of the distribution of $\hat{\theta}$ (after scale and
location normalization) under the local alternatives. For this purpose, we
define the sequence of probability sets that admit the normal approximation
under local alternatives. For $c_{1},c_{2}>0,$ let $B_{n}^{0}(c_{1},c_{2})$
and $B_{n,A}^{0}(c_{1},c_{2})$ denote $B_{n}(c_{1},c_{2})$ and $%
B_{n,A}(c_{1},c_{2})$ except that $v_{n,\tau ,j}(x)$'s are replaced by $%
v_{n,\tau ,j}^{0}(x)$'s in Assumption C2. As before, we write $B_n^{0}(c)
\equiv B_{n}^{0}(c,c)$.

\begin{defn}
For any positive sequence $\lambda _{n}\rightarrow 0,\ $define%
\begin{equation*}
\mathcal{P}_{n}^{0}(\lambda _{n})\equiv \left\{ P\in \mathcal{\tilde{P}}%
_{n}^{0}(\lambda _{n}):\sigma _{n}^{2}(0)\geq \eta /\Phi ^{-1}(1-\alpha
)\right\} ,
\end{equation*}%
where $\mathcal{\tilde{P}}_{n}^{0}(\lambda _{n})$ is equal to $\mathcal{%
\tilde{P}}_{n}(\lambda _{n},q_{n})$ except that $B_{n,A}(c_{n,U},c_{n,L})$
and $B_{n,A}(q_{n})$ are replaced by $B_{n,A}^{0}(c_{n,U},c_{n,L})$ and $%
B_{n,A}^{0}(q_{n})$ for all $A\in \mathcal{N}_{J}$, and $q_{n}$ is set to be
zero.
\end{defn}

To give a general form of the local power function, let us define $\psi
_{n,A,\tau }(\cdot ;x):\mathbf{R}^{J}\mathbf{\rightarrow \lbrack }0,\infty
), $ $(x,\tau )\in \mathcal{X}\times \mathcal{T}$ and $A\subset \mathbb{N}%
_{J}$, as%
\begin{equation*}
\psi _{n,A,\tau }(\mathbf{y};x)=\frac{1}{\sigma _{n}(0)}\mathbf{E}\left[
\Lambda _{A,p}\left( \mathbb{W}_{n,\tau ,\tau }^{(1)}(x,0)+\mathbf{y}\right) %
\right] \cdot 1\left\{ (x,\tau )\in B_{n,A}^{0}(0)\right\} .
\end{equation*}%
The local power properties of the bootstrap test are mainly determined by
the slope and the curvature of this function. So, we define%
\begin{equation}
\psi _{n,A,\tau }^{(1)}(\mathbf{y};x)\equiv \frac{\partial }{\partial 
\mathbf{y}}\psi _{n,A,\tau }(\mathbf{y};x)\text{ and }\psi _{n,A,\tau
}^{(2)}(\mathbf{y};x)\equiv \frac{\partial ^{2}}{\partial \mathbf{y}\partial 
\mathbf{y}^{\top }}\psi _{n,A,\tau }(\mathbf{y};x),  \label{psi1}
\end{equation}%
if the first derivatives and the second derivatives in the definition exist
respectively.

\begin{AssumptionC}
(i) There exists $\varepsilon _{1}>0$ such that for all $(\tau ,A)\in 
\mathcal{T}\times \mathcal{N}_{J}$ and all $x$\ in the interior of $\mathcal{%
S}_{\tau }(\varepsilon _{1})$, $\psi _{n,A,\tau }^{(1)}(\mathbf{0};x)$
exists for all $n\geq 1$ and%
\begin{equation*}
\psi _{A,\tau }^{(1)}(\mathbf{0};x)\equiv \lim_{n\rightarrow \infty }\psi
_{n,A,\tau }^{(1)}(\mathbf{0};x)
\end{equation*}%
exists, and $\lim \sup_{n\rightarrow \infty }\sup_{(x,\tau )\in \mathcal{S}%
}|\psi _{n,A,\tau }^{(1)}(0;x)|<C$\ \textit{for some }$C>0$.\newline
(ii) \textit{There exists }$\varepsilon _{1}>0$ \textit{such that for all }$%
(\tau ,A)\in \mathcal{T}\times \mathcal{N}_{J}$ \textit{and all }$x$\textit{%
\ in the interior of }$\mathcal{S}_{\tau }(\varepsilon _{1})$, $\psi
_{n,A,\tau }^{(2)}(\mathbf{0};x)$ exists for all $n\geq 1$ and%
\begin{equation*}
\psi _{A,\tau }^{(2)}(\mathbf{0};x)\equiv \lim_{n\rightarrow \infty }\psi
_{n,A,\tau }^{(2)}(\mathbf{0};x)
\end{equation*}%
\textit{exists, and }$\lim \sup_{n\rightarrow \infty }\sup_{(x,\tau )\in 
\mathcal{S}}|\psi _{n,A,\tau }^{(2)}(\mathbf{0};x)|<C$\textit{\ for some }$%
C>0$.
\end{AssumptionC}

To appreciate Assumption C3, consider the case where $J=2$, $A=\{1,2\}$, and 
$\mathbb{W}_{n,\tau ,\tau }^{(1)}(x,0)$ has a distribution denoted by $G_{n}$%
. Choose $y_{1}\geq y_{2}$ without loss of generality. We take $\Lambda
_{p}(v_{1},v_{2})=\max \{v_{1},v_{2},0\}^{p}$. Then we can write $\mathbf{E}%
[\Lambda _{A,p}(\mathbb{W}_{n,\tau ,\tau }^{(1)}(x,0)+\mathbf{y)}]$ as%
\begin{eqnarray*}
&&\int_{\mathbf{R}^{2}}(w_{1}+y_{1})^{p}1\left\{ w_{1}\in \lbrack
w_{2}+y_{2}-y_{1},\infty )\text{ and }w_{2}\in \lbrack -y_{2},\infty
)\right\} dG_{n}(w_{1},w_{2}) \\
&&+\int_{\mathbf{R}^{2}}\left( w_{2}+y_{2}\right) ^{p}1\left\{ w_{1}\in
(-\infty ,w_{2}+y_{2}-y_{1})\text{ and }w_{2}\in \lbrack -y_{2},\infty
)\right\} dG_{n}(w_{1},w_{2}) \\
&&+\int_{\mathbf{R}^{2}}(w_{1}+y_{1})^{p}1\left\{ w_{1}\in \lbrack
-y_{1},\infty )\text{ and }w_{2}\in (-\infty ,-y_{2})\right\}
dG_{n}(w_{1},w_{2}).
\end{eqnarray*}%
Certainly the three quantities are all differentiable in $(y_{1},y_{2})$.

The following theorem offers the local power function of the bootstrap test
in a general form.

\begin{thm}
\label{Thm4} \textit{Suppose that Assumptions A1-A6, B\ref{assumption-B1}%
-B4, C1-C2, and C3(i) hold and that}%
\begin{equation}
h^{-d/2}\left( \log n\right) ^{p/2}\lambda _{n}\rightarrow 0,  \label{egref2}
\end{equation}%
\textit{as }$n\rightarrow \infty $. \textit{Then for each sequence }$%
P_{n}\in \mathcal{P}_{n}^{0}(\lambda _{n})$, $n\geq 1,$\textit{\ which
satisfies the local alternative hypothesis }$(b_{n},\delta )$\textit{\ for
some }$\delta \in \mathcal{D}$\textit{\ with }$b_{n} =
(r_{n,j}h^{-d/2})_{j=1}^J$,%
\begin{equation*}
\lim_{n\rightarrow \infty }P_{n}\{\hat{\theta}>c_{\alpha ,\eta }^{\ast
}\}=1-\Phi \left( z_{1-\alpha }-\mu _{1}(\delta )\right) ,
\end{equation*}%
\textit{where }$\Phi $\textit{\ denotes the standard normal cdf,}%
\begin{equation*}
\mu _{1}(\delta )\equiv \sum_{A\in \mathcal{N}_{J}}\int \psi _{A,\tau
}^{(1)}(\mathbf{0};x)^{\top }\delta _{\tau ,\sigma }(x)dQ(x,\tau ),
\end{equation*}%
\textit{and} 
\begin{equation}
\delta _{\tau ,\sigma }(x)\equiv \left( \frac{\delta _{\tau ,1}(x)}{\sigma
_{\tau ,1}(x)},\cdot \cdot \cdot ,\frac{\delta _{\tau ,J}(x)}{\sigma _{\tau
,J}(x)}\right) .  \label{deltas}
\end{equation}
\end{thm}

Theorem \ref{Thm4} shows that if we take $b_{n}$ such that $%
b_{n,j}=r_{n,j}h^{-d/2}$ for each $j=1,\ldots,J$, the local asymptotic power
of the test against $(b_{n},\delta )$ is determined by the shift $\mu
_{1}(\delta )$. Thus, the bootstrap test has nontrivial local power against $%
(b_{n},\delta )$ if and only if%
\begin{equation*}
\mu _{1}(\delta )>0.
\end{equation*}%
The test is asymptotically biased against $(b_{n},\delta )$ such that $\mu
_{1}(\delta )<0$.

Suppose that 
\begin{equation}
\mu _{1}(\delta )=0,  \label{eq2}
\end{equation}%
for all $A\in \mathcal{N}_{J},$ i.e., when $\delta _{\tau ,\sigma }$ has
positive and negative parts which precisely cancels out in the integration.
Then, we show that the bootstrap test has nontrivial asymptotic power
against local alternatives that converges at a rate slower than $n^{-1/2}$
to the null hypothesis.

\begin{thm}
\label{Thm5} \textit{Suppose that the conditions of Theorem \ref{Thm4} and
Assumption C3(ii) hold. Then for each sequence }$P_{n}\in \mathcal{P}%
_{n}^{0}(\lambda _{n})$, $n\geq 1,$\textit{\ which satisfies the local
alternative hypothesis }$(b_{n},\delta )$\textit{\ for some }$\delta \in 
\mathcal{D}$\textit{\ such that }$\mu _{1}(\delta )=0$\textit{\ and }$%
b_{n}=(r_{n,j}h^{-d/4})_{j=1}^J$,%
\begin{equation*}
\lim_{n\rightarrow \infty }P_{n}\{\hat{\theta}>c_{\alpha ,\eta }^{\ast
}\}=1-\Phi \left( z_{1-\alpha }-\mu _{2}(\delta )\right) ,
\end{equation*}%
\textit{where} 
\begin{equation*}
\mu _{2}(\delta )\equiv \frac{1}{2}\sum_{A\in \mathcal{N}_{J}}\int \delta
_{\tau ,\sigma }^{\top }(x)\psi _{A,\tau }^{(2)}(\mathbf{0};x)\delta _{\tau
,\sigma }(x)dQ(x,\tau ).
\end{equation*}
\end{thm}

The local power function depends on the limit of the curvature of the
function $\psi _{n,A,\tau }(\mathbf{y};x)$ at $\mathbf{y}=\mathbf{0},$ for
all $A\in \mathcal{N}_{J}$. When the function is strictly concave at $%
\mathbf{0}$ in the limit, $\psi _{A,\tau }^{(2)}(\mathbf{0};x)$ is positive
definite on $\mathcal{X}\times \mathcal{T}$, and in this case, the bootstrap
test has nontrivial power whenever $\delta _{\tau ,\sigma }(x)$ is nonzero
on a set whose intersection with $B_{n}^{0}(0)$ has Lebesgue measure greater
than $c>0$ for all $n\geq 1$, for some $c>0.$

From Theorems \ref{Thm4}\ and \ref{Thm5}, it is seen that the phenomenon of
dual convergence rates generally hold for our tests. To formally state the
result, define 
\begin{align*}
\mathcal{D}_{1} &\equiv \left\{ \delta \in \mathcal{D}:\mu _{1}(\delta )\neq
0\right\} \text{ and } \\
\mathcal{D}_{2} &\equiv \left\{ \delta \in \mathcal{D}:\mu _{1}(\delta )=0%
\text{ and }\mu _{2}(\delta )>0\right\} .
\end{align*}%
When $\liminf_{n \rightarrow \infty} Q(B_{n}^{0}(0))>0$, the set $\{\mathcal{%
D}_{1},\mathcal{D}_{2}\}$ becomes a partition of the space of Pitman
directions $\mathcal{D}$.

\begin{cor}
\label{partition-cor} Suppose that the conditions of Theorem 5 hold. Then
the bootstrap test has convergence rate $b_{n} = (r_{n,j}h^{-d/2})_{j=1}^J$
against $\mathcal{D}_{1}$, and convergence rate $%
b_{n}=(r_{n,j}h^{-d/4})_{j=1}^J$ against $\mathcal{D}_{2}$.
\end{cor}

When $r_{n,j}$'s diverge to infinity at the usual nonparametric rate $%
r_{n,j}=n^{1/2}h^{d/2}$ as in many kernel-based estimators, the test has a
parametric rate of convergence $b_{n}=\sqrt{n}$ and nontrivial local power
against $\mathcal{D}_{1}$. However, the test has a convergence rate slower
than the parametric rate against $\mathcal{D}_{2}$.

%When $r_{n,j}$'s diverge slower than the rate 
%$n^{1/2}h^{d/2}$ as in the case of kernel-based derivative estimators, the
%test has a convergence rate slower than the parametric rate. In Appendix \ref%
%{sec:ntm}, we present several nonparametric tests for monotonicity where $d=1
%$, $J=1$, and $r_{n,1} = n^{1/2}h^{3/2}$. In this case, the monotonicity
%tests have convergence rate with $b_{n}= n^{1/2}h$ against $\mathcal{D}_{1}$%
%, and convergence rate with $b_{n}= n^{1/2} h^{5/4}$ against $\mathcal{D}_{2}
%$.

\section{Monte Carlo Experiments}

\label{sec:mc:AS}

In this section, we report the finite-sample performance of our proposed
test for the Monte Carlo design considered in 
\citeasnoun[Section
10.3, hereafter AS]{Andrews/Shi:13a}. The null hypothesis has the form 
\begin{equation*}
H_{0}:\mathbf{E}(Y-\theta |X=x)\leq 0\;\;\text{for each $x\in \mathcal{X}$}
\end{equation*}%
with a fixed $\theta $. AS generated a random sample of $(Y,X)$ from the
following model: 
\begin{equation*}
Y=f(X)+U,
\end{equation*}%
where $X\sim \text{Unif}[-2,2]$, $U$ follows truncated normal such that $%
U_{i}=\min \{\max \{-3,\sigma \tilde{U}_{i}\},3\}$ with $\tilde{U}_{i}\sim
N(0,1)$ and $\sigma =1$, and $f(\cdot )$ is a function with an alternative
shape. AS considered two functions: 
\begin{align*}
f_{AS1}(x)& :=L\phi (x^{10}), \\
f_{AS2}(x)& :=L\cdot \max {\{\phi ((x-1.5)^{10}),\phi ((x+1.5)^{10})\}},
\end{align*}%
These two functions have steep slopes, $f_{AS1}$ being a roughly
plateau-shaped function and $f_{AS2}$ a roughly double-plateau-shaped
function, respectively. AS considered the following Monte Carlo designs: 
\begin{eqnarray*}
&&\text{DGP1: }f(x)=f_{AS1}(x)\text{ and }L=1\text{; }\ \ \ \text{DGP2: }%
f(x)=f_{AS1}(x)\text{ and }L=5\text{;} \\
&&\text{DGP3: }f(x)=f_{AS2}(x)\text{ and }L=1\text{; }\ \ \ \text{DGP4: }%
f(x)=f_{AS2}(x)\text{ and }L=5\text{.}
\end{eqnarray*}%
AS compared their tests with \citeasnoun[hereafter CLR]{CLR} and %
\citeasnoun{LSW}. The latter test uses conservative standard normal critical
values based on the least favorable configuration.

In this paper, we used the same statistic for \citeasnoun{LSW} as reported
in AS. Specifically, we used the $L_1$ version of the test with the inverse
standard error weight function. In implementing the test, we used $%
K(u)=(3/2)(1-(2u)^{2})I(|u|\leq 1/2)$ and $h= 2 \times \hat{s}_{X}\times
n^{-1/5}$, where $I(A)$ is the usual indicator function that has value one
if $A$ is true and zero otherwise and $\hat{s}_{X}$ is the sample standard
deviation of $X$. Thus, the only difference between the new test (which we
call LSW2) and \citeasnoun{LSW} (which we call LSW1) is the use of critical
values: LSW1 uses the standard normal critical values based on the least
favorable configuration, whereas LSW2 uses bootstrap critical values based
on the estimated contact set. See the next subsection for details regarding contact set estimation.

The experiments considered sample sizes of $n=100,250,500,1000$ and the
nominal level of $\alpha =0.05$. We performed 1000 Monte Carlo replications
in each experiment. The number of bootstrap replications was $200$.

The null hypothesis is tested on $\mathcal{X} = [-1.8,1.8]$. To compare
simulation results from AS, the coverage probability (CP) is computed at
nominal level 95\% when $\theta = \max_{x \in \mathcal{X}} f(x)$ and the
false coverage probability (FCP) is computed at nominal level 95\% when $%
\theta = \max_{x \in \mathcal{X}} f(x) - 0.02$.

%\clearpage
%\newpage
\begin{table}[htbp]
\caption{Results for Monte Carlo Experiments: Coverage Probability}
\label{table-mc1}
\begin{center}
\begin{tabular}{cccccccccc}
\hline\hline
&  & (1) & (2) & (3) & (4) & (5) & (6) & (7) & (8) \\ 
&  & \multicolumn{2}{c}{AS} & \multicolumn{2}{c}{CLR} & LSW1 & 
\multicolumn{3}{c}{LSW2} \\ 
& $n$ & CvM & KS & series & local &  & \multicolumn{1}{c}{$C_{\text{cs}} =
0.4$} & \multicolumn{1}{c}{$C_{\text{cs}} = 0.5$} & \multicolumn{1}{c}{$C_{%
\text{cs}} = 0.6$} \\ 
&  &  &  &  & linear &  & \multicolumn{3}{c}{} \\ \hline
DGP1 & 100 & .986 & .986 & .707 & .804 & 1.00 & .980 & .990 & .999 \\ 
& 250 & .975 & .973 & .805 & .893 & 1.00 & .951 & .960 & .971 \\ 
& 500 & .975 & .970 & .872 & .925 & 1.00 & .968 & .976 & .977 \\ 
& 1000 & .971 & .966 & .909 & .935 & 1.00 & .962 & .971 & .973 \\ \hline
DGP2 & 100 & 1.00 & 1.00 & .394 & .713 & 1.00 & .996 & .999 & 1.00 \\ 
& 250 & 1.00 & 1.00 & .683 & .856 & 1.00 & .953 & .963 & .975 \\ 
& 500 & 1.00 & 1.00 & .833 & .908 & 1.00 & .963 & .972 & .976 \\ 
& 1000 & 1.00 & 1.00 & .900 & .927 & 1.00 & .965 & .968 & .968 \\ \hline
DGP3 & 100 & .970 & .969 & .620 & .721 & 1.00 & .987 & .991 & .993 \\ 
& 250 & .969 & .964 & .762 & .854 & 1.00 & .952 & .965 & .973 \\ 
& 500 & .963 & .957 & .854 & .900 & 1.00 & .966 & .971 & .976 \\ 
& 1000 & .969 & .963 & .901 & .927 & 1.00 & .949 & .957 & .962 \\ \hline
DGP3 & 100 & .998 & .999 & .321 & .655 & 1.00 & .998 & .999 & 1.00 \\ 
& 250 & .997 & .998 & .612 & .826 & 1.00 & .952 & .965 & .976 \\ 
& 500 & .994 & .994 & .808 & .890 & 1.00 & .964 & .971 & .973 \\ 
& 1000 & .994 & .991 & .893 & .918 & 1.00 & .943 & .950 & .958 \\ \hline
\end{tabular}%
\end{center}
\par
\parbox{5in}{Notes:
Figures in columns (1)-(5) are from Table V of \citeasnoun{Andrews/Shi:13a}, whereas those in columns  (6)-(8) are
based 1000 Monte Carlo replications in each
experiment, with the number of bootstrap replications being $200$.
LSW1 refers to the test of
\citeasnoun{LSW}, which uses
conservative  standard normal critical values based on the least favorable configuration.
LSW2 refers to this paper that uses bootstrap critical values based on the estimated contact set.
The tuning parameter is chosen by  the rule $\hat{c}_n = C_{\text{cs}} \log \log (n) q_{1-0.1/\log(n)}(S_n^\ast)$, where $C_{\text{cs}} \in \{ 0.4, 0.5, 0.6 \}$.
 }
\end{table}

%\clearpage
%\newpage
\begin{table}[htbp]
\caption{Results for Monte Carlo Experiments: False Coverage Probability}
\label{table-mc2}
\begin{center}
\begin{tabular}{cccccccccc}
\hline\hline
&  & (1) & (2) & (3) & (4) & (5) & (6) & (7) & (8) \\ 
&  & \multicolumn{2}{c}{AS} & \multicolumn{2}{c}{CLR} & LSW1 & 
\multicolumn{3}{c}{LSW2} \\ 
& $n$ & CvM & KS & series & local &  & \multicolumn{1}{c}{$C_{\text{cs}} =
0.4$} & \multicolumn{1}{c}{$C_{\text{cs}} = 0.5$} & \multicolumn{1}{c}{$C_{%
\text{cs}} = 0.6$} \\ 
&  &  &  &  & linear &  & \multicolumn{3}{c}{} \\ \hline
DGP1 & 100 & .84 & .89 & .88 & .83 & .98 & .81 & .90 & .95 \\ 
& 250 & .57 & .67 & .82 & .69 & .92 & .44 & .49 & .54 \\ 
& 500 & .25 & .37 & .72 & .50 & .70 & .17 & .18 & .20 \\ 
& 1000 & .03 & .07 & .57 & .26 & .25 & .02 & .02 & .02 \\ \hline
DGP2 & 100 & 1.0 & 1.0 & .91 & .89 & .99 & .94 & .98 & 1.0 \\ 
& 250 & 1.0 & 1.0 & .85 & .73 & .96 & .48 & .54 & .62 \\ 
& 500 & .97 & .99 & .77 & .56 & .82 & .19 & .21 & .23 \\ 
& 1000 & .70 & .89 & .61 & .33 & .40 & .03 & .03 & .03 \\ \hline
DGP3 & 100 & .70 & .79 & .89 & .84 & .90 & .69 & .79 & .86 \\ 
& 250 & .30 & .46 & .83 & .66 & .65 & .27 & .32 & .35 \\ 
& 500 & .06 & .15 & .70 & .47 & .26 & .06 & .06 & .08 \\ 
& 1000 & .00 & .01 & .55 & .23 & .02 & .00 & .00 & .00 \\ \hline
DGP4 & 100 & .95 & .99 & .91 & .88 & .95 & .89 & .95 & .97 \\ 
& 250 & .66 & .83 & .86 & .70 & .75 & .30 & .35 & .42 \\ 
& 500 & .23 & .42 & .74 & .51 & .36 & .07 & .08 & .09 \\ 
& 1000 & .01 & .04 & .59 & .29 & .04 & .00 & .00 & .00 \\ \hline
\end{tabular}%
\end{center}
\par
\parbox{5in}{Notes:
See notes in Table \ref{table-mc1}.
Figures in columns (1)-(5) are
``CP-corrected'', where  those in columns  (6)-(8) are
not ``CP-corrected''.
 }
\end{table}

\subsection{Obtaining $\hat{c}_{n}$}\label{mc-tuning-parameters}

To construct $\hat{c}_{n}$, we suggest the following procedure. First, define%
\begin{equation*}
S_{n}^{\ast }\equiv \max \left\{ \sup_{(j,\tau ,x)} \hat{s}_{\tau ,j}^{\ast
}(x), \sqrt{\log n} \right\}.
\end{equation*}
Then, set 
\begin{align}  \label{cn-suggestion}
\hat{c}_{n}&=C_{\text{cs}}(\log \log n)q_{n}(S_{n}^{\ast }),
\end{align}%
where $q_{n}(S_{n}^{\ast })$ is the $(1- 0.1/\log n)$ quantile of the
bootstrap distribution of $S_{n}^{\ast }$, and  $C_{\text{cs}} \in \{ 0.4, 0.5, 0.6 \}$.

Although the rule-of-thumb for $\hat{c}_{n}$ in %
\eqref{cn-suggestion} is not completely data-driven, it has the advantage
that the scale of $\hat{u}_{\tau ,j}(x)$ is invariant, due to the term $%
q_{n}(S_{n}^{\ast })$; see \citeasnoun{CLR} for a similar idea.\footnote{%
Note that $q_{n}(S_{n}^{\ast })$ is the $(1- 0.1/\log n)$ quantile of the
supremum of $\hat{s}_{\tau ,j}^{\ast }(x)$ over $(j,\tau ,x)$ and that $(1-
0.1/\log n)$ converges to 1 as $n$ gets large. Thus, this observation leads
to the choice of $\hat{c}_{n} $ in \eqref{cn-suggestion} that is
proportional to $q_{n}(S_{n}^{\ast })$ times a very slowing growing term
such as $\log \log n$, to insure that $\hat{c}_{n}$ diverges to infinity but
as slowly as possible, while having the property of scale invariance.} This
data-dependent choice of $\hat{c}_{n}$ is encompassed by the theoretical
framework of this paper, while many other choices are also admitted.%
\footnote{%
See Assumption A4(ii) below for sufficient conditions for a data dependent
choice of $\hat{c}_{n}$. It is not hard to see that the conditions are
satisfied, once the uniform convergence rates of $\hat{v}_{\tau ,j}(x)$ and $%
\hat{\sigma}_{\tau ,j}(x)$ and their bootstrap versions hold as required in
Assumptions A3, A5, and B2 and B3.}

\subsection{Simulation results}

Tables \ref{table-mc1} and \ref{table-mc2} report the results of Monte Carlo
experiments. In each table, figures in columns (1)-(5) are from Table V of %
\citeasnoun{Andrews/Shi:13a}, whereas those in columns (6)-(8) are from our
Monte Carlo experiments. Table \ref{table-mc1} shows that coverage
probabilities of LSW2 are much closer to the nominal level than those of
LSW1. When $c=0.4$ and $n = 100$ or $250$, we see some under-coverage for
LSW2, but it disappears as $n$ gets larger. Table \ref{table-mc2} reports
the false coverage probabilities (FCPs). Figures in columns (1)-(5) are
``CP-corrected'' by AS, where those in columns (6)-(8) are not
``CP-corrected''. However, CP-correction would not change the results for
either $n \geq 500$ or $c \geq 0.5$ since in each of these cases, we have
over-coverage. We can see that in terms of FCPs, LSW2 performs much better
than LSW1 in all DGPs. Furthermore, the performance of LSW2 is equivalent to
that of AS for DGP1, DGP3, and DGP4, and is superior to AS for DGP2.
Overall, our simulation results show that our new test is a substantial
improved version of LSW1 and is now very much comparable to AS. The relative
poor performance of CLR in tables \ref{table-mc1} and \ref{table-mc2} are
mainly due to the experimental design. If the underlying function is sharply
peaked, as those in the reported simulations of \citeasnoun{CLR}, CLR
performs better than AS. In unreported simulations, we confirmed that CLR
performs better than LSW2 as well. This is very reasonable since CLR is
based on the sup-norm statistic, whereas ours is based on the one-sided $L_p$
norm. Therefore, we may conclude that AS, CLR, and LSW2 complement each
other.

\section{Empirical Example 1: Testing Functional Inequalities in Auction Models}\label{emp-example auction}

In this example, we go back to the auction environment of GPV mentioned earlier. We first state the testing problem formally, give the form of test statistic, and present empirical results.

\subsection{Testing Problem}

Suppose that the number $I$ of bidders can take two values, 2 and 3 (that
is, $I\in \{2,3\}$). For each $\tau $ such that $0<\tau <1$, let $q_{k}(\tau
|x)$ denote the $\tau $-th conditional quantile (given $X=x$) of the
observed equilibrium bid distribution when the number of bidders is $I=k$,
where $k=2,3$. A conditional version of Equation (5) of GPV (with $I_{1}=2$
and $I_{2}=3$ in their notation) provides the following testing
restrictions: 
\begin{equation}
\begin{split}
q_{2}(\tau |x)-q_{3}(\tau |x)& <0,\text{ and} \\
\underline{b}-2q_{2}(\tau |x)+q_{3}(\tau |x)& <0,
\end{split}
\label{test-intro-cond}
\end{equation}%
for any $\tau \in (0,1]$ and for any $x\in \text{supp}(X)$, where $\text{supp%
}(X)$ is the (common) support of $X$, and $\underline{b}$ is the left
endpoint of the support of the observed bids.\footnote{%
In GPV, it is assumed that for $I=k$, the support of the observed
equilibrium bid distribution is $[\underline{b},\overline{b}_{k}]\subset
\lbrack 0,\infty )$ with $\underline{b}<\overline{b}_{k}$, where $k=2,3$.
Note that $\underline{b}$ is common across $k$'s, while $\overline{b}_{k}$'s
are not.} The restrictions in (\ref{test-intro-cond}) are based on
conditionally exogenous participation for which the latent private value
distribution is independent of the number of bidders conditional on observed
characteristics ($X$), e.g. appraisal values. A slightly weaker version of %
\eqref{test-intro-cond} can be put into our general problem of testing the
null hypothesis:\footnote{%
If necessary, we may test the strict inequalities (3.1), instead of the weak
inequalities (3.2). However, such test would require a test statistic that
is different from ours and needs a separate treatment.}%
\begin{equation}
\begin{split}
v_{\tau ,1}(x)\equiv q_{2}(\tau |x)-q_{3}(\tau |x)& \leq 0,\text{ and} \\
v_{\tau ,2}(x)\equiv \underline{b}-2q_{2}(\tau |x)+q_{3}(\tau |x)& \leq 0,
\end{split}
\label{test-intro-cond1}
\end{equation}%
for any $(\tau ,x)\in \mathcal{T}\times \mathcal{X}\subset (0,1]\times \text{%
supp}(X)$.%\footnote{%
%Since we test \eqref{test-intro-cond} in weak inequalities, we are unable to
%detect the departure from the null hypothesis such that $v_{\tau,k}(x) = 0$
%for some $(\tau,x,k)$. We acknowledge this limitation and view this as a
%future research topic.}

The example in \eqref{test-intro-cond1} illustrates that in order to test
the implications of auction theory, it is essential to test the null
hypothesis uniformly in $\tau$ and $x$. More specifically, testing for a
wide range of $\tau$ is important because testable implications are
expressed in terms of conditional stochastic dominance relations.
Furthermore, testing the relations uniformly over $x$ is natural since
theoretical predictions given by conditionally exogenous participation
should hold for any realization of observed auction heterogeneity. It also
shows that it is important to go beyond the $J=1$ case and to include a
general $J > 1$. In fact, if the number of bidders can take more than two
values, there could be many more functional inequalities (see Corollary 1 of
GPV). Finally, we note that $v_{\tau,1}(x)$ and $v_{\tau,2}(x)$ are not
forms of conditional moment inequalities and each involves two different
conditional quantile functions indexed by $\tau$. Therefore, tests developed
for conditional moment inequalities are not directly applicable to this
empirical example. There exist related but distinct papers regarding this
empirical example. See, e.g., \citeasnoun{Marmer-et-al:13} who developed a
nonparametric test for selective entry, and \citeasnoun{Gimens/Guerre:13}
who proposed augmented quantile regression for first-price auction models.

\subsection{Test Statistic}

To implement the test, it is necessary to estimate conditional quantile
functions. In estimation of $q_{j}(\tau |x)$, $j=2,3$, we may use a local
polynomial quantile regression estimator, say $\widehat{q}_{j}(\tau |x)$.
Now write 
\begin{align*}
\hat{v}_{\tau ,1}(x)& =\hat{q}_{2}(\tau |x)-\hat{q}_{3}(\tau |x), \\
\hat{v}_{\tau ,2}(x)& =\hat{\underline{b}}-2\hat{q}_{2}(\tau |x)+\hat{q}%
_{3}(\tau |x),
\end{align*}%
where $\hat{\underline{b}}$ is a consistent estimator of $\underline{b}$.%
\footnote{%
In our application, we set $\hat{\underline{b}}$ to be the observed minimum
value.} Then testing \eqref{test-intro-cond1} can be carried out using $\{%
\hat{v}_{\tau ,j}(x):j=1,2\}$ based on our general framework. In this
application, our test statistics take the following forms: 
\begin{align}  \label{test-form-cqi}
\begin{split}
\hat{\theta}_{\text{sum}}& =\int_{\mathcal{X}\times \mathcal{T}}\left[ r_{n}%
\hat{v}_{\tau ,1}(x)\right] _{+}^{p}dQ(x,\tau ) + \int_{\mathcal{T}\times 
\mathcal{X}}\left[ r_{n}\hat{v}_{\tau ,2}(x)\right] _{+}^{p}dQ(x,\tau ),\ 
\text{or} \\
\hat{\theta}_{\text{max}}& =\int_{\mathcal{X}\times \mathcal{T}}\left( \max
\left\{ \left[ r_{n}\hat{v}_{\tau ,1}(x)\right] _{+},\left[ r_{n}\hat{v}%
_{\tau ,2}(x)\right] _{+}\right\} \right) ^{p}dQ(x,\tau ),
\end{split}%
\end{align}%
where $r_{n}=n^{1/2}h^{d/2}$.
Note that in \eqref{test-form-cqi}, we set $\hat{\sigma}_{\tau ,j}(x)\equiv
1 $.

\subsection{Details on Estimating Conditional Quantile Functions}

\label{auction-example-detail}

Assume that $q_{k}(\tau |x)$ is $(r+1)$-times
continuously differentiable with respect to $x$, where $r\geq 1$. We use a
local polynomial estimator $\widehat{q}_{k}(\tau |x)$. For $u\equiv
(u_{1},\ldots ,u_{d})$, a $d$-dimensional vector of nonnegative integers,
let $[u]=u_{1}+\cdots +u_{d}$. Let $A_{r}$ be the set of all $d$-dimensional
vectors $u$ such that $[u]\leq r$, and let $|A_{r}|$ denote the number of
elements in $A_{r}$. For $z=(z_{1},\cdot \cdot \cdot ,z_{d})^{\top }\in 
\mathbf{R}^{d}$ with $u=(u_{1},\cdot \cdot \cdot ,u_{d})^{\top }\in A_{r}$,
let $z^{u}=\prod_{m=1}^{d}z_{m}^{u_{m}}$. Now define $c(z)=(z^{u})_{u\in
A_{r}},$ for $z\in \mathbf{R}^{d}$. Note that $c(z)$ is a vector of
dimension $|A_{r}|$.

Let $\{(B_{\ell i},X_{i},L_{i}):\ell =1,\ldots ,L_{i},i=1,\ldots ,n\}$
denote the observed data, where $\{B_{\ell i}:\ell =1,\ldots ,L_{i}\}$
denotes the $L_{i}$ number of observed bids in the $i$-th auction, $X_{i}$ a
vector of observed characteristics for the $i$-th auction, and $L_{i}$ the
number of bids for the $i$-th auction, taking values from $\mathbb{N}%
_{L}\equiv \{2,\cdot \cdot \cdot ,\bar{L}\}$. In our application, $\bar{L}=3$%
.

Assume that the data $\{(B_{\ell i},X_{i},L_{i}):\ell =1,\ldots
,L_{i},i=1,\ldots ,n\}$ are i.i.d. over $i$ and that $B_{\ell i}$'s are also
i.i.d. over $\ell $ conditional on $X_{i}$ and $L_{i}$. To implement the
test, it is necessary to estimate $\underline{b}$. In our application, we
use $\underline{\hat{b}}=\min \{B_{\ell i}:\ell =1,\ldots ,L_{i},i=1,\ldots
,n\}$, that is the overall sample minimum.

For each $x=(x_{1},\ldots ,x_{d})$, the $r$-th order local polynomial
quantile regression estimator of $q_{k}(\tau |x)$ can be obtained by
minimizing 
\begin{equation*}
S_{n,x,\tau ,k}(\gamma )\equiv \sum_{i=1}^{n}1\{L_{i}=k\}\sum_{\ell
=1}^{L_{i}}l_{\tau }\left[ B_{\ell i}-\gamma ^{\top }c\left( \frac{X_{i}-x}{h%
}\right) \right] K\left( \frac{x-X_{i}}{h}\right)
\end{equation*}%
with respect to $\gamma \in \mathbf{R}^{|A_{r}|}$, where $l_{\tau }(u)\equiv
\{|u|+(2\tau -1)u\}/2$ for any $u\in \mathbf{R},$ and $K(\cdot )$ is a $d$%
-dimensional kernel function and $h$ a bandwidth. More specifically, let $%
\widehat{q}_{k}(\tau |x)=\mathbf{e}_{1}^{\top }\hat{\gamma}_{k}(x)$, where $%
\hat{\gamma}_{k}(x)\equiv \arg \min_{\gamma \in \mathbf{R}%
^{|A_{r}|}}S_{n,x,\tau ,k}(\gamma )$ and $\mathbf{e}_{1}$ is a column vector
whose first entry is one, and the rest zero. Note that all bids are combined
in each auction since we consider symmetric bidders.

\subsection{Primitive Conditions}

\label{auction-example-low-level-condition}

Let us present primitive conditions for the auction example of GPV. Let $%
\mathcal{P}$ denote the collection of the potential joint distributions of $%
(B^{\top },X^{\top },L)^{\top }$ and define $\mathcal{V}=\mathcal{T}\times 
\mathcal{P}$ as before. 

For $u=(u_{1},\cdot
\cdot \cdot ,u_{d})^{\top }\in A_{r}$, and $r+1$ times differentiable map $f$
on $\mathbf{R}^{d}$, we define the following derivative:%
\begin{equation*}
(D^{u}f)(x)\equiv \frac{\partial ^{\lbrack u]}}{\partial x_{1}^{u_{1}}\cdot
\cdot \cdot \partial x_{d}^{u_{d}}}f(x),
\end{equation*}%
where $[u]=u_{1}+\cdot \cdot \cdot +u_{d}$. Then we define $\gamma _{\tau
,k}(x)\equiv \left( \gamma _{\tau ,k,u}(x)\right) _{u\in A_{r}}$, where%
\begin{equation*}
\gamma _{\tau ,k,u}(x)\equiv \frac{1}{u_{1}!\cdot \cdot \cdot u_{d}!}%
D^{u}q_{k}(\tau |x).
\end{equation*}

In order to reduce the redundancy of the statements, let us introduce the
following definitions.

\begin{defn}
Let $\mathcal{G}$ be a set of functions $g_{v}:\mathbf{R}^{m}\rightarrow 
\mathbf{R}^{s}\ $indexed by a set $\mathcal{V}$, and let $S\subset \mathbf{R}%
^{m}$ be a given set and for $\varepsilon >0$, let $S_{v}(\varepsilon )$ be
an $\varepsilon $-enlargement of $S_{v}=\{x\in S:(x,v)\in S\times V\}$,
i.e., $S_{v}(\varepsilon )=\{x+a:x\in S$ and $a\in \lbrack -\varepsilon
,\varepsilon ]^{m}\}.$ Then we define the following conditions for $\mathcal{%
G}$: 
\begin{enumerate}
\item[(a)] B($S,\varepsilon $): $g_{v}$ is bounded on $S_{v}(\varepsilon
)\ $uniformly over $v\in \mathcal{V}.$
\item[(b)] BZ$(S,\varepsilon $): $g_{v}$ is bounded away from zero on $%
S_{v}(\varepsilon )\ $uniformly over $v\in \mathcal{V}.$
\item[(c)] BD($S,\varepsilon ,r$): $\mathcal{G}$ satisfies B$%
(S,\varepsilon )$ and $g_{v}$ is $r$ times continuously differentiable on $%
S_{v}(\varepsilon )$ with derivatives bounded on $S_{v}(\varepsilon )$
uniformly over $v\in \mathcal{V}.$
\item[(d)] BZD($S,\varepsilon ,r$): $\mathcal{G}$ satisfies BZ$%
(S,\varepsilon )$ and $g_{v}$ is $r$ times continuously differentiable on $%
S_{v}(\varepsilon )$ with derivatives bounded on $S_{v}(\varepsilon )$
uniformly over $v\in \mathcal{V}$.
\item[(e)] LC: $g_{v}$ is Lipschitz continuous with Lipschitz coefficient
bounded uniformly over $v\in V$.
\end{enumerate}
\end{defn}

Let $\mathcal{P}$ denote the collection of the potential joint distributions
of $(B^{\top },X^{\top },L)^{\top }$ and define $\mathcal{V}=\mathcal{T}%
\times \mathcal{P}$, and for each $k\in \mathbb{N}_{L},$%
\begin{eqnarray}
\mathcal{G}_{q}(k) &=&\left\{ q_{k}(\tau |\cdot ):(\tau ,P)\in \mathcal{V}%
\right\} ,  \label{fns} \\
\mathcal{G}_{f}(k) &=&\left\{ f_{\tau ,k}(\cdot |\cdot ):(\tau ,P)\in 
\mathcal{V}\right\} ,  \notag \\
\mathcal{G}_{L}(k) &=&\left\{ P\left\{ L_{i}=k|X_{i}=\cdot \right\} :P\in 
\mathcal{P}\right\} ,\text{ and}  \notag \\
\mathcal{G}_{f} &=&\left\{ f(\cdot ):P\in \mathcal{P}\right\} ,  \notag
\end{eqnarray}%
where $f_{\tau ,k}(0|x)$ being the conditional density of $B_{li}-q_{k}(\tau
|X_{i})$ given $X_{i}=x$ and $L_{i}=k.$ Also, define%
\begin{equation}
\mathcal{G}_{f,2}(k)=\left\{ f_{\cdot ,k}(\cdot |\cdot ):P\in \mathcal{P}%
\right\} \text{ and\ }\mathcal{G}_{\gamma }(k)=\left\{ \gamma _{\cdot
,k}(\cdot ):P\in \mathcal{P}\right\} .  \label{fns2}
\end{equation}%
We make the following assumptions.

\begin{AssumptionAUC}
(i) $\mathcal{G}_{f}$ satisfies BD($\mathcal{S},\varepsilon ,1)$.\newline
\noindent (ii) For each $k\in \mathbb{N}_{L}$, $\mathcal{G}_{f}(k)$ and $%
\mathcal{G}_{L}(k)\ $satisfy BD($\mathcal{S},\varepsilon ,1)$ and BZD($%
\mathcal{S},\varepsilon ,1).$\newline
\noindent (iii) For each $k\in \mathbb{N}_{L}$, $\mathcal{G}_{q}(k)\ $%
satisfies BD($\mathcal{S},\varepsilon ,r+1)$  for some $r>3d/2-1.$ \newline
\noindent (iv) For each $k\in \mathbb{N}_{L}$, $\mathcal{G}_{f,2}(k)$ and $%
\mathcal{G}_{\gamma }(k)\ $satisfy LC.
\end{AssumptionAUC}

Assumption AUC1(i) and (iii) are standard assumptions used in the local
polynomial approach where one approximates $q_{k}(\cdot |x)$ by a linear
combination of its derivatives through Taylor expansion, except only that
the approximation here is required to behave well uniformly over $P\in 
\mathcal{P}$. Assumption AUC1(ii) is made to prevent the degeneracy of the
asymptotic linear representation of $\hat{\gamma}_{\tau ,k}(x)-\gamma _{\tau
,k}(x)$ that is uniform over $x\in \mathcal{S}_{\tau }(\varepsilon ),\ \tau
\in \mathcal{T}$ and over $P\in \mathcal{P}$. Assumption AUC(iv) requires
that the conditional density function of $B_{li}-q_{k}(\tau |X_{i})$ given $%
X_{i}=x$ and $L_{i}=k$ and $\gamma _{\tau ,k}(\cdot )$ behave smoothly as we
perturb $\tau $ locally. This requirement is used to control the size of the
function spaces indexed by $\tau $, so that when the stochastic convergence
of random sequences holds, it is ensured to hold uniformly in $\tau $.

 Assumption
AUR2 lists conditions for the kernel function and the bandwidth.

\begin{AssumptionAUC}
(i)  $K$ is compact-supported, nonnegative, bounded, and Lipschitz continuous
on the interior of its support, $\int K(u)du=1$, and $\int K\left( u\right)
||u||^{2}du>0.$\newline
(ii) $n^{-1/2}h^{-3(d+\nu )/2}+\sqrt{n}h^{r+d+1}/\sqrt{\log n}\rightarrow 0,$
as $n\rightarrow \infty ,\ $for some small $\nu >0$, with $r$ in Assumption
AUC1(iii).
\end{AssumptionAUC}

As for Assumption AUC2(ii), the choice of $h=n^{-s}$ with the condition $%
1/(2(r+d+1))<s<1/(3(d+\nu ))$ satisfies the bandwidth condition. The small $%
\nu >0$ there is introduced to satisfy Assumption \ref{assumption-B4}.

\begin{AssumptionAUC}
$\underline{\hat{b}}=\underline{b}+o_{P}\left( n^{-1/2}\right) ,\ \mathcal{P}
$-uniformly.
\end{AssumptionAUC}

\begin{AssumptionAUC}
(i) There exist nonstochastic sequences $c_{n,L}>0$ and $c_{n,U}>0$ such
that $c_{n,L}<c_{n,U}$, and as $n\rightarrow \infty ,$%
\begin{equation*}
\inf_{P\in \mathcal{P}}P\left\{ \hat{c}_{n}\in \lbrack
c_{n,L},c_{n,U}]\right\} \rightarrow 1,\text{ and\ }\sqrt{\log n}%
/c_{n,L}+n^{-1/2}h^{-d/2}c_{n,U}\rightarrow 0.
\end{equation*}%
\newline
(ii) For each $a\in (0,1/2)$, there exists a compact set $\mathcal{C}%
_{a}\subset \mathbf{R}^{d}$ such that 
\begin{equation*}
0<\inf_{P\in \mathcal{P}}P\{X_{i}\in \mathbf{R}^{d}\backslash \mathcal{C}%
_{a}\}\leq \sup_{P\in \mathcal{P}}P\{X_{i}\in \mathbf{R}^{d}\backslash 
\mathcal{C}_{a}\}<a.
\end{equation*}
\end{AssumptionAUC}

 Assumption AUC3 holds
in general because the extreme order statistic is super-consistent with the $%
n^{-1}$ rate of convergence.  Assumption
AUC4(i) requires that $\hat{c}_{n}$ increase faster than $\sqrt{\log n}$ but
slower than $r_{n}$ with probability approaching one. Assumption AUC4(ii)
imposes some regularity on the behavior of the support of $X_{i}$ as $P$
moves around $\mathcal{P}$.

The following result establishes the uniform validity of the bootstrap test.

\begin{TheoremAUC}
\label{thm:AUC1} Suppose that Assumptions AUC1-AUC4 hold. Then 
\begin{equation*}
\limsup_{n\rightarrow \infty }\sup_{P\in \mathcal{P}_{0}}P\left\{
\hat{\theta}_{\textrm{sum}}>c_{\alpha ,\eta }^{\ast }\right\} \leq \alpha 
\; \text{ and } \;
\limsup_{n\rightarrow \infty }\sup_{P\in \mathcal{P}_{0}}P\left\{
\hat{\theta}_{\textrm{max}}>c_{\alpha ,\eta }^{\ast }\right\} \leq \alpha. 
\end{equation*}
\end{TheoremAUC}

Theorem AUC1 gives  the uniform asymptotic validity of the bootstrap
test. It is straightforward to characterize 
the class of distributions under the null
hypothesis renders the test asymptotically exact, using Theorem \ref{Thm2}. 
We omit the details for the brevity of the paper.

Our asymptotic approximation is based on plugging the asymptotic linear
expansion directly. There is a recent proposal by %
\citeasnoun{Mammen-et-al:13}, who developed nonparametric tests for
parametric specifications of regression quantiles and showed that
calculating moments of linear expansions of nonparametric quantile
regression estimators might work better in a sense that their approach
requires less stringent conditions for the dimension of covariates and the
choice of the bandwidth. It is an interesting future research topic whether
their ideas can be applied to our setup.

\subsection{Empirical Results}

We now present empirical results using the timber auction data used in %
\citeasnoun{Lu/Perrigne:08}.\footnote{%
The data are available on the \textit{Journal of Applied Econometrics}
website.} They used the timer auction data to estimate bidders' risk
aversion, taking advantage of bidding data from ascending auctions as well
as those from first-price sealed-bid auctions. In our empirical example, we
use only the latter auctions with 2 and 3 bidders, and we use the appraisal
value as the only covariate $X_{i}$ ($d=1$). Summary statistics and visual
presentation of data are given in Table \ref{table-auction-sum-stat} and
Figure \ref{figure-auction-1}. It can be seen from Table \ref%
{table-auction-sum-stat} that average bids become higher as the number of
bidders increases from 2 to 3. The top panel of Figure \ref{figure-auction-1}
suggests that this more aggressive bidding seems to be true, conditional on
appraisal values.

%\clearpage
\begin{table}[htbp]
\caption{Summary Statistics for Empirical Example 1}
\label{table-auction-sum-stat}
\begin{center}
\begin{tabular}{lrrrr}
\hline\hline
& \multicolumn{2}{c}{2 bidders} & \multicolumn{2}{c}{3 bidders} \\ 
& \multicolumn{2}{c}{(Sample size $=107$)} & \multicolumn{2}{c}{(Sample size 
$=108$)} \\ \hline
&  & Standard &  & Standard \\ 
& Mean & Deviation & Mean & Deviation \\ \hline
Appraisal Value & 66.0 & 47.7 & 53.3 & 41.4 \\ 
Highest bid & 96.1 & 55.6 & 100.8 & 56.7 \\ 
Second highest bid & 80.9 & 49.2 & 83.1 & 51.5 \\ 
Third highest bid &  &  & 69.4 & 44.6 \\ \hline
\end{tabular}%
\end{center}
\par
\parbox{5in}{Notes: Bids and appraisal values are given in dollars per thousand
board-feet (MBF).
Source:  Timber auction data are from
 the \textit{Journal of Applied Econometrics} website.}
\end{table}

%\clearpage
\begin{figure}[htbp]
\caption{Data for Empirical Illustration for Empirical Example 1}
\label{figure-auction-1}
\begin{center}
\hspace{1ex}
\par
\makebox{
\includegraphics[origin=bl,scale=.3,angle=90]{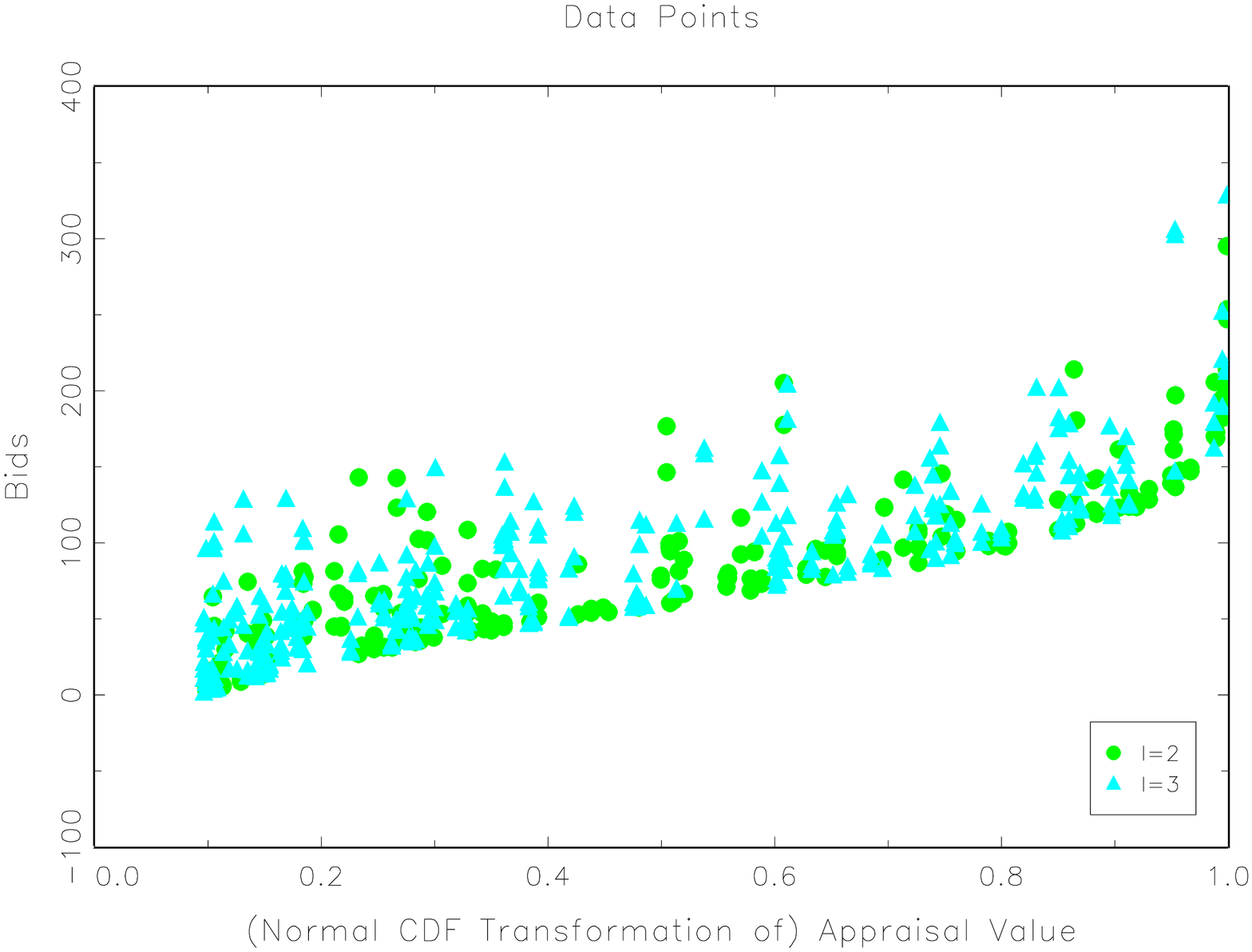}
}
\par
\hspace{1ex}
\par
\makebox{
\includegraphics[origin=bl,scale=.3,angle=90]{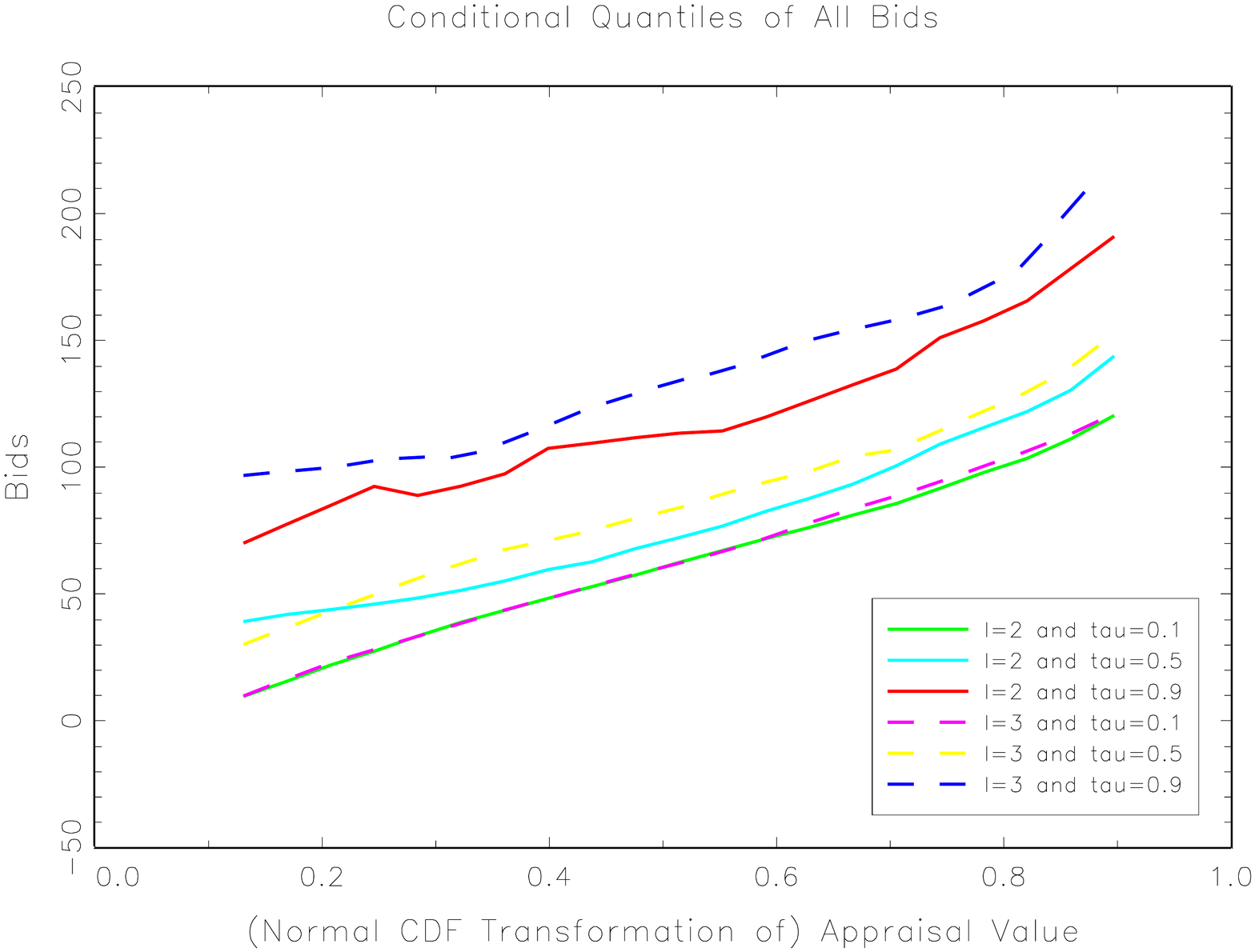}
}
\end{center}
\par
\parbox{5in}{
Note: The top panel of the figure shows observations and the bottom panel
depicts local linear quantile regression estimates.}
\end{figure}

%\clearpage
\begin{figure}[htbp]
\caption{Estimates of ${v}_{\protect\tau,1}(x)$ and ${v}_{\protect\tau,2}(x)$
for Empirical Example 1}
\label{figure-auction-2}
\begin{center}
\hspace{1ex}
\par
\makebox{
\includegraphics[origin=bl,scale=.3,angle=90]{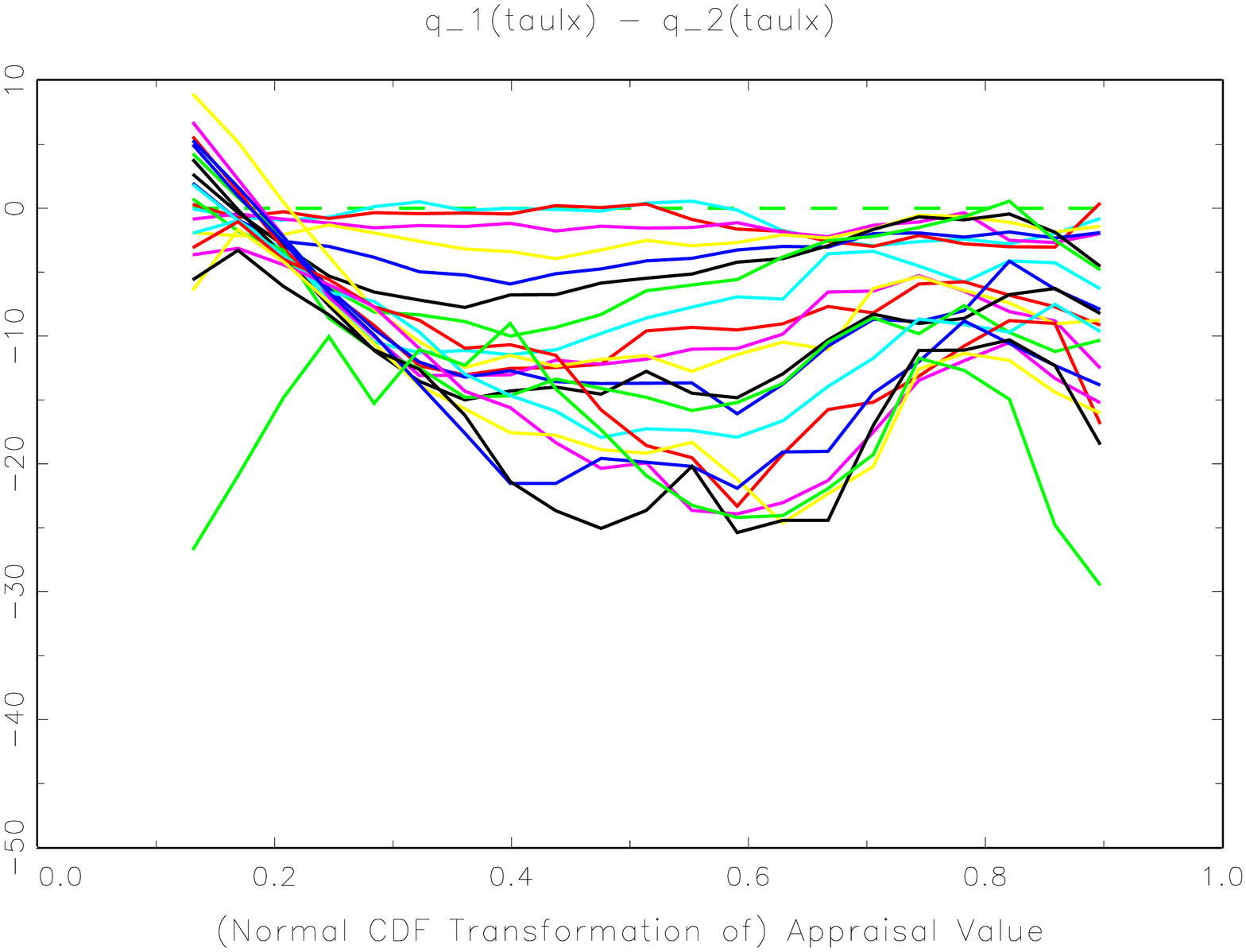}
}
\par
\hspace{1ex}
\par
\makebox{
\includegraphics[origin=bl,scale=.3,angle=90]{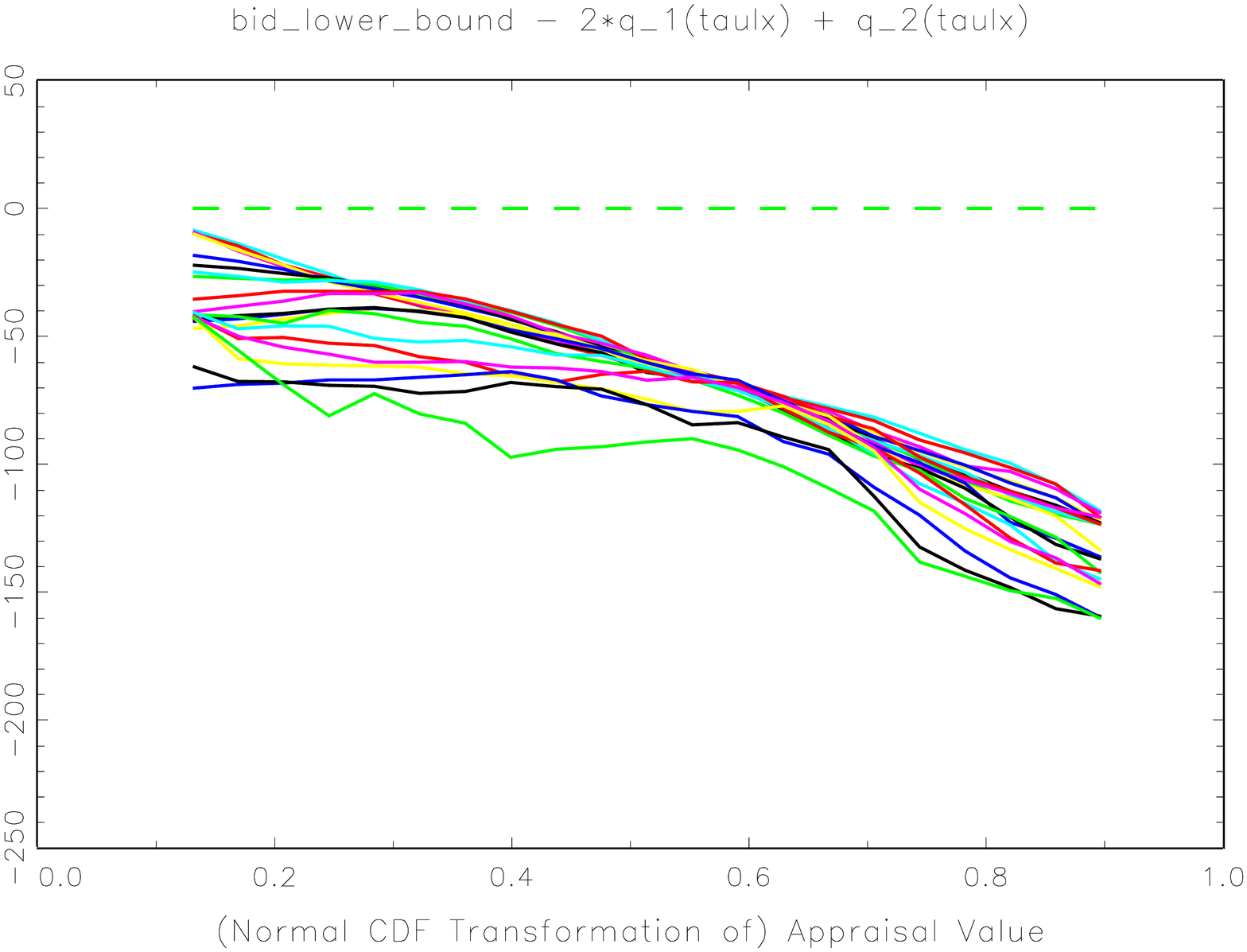}
}
\end{center}
\par
\parbox{5in}{
Note: The top and bottom panels of the figure show estimates of ${v}_{\tau,1}(x)$ and ${v}_{\tau,2}(x)$, respectively, where $\hat{v}_{\tau,1}(x) = \hat{q}_1(\tau|x) - \hat{q}_2(\tau|x)$ and $\hat{v}_{\tau,2}(x) = \underline{b} - 2 \hat{q}_1(\tau|x) + \hat{q}_2(\tau|x)$.
}
\end{figure}

Before estimation, the covariate was transformed to lie between 0 and 1 by
studentizing it and then applying the standard normal CDF transformation.
The bottom panel of Figure \ref{figure-auction-1} shows local linear
estimates of conditional quantile functions at $\tau = 0.1, 0.5, 0.9$.%
\footnote{%
Specifically, the conditional quantile functions ${q}_{2}(\tau |x)$ and ${q}%
_{3}(\tau |x)$ are estimated via the local linear quantile regression
estimator with the kernel function $K(u)=1.5[1-(2u)^{2}]\times 1\{|u|\leq
0.5\}$ and the bandwidth $h=0.6$. See Section \ref{auction-example-detail}
for more details on estimating conditional quantile functions.} In this
figure, estimates are only shown between the 10\% and 90\% sample quantiles
of the covariate. 

On one hand, the 10\% conditional quantiles are almost
identical between auctions with two bidders ($I=2$) and those with three
bidders ($I=3$). On the other hand, the 50\% and 90\% conditional quantiles
are higher with three bidders for most values of appraisal values. There is
a crossing of two conditional median curves at the lower end of appraisal
values.

To check whether inequalities in \eqref{test-intro-cond1} hold in this
empirical example, we plot estimates of ${v}_{\tau,1}(x)$ and ${v}%
_{\tau,2}(x)$ in Figure \ref{figure-auction-2}. The top panel of the figure
shows that 20 estimated curves of ${v}_{\tau,1}(x)$, each representing a
particular conditional quantile, ranging from the 10th percentile to the
90th percentile. There are strictly positive values of ${v}_{\tau,1}(x)$ at
the lower end of appraisal values. The bottom panel of Figure \ref%
{figure-auction-2} depicts 20 estimated curves of ${v}_{\tau,2}(x)$%
, showing that they are all strictly negative. The test based on %
\eqref{test-form-cqi} can tell formally whether positive values of ${v}%
_{\tau,1}(x)$ at the lower end of appraisal values can be viewed as evidence
against economic restrictions imposed by \eqref{test-intro-cond1}.

We considered both the $L_{1}$ and $L_{2}$ test statistics described in %
\eqref{test-form-cqi}. We set $\mathcal{T}$ to be the interval between the
10th and 90th percentiles of the covariate, and also set $\mathcal{X}
=[0.1,0.9]$. The contact set was estimated with $\hat{c}_n = C_{\text{cs}}
\log \log (n) q_{1-0.1/\log(n)}(S_n^\ast)$ with $r_{n}=\sqrt{nh}$. We
checked the sensitivity to the tuning parameters with $C_{\text{cs}}\in
\{0.5,1,1.5\}$ and $h\in \{0.3,0.6,0.9\}$. All cases resulted in bootstrap
p-values of 1, thereby suggesting that positive values of ${v}_{\tau ,1}(x)$
at the lower end of appraisal values cannot be interpreted as evidence
against the null hypothesis beyond random sampling errors. Therefore, we
have not found any evidence against economic implications imposed by %
\eqref{test-intro-cond1}.

\section{Empirical Example 2: Testing Functional Inequalities in the Context of Wage Inequality}\label{emp-example wage}

In this section, we give an empirical example regarding testing functional
inequalities via differences-in-differences in conditional quantiles,
inspired by \citeasnoun{Acemoglu:Autor:11}.

\subsection{Testing Problem}

Figures 9a-9c in \citeasnoun{Acemoglu:Autor:11} depict changes in log hourly
wages by percentile relative the median. Specifically, they consider the
following differences-in-differences in quantiles: 
\begin{align*}
\Delta_{t,s}(\tau,x)\equiv [q_t(\tau|x) - q_s(\tau|x)] - [q_t(0.5|x) -
q_s(0.5|x)]
\end{align*}
for time periods $t$ and $s$ and for quantiles $\tau$, where $q_t(\tau|x)$
denotes the $\tau$-quantile of log hourly wages conditional on $X=x$ in year 
$t$. \citeasnoun{Acemoglu:Autor:11} consider males and females together in
Figure 9a, males only in Figure 9b, and females only in Figure 9c. Thus, in
their setup, the only covariate $X$ is gender.

Figures 9a-9c in \citeasnoun{Acemoglu:Autor:11} suggest that (1) $%
\Delta_{1988,1974}(\tau,x) \geq 0$ for quantiles above the median, but $%
\Delta_{1988,1974}(\tau,x) \leq 0$ for quantiles below the median (hence,
widening the wage inequality, while the lower quantiles losing most), and
that (2) $\Delta_{2008,1988}(\tau,x) \geq 0$ for most of quantiles (hence,
`polarization' of wage growth, while middle quantiles losing most). In this
subsection, we consider testing 
\begin{align}  \label{test-CPS}
H_0: \Delta_{t,s}(\tau,x) \geq 0 \; \forall (x,\tau) \in \mathcal{X} \times 
\mathcal{T},
\end{align}
with a continuous covariate (age in our empirical example), where $(t,s) = (1988,1974)$ or $(t,s) =
(2008,1988)$.\footnote{%
Note that $H_0$ in \eqref{test-CPS} includes the case $\Delta_{t,s}(\tau,x)
\equiv 0$, which does not correspond to the notion of polarization. In view
of this, our null hypothesis in \eqref{test-CPS} can be regarded as a weak
form of polarization hypothesis, whereas a more strict version can be
written as the inequality in \eqref{test-CPS} holds strictly for some high
and low quantiles.} Note that degeneracy of the test statistic could occur
if the contact set consists of values of $(x,\tau)$ only around $\tau=0.5$.
Therefore, the uniformity of our test could be potentially important in this
example.

\subsection{Test Statistic}

To implement the test, we again use a local polynomial quantile regression
estimator, say $\hat{q}_t(\tau|x)$. Then $\Delta_{t,s}(\tau,x)$ can be
estimated by 
\begin{align*}
\hat{\Delta}_{t,s}(\tau,x)\equiv [\hat{q}_t(\tau|x) - \hat{q}_s(\tau|x)] - [%
\hat{q}_t(0.5|x) - \hat{q}_s(0.5|x)].
\end{align*}
Then testing \eqref{test-CPS} can be carried out using 
\begin{align}  \label{test-form-CPS}
\hat{\theta}_{t,s} &\equiv \int_{ \mathcal{X} \times \mathcal{T} } \left[
r_n \hat{v}_{\tau,t,s}(x) \right]_{+}^p dQ(x,\tau ),
\end{align}
where $\hat{v}_{\tau,t,s}(x) = -\hat{\Delta}_{t,s}(\tau,x)$.\footnote{%
Note that the null hypothesis is written as positivity in \eqref{test-CPS}.
Hence $\hat{v}_{\tau,t,s}(x)$ is defined accordingly.} Here, to reflect
different sample sizes between two time periods, we set 
\begin{align*}
r_n = \sqrt{ \frac{(n_t h_t) \times (n_s h_s)}{(n_t h_t) + (n_s h_s)}},
\end{align*}
where $n_j$ and $h_j$ are the sample size and the bandwidth used for
nonparametric estimation for year $j=t,s$.

%\clearpage
\begin{table}[htbp]
\caption{Summary Statistics for Empirical Example 2}
\label{table-cps-sum-stat}
\begin{center}
\begin{tabular}{lrrr}
\hline\hline
Year & 1974 & 1988 & 2008 \\ \hline
Log Real Hourly Wages & 2.780 & 2.769 & 2.907 \\ 
Age in Years & 35.918 & 35.501 & 39.051 \\ 
Sample Size & 19575 & 64682 & 48341 \\ \hline
\end{tabular}%
\end{center}
\par
\parbox{5in}{
Notes:
The sample is restricted to white males, with age between 16 and 64.
Entries for log real hourly wages and age  show CPS sample weighted means.
Source: May/ORG CPS data extract from David Autor's web site.
 }
\end{table}

\begin{figure}[htbp]
\caption{Changes in Log Hourly Wages by Percentile Relative to the Median}
\label{figure-cps-1}
\begin{center}
%\hspace{1ex}
\makebox{
\includegraphics[origin=bl,scale=.5,angle=90]{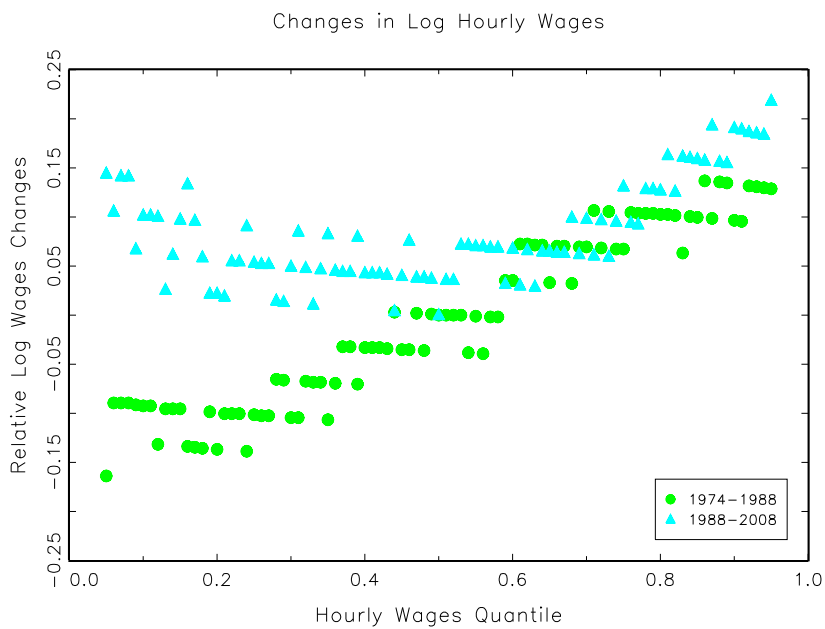}
}
\end{center}
\par
\parbox{5in}{
Notes: The figure shows differences-in-differences in quantiles of log hourly wages, measured by
$[q_t(\tau) - q_s(\tau)] - [q_t(0.5) - q_s(0.5)]$. Triangles correspond to changes from 1974 to 1988, whereas
circles those from 1988 to 2008.
All quantiles are computed using CPS sample weight.
Source: May/ORG CPS data extract from David Autor's web site.
 }
\end{figure}

\begin{figure}[htbp]
\caption{Estimates of $\hat{v}_{\protect\tau,t,s}(x)$}
\label{figure-cps-2}
\begin{center}
%\hspace{1ex}
\makebox{
\includegraphics[origin=bl,scale=.3,angle=90]{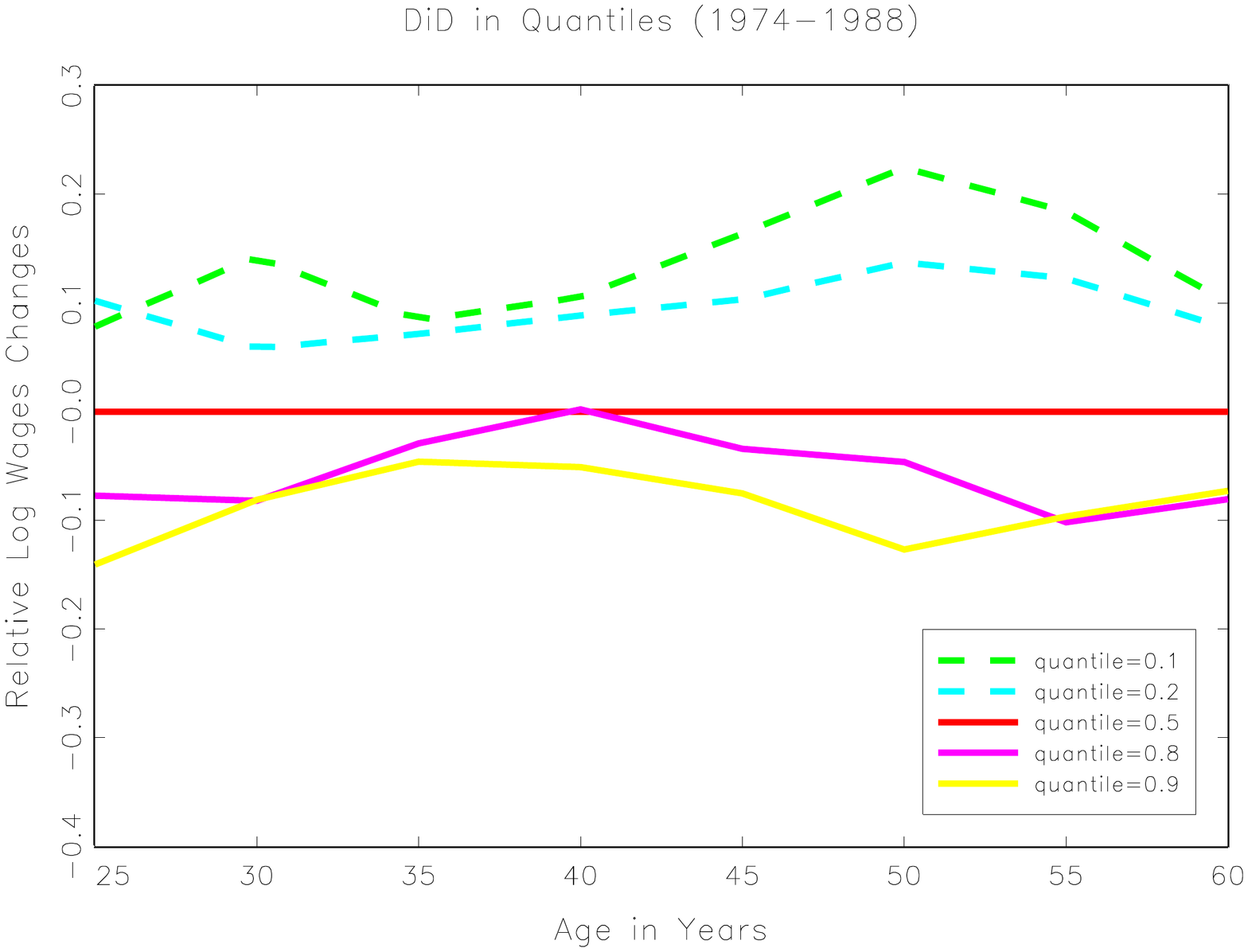}
} %\par
%\hspace{1ex}
%\par
\makebox{
\includegraphics[origin=bl,scale=.3,angle=90]{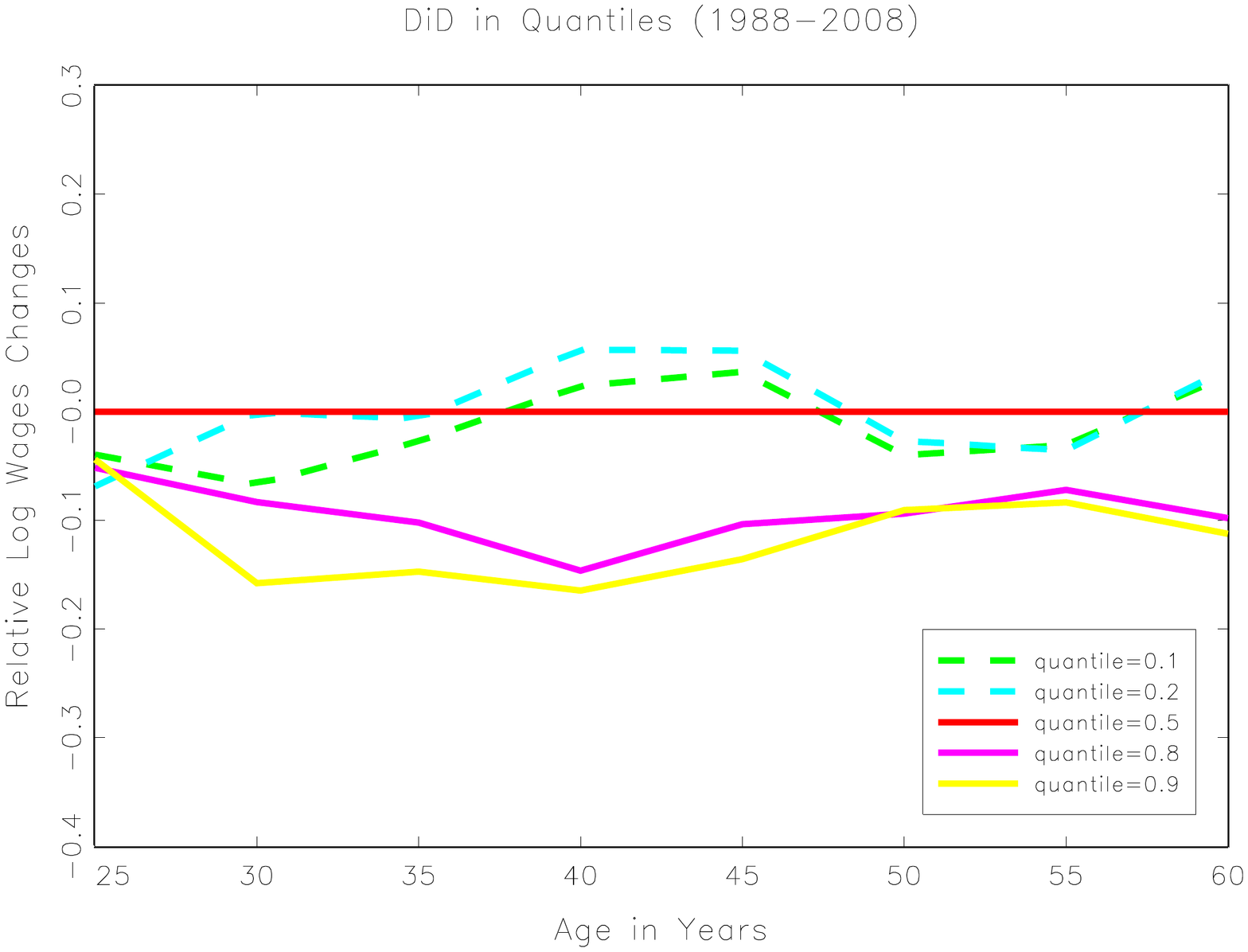}
}
\end{center}
\par
\parbox{5in}{
Note: The top and bottom panels of the figure show local linear estimates of
$-\Delta_{1988,1974}(\tau,x)$ and $-\Delta_{2008,1988}(\tau,x)$, respectively,
where $x$ is age in years.
}
\end{figure}

\subsection{Empirical Results}

We used the CPS data extract of \citeasnoun{Acemoglu:Autor:11}.\footnote{%
The data are available on David Autor's web site. We would like to thank him
for posting the data set on a public domain. They used three-year averages
around the year of interest to produce Figures 9a-9c in %
\citeasnoun{Acemoglu:Autor:11}; however, we used just annual data.} In our
empirical example, we use age in years as the only covariate. Summary
statistics and visual presentation of data are given in Table \ref%
{table-cps-sum-stat} and Figure \ref{figure-cps-1}. Note that Figure \ref%
{figure-cps-1} replicates the basic patterns of Figures 9 of %
\citeasnoun{Acemoglu:Autor:11}.

We now turn to the conditional version of Figure \ref{figure-cps-1}, using
age as a conditioning variable. As an illustration, let $\mathcal{X}$ be an
interval of ages between 25 and 60 and let $\mathcal{T} = [0.1,0.9]$. To
check whether inequalities in $\hat{\Delta}_{t,s}(\tau,x) \geq 0$ hold for
each value of $(x,\tau) \in \mathcal{X} \times \mathcal{T}$, we plot
estimates of $\hat{v}_{\tau,t,s}(x) = - \hat{\Delta}_{t,s}(\tau,x) $ in
Figure \ref{figure-cps-2}. The top panel of the figure shows that 5
estimated curves of $\hat{v}_{\tau,1988,1974}(x)$, each representing a
particular conditional quantile, and the bottom panel shows the
corresponding graph for period 1988-2008.\footnote{%
As before, underlying conditional quantile functions are estimated via the
local linear quantile regression estimator with the kernel function $%
K(u)=1.5[1-(2u)^{2}]\times 1\{|u|\leq 0.5\}$. One important difference from
the first empirical example is that we used the CPS sample weight, which
were incorporated by multiplying it to the kernel weight for each
observation. Finally, the bandwidth was $h=2.5$ for all years.} By
construction, the estimated curve is a flat line at zero when $\tau = 0.5$.
As consistent with Figure \ref{figure-cps-1}, the lower quantiles seem to
violate the null hypothesis, especially for the period 1974-1988. As before,
our test can tell formally whether positive values of $\hat{v}_{\tau,t,s}(x)$
lead to rejection of the null hypothesis of polarization of wage growth.

We considered both the $L_{1}$ and $L_{2}$ test statistics described in %
\eqref{test-form-CPS}. As before, the contact set was estimated with $\hat{c}%
_n = C_{\text{cs}} \log \log (n) q_{1-0.1/\log(n)}(S_n^\ast)$ with $r_{n}=%
\sqrt{nh}$.\footnote{%
To accommodate different sample sizes across years, we set $n =
(n_{1974}+n_{1988}+n_{2008})/3$ in computing $\hat{c}_n$.} We checked the
sensitivity to the tuning parameters with $C_{\text{cs}}\in \{0.5,1,1.5\}$.

For period 1974-1988, we rejected the null hypothesis at the 1\% level
across all three values of $C_{\text{cs}}$. However, for period 1988-2008,
we fail to reject the null hypothesis at the 5\% level for any value of $C_{%
\text{cs}}$. Therefore, the changing patterns of the US wage distribution
around 1988, reported in \citeasnoun{Acemoglu:Autor:11}, seem to hold up
conditionally on age as well.

\section{Conclusions}

\label{sec:conclusion}

In this paper, we have proposed a general method for testing inequality
restrictions on nonparametric functions and have illustrated its usefulness
by looking at two particular empirical applications. We regard our examples
as just some illustrative applications and believe that our framework can be
useful in a number of other settings.

Our bootstrap test is based on a one-sided version of $L_{p}$ functionals of
kernel-type estimators $(1\leq p <\infty )$. We have provided regularity
conditions under which the bootstrap test is asymptotically valid uniformly
over a large class of distributions and have also provided a class of
distributions for which the asymptotic size is exact. We have shown the
consistency of our test and have obtained a general form of the local power
function.

There are different notions of efficiency for nonparametric tests and hence
there is no compelling sense of an asymptotically optimal test for the
hypothesis considered in this paper. See \citeasnoun{Nikitin:95} and %
\citeasnoun{bickel2006} for a general discussion. It would be interesting to
consider a multiscale version of our test based on a range of bandwidths to
see if it achieves adaptive rate-optimality against a sequence of smooth
alternatives along the lines of \citeasnoun{Armstrong/Chan:13} and %
\citeasnoun{Chetverikov:11}.

\renewcommand\thesection{\Roman{section}} \setcounter{section}{0} %

\section*{\protect\large Appendices}

Appendix \ref{appendix-proof-A} gives the proofs of Theorems \ref{Thm1}-\ref%
{Thm5}, and \ref{appendix:C} and \ref{appendix:D} offer auxiliary results
for the proofs of Theorems \ref{Thm1}-\ref{Thm5}.
Finally, Appendix \ref{sec:appendix-D} contains the proof of Theorem AUC\ref{thm:AUC1}.

%\clearpage

\appendix

\section{Proofs of Theorems \protect\ref{Thm1}-\protect\ref{Thm5}}

\label{appendix-proof-A}

The roadmap of Appendix A is as follows. Appendix A begins with the proofs
of Lemma 1 (the representation of $\hat{\theta}$) and Lemma 2 (the uniform
convergence of $\hat{v}_{\tau ,j}(x)$). Then we establish auxiliary results,
Lemmas A1-A4, to prepare for the proofs of Theorems 1-3. The brief
descriptions of these auxiliary results are given below.

Lemma A1 establishes asymptotic representation of the location normalizers
for the test statistic both in the population and in the bootstrap
distribution. The crucial implication is that the difference between the
population version and the bootstrap version is of order $o_{P}(h^{d/2})$, $%
\mathcal{P}$-uniformly. The result is in fact an immediate consequence of
Lemma C12 in Appendix C.

Lemma A2 establishes uniform asymptotic normality of the representation of $%
\hat{\theta}$ and its bootstrap version. The asymptotic normality results
use the method of Poissonization as in \citeasnoun{GMZ} and \citeasnoun{LSW}%
. However, in contrast to the preceding researches, the results established
here are much more general, and hold uniformly over a wide class of
probabilities. The lemma relies on Lemmas B7-B9 in Appendix B and their
bootstrap versions in Lemmas C7-C9 in Appendix C. These results are employed
to obtain the uniform asymptotic normality of the representation of $\hat{%
\theta}$ in Lemma A2.

Lemma A3 establishes that the estimated contact sets $\hat{B}_{A}(\hat{c}%
_{n})$ are covered by its enlarged population version, and covers its shrunk
population version with probability approaching one uniformly over $P\in 
\mathcal{P}$. In fact, this is an immediate consequence of the uniform
convergence results for $\hat{v}_{\tau ,j}(x)$ and $\hat{\sigma}_{\tau
,j}(x) $ in Assumptions 3 and 5. Lemma A3 is used later, when we replace the
estimated contact sets by their appropriate population versions, eliminating
the nuisance to deal with the estimation errors in $\hat{B}_{A}(\hat{c}_{n})$%
.

Lemma A4 presents the approximation result of the critical values for the
original and bootstrap test statistics in Lemma A2, by critical values from
the standard normal distribution uniformly over $P\in \mathcal{P}.$ Although
we do not propose using the normal critical values, the result is used as an
intermediate step for justifying the use of the bootstrap method in this
paper. Obviously, Lemma A4 follows as a consequence of Lemma A2.

Equipped with Lemmas A1-A4, we proceed to prove Theorem 1. For this, we
first use the representation result of Lemma 1 for $\hat{\theta}$. In doing
so, we use $B_{A}(c_{n,L},c_{n,U})$ as a population version of $\hat{B}_{A}(%
\hat{c}_{n})$. This is because%
\begin{equation*}
B_{A}(c_{n,L},c_{n,U})\subset \hat{B}_{A}(\hat{c}_{n})
\end{equation*}%
with probability approaching one by Lemma A3, and thus, makes the bootstrap
test statistic $\hat{\theta}^{\ast }$ dominate the one that involves $%
B_{A}(c_{n,L},c_{n,U})$ in place of $\hat{B}_{A}(\hat{c}_{n})$. The
distribution of the latter bootstrap version with $B_{A}(c_{n,L},c_{n,U})$
is asymptotically equivalent to the representation of $\hat{\theta}$ with $%
B_{A}(c_{n,L},c_{n,U})$ after location-scale normalization, as long as the
limiting distribution is nondegenerate. When the limiting distribution is
degenerate, we use the second component $h^{d/2}\eta +\hat{a}^{\ast }$ in
the definition of $c_{\alpha ,\eta }^{\ast }$ to ensure the asymptotic
validity of the bootstrap procedure. For both cases of degenerate and
nondegenerate limiting distributions, Lemma A1 which enables one to replace $%
\hat{a}^{\ast }$ by an appropriate population version is crucial.

The proof of Theorem 2 that shows the asymptotic exactness of the bootstrap
test modifies the proof of Theorem 1 substantially. Instead of using the
representation result of Lemma 1 for $\hat{\theta}$ with $%
B_{n,A}(c_{n,L},c_{n,U})$, we now use the same version but with $%
B_{n,A}(c_{n,U},c_{n,L})$. This is because for asymptotic exactness, we need
to approximate the original and bootstrap quantities by versions using $%
B_{n,A}(q_{n})$ for small $q_{n}$, and to do this, we need to control the
remainder term in the bootstrap statistic with the integral domain $\hat{B}%
_{A}(\hat{c}_{n})\backslash B_{n,A}(q_{n})$. By our choice of $%
B_{n,A}(c_{n,U},c_{n,L})$ and by the fact that we have 
\begin{equation*}
\hat{B}_{A}(\hat{c}_{n})\subset B_{n,A}(c_{n,U},c_{n,L}),
\end{equation*}%
with probability approaching one by Lemma A3, we can bound the remainder
term with a version with the integral domain $B_{n,A}(c_{n,U},c_{n,L})%
\backslash B_{n,A}(q_{n})$. Thus this remainder term vanishes by the
condition for $\lambda _{n}$ and $q_{n}$ in the definition of $\mathcal{P}%
_{n}(\lambda _{n},q_{n})$.

The rest of the proofs are devoted to proving the power properties of the
bootstrap procedure. Theorem 3 establishes consistency of the bootstrap
test. Theorems 4 and 5 establish local power functions under Pitman local
drifts. The proofs of Theorems 4-5 are similar to the proof of Theorem 2, as
we need to establish the asymptotically exact form of the rejection
probability for the bootstrap test statistic. Nevertheless, we need to
employ some delicate arguments to deal with the Pitman local alternatives,
and need to expand the rejection probability to obtain the final results.
For this, we first establish Lemmas A5-A7. Essentially, Lemma A5 is a
version of the representation result of Lemma 1 under local alternatives.
Lemma A6 and Lemma A7 parallel Lemma A1 and Lemma 2 under local
alternatives.\bigskip

Let us begin by proving Lemma 1. First, recall the following definitions%
\begin{equation}
\mathbf{\hat{s}}_{\tau }(x)\equiv \left[ \frac{r_{n,j}\{\hat{v}_{\tau
,j}(x)-v_{n,\tau ,j}(x)\}}{\hat{\sigma}_{\tau ,j}(x)}\right] _{j\in \mathbb{N%
}_{J}}\text{ and\ }\mathbf{\hat{s}}_{\tau }^{\ast }(x)\equiv \left[ \frac{%
r_{n,j}\{\hat{v}_{\tau ,j}^{\ast }(x)-\hat{v}_{,\tau ,j}(x)\}}{\hat{\sigma}%
_{\tau ,j}^{\ast }(x)}\right] _{j\in \mathbb{N}_{J}}.  \label{s and s_star}
\end{equation}%
Also, define 
\begin{equation}
\mathbf{\hat{u}}_{\tau }(x)\equiv \left[ \frac{r_{n,j}\hat{v}_{\tau ,j}(x)}{%
\hat{\sigma}_{\tau ,j}(x)}\right] _{j\in \mathbb{N}_{J}}\text{ and\ }\mathbf{%
u}_{\tau }(x;\hat{\sigma})\equiv \left[ \frac{r_{n,j}v_{n,\tau ,j}(x)}{\hat{%
\sigma}_{\tau ,j}(x)}\right] _{j\in \mathbb{N}_{J}}.  \label{v}
\end{equation}

\begin{proof}[Proof of Lemma 1]
It suffices to show the following two statements:

\textbf{Step 1:} As\textit{\ }$n\rightarrow \infty ,$%
\begin{equation*}
\inf_{P\in \mathcal{P}_{0}}P\left\{ \int_{\mathcal{S}\backslash
B_{n}(c_{n,1},c_{n,2})}\Lambda _{p}\left( \mathbf{\hat{u}}_{\tau }(x)\right)
dQ(x,\tau )=0\right\} \rightarrow 1\text{\textit{,}}
\end{equation*}%
where we recall $B_{n}(c_{n,1},c_{n,2})\equiv \cup _{A\in \mathcal{N}%
_{J}}B_{n,A}(c_{n,1},c_{n,2})$.

\textbf{Step 2:} For each $A\in \mathcal{N}_{J}$, as $n\rightarrow \infty ,$%
\begin{equation*}
\inf_{P\in \mathcal{P}_{0}}P\left\{ \int_{B_{n,A}(c_{n,1},c_{n,2})}\left\{
\Lambda _{p}\left( \mathbf{\hat{u}}_{\tau }(x)\right) -\Lambda _{A,p}\left( 
\mathbf{\hat{u}}_{\tau }(x)\right) \right\} dQ(x,\tau )=0\right\}
\rightarrow 1\text{.}
\end{equation*}%
First, we prove Step 1. We write the integral in the probability as%
\begin{equation}
\int_{\mathcal{S}\backslash B_{n}(c_{n,1},c_{n,2})}\Lambda _{p}\left( 
\mathbf{\hat{s}}_{\tau }(x)+\mathbf{u}_{\tau }(x;\hat{\sigma})\right)
dQ(x,\tau ).  \label{eq35}
\end{equation}
Let 
\begin{equation*}
A_{n}(x,\tau )\equiv \left\{ j\in \mathbb{N}_{J}:\frac{r_{n,j}v_{n,\tau
,j}(x)}{\sigma _{n,\tau ,j}(x)}\geq - (c_{n,1}\wedge c_{n,2}) \right\} .
\end{equation*}
We first show that when $(x,\tau )\in \mathcal{S}\backslash
B_{n}(c_{n,1},c_{n,2})$, we have $A_{n}(x,\tau )=\varnothing $ under the
null hypothesis. Suppose that $(x,\tau )\in \mathcal{S}\backslash
B_{n}(c_{n,1},c_{n,2})$ but to the contrary, $A_{n}(x,\tau )$ is nonempty.
By the definition of $A_{n}(x,\tau )$, we have $(x,\tau )\in
B_{n,A_{n}(x,\tau )}(c_{n,1},c_{n,2})$. However, since%
\begin{equation*}
\mathcal{S}\backslash B_{n}(c_{n,1},c_{n,2})=\mathcal{S}\cap \left( \cap
_{A\in \mathcal{N}_{J}}B_{n,A}^{c}(c_{n,1},c_{n,2})\right) \subset
B_{n,A_{n}(x,\tau )}^{c}(c_{n,1},c_{n,2}),
\end{equation*}%
this contradicts the fact that $(x,\tau )\in \mathcal{S}\backslash
B_{n}(c_{n,1},c_{n,2})$. Hence whenever $(x,\tau )\in \mathcal{S}\backslash
B_{n}(c_{n,1},c_{n,2})$, we have $A_{n}(x,\tau )=\varnothing $.

Note that%
\begin{equation*}
\frac{v_{n,\tau ,j}(x)}{\hat{\sigma}_{\tau ,j}(x)}=\frac{v_{n,\tau ,j}(x)}{%
\sigma _{n,\tau ,j}(x)}\left\{ 1+\frac{\sigma _{n,\tau ,j}(x)-\hat{\sigma}%
_{\tau ,j}(x)}{\hat{\sigma}_{\tau ,j}(x)}\right\} =\frac{v_{n,\tau ,j}(x)}{%
\sigma _{n,\tau ,j}(x)}\left\{ 1+o_{P}(1)\right\} ,
\end{equation*}%
where $o_{P}(1)$ is uniform over $(x,\tau )\in \mathcal{S}$ and over $P\in 
\mathcal{P}$ by Assumption A\ref{assumption-A5}. Fix a small $\varepsilon >0$%
. We have for all $j\in \mathbb{N}_{J},$%
\begin{eqnarray*}
&&\inf_{P\in \mathcal{P}_{0}}P\left\{ \frac{r_{n,j}v_{n,\tau ,j}(x)}{\hat{%
\sigma}_{\tau ,j}(x)}<-\frac{c_{n,1}\wedge c_{n,2}}{1+\varepsilon }\text{
for all }(x,\tau )\in \mathcal{S}\backslash B_{n}(c_{n,1},c_{n,2})\right\} \\
&\geq &\inf_{P\in \mathcal{P}_{0}}P\left\{ \frac{r_{n,j}v_{n,\tau ,j}(x)}{%
\sigma _{n,\tau ,j}(x)}<-\frac{c_{n,1}\wedge c_{n,2}}{(1+\varepsilon
)\left\{ 1+o_{P}(1)\right\} }\text{ for all }(x,\tau )\in \mathcal{S}%
\backslash B_{n}(c_{n,1},c_{n,2})\right\} \rightarrow 1,
\end{eqnarray*}%
as $n\rightarrow \infty $, where the last convergence follows because $%
A_{n}(x,\tau )=\varnothing $ for all $(x,\tau )\in \mathcal{S}\backslash
B_{n}(c_{n,1},c_{n,2})$. Therefore, with probability approaching one, the
term in (\ref{eq35}) is bounded by%
\begin{equation}
\int_{\mathcal{S}\backslash B_{n}(c_{n,1},c_{n,2})}\Lambda _{p}\left( 
\mathbf{\hat{s}}_{\tau }(x)-\left( \frac{c_{n,1}\wedge c_{n,2}}{%
1+\varepsilon }\right) \mathbf{1}_{J}\right) dQ(x,\tau ),  \label{inet}
\end{equation}%
where $\mathbf{1}_{J}$ is a $J$-dimensional vector of ones. Using the
definition of $\Lambda _{p}(\mathbf{v})$, bound the above integral by%
\begin{equation}
J^{p/2}\left( \sum_{j=1}^{J}\left[ r_{n,j}\sup_{(x,\tau )\in \mathcal{S}%
}\left\vert \frac{\hat{v}_{\tau ,j}(x)-v_{n,\tau ,j}(x)}{\hat{\sigma}_{\tau
,j}(x)}\right\vert -\frac{c_{n,1}\wedge c_{n,2}}{1+\varepsilon }\right]
_{+}^{2}\right) ^{p/2}.  \label{dec45}
\end{equation}%
Note that by Assumption A\ref{assumption-A3}, 
\begin{equation*}
r_{n,j}\sup_{(x,\tau )\in \mathcal{S}}\left\vert \frac{\hat{v}_{\tau
,j}(x)-v_{n,\tau ,j}(x)}{\hat{\sigma}_{\tau ,j}(x)}\right\vert =O_{P}\left( 
\sqrt{\log n}\right) .
\end{equation*}%
Fix any arbitrarily large $M>0$ and denote by $E_{n}$ the event that%
\begin{equation*}
r_{n,j}\sup_{(x,\tau )\in \mathcal{S}}\left\vert \frac{\hat{v}_{\tau
,j}(x)-v_{n,\tau ,j}(x)}{\hat{\sigma}_{\tau ,j}(x)}\right\vert \leq M\sqrt{%
\log n}.
\end{equation*}%
The term (\ref{dec45}), when restricted to this event $E_{n},$ is bounded by%
\begin{equation*}
J^{p/2}\left( \sum_{j=1}^{J}\left[ M\sqrt{\log n}-\frac{c_{n,1}\wedge c_{n,2}%
}{1+\varepsilon }\right] _{+}^{2}\right) ^{p/2}
\end{equation*}%
which becomes zero from some large $n$ on, given that $(c_{n,1}\wedge
c_{n,2})/\sqrt{\log n}\rightarrow \infty $. Since sup$_{P\in \mathcal{P}%
_{0}}PE_{n}^{c}\rightarrow 0$ as $n\rightarrow \infty $ and then $%
M\rightarrow \infty $ by Assumption A3, we obtain the desired result of Step
1.

As for Step 2, we have for any small $\varepsilon >0$, and for all $j\in 
\mathbb{N}_{J}\backslash A$,%
\begin{eqnarray}
&&P\left\{ \frac{r_{n,j}v_{n,\tau ,j}(x)}{\hat{\sigma}_{\tau ,j}(x)}<-\frac{%
c_{n,1}\wedge c_{n,2}}{1+\varepsilon }\text{ for all }(x,\tau )\in
B_{n,A}(c_{n,1},c_{n,2})\right\}  \label{ineq54} \\
&\geq &P\left\{ \frac{r_{n,j}v_{n,\tau ,j}(x)}{\sigma _{n,\tau ,j}(x)}<-%
\frac{c_{n,1}\wedge c_{n,2}}{(1+\varepsilon )\left\{ 1+o_{P}(1)\right\} }%
\text{ for all }(x,\tau )\in B_{n,A}(c_{n,1},c_{n,2})\right\} \rightarrow 1,
\notag
\end{eqnarray}%
similarly as before. Let $\mathbf{\bar{s}}_{\tau ,A}(x)$ be a $J$%
-dimensional vector whose $j$-th entry is $r_{n,j}\hat{v}_{n,\tau ,j}(x)/%
\hat{\sigma}_{\tau ,j}(x)$ if $j\in A$, and $r_{n,j}\{\hat{v}_{n,\tau
,j}(x)-v_{n,\tau ,j}(x)\}/\hat{\sigma}_{\tau ,j}(x)$ if $j\in \mathbb{N}%
_{J}\backslash A$. Since by Assumption A5, we have 
\begin{equation*}
\inf_{P\in \mathcal{P}_{0}}P\left\{ \mathbf{u}_{\tau }(x;\hat{\sigma})\leq 0%
\text{ for all }(x,\tau )\in \mathcal{S}\right\} \rightarrow 1,
\end{equation*}%
as $n\rightarrow \infty $, using either definition of $\Lambda _{p}(\mathbf{v%
})$ in \eqref{lambda_p}, 
\begin{eqnarray}
&&\int_{B_{n,A}(c_{n,1},c_{n,2})}\Lambda _{A,p}\left( \mathbf{\hat{u}}_{\tau
}(x)\right) dQ(x,\tau )  \label{ineqs2} \\
&\leq &\int_{B_{n,A}(c_{n,1},c_{n,2})}\Lambda _{p}\left( \mathbf{\hat{u}}%
_{\tau }(x)\right) dQ(x,\tau )  \notag \\
&\leq &\int_{B_{n,A}(c_{n,1},c_{n,2})}\Lambda _{p}\left( \mathbf{\bar{s}}%
_{\tau ,A}(x)-\frac{c_{n,1}\wedge c_{n,2}}{1+\varepsilon }\mathbf{1}%
_{-A}\right) dQ(x,\tau ),  \notag
\end{eqnarray}%
where $\mathbf{1}_{-A}$ is the $J$-dimensional vector whose $j$-th entry is
zero if $j\in A$ and one if $j\in \mathbb{N}_{J}\backslash A$, and the last
inequality holds with probability approaching one by (\ref{ineq54}). Note
that by Assumption A\ref{assumption-A3} and by the assumption that $\sqrt{%
\log n}\{c_{n,1}^{-1}+c_{n,2}^{-1}\}\rightarrow \infty $, we deduce that for
any $j\in \mathbb{N}_{J},$%
\begin{equation*}
\inf_{P\in \mathcal{P}_{0}}P\left\{ r_{n,j}\sup_{(x,\tau )\in \mathcal{S}%
}\left\vert \frac{\hat{v}_{\tau ,j}(x)-v_{n,\tau ,j}(x)}{\hat{\sigma}_{\tau
,j}(x)}\right\vert \leq \frac{c_{n,1}\wedge c_{n,2}}{1+\varepsilon }\right\}
\rightarrow 1,
\end{equation*}%
as $n\rightarrow \infty $. Hence, as $n\rightarrow \infty ,$%
\begin{equation*}
\inf_{P\in \mathcal{P}_{0}}P\left\{ 
\begin{array}{c}
\int_{B_{n,A}(c_{n,1},c_{n,2})}\Lambda _{p}\left( \mathbf{\bar{s}}_{\tau
,A}(x)-((c_{n,1}\wedge c_{n,2})/(1+\varepsilon ))\mathbf{1}_{-A}\right)
dQ(x,\tau ) \\ 
=\int_{B_{n,A}(c_{n,1},c_{n,2})}\Lambda _{A,p}\left( \mathbf{\bar{s}}_{\tau
,A}(x)\right) dQ(x,\tau )%
\end{array}%
\right\} \rightarrow 1.
\end{equation*}%
Since%
\begin{equation*}
\int_{B_{n,A}(c_{n,1},c_{n,2})}\Lambda _{A,p}\left( \mathbf{\bar{s}}_{\tau
,A}(x)\right) dQ(x,\tau )=\int_{B_{n,A}(c_{n,1},c_{n,2})}\Lambda
_{A,p}\left( \mathbf{\hat{u}}_{\tau }(x)\right) dQ(x,\tau ),
\end{equation*}%
we obtain the desired result from (\ref{ineqs2}).
\end{proof}

Now let us turn to the proof of Lemma 2 in Section 4.4.

\begin{proof}[Proof of Lemma 2]
(i) Recall the definition $b_{n,ij}(x,\tau )\equiv \beta _{n,x,\tau
,j}\left( Y_{ij},(X_{i}-x)/h\right) )$. Take $M_{n,j}\equiv \sqrt{nh^{d}}/%
\sqrt{\log n},$ and let%
\begin{equation*}
b_{n,ij}^{a}(x,\tau )\equiv b_{n,ij}(x,\tau )1_{n,ij}\text{ and\ }%
b_{n,ij}^{b}(x,\tau )\equiv b_{n,ij}(x,\tau )\left( 1-1_{n,ij}\right) ,
\end{equation*}%
where $1_{n,ij}\equiv 1\{$sup$_{(x,\tau )\in \mathcal{S}}|b_{n,ij}(x,\tau
)|\leq M_{n,j}/2\}.$ First, note that by Assumption A1,%
\begin{eqnarray}
&&r_{n,j}\sqrt{h^{d}}\sup_{(x,\tau )\in \mathcal{S}}\left\vert \frac{\hat{v}%
_{\tau ,j}(x)-v_{n,\tau ,j}(x)}{\hat{\sigma}_{\tau ,j}(x)}\right\vert
\label{decomp2} \\
&\leq &\sup_{(x,\tau )\in \mathcal{S}}\left\vert \frac{1}{\sqrt{n}}%
\sum_{i=1}^{n}\left( b_{n,ij}^{a}(x,\tau )-\mathbf{E}\left[
b_{n,ij}^{a}(x,\tau )\right] \right) \right\vert  \notag \\
&&+\sup_{(x,\tau )\in \mathcal{S}}\left\vert \frac{1}{\sqrt{n}}%
\sum_{i=1}^{n}\left( b_{n,ij}^{b}(x,\tau )-\mathbf{E}\left[
b_{n,ij}^{b}(x,\tau )\right] \right) \right\vert +o_{P}(1),\ \mathcal{P}%
\text{-uniformly.}
\end{eqnarray}%
We now prove part (i) by proving the following two steps.\newline
\textbf{Step 1:}%
\begin{equation*}
\sup_{(x,\tau )\in \mathcal{S}}\left\vert \frac{1}{\sqrt{nh^{d}}}%
\sum_{i=1}^{n}\left( b_{n,ij}^{b}(x,\tau )-\mathbf{E}\left[
b_{n,ij}^{b}(x,\tau )\right] \right) \right\vert =o_{P}(\sqrt{\log n}),\ 
\mathcal{P}\text{-uniformly.}
\end{equation*}%
\newline
\textbf{Step 2:}%
\begin{equation*}
\sup_{(x,\tau )\in \mathcal{S}}\left\vert \frac{1}{\sqrt{nh^{d}}}%
\sum_{i=1}^{n}\left( b_{n,ij}^{a}(x,\tau )-\mathbf{E}\left[
b_{n,ij}^{a}(x,\tau )\right] \right) \right\vert =O_{P}(\sqrt{\log n}),\ 
\mathcal{P}\text{-uniformly.}
\end{equation*}%
\newline
Step 1 is carried out by some elementary moment calculations, whereas Step 2
is proved using a maximal inequality of \citeasnoun[Theorem
6.8]{Massart:07}.

\textbf{Proof of Step 1:} It is not hard to see that%
\begin{eqnarray*}
&&\mathbf{E}\left[ \sup_{(x,\tau )\in \mathcal{S}}\left\vert \frac{1}{\sqrt{n%
}}\sum_{i=1}^{n}\left( b_{n,ij}^{b}(x,\tau )-\mathbf{E}\left[
b_{n,ij}^{b}(x,\tau )\right] \right) \right\vert \right] \\
&\leq &2\sqrt{n}\mathbf{E}\left[ \sup_{(x,\tau )\in \mathcal{S}}\left\vert
b_{n,ij}(x,\tau )\right\vert \left( 1-1_{n,ij}\right) \right] \\
&\leq &C\sqrt{n}\left( \frac{M_{n,j}}{2}\right) ^{-3}\mathbf{E}\left[
\sup_{(x,\tau )\in \mathcal{S}}\left\vert b_{n,ij}(x,\tau )\right\vert ^{4}%
\right] \leq C_{1}\sqrt{n}\left( \frac{M_{n,j}}{2}\right) ^{-3},
\end{eqnarray*}%
for some $C_{1}>0$, $C>0$. The last bound follows by the uniform fourth
moment bound for $b_{n,ij}(x,\tau )$ assumed in Lemma 2. Note that 
\begin{equation*}
\sqrt{n}\left( M_{n,j}\right) ^{-3}=n^{-1}h^{-3d/2}\left( \log n\right)
^{3/2}=o\left( \sqrt{\log n}h^{d/2}\right) ,
\end{equation*}%
by the condition that $n^{-1/2}h^{-d-\nu }\rightarrow 0$ for some small $\nu
>0$.\newline
\textbf{Proof of Step 2: }For each $j\in \mathbb{N}_{J}$, let $\mathcal{F}%
_{n,j}\equiv \{\beta _{n,x,\tau ,j}^{a}(\cdot ,(\cdot -x)/h)/M_{n,j}:(x,\tau
)\in \mathcal{S}\}$, where $\beta _{n,x,\tau
,j}^{a}(Y_{ij},(X_{i}-x)/h)\equiv b_{n,ij}^{a}(x,\tau )$. Note that the
indicator function $1_{n,ij}$ in the definition of $\beta _{n,x,\tau ,j}^{a}$
does not depend on $(x,\tau )$ of $\beta _{n,x,\tau ,j}^{a}$. Using (\ref%
{Lp-Cont}) in Lemma 2 and following (part of) the arguments in the proof of
Theorem 3 of \citeasnoun{Chen/Linton/VanKeilegom:03}, we find that there
exist $C_{1}>0$ and $C_{2,j}>0$ such that for all $\varepsilon >0,$%
\begin{equation*}
N_{[]}\left( \varepsilon ,\mathcal{F}_{n,j},L_{2}(P)\right) \leq N\left(
\left( \frac{\varepsilon M_{n,j}}{\delta _{n,j}}\right) ^{2/\gamma _{j}},%
\mathcal{X}\times \mathcal{T},||\cdot ||\right) \leq C_{1}\left( \frac{%
\varepsilon M_{n,j}}{\delta _{n,j}}\wedge 1\right) ^{-C_{2,j}},
\end{equation*}%
where $N_{[]}\left( \varepsilon ,\mathcal{F}_{n,j},L_{2}(P)\right) $ denotes
the $\varepsilon $-bracketing number of the class $\mathcal{F}_{n,j}$ with
respect to the $L_{2}(P)$-norm and $N\left( \varepsilon ,\mathcal{X}\times 
\mathcal{T},||\cdot ||\right) $ denotes the $\varepsilon $-covering number
of the space $\mathcal{X}\times \mathcal{T}$ with respect to the Euclidean
norm $||\cdot ||$. The last inequality follows by the assumption that $%
\mathcal{X}$ and $\mathcal{T}$ are compact subsets of a Euclidean space. The
class $\mathcal{F}_{n,j}$ is uniformly bounded by $1/2$. \newline
Let $\{[\beta _{n,x_{k},\tau _{k},j}^{a}(\cdot ,(\cdot
-x_{k})/h)/M_{n,j}-\Delta _{k}(\cdot ,\cdot )/M_{n,j},\beta _{n,x_{k},\tau
_{k},j}^{a}(\cdot ,(\cdot -x_{k})/h)/M_{n,j}+\Delta _{k}(\cdot ,\cdot
)/M_{n,j}]:k=1,\cdot \cdot \cdot ,N_{n,j}\}$ constitutes $\varepsilon $%
-brackets, where $\Delta _{k}(Y_{ij},X_{i})\equiv \sup |\beta _{n,x,\tau
,j}^{a}(Y_{ij},(X_{i}-x)/h)-\beta _{n,x_{k},\tau
_{k},j}^{a}(Y_{ij},(X_{i}-x_{k})/h)|$ and the supremum is over $(x,\tau )\in 
\mathcal{S}$ such that 
\begin{equation*}
\sqrt{||x-x_{k}||^{2}+||\tau -\tau _{k}||^{2}}\leq C_{1}(\varepsilon
M_{n,j}/\delta _{n,j})^{2/\gamma _{j}}.
\end{equation*}%
By the previous covering number bound, we can take $N_{n,j}\leq C_{1}\left(
(\varepsilon M_{n,j}/\delta _{n,j})\wedge 1\right) ^{-C_{2,j}},$ and 
\begin{equation*}
\mathbf{E}\Delta _{k}^{2}(Y_{ij},X_{i})M_{n,j}^{-2}<\varepsilon ^{2}.
\end{equation*}%
\newline
Note that for any $k\geq 2,$%
\begin{equation*}
\mathbf{E}\left[ |b_{n,ij}^{a}(x,\tau )/M_{n,j}|^{k}\right] \leq \mathbf{E}%
\left[ b_{n,ij}^{2}(x,\tau )\right] /M_{n,j}^{2}\leq
CM_{n,j}^{-2}h^{d}=C(\log n)/n\text{,}
\end{equation*}%
by the fact that $|b_{n,ij}^{a}(x,\tau )/M_{n,j}|\leq 1/2$. Furthermore,%
\begin{equation*}
\mathbf{E}\left[ |\Delta _{k}(Y_{ij},X_{i})/M_{n,j}|^{k}\right] \leq \mathbf{%
E}\left[ \Delta _{k}^{2}(Y_{ij},X_{i})/M_{n,j}^{2}\right] \leq \varepsilon
^{2},
\end{equation*}%
where the first inequality follows because $|\Delta
_{k}(Y_{ij},X_{i})/M_{n,j}|\leq 1$. Therefore, by Theorem 6.8 of %
\citeasnoun{Massart:07}, we have (from sufficiently large $n$ on)%
\begin{eqnarray}
&&\mathbf{E}\left[ \sup_{(x,\tau )\in \mathcal{S}}\left\vert \frac{1}{M_{n,j}%
\sqrt{n}}\sum_{i=1}^{n}\left( b_{n,ij}^{a}(x,\tau )-\mathbf{E}\left[
b_{n,ij}^{a}(x,\tau )\right] \right) \right\vert \right]  \label{bd5} \\
&\leq &C_{1}\int_{0}^{\frac{C_{2}h^{d/2}}{M_{n,j}}}\left\{ \left( -C_{3}\log
\left( \frac{\varepsilon M_{n,j}}{\delta _{n,j}}\wedge 1\right) \right)
\wedge n\right\} ^{1/2}d\varepsilon -\frac{C_{4}}{\sqrt{n}}\log \left( \frac{%
\sqrt{\log n}}{\sqrt{n}}\right) ,  \notag
\end{eqnarray}%
where $C_{1},C_{2},C_{3}$, and $C_{4}$ are positive constants. (The
inequality above follows because$\sqrt{\log n}/\sqrt{n}\rightarrow 0$ as $%
n\rightarrow \infty $.) The leading integral has a domain restricted to $%
[0,\delta _{n,j}/M_{n,j}]$, so that it is equal to%
\begin{eqnarray*}
&&C_{1}\int_{0}^{\frac{C_{2}h^{d/2}}{M_{n,j}}\wedge \frac{\delta _{n,j}}{%
M_{n,j}}}\left\{ \left( -C_{3}\log \left( \frac{\varepsilon M_{n,j}}{\delta
_{n,j}}\right) \right) \wedge n\right\} ^{1/2}d\varepsilon \\
&=&\frac{C_{1}\delta _{n,j}}{M_{n,j}}\int_{0}^{\frac{C_{2}h^{d/2}}{\delta
_{n,j}}\wedge 1}\sqrt{\left( -C_{3}\log \varepsilon \right) \wedge n}%
d\varepsilon \\
&=&O\left( \frac{\delta _{n,j}}{M_{n,j}}\left( \frac{h^{d/2}}{\delta _{n,j}}%
\wedge 1\right) \sqrt{-\log \left( \frac{h^{d/2}}{\delta _{n,j}}\wedge
1\right) }\right) .
\end{eqnarray*}%
After multiplying by $M_{n,j}/h^{d/2}$, the last term is of order%
\begin{equation*}
O\left( \left( 1\wedge \frac{\delta _{n,j}}{h^{d/2}}\right) \sqrt{-\log
\left( \frac{h^{d/2}}{\delta _{n,j}}\wedge 1\right) }\right) =O\left( \sqrt{%
-\log \left( \frac{h^{d/2}}{\delta _{n,j}}\wedge 1\right) }\right) =O(\sqrt{%
\log n}),
\end{equation*}%
because $\delta _{n,j}=n^{s_{1,j}}$ and $h=n^{s_{2}}$ for some $%
s_{1,j},s_{2}\in \mathbf{R}$. \newline
Also, note that after multiplying by $M_{n,j}/h^{d/2}=\sqrt{n}/\sqrt{\log n}$%
, the last term in (\ref{bd5}) (with minus sign) becomes%
\begin{equation*}
-\frac{C_{4}}{\sqrt{\log n}}\log \left( \frac{\sqrt{\log n}}{\sqrt{n}}%
\right) \leq \frac{C_{4}\sqrt{\log n}}{2}-\frac{C_{4}\log \sqrt{\log n}}{%
\sqrt{\log n}}=O\left( \sqrt{\log n}\right) ,
\end{equation*}%
where the inequality follows because $\sqrt{\log n}\geq 1$ for all $n\geq
e\equiv \exp (1)$. Collecting the results for both the terms on the right
hand side of (\ref{bd5}), we obtain the desired result of Step 2.

\noindent (ii) Define $b_{n,ij}^{\ast }(x,\tau )\equiv \beta _{n,x,\tau
,j}(Y_{ij}^{\ast },(X_{i}^{\ast }-x)/h)$. By Assumptions B\ref{assumption-B1}
and B3, it suffices to show that%
\begin{equation*}
\sup_{(x,\tau )\in \mathcal{S}}\left\vert \frac{1}{\sqrt{nh^{d}}}%
\sum_{i=1}^{n}\left( b_{n,ij}^{\ast }(x,\tau )-\mathbf{E}^{\ast }\left[
b_{n,ij}^{\ast }(x,\tau )\right] \right) \right\vert =O_{P^{\ast }}(\sqrt{%
\log n}),\text{ }\mathcal{P}\text{-uniformly.}
\end{equation*}%
Using Le Cam's Poissonization lemma in \citeasnoun{Gine/Zinn:90}
(Proposition 2.2 on p.855) and following the arguments in the proof of
Theorem 2.2 of \citeasnoun{Gine:97}, we deduce that%
\begin{eqnarray*}
&&\mathbf{E}\left[ \mathbf{E}^{\ast }\left( \sup_{(x,\tau )\in \mathcal{S}%
}\left\vert \frac{1}{\sqrt{nh^{d}}}\sum_{i=1}^{n}\left( b_{n,ij}^{\ast
}(x,\tau )-\mathbf{E}^{\ast }\left[ b_{n,ij}^{\ast }(x,\tau )\right] \right)
\right\vert \right) \right] \\
&\leq &\frac{e}{e-1}\mathbf{E}\left[ \sup_{(x,\tau )\in \mathcal{S}%
}\left\vert \frac{1}{\sqrt{nh^{d}}}\sum_{i=1}^{n}\left( N_{i}-1\right)
\left\{ b_{n,ij}(x,\tau )-\frac{1}{n}\sum_{k=1}^{n}b_{n,kj}(x,\tau )\right\}
\right\vert \right] ,
\end{eqnarray*}%
where $N_{i}$'s are i.i.d. Poisson random variables with mean $1$ and
independent of $\{(X_{i},Y_{i})\}_{i=1}^{n}$. The last expectation is
bounded by%
\begin{eqnarray*}
&&\mathbf{E}\left[ \sup_{(x,\tau )\in \mathcal{S}}\left\vert \frac{1}{\sqrt{%
nh^{d}}}\sum_{i=1}^{n}\left\{ \left( N_{i}-1\right) b_{n,ij}(x,\tau )-%
\mathbf{E}\left[ \left( N_{i}-1\right) b_{n,ij}(x,\tau )\right] \right\}
\right\vert \right] \\
&&+\mathbf{E}\left[ \sup_{(x,\tau )\in \mathcal{S}}\left\vert \frac{1}{n}%
\sum_{i=1}^{n}\left( N_{i}-1\right) \right\vert \left\vert \frac{1}{\sqrt{%
nh^{d}}}\sum_{k=1}^{n}\left( b_{n,kj}(x,\tau )-\mathbf{E}\left[
b_{n,kj}(x,\tau )\right] \right) \right\vert \right] .
\end{eqnarray*}%
Using the same arguments as in the proof of (i), we find that the first
expectation is $O\left( \sqrt{\log n}\right) $ uniformly in $P\in \mathcal{P}
$. Using independence, we write the second expectation as 
\begin{equation*}
\mathbf{E}\left[ \left\vert \frac{1}{n}\sum_{i=1}^{n}\left( N_{i}-1\right)
\right\vert \right] \cdot \mathbf{E}\left[ \sup_{(x,\tau )\in \mathcal{S}%
}\left\vert \frac{1}{\sqrt{nh^{d}}}\sum_{k=1}^{n}\left( b_{n,kj}(x,\tau )-%
\mathbf{E}\left[ b_{n,kj}(x,\tau )\right] \right) \right\vert \right]
\end{equation*}%
which, as shown in the proof of part (i), is $O(\sqrt{\log n})$, uniformly
in $P\in \mathcal{P}$.
\end{proof}

For further proofs, we introduce new notation. Define for any positive
sequences $c_{n,1}$ and $c_{n,2}$, and any $\mathbf{v}\in \mathbf{R}^{J}$,%
\begin{equation}
\bar{\Lambda}_{x,\tau }(\mathbf{v})\equiv \sum_{A\in \mathcal{N}_{J}}\Lambda
_{A,p}(\mathbf{v})1\{(x,\tau )\in B_{n,A}(c_{n,1},c_{n,2})\}.  \label{lambe}
\end{equation}%
We let%
\begin{eqnarray}
a_{n}^{R}(c_{n,1},c_{n,2}) &\equiv &\int_{\mathcal{X}\times \mathcal{T}}%
\mathbf{E}\left[ \bar{\Lambda}_{x,\tau }(\sqrt{nh^{d}}\mathbf{z}_{N,\tau
}(x))\right] dQ(x,\tau ),\text{ and}  \label{atilde} \\
a_{n}^{R\ast }(c_{n,1},c_{n,2}) &\equiv &\int_{\mathcal{X}\times \mathcal{T}}%
\mathbf{E}^{\ast }\left[ \bar{\Lambda}_{x,\tau }(\sqrt{nh^{d}}\mathbf{z}%
_{N,\tau }^{\ast }(x))\right] dQ(x,\tau ),  \notag
\end{eqnarray}%
where $\mathbf{z}_{N,\tau }(x)$ and $\mathbf{z}_{N,\tau }^{\ast }(x)$ are
random vectors whose $j$-th entry is respectively given by 
\begin{eqnarray*}
z_{N,\tau ,j}(x) &\equiv &\frac{1}{nh^{d}}\sum_{i=1}^{N}\left( \beta
_{n,x,\tau ,j}\left( Y_{ij},\frac{X_{i}-x}{h}\right) -\mathbf{E}\left[ \beta
_{n,x,\tau ,j}\left( Y_{ij},\frac{X_{i}-x}{h}\right) \right] \right) \text{
and} \\
z_{N,\tau ,j}^{\ast }(x) &\equiv &\frac{1}{nh^{d}}\sum_{i=1}^{N}\left( \beta
_{n,x,\tau ,j}\left( Y_{ij}^{\ast },\frac{X_{i}^{\ast }-x}{h}\right) -%
\mathbf{E}^{\ast }\left[ \beta _{n,x,\tau ,j}\left( Y_{ij}^{\ast },\frac{%
X_{i}^{\ast }-x}{h}\right) \right] \right) ,
\end{eqnarray*}%
and $N$ is a Poisson random variable with mean $n$ and independent of $%
\{Y_{i},X_{i}\}_{i=1}^{\infty }$. We also define%
\begin{equation*}
a_{n}(c_{n,1},c_{n,2})\equiv \int \mathbf{E}\left[ \bar{\Lambda}_{x,\tau }(%
\mathbb{W}_{n,\tau ,\tau }^{(1)}(x,0))\right] dQ(x,\tau ).
\end{equation*}%
(See Section 6.3 for the definition of $\mathbb{W}_{n,\tau ,\tau
}^{(1)}(x,u) $.)

\begin{LemmaA}
Suppose that Assumptions A\ref{assumption-A6}(i) and B\ref{assumption-B4}
hold and let $c_{n,1}$ and $c_{n,2}$ be any nonnegative sequences. Then%
\begin{eqnarray*}
\left\vert a_{n}^{R}(c_{n,1},c_{n,2})-a_{n}(c_{n,1},c_{n,2})\right\vert
&=&o(h^{d/2})\text{, uniformly in }P\in \mathcal{P}\text{, and} \\
\left\vert a_{n}^{R\ast }(c_{n,1},c_{n,2})-a_{n}(c_{n,1},c_{n,2})\right\vert
&=&o_{P}(h^{d/2}),\ \mathcal{P}\text{-uniformly.}
\end{eqnarray*}
\end{LemmaA}

\begin{proof}[Proof of Lemma A1]
The proof is essentially the same as the proof of Lemma C12 in Appendix C.
\end{proof}

For any given nonnegative sequences $c_{n,1},c_{n,2},$ we define%
\begin{equation}
\sigma _{n}^{2}(c_{n,1},c_{n,2})\equiv \int_{\mathcal{T}}\int_{\mathcal{T}%
}\int_{\mathcal{X}}\bar{C}_{\tau _{1},\tau _{2}}(x)dxd\tau _{1}d\tau _{2},
\label{sigmae}
\end{equation}%
where 
\begin{equation*}
\bar{C}_{\tau _{1},\tau _{2}}(x)\equiv \int_{\mathcal{U}}Cov\left( \bar{%
\Lambda}_{n,x,\tau _{1}}(\mathbb{W}_{n,\tau _{1},\tau _{2}}^{(1)}(x,u)),\bar{%
\Lambda}_{n,x,\tau _{2}}(\mathbb{W}_{n,\tau _{1},\tau
_{2}}^{(2)}(x,u))\right) du\text{.}
\end{equation*}

Let%
\begin{equation}
\bar{\theta}_{n}(c_{n,1},c_{n,2})\equiv \int \bar{\Lambda}_{x,\tau }\left( 
\mathbf{\hat{s}}_{\tau }(x)\right) dQ(x,\tau )\text{,}  \label{theta_hate}
\end{equation}%
and%
\begin{equation}
\bar{\theta}_{n}^{\ast }(c_{n,1},c_{n,2})\equiv \int \bar{\Lambda}_{x,\tau
}\left( \mathbf{\hat{s}}_{\tau }^{\ast }(x)\right) dQ(x,\tau ).
\label{theta*e}
\end{equation}

From here on, for any sequence of random quantities $Z_{n}$ and a random
vector $Z$, we write 
\begin{equation*}
Z_{n}\overset{d}{\rightarrow }N(0,1)\text{, }\mathcal{P}_{0}\text{-uniformly,%
}
\end{equation*}%
if for each $t>0$,%
\begin{equation*}
\sup_{P\in \mathcal{P}_{0}}\left\vert P\left\{ Z_{n}\leq t\right\} -\Phi
(t)\right\vert =o(1)\text{.}
\end{equation*}%
And for any sequence of bootstrap quantities $Z_{n}^{\ast }$ and a random
vector $Z$, we write 
\begin{equation*}
Z_{n}^{\ast }\overset{d^{\ast }}{\rightarrow }N(0,1)\text{, }\mathcal{P}_{0}%
\text{-uniformly,}
\end{equation*}%
if for each $t>0$,%
\begin{equation*}
\left\vert P^{\ast }\left\{ Z_{n}^{\ast }\leq t\right\} -\Phi (t)\right\vert
=o_{P^{\ast }}(1)\text{, }\mathcal{P}_{0}\text{-uniformly.}
\end{equation*}

\begin{LemmaA}
(i) \textit{Suppose that Assumptions }A\ref{assumption-A1}-A\ref%
{assumption-A3}, A4(i)\textit{, and }A5-A6\textit{\ are satisfied. Then for
any sequences }$c_{n,1},c_{n,2}>0$\textit{\ such that }$\liminf_{n%
\rightarrow \infty }\inf_{P\in \mathcal{P}_{0}}\sigma
_{n}^{2}(c_{n,1},c_{n,2})>0$ and $\sqrt{\log n}/c_{n,2}\rightarrow 0$, as $%
n\rightarrow \infty ,$%
\begin{equation*}
h^{-d/2}\left( \frac{\bar{\theta}%
_{n}(c_{n,1},c_{n,2})-a_{n}^{R}(c_{n,1},c_{n,2})}{\sigma
_{n}(c_{n,1},c_{n,2})}\right) \overset{d}{\rightarrow }N(0,1),\text{ }%
\mathcal{P}_{0}\text{-\textit{uniformly}}.
\end{equation*}

\noindent (ii) \textit{Suppose that Assumptions A\ref{assumption-A1}-A\ref%
{assumption-A3}, A4(i), A5-A6, B\ref{assumption-B1} and B\ref{assumption-B4}
are satisfied. Then for any sequences }$c_{n,1},c_{n,2}\geq 0$\textit{\ such
that }$\liminf_{n\rightarrow \infty }\inf_{P\in \mathcal{P}_{0}}\sigma
_{n}^{2}(c_{n,1},c_{n,2})>0$ and $\sqrt{\log n}/c_{n,2}\rightarrow 0$, as $%
n\rightarrow \infty ,$%
\begin{equation*}
h^{-d/2}\left( \frac{\bar{\theta}_{n}^{\ast }(c_{n,1},c_{n,2})-a_{n}^{R\ast
}(c_{n,1},c_{n,2})}{\sigma _{n}(c_{n,1},c_{n,2})}\right) \overset{d^{\ast }}{%
\rightarrow }N(0,1)\text{, }\mathcal{P}_{0}\text{-\textit{uniformly}}.
\end{equation*}
\end{LemmaA}

\begin{proof}[Proof of Lemma A2]
(i) By Lemma 1, we have (with probability approaching one)%
\begin{equation*}
\bar{\theta}_{n}(c_{n,1},c_{n,2})=\sum_{A\in \mathcal{N}_{J}}%
\int_{B_{n,A}(c_{n,1},c_{n,2})}\Lambda _{p}(\mathbf{\hat{s}}_{\tau
}(x))dQ(x,\tau )=\sum_{A\in \mathcal{N}_{J}}\int_{B_{n,A}(c_{n,1},c_{n,2})}%
\Lambda _{A,p}(\mathbf{\hat{s}}_{\tau }(x))dQ(x,\tau ).
\end{equation*}%
Note that $a_{n}^{R}(c_{n,1},c_{n,2})=\sum_{A\in \mathcal{N}%
_{J}}a_{n,A}^{R}(c_{n,1},c_{n,2})$, where%
\begin{equation*}
a_{n,A}^{R}(c_{n,1},c_{n,2})\equiv \int_{B_{n,A}(c_{n,1},c_{n,2})}\mathbf{E}%
\left[ \Lambda _{A,p}(\sqrt{nh^{d}}\mathbf{z}_{N,\tau }(x))\right] dQ(x,\tau
).
\end{equation*}%
Using Assumption A1, we find that $h^{-d/2}\{\bar{\theta}%
_{n}(c_{n,1},c_{n,2})-a_{n}^{R}(c_{n,1},c_{n,2})\}$ is equal to%
\begin{equation*}
h^{-d/2}\sum_{A\in \mathcal{N}_{J}}\{\zeta _{n,A}(B_{n,A}(c_{n,1},c_{n,2}))-%
\mathbf{E}\zeta _{N,A}(B_{n,A}(c_{n,1},c_{n,2}))\}+o_{P}(1),
\end{equation*}%
where for any Borel set $B\subset \mathcal{S}$,%
\begin{eqnarray*}
\zeta _{n,A}(B) &\equiv &\int_{B}\Lambda _{A,p}(\sqrt{nh^{d}}\mathbf{z}%
_{n,\tau }(x))dQ(x,\tau ), \\
\zeta _{N,A}(B) &\equiv &\int_{B}\Lambda _{A,p}(\sqrt{nh^{d}}\mathbf{z}%
_{N,\tau }(x))dQ(x,\tau ),
\end{eqnarray*}%
and%
\begin{equation*}
\mathbf{z}_{n,\tau }(x)\equiv \frac{1}{nh^{d}}\sum_{i=1}^{n}\beta _{n,x,\tau
}(Y_{i},(X_{i}-x)/h)-\frac{1}{h^{d}}\mathbf{E}\left[ \beta _{n,x,\tau
}(Y_{i},(X_{i}-x)/h)\right] \text{,}
\end{equation*}%
with 
\begin{equation*}
\beta _{n,x,\tau }(Y_{i},(X_{i}-x)/h)=(\beta _{n,x,\tau
,1}(Y_{i1},(X_{i}-x)/h),\cdot \cdot \cdot ,\beta _{n,x,\tau
,J}(Y_{iJ},(X_{i}-x)/h))^{\top }.
\end{equation*}%
We take $0<\bar{\varepsilon}_{l}\rightarrow 0$ as $l\rightarrow \infty $ and
take $\mathcal{C}_{l}\subset \mathbf{R}^{d}$ such that%
\begin{equation*}
0<P\left\{ X_{i}\in \mathbf{R}^{d}\backslash \mathcal{C}_{l}\right\} \leq 
\bar{\varepsilon}_{l},
\end{equation*}%
and $Q((\mathcal{X}\backslash \mathcal{C}_{l})\times \mathcal{T})\rightarrow
0$ as $l\rightarrow \infty $. Such a sequence $\{\bar{\varepsilon}%
_{l}\}_{l=1}^{\infty }$ exists by Assumption A6(ii) by the condition that $%
\mathcal{S}$ is compact. We write%
\begin{eqnarray}
&&\frac{h^{-d/2}\sum_{A\in \mathcal{N}_{J}}\{\zeta
_{n,A}(B_{n,A}(c_{n,1},c_{n,2}))-\mathbf{E}\zeta
_{N,A}(B_{n,A}(c_{n,1},c_{n,2}))\}}{\sigma _{n}^{2}(c_{n,1},c_{n,2})}
\label{dec56} \\
&=&\frac{h^{-d/2}\sum_{A\in \mathcal{N}_{J}}\{\zeta
_{n,A}(B_{n,A}(c_{n,1},c_{n,2})\cap (\mathcal{C}_{l}\times \mathcal{T}))-%
\mathbf{E}\zeta _{N,A}(B_{n,A}(c_{n,1},c_{n,2})\cap (\mathcal{C}_{l}\times 
\mathcal{T}))\}}{\sigma _{n}^{2}(c_{n,1},c_{n,2})}  \notag \\
&&+\frac{h^{-d/2}\sum_{A\in \mathcal{N}_{J}}\{\zeta
_{n,A}(B_{n,A}(c_{n,1},c_{n,2})\backslash (\mathcal{C}_{l}\times \mathcal{T}%
))-\mathbf{E}\zeta _{N,A}(B_{n,A}(c_{n,1},c_{n,2})\backslash (\mathcal{C}%
_{l}\times \mathcal{T}))\}}{\sigma _{n}^{2}(c_{n,1},c_{n,2})}  \notag \\
&=&A_{1n}+A_{2n}\text{, say.}  \notag
\end{eqnarray}%
As for $A_{2n}$, we apply Lemma B7 in Appendix B, and the condition that $Q((%
\mathcal{X}\backslash \mathcal{C}_{l})\times \mathcal{T})\rightarrow 0$, as $%
l\rightarrow \infty $, and 
\begin{equation*}
\text{liminf}_{n\rightarrow \infty }\text{inf}_{P\in \mathcal{P}_{0}}\sigma
_{n}(c_{1n},c_{2n})>0\text{,}
\end{equation*}%
to deduce that $A_{2n}=o_{P}(1)$, as $n\rightarrow \infty $ and then $%
l\rightarrow \infty $. As for $A_{1n}$, first observe that as $n\rightarrow
\infty $ and then $l\rightarrow \infty ,$%
\begin{equation}
\left\vert \sigma _{n}^{2}(c_{n,1},c_{n,2})-\bar{\sigma}%
_{n,l}^{2}(c_{n,1},c_{n,2})\right\vert \rightarrow 0,  \label{conv}
\end{equation}%
where $\bar{\sigma}_{n,l}^{2}(c_{n,1},c_{n,2})$ is equal to $\sigma
_{n}^{2}(c_{n,1},c_{n,2})$ except that $B_{n,A}(c_{n,1},c_{n,2})$'s are
replaced by $B_{n,A}(c_{n,1},c_{n,2})\cap (\mathcal{C}_{l}\times \mathcal{T}%
) $. The convergence follows by Assumption 6(i). Also by Lemma B9(i) and the
convergence in (\ref{conv}) and the fact that 
\begin{equation*}
\liminf_{n\rightarrow \infty }\inf_{P\in \mathcal{P}_{0}}\sigma
_{n}^{2}(c_{n,1},c_{n,2})>0,
\end{equation*}%
we have%
\begin{equation*}
A_{1n}\overset{d}{\rightarrow }N(0,1),\text{ }\mathcal{P}_{0}\text{%
-uniformly,}
\end{equation*}%
as $n\rightarrow \infty $ and as $l\rightarrow \infty $. Hence we obtain (i).

\noindent (ii) The proof can be done in the same way as in the proof of (i),
using Lemmas C7 and C9(i) in Appendix C instead of Lemmas B7 and B9(i) in
Appendix B.
\end{proof}

\begin{LemmaA}
\textit{Suppose that Assumptions A\ref{assumption-A1}-A\ref{assumption-A5}
hold. Then for any sequences }$c_{n,L},c_{n,U}>0$\textit{\ satisfying
Assumption A4(ii), and for each} $A\in \mathcal{N}_{J},$%
\begin{equation*}
\inf_{P\in \mathcal{P}}P\left\{ B_{n,A}(c_{n,L},c_{n,U})\subset \hat{B}_{A}(%
\hat{c}_{n})\subset B_{n,A}(c_{n,U},c_{n,L})\right\} \rightarrow 1\text{%
\textit{, as} }n\rightarrow \infty \text{.}
\end{equation*}
\end{LemmaA}

\begin{proof}[Proof of Lemma A3]
By using Assumptions A\ref{assumption-A3}-A\ref{assumption-A5}, and
following the proof of Theorem 2, Claim 1 in %
\citeasnoun{Linton/Song/Whang:10}, we can complete the proof. Details are
omitted.
\end{proof}

Define for $c_{n,1},c_{n,2}>0,$%
\begin{eqnarray*}
T_{n}(c_{n,1},c_{n,2}) &\equiv &h^{-d/2}\left( \frac{\bar{\theta}%
_{n}(c_{n,1},c_{n,2})-a_{n}(c_{n,1},c_{n,2})}{\sigma _{n}(c_{n,1},c_{n,2})}%
\right) \text{ and} \\
T_{n}^{\ast }(c_{n,1},c_{n,2}) &\equiv &h^{-d/2}\left( \frac{\bar{\theta}%
_{n}^{\ast }(c_{n,1},c_{n,2})-a_{n}(c_{n,1},c_{n,2})}{\sigma
_{n}(c_{n,1},c_{n,2})}\right) .
\end{eqnarray*}%
We introduce critical values for the finite sample distribution of $\hat{%
\theta}$ as follows:%
\begin{equation*}
\gamma _{n}^{\alpha }(c_{n,1},c_{n,2})\equiv \inf \left\{ c\in \mathbf{R}%
:P\left\{ T_{n}(c_{n,1},c_{n,2})\leq c\right\} >1-\alpha \right\} \text{.}
\end{equation*}%
Similarly, let us introduce bootstrap critical values:%
\begin{equation}
\gamma _{n}^{\alpha \ast }(c_{n,1},c_{n,2})\equiv \inf \left\{ c\in \mathbf{R%
}:P^{\ast }\left\{ T_{n}^{\ast }(c_{n,1},c_{n,2})\leq c\right\} >1-\alpha
\right\} \text{.}  \label{c_tildee*}
\end{equation}%
Finally, we introduce asymptotic critical values:\ $\gamma _{\infty
}^{\alpha }\equiv \Phi ^{-1}(1-\alpha ),$ where $\Phi $ denotes the standard
normal CDF.

\begin{LemmaA}
Suppose that Assumptions A\ref{assumption-A1}-A\ref{assumption-A3}, A4(i)%
\textit{, and }A5-A6 hold. Then the following holds.

\noindent (i)\textit{\ }For any $c_{n,1},c_{n,2}\rightarrow \infty $ such
that 
\begin{equation*}
\liminf_{n\rightarrow \infty }\inf_{P\in \mathcal{P}}\sigma
_{n}^{2}(c_{n,1},c_{n,2})>0,
\end{equation*}%
it is satisfied that%
\begin{equation*}
\sup_{P\in \mathcal{P}}\left\vert \gamma _{n}^{\alpha
}(c_{n,1},c_{n,2})-\gamma _{\infty }^{\alpha }\right\vert \rightarrow 0,\ 
\text{as\ }n\rightarrow \infty .
\end{equation*}

\noindent (ii) \textit{Suppose further that Assumptions B\ref{assumption-B1}
and B\ref{assumption-B4} hold. Then }for any $c_{n,1},c_{n,2}\rightarrow
\infty $ such that 
\begin{equation*}
\liminf_{n\rightarrow \infty }\inf_{P\in \mathcal{P}}\sigma
_{n}^{2}(c_{n,1},c_{n,2})>0,
\end{equation*}%
it is satisfied that%
\begin{equation*}
\sup_{P\in \mathcal{P}}\left\vert \gamma _{n}^{\alpha \ast
}(c_{n,1},c_{n,2})-\gamma _{\infty }^{\alpha }\right\vert \rightarrow 0,\ 
\text{as\ }n\rightarrow \infty .
\end{equation*}
\end{LemmaA}

\begin{proof}[Proof of Lemma A4]
(i) The statement immediately follows from the first statement of Lemma
A2(i) and Lemma A1.

\noindent (ii) We show only the second statement. Fix $a>0$. Let us
introduce two events:%
\begin{equation*}
E_{n,1}\equiv \left\{ \gamma _{n}^{\alpha \ast }(c_{n,1},c_{n,2})-\gamma
_{\infty }^{\alpha }<-a\right\} \text{ and }E_{n,2}\equiv \left\{ \gamma
_{n}^{\alpha \ast }(c_{n,1},c_{n,2})-\gamma _{\infty }^{\alpha }>a\right\} .
\end{equation*}%
On the event $E_{n,1},$ we have%
\begin{eqnarray*}
\alpha &=&P^{\ast }\left\{ h^{-d/2}\left( \frac{\bar{\theta}_{n}^{\ast
}(c_{n,1},c_{n,2})-a_{n}(c_{n,1},c_{n,2})}{\sigma _{n}(c_{n,1},c_{n,2})}%
\right) >\gamma _{n}^{\alpha \ast }(c_{n,1},c_{n,2})\right\} \\
&\geq &P^{\ast }\left\{ h^{-d/2}\left( \frac{\bar{\theta}_{n}^{\ast
}(c_{n,1},c_{n,2})-a_{n}(c_{n,1},c_{n,2})}{\sigma _{n}(c_{n,1},c_{n,2})}%
\right) >\gamma _{\infty }^{\alpha }-a\right\} .
\end{eqnarray*}%
By Lemma A2(ii) and Lemma A1, the last probability is equal to%
\begin{equation*}
1-\Phi \left( \gamma _{\infty }^{\alpha }-a\right) +o_{P}(1)>\alpha
+o_{P}(1),
\end{equation*}%
where $o_{P}(1)$ is uniform over $P\in \mathcal{P}$ and the last strict
inequality follows by the definition of $\gamma _{\infty }^{\alpha }$ and $%
a>0.$ Hence $\sup_{P\in \mathcal{P}}PE_{n,1}\rightarrow 0$ as $n\rightarrow
\infty $. Similarly, on the event $E_{n,2}$, we have%
\begin{eqnarray*}
\alpha &=&P^{\ast }\left\{ h^{-d/2}\left( \frac{\bar{\theta}_{n}^{\ast
}(c_{n,1},c_{n,2})-a_{n}(c_{n,1},c_{n,2})}{\sigma _{n}(c_{n,1},c_{n,2})}%
\right) >\gamma _{n}^{\alpha \ast }(c_{n,1},c_{n,2})\right\} \\
&\leq &P^{\ast }\left\{ h^{-d/2}\left( \frac{\bar{\theta}_{n}^{\ast
}(c_{n,1},c_{n,2})-a_{n}(c_{n,1},c_{n,2})}{\sigma _{n}(c_{n,1},c_{n,2})}%
\right) >\gamma _{\infty }^{\alpha }+a\right\} .
\end{eqnarray*}%
By the first statement of Lemma A2(ii) and Lemma A1, the last bootstrap
probability is bounded by%
\begin{equation*}
1-\Phi \left( \gamma _{\infty }^{\alpha }+a\right) +o_{P}(1)<\alpha
+o_{P}(1),
\end{equation*}%
so that we have $\sup_{P\in \mathcal{P}}PE_{n,2}\rightarrow 0$ as $%
n\rightarrow \infty $.\textbf{\ }We conclude that%
\begin{equation*}
\sup_{P\in \mathcal{P}}P\left\{ |\gamma _{n}^{\alpha \ast
}(c_{n,1},c_{n,2})-\gamma _{\infty }^{\alpha }|>a\right\} =\sup_{P\in 
\mathcal{P}}\left( PE_{n,1}+PE_{n,2}\right) \rightarrow 0,
\end{equation*}%
as $n\rightarrow \infty $, obtaining the desired result.
\end{proof}

\begin{proof}[Proof of Theorem \protect\ref{Thm1}]
By Lemma 1, we have 
\begin{equation*}
\inf_{P\in \mathcal{P}_{0}}P\left\{ \hat{\theta}=\sum_{A\in \mathcal{N}%
_{J}}\int_{B_{n,A}(c_{n,L},c_{n,U})}\Lambda _{A,p}\left( \mathbf{\hat{u}}%
_{\tau }(x)\right) dQ(x,\tau )\right\} \rightarrow 1,
\end{equation*}%
as $n\rightarrow \infty $. Since under the null hypothesis, we have $%
v_{n,\tau ,j}(\cdot )/\hat{\sigma}_{\tau ,j}(\cdot )\leq 0$ for all $j\in 
\mathbb{N}_{J},$ with probability approaching one by Assumption A5, we have%
\begin{eqnarray*}
&&\sum_{A\in \mathcal{N}_{J}}\int_{B_{n,A}(c_{n,L},c_{n,U})}\Lambda
_{A,p}\left( \mathbf{\hat{u}}_{\tau }(x)\right) dQ(x,\tau ) \\
&\leq &\sum_{A\in \mathcal{N}_{J}}\int_{B_{n,A}(c_{n,L},c_{n,U})}\Lambda
_{A,p}\left( \mathbf{\hat{s}}_{\tau }(x)\right) dQ(x,\tau )\equiv \bar{\theta%
}_{n}(c_{n,L},c_{n,U}).
\end{eqnarray*}%
Thus, we have as $n\rightarrow \infty ,$%
\begin{equation}
\inf_{P\in \mathcal{P}_{0}}P\left\{ \hat{\theta}\leq \bar{\theta}%
_{n}(c_{n,L},c_{n,U})\right\} \rightarrow 1.  \label{ineq42b}
\end{equation}

Let the $(1-\alpha )$-th percentile of the bootstrap distribution of%
\begin{equation*}
\bar{\theta}_{n}^{\ast }(c_{n,L},c_{n,U})=\sum_{A\in \mathcal{N}%
_{J}}\int_{B_{n,A}(c_{n,L},c_{n,U})}\Lambda _{A,p}(\mathbf{\hat{s}}_{\tau
}^{\ast }(x))dQ(x,\tau )
\end{equation*}%
be denoted by $\bar{c}_{n,L}^{\alpha \ast }$. By Lemma A3 and Assumption
A4(ii), with probability approaching one,%
\begin{equation}
\sum_{A\in \mathcal{N}_{J}}\int_{B_{n,A}(c_{n,L},c_{n,U})}\Lambda
_{A,p}\left( \mathbf{\hat{s}}_{\tau }^{\ast }(x)\right) dQ(x,\tau )\leq
\sum_{A\in \mathcal{N}_{J}}\int_{\hat{B}_{A}(\hat{c}_{n})}\Lambda
_{A,p}\left( \mathbf{\hat{s}}_{\tau }^{\ast }(x)\right) dQ(x,\tau ).
\label{comp}
\end{equation}%
This implies that as $n\rightarrow \infty ,$%
\begin{equation}
\inf_{P\in \mathcal{P}}P\left\{ c_{\alpha }^{\ast }\geq \bar{c}%
_{n,L}^{\alpha \ast }\right\} \rightarrow 1\text{.}  \label{comp2}
\end{equation}%
There exists a sequence of probabilities $\{P_{n}\}_{n\geq 1}\subset 
\mathcal{P}_{0}$ such that 
\begin{eqnarray}
\underset{n\rightarrow \infty }{\text{limsup}}\sup_{P\in \mathcal{P}%
_{0}}P\left\{ \hat{\theta}>c_{\alpha ,\eta }^{\ast }\right\} &=&\ \underset{%
n\rightarrow \infty }{\text{limsup}}P_{n}\left\{ \hat{\theta}>c_{\alpha
,\eta }^{\ast }\right\}  \label{limit} \\
&=&\text{lim}_{n\rightarrow \infty }P_{w_{n}}\left\{ \hat{\theta}%
_{w_{n}}>c_{w_{n},\alpha ,\eta }^{\ast }\right\} ,  \notag
\end{eqnarray}%
where $\{w_{n}\}\subset \{n\}$ is a certain subsequence, and $\hat{\theta}%
_{w_{n}}$ and $c_{w_{n},\alpha ,\eta }^{\ast }$ are the same as $\hat{\theta}
$ and $c_{\alpha ,\eta }^{\ast }$ except that the sample size $n$ is now
replaced by $w_{n}$.\newline
By Assumption A6(i), $\{\sigma _{n}(c_{n,L},c_{n,U})\}_{n\geq 1}$ is a
bounded sequence. Therefore, there exists a subsequence $\{u_{n}\}_{n\geq
1}\subset \{w_{n}\}_{n\geq 1},$ such that $\sigma
_{u_{n}}(c_{u_{n},L},c_{u_{n},U})$ converges. We consider two cases:

\textbf{Case 1: }lim$_{n\rightarrow \infty }\sigma
_{u_{n}}(c_{u_{n},L},c_{u_{n},U})>0,$ and

\textbf{Case 2:} lim$_{n\rightarrow \infty }\sigma
_{u_{n}}(c_{u_{n},L},c_{u_{n},U})=0.$

In both case, we will show below that 
\begin{equation}
\underset{n\rightarrow \infty }{\text{limsup}}P_{u_{n}}\{\hat{\theta}%
_{u_{n}}>c_{u_{n},\alpha ,\eta }^{\ast }\}\leq \alpha .  \label{control}
\end{equation}%
Since along $\{w_{n}\}$, $P_{w_{n}}\{\hat{\theta}_{w_{n}}>c_{w_{n},\alpha
,\eta }^{\ast }\}$ converges, it does so along any subsequence of $\{w_{n}\}$%
. Therefore, the above limsup is equal to the last limit in (\ref{limit}).
This completes the proof.\newline
\textbf{Proof of (\ref{control}) in Case 1: }We write $P_{u_{n}}\{\hat{\theta%
}_{u_{n}}>c_{u_{n},\alpha ,\eta }^{\ast }\}$ as%
\begin{eqnarray*}
&&P_{u_{n}}\left( h^{-d/2}\left( \frac{\hat{\theta}%
_{u_{n}}-a_{u_{n}}(c_{u_{n},L},c_{u_{n},U})}{\sigma
_{u_{n}}(c_{u_{n},L},c_{u_{n},U})}\right) >h^{-d/2}\left( \frac{%
c_{u_{n},\alpha ,\eta }^{\ast }-a_{u_{n}}(c_{u_{n},L},c_{u_{n},U})}{\sigma
_{u_{n}}(c_{u_{n},L},c_{u_{n},U})}\right) \right) \\
&\leq &P_{u_{n}}\left( h^{-d/2}\left( \frac{\hat{\theta}%
_{u_{n}}-a_{u_{n}}(c_{u_{n},L},c_{u_{n},U})}{\sigma
_{u_{n}}(c_{u_{n},L},c_{u_{n},U})}\right) >h^{-d/2}\left( \frac{\bar{c}%
_{u_{n},L}^{\alpha \ast }-a_{u_{n}}(c_{u_{n},L},c_{u_{n},U})}{\sigma
_{u_{n}}(c_{u_{n},L},c_{u_{n},U})}\right) \right) +o(1),
\end{eqnarray*}%
where the inequality follows by the fact that $c_{\alpha ,\eta }^{\ast }\geq
c_{\alpha }^{\ast }\geq \bar{c}_{n,L}^{\alpha \ast }$ with probability
approaching one by (\ref{comp2}). Using (\ref{ineq42b}), we bound the last
probability by 
\begin{equation}
P_{u_{n}}\left\{ h^{-d/2}\left( \frac{\bar{\theta}%
_{u_{n}}(c_{u_{n},L},c_{u_{n},U})-a_{u_{n}}(c_{u_{n},L},c_{u_{n},U})}{\sigma
_{u_{n}}(c_{u_{n},L},c_{u_{n},U})}\right) >h^{-d/2}\left( \frac{\bar{c}%
_{u_{n},L}^{\alpha \ast }-a_{u_{n}}(c_{u_{n},L},c_{u_{n},U})}{\sigma
_{u_{n}}(c_{u_{n},L},c_{u_{n},U})}\right) \right\} +o(1).  \label{probs}
\end{equation}%
Therefore, since lim$_{n\rightarrow \infty }\sigma
_{u_{n}}(c_{u_{n},L},c_{u_{n},U})>0$, by Lemmas A2 and A4, we rewrite the
last probability in (\ref{probs}) as%
\begin{eqnarray*}
&&P_{u_{n}}\left\{ h^{-d/2}\left( \frac{\bar{\theta}%
_{u_{n}}(c_{u_{n},L},c_{u_{n},U})-a_{u_{n}}(c_{u_{n},L},c_{u_{n},U})}{\sigma
_{u_{n}}(c_{u_{n},L},c_{u_{n},U})}\right) >\gamma _{u_{n}}^{\alpha \ast
}(c_{u_{n},L},c_{u_{n},U})\right\} +o(1) \\
&=&P_{u_{n}}\left\{ h^{-d/2}\left( \frac{\bar{\theta}%
_{u_{n}}(c_{u_{n},L},c_{u_{n},U})-a_{u_{n}}(c_{u_{n},L},c_{u_{n},U})}{\sigma
_{u_{n}}(c_{u_{n},L},c_{u_{n},U})}\right) >\gamma _{\infty }^{\alpha
}\right\} +o(1)=\alpha +o(1).
\end{eqnarray*}%
This completes the proof of Step 1.\bigskip \newline
\textbf{Proof of (\ref{control}) in Case 2:} First, observe that%
\begin{equation*}
a_{u_{n}}^{\ast }(c_{u_{n},L},c_{u_{n},U})\leq a_{u_{n}}^{\ast }(\hat{c}%
_{u_{n}}),
\end{equation*}%
with probability approaching one by Lemma A3. Hence using this and (\ref%
{ineq42b}),%
\begin{eqnarray*}
P_{u_{n}}\left\{ \hat{\theta}_{u_{n}}>c_{u_{n},\alpha ,\eta }^{\ast
}\right\} &=&P_{u_{n}}\left\{ h^{-d/2}\left( \hat{\theta}%
_{u_{n}}-a_{u_{n}}(c_{u_{n},L},c_{u_{n},U})\right) >h^{-d/2}\left(
c_{u_{n},\alpha ,\eta }^{\ast }-a_{u_{n}}(c_{u_{n},L},c_{u_{n},U})\right)
\right\} \\
&\leq &P_{u_{n}}\left\{ 
\begin{array}{c}
h^{-d/2}\left( \bar{\theta}%
_{u_{n}}(c_{u_{n},L},c_{u_{n},U})-a_{u_{n}}(c_{u_{n},L},c_{u_{n},U})\right)
\\ 
>h^{-d/2}\left( h^{d/2}\eta +a_{u_{n}}^{\ast
}(c_{u_{n},L},c_{u_{n},U})-a_{u_{n}}(c_{u_{n},L},c_{u_{n},U})\right)%
\end{array}%
\right\} +o(1).
\end{eqnarray*}%
By Lemma A1, the leading probability is equal to%
\begin{equation*}
P_{u_{n}}\left\{ h^{-d/2}\left( \bar{\theta}%
_{u_{n}}(c_{u_{n},L},c_{u_{n},U})-a_{u_{n}}(c_{u_{n},L},c_{u_{n},U})\right)
>\eta +o_{P}(1)\right\} +o(1).
\end{equation*}%
Since $\eta >0$ and lim$_{n\rightarrow \infty }\sigma
_{u_{n}}(c_{u_{n},L},c_{u_{n},U})=0$, the leading probability vanishes by
Lemma B9(ii).
\end{proof}

\begin{proof}[Proof of Theorem \protect\ref{Thm2}]
We focus on probabilities $P\in \mathcal{P}_{n}(\lambda _{n},q_{n})\cap 
\mathcal{P}_{0}$. Recalling the definition of $\mathbf{u}_{n,\tau }(x;\hat{%
\sigma})\equiv \left[ r_{n,j}v_{n,\tau ,j}(x)/\hat{\sigma}_{\tau ,j}(x)%
\right] _{j\in \mathbb{N}_{J}}$ and applying Lemma 1 along with the
condition that%
\begin{equation*}
\sqrt{\log n}/c_{n,U}<\sqrt{\log n}/c_{n,L}\rightarrow 0,
\end{equation*}%
as $n\rightarrow \infty $, we find that with probability approaching one,%
\begin{eqnarray*}
\hat{\theta} &=&\sum_{A\in \mathcal{N}_{J}}\int_{B_{n,A}(c_{n,U},c_{n,L})}%
\Lambda _{A,p}\left( \mathbf{\hat{s}}_{\tau }(x)+\mathbf{u}_{n,\tau }(x;\hat{%
\sigma})\right) dQ(x,\tau ) \\
&=&\sum_{A\in \mathcal{N}_{J}}\int_{B_{n,A}(q_{n})}\Lambda _{A,p}\left( 
\mathbf{\hat{s}}_{\tau }(x)+\mathbf{u}_{n,\tau }(x;\hat{\sigma})\right)
dQ(x,\tau ) \\
&&+\sum_{A\in \mathcal{N}_{J}}\int_{B_{n,A}(c_{n,U},c_{n,L})\backslash
B_{n,A}(q_{n})}\Lambda _{A,p}\left( \mathbf{\hat{s}}_{\tau }(x)+\mathbf{u}%
_{n,\tau }(x;\hat{\sigma})\right) dQ(x,\tau ).
\end{eqnarray*}%
Since under $P\in \mathcal{P}_{0}$, $\mathbf{u}_{n,\tau }(x;\hat{\sigma}%
)\leq 0$ for all $x\in \mathcal{S}$, with probability approaching one by
Assumption 5, the last term multiplied by $h^{-d/2}$ is bounded by (from
some large $n$ on)%
\begin{eqnarray*}
&&h^{-d/2}\sum_{A\in \mathcal{N}_{J}}\int_{B_{n,A}(c_{n,U},c_{n,L})%
\backslash B_{n,A}(q_{n})}\Lambda _{A,p}\left( \mathbf{\hat{s}}_{\tau
}(x)\right) dQ(x,\tau ) \\
&\leq &h^{-d/2}\sum_{A\in \mathcal{N}_{J}}\left( \sup_{(x,\tau )\in \mathcal{%
S}}||\mathbf{\hat{s}}_{\tau }(x)||\right) ^{p}Q\left(
B_{n,A}(c_{n,U},c_{n,L})\backslash B_{n,A}(q_{n})\right) \\
&=&O_{P}\left( h^{-d/2}(\log n)^{p/2}\lambda _{n}\right) =o_{P}(1),
\end{eqnarray*}%
where the second to the last equality follows because $Q\left(
B_{n,A}(c_{n,U},c_{n,L})\backslash B_{n,A}(q_{n})\right) \leq \lambda _{n}$
by the definition of $\mathcal{P}_{n}(\lambda _{n},q_{n})$, and the last
equality follows by (\ref{lamq}).

On the other hand,%
\begin{eqnarray*}
&&h^{-d/2}\sum_{A\in \mathcal{N}_{J}}\int_{B_{n,A}(q_{n})}\Lambda
_{A,p}\left( \mathbf{\hat{s}}_{\tau }(x)+\mathbf{u}_{n,\tau }(x;\hat{\sigma}%
)\right) dQ(x,\tau ) \\
&=&h^{-d/2}\sum_{A\in \mathcal{N}_{J}}\int_{B_{n,A}(q_{n})}\Lambda
_{A,p}\left( \mathbf{\hat{s}}_{\tau }(x)\right) dQ(x,\tau ) \\
&&+h^{-d/2}\sum_{A\in \mathcal{N}_{J}}\int_{B_{n,A}(q_{n})}\Lambda
_{A,p}\left( \mathbf{\hat{s}}_{\tau }(x)+\mathbf{u}_{n,\tau }(x;\hat{\sigma}%
)\right) dQ(x,\tau ) \\
&&-h^{-d/2}\sum_{A\in \mathcal{N}_{J}}\int_{B_{n,A}(q_{n})}\Lambda
_{A,p}\left( \mathbf{\hat{s}}_{\tau }(x)\right) dQ(x,\tau ).
\end{eqnarray*}%
From the definition of $\Lambda _{p}$ in (\ref{lambda_p}), the last
difference (in absolute value) is bounded by%
\begin{eqnarray*}
&&Ch^{-d/2}\sum_{A\in \mathcal{N}_{J}}\int_{B_{n,A}(q_{n})}\left\Vert [%
\mathbf{u}_{n,\tau }(x;\hat{\sigma})]_{A}\right\Vert \left\Vert [\mathbf{%
\hat{s}}_{\tau }(x)]_{A}\right\Vert ^{p-1}dQ(x,\tau ) \\
&&+Ch^{-d/2}\sum_{A\in \mathcal{N}_{J}}\int_{B_{n,A}(q_{n})}\left\Vert [%
\mathbf{u}_{n,\tau }(x;\hat{\sigma})]_{A}\right\Vert \left\Vert [\mathbf{u}%
_{n,\tau }(x;\hat{\sigma})]_{A}\right\Vert ^{p-1}dQ(x,\tau ),
\end{eqnarray*}%
where $[a]_{A}$ is a vector $a$ with the $j$-th entry is set to be zero for
all $j\in \mathbb{N}_{J}\backslash A$ and $C>0$ is a constant that does not
depend on $n\geq 1$ or $P\in \mathcal{P}$. We have sup$_{(x,\tau )\in
B_{n,A}(q_{n})}\left\Vert [\mathbf{u}_{n,\tau }(x;\hat{\sigma}%
)]_{A}\right\Vert \leq q_{n}(1+o_{P}(1))$, by the null hypothesis and by
Assumption A5. Also, by Assumptions A3 and A5, 
\begin{equation*}
\text{sup}_{(x,\tau )\in B_{n,A}(q_{n})}\left\Vert [\mathbf{\hat{s}}_{\tau
}(x)]_{A}\right\Vert =O_{P}\left( \sqrt{\log n}\right) .
\end{equation*}%
Therefore, we conclude that%
\begin{eqnarray*}
&&h^{-d/2}\sum_{A\in \mathcal{N}_{J}}\int_{B_{n,A}(q_{n})}\Lambda
_{A,p}\left( \mathbf{\hat{s}}_{\tau }(x)+\mathbf{u}_{n,\tau }(x;\hat{\sigma}%
)\right) dQ(x,\tau ) \\
&=&h^{-d/2}\sum_{A\in \mathcal{N}_{J}}\int_{B_{n,A}(q_{n})}\Lambda
_{A,p}\left( \mathbf{\hat{s}}_{\tau }(x)\right) dQ(x,\tau )+O_{P}\left(
h^{-d/2}q_{n}\{(\log n)^{(p-1)/2}+q_{n}^{p-1}\}\right) .
\end{eqnarray*}%
The last $O_{P}(1)$ term is $o_{P}(1)$ by the condition for $q_{n}$ in (\ref%
{lamq}). Thus we find that%
\begin{equation}
\hat{\theta}=\bar{\theta}_{n}(q_{n})+o_{P}(h^{d/2}),  \label{cv33}
\end{equation}%
where $\bar{\theta}_{n}(q_{n})=\sum_{A\in \mathcal{N}_{J}}%
\int_{B_{n,A}(q_{n})}\Lambda _{A,p}\left( \mathbf{\hat{s}}_{\tau }(x)\right)
dQ(x,\tau ).$

Now let us consider the bootstrap statistic. We write%
\begin{eqnarray*}
\hat{\theta}^{\ast } &=&\sum_{A\in \mathcal{N}_{J}}\int_{\hat{B}_{A}(\hat{c}%
_{n})}\Lambda _{A,p}\left( \mathbf{\hat{s}}_{\tau }^{\ast }(x)\right)
dQ(x,\tau ) \\
&=&\sum_{A\in \mathcal{N}_{J}}\int_{B_{n,A}(q_{n})}\Lambda _{A,p}\left( 
\mathbf{\hat{s}}_{\tau }^{\ast }(x)\right) dQ(x,\tau )+\sum_{A\in \mathcal{N}%
_{J}}\int_{\hat{B}_{A}(\hat{c}_{n})\backslash B_{n,A}(q_{n})}\Lambda
_{A,p}\left( \mathbf{\hat{s}}_{\tau }^{\ast }(x)\right) dQ(x,\tau ).
\end{eqnarray*}%
By Lemma A3, we find that 
\begin{equation*}
\inf_{P\in \mathcal{P}}P\left\{ \hat{B}_{n,A}(\hat{c}_{n})\subset
B_{n,A}(c_{n,U},c_{n,L})\right\} \rightarrow 1,\text{ as }n\rightarrow
\infty ,
\end{equation*}%
so that%
\begin{equation*}
\sum_{A\in \mathcal{N}_{J}}\int_{\hat{B}_{A}(\hat{c}_{n})\backslash
B_{n,A}(q_{n})}\Lambda _{A,p}\left( \mathbf{\hat{s}}_{\tau }^{\ast
}(x)\right) dQ(x,\tau )\leq \sum_{A\in \mathcal{N}_{J}}%
\int_{B_{n,A}(c_{n,U},c_{n,L})\backslash B_{n,A}(q_{n})}\Lambda _{A,p}\left( 
\mathbf{\hat{s}}_{\tau }^{\ast }(x)\right) dQ(x,\tau ),
\end{equation*}%
with probability approaching one. The last term multiplied by $h^{-d/2}$ is
bounded by%
\begin{eqnarray*}
&&h^{-d/2}\left( \sup_{(x,\tau )\in \mathcal{S}}||\mathbf{\hat{s}}_{\tau
}^{\ast }(x)||\right) ^{p}\sum_{A\in \mathcal{N}_{J}}Q\left(
B_{n,A}(c_{n,U},c_{n,L})\backslash B_{n,A}(q_{n})\right) \\
&=&O_{P^{\ast }}\left( h^{-d/2}(\log n)^{p/2}\lambda _{n}\right) =o_{P^{\ast
}}(1),\text{ }\mathcal{P}_{n}(\lambda _{n},q_{n})\text{-uniformly,}
\end{eqnarray*}%
where the second to the last equality follows by Assumption B2 and the
definition of $\mathcal{P}_{n}(\lambda _{n},q_{n})$, and the last equality
follows by (\ref{lamq}). Thus, we conclude that%
\begin{equation}
\frac{h^{-d/2}(\hat{\theta}^{\ast }-a_{n}(q_{n}))}{\sigma _{n}(q_{n})}=\frac{%
h^{-d/2}\left( \bar{\theta}_{n}^{\ast }(q_{n})-a_{n}(q_{n})\right) }{\sigma
_{n}(q_{n})}+o_{P^{\ast }}(1),\text{ }\mathcal{P}_{n}(\lambda _{n},q_{n})%
\text{-uniformly,}  \label{cv22}
\end{equation}%
where%
\begin{equation*}
\bar{\theta}^{\ast }(q_{n})\equiv \sum_{A\in \mathcal{N}_{J}}%
\int_{B_{n,A}(q_{n})}\Lambda _{A,p}\left( \mathbf{\hat{s}}_{\tau }^{\ast
}(x)\right) dQ(x,\tau ).
\end{equation*}

Using the same arguments, we also observe that 
\begin{equation}
\hat{a}^{\ast }=\hat{a}^{\ast
}(q_{n})+o_{P}(h^{d/2})=a_{n}(q_{n})+o_{P}(h^{d/2}),  \label{a}
\end{equation}%
where the last equality uses Lemma A1. Let the $(1-\alpha )$-th percentile
of the bootstrap distribution of $\bar{\theta}^{\ast }(q_{n})$ be denoted by 
$\bar{c}_{n}^{\alpha \ast }(q_{n})$. Then by (\ref{cv22}), we have 
\begin{equation}
\frac{h^{-d/2}\left( c_{\alpha }^{\ast }-a_{n}(q_{n})\right) }{\sigma
_{n}(q_{n})}=\frac{h^{-d/2}\left( \bar{c}_{n}^{\alpha \ast
}(q_{n})-a_{n}(q_{n})\right) }{\sigma _{n}(q_{n})}+o_{P^{\ast }}(1),\text{ }%
\mathcal{P}_{n}(\lambda _{n},q_{n})\text{-uniformly.}  \label{cv7}
\end{equation}%
By Lemma A4(ii) and by the condition that $\sigma _{n}(q_{n})\geq \eta /\Phi
^{-1}(1-\alpha )$, the leading term on the right hand side is equal to%
\begin{equation*}
\Phi ^{-1}(1-\alpha )+o_{P^{\ast }}(1)\text{, }\mathcal{P}_{n}(\lambda
_{n},q_{n})\text{-uniformly.}
\end{equation*}%
Note that 
\begin{equation}
c_{\alpha }^{\ast }\geq h^{d/2}\eta +\hat{a}_{n}^{\ast }+o_{P}(h^{d/2}),
\label{c_ehta}
\end{equation}%
by the restriction $\sigma _{n}(q_{n})\geq \eta /\Phi ^{-1}(1-\alpha )$ in
the definition of $\mathcal{P}_{n}(\lambda _{n},q_{n})$ and (\ref{a}). Using
this, and following the proof of Step 1 in the proof of Theorem 2, we deduce
that%
\begin{eqnarray*}
&&P\left\{ h^{-d/2}\left( \frac{\hat{\theta}-a_{n}(q_{n})}{\sigma _{n}(q_{n})%
}\right) >h^{-d/2}\left( \frac{c_{\alpha ,\eta }^{\ast }-a_{n}(q_{n})}{%
\sigma _{n}(q_{n})}\right) \right\} \\
&=&P\left\{ h^{-d/2}\left( \frac{\bar{\theta}_{n}(q_{n})-a_{n}(q_{n})}{%
\sigma _{n}(q_{n})}\right) >h^{-d/2}\left( \frac{c_{\alpha }^{\ast
}-a_{n}(q_{n})}{\sigma _{n}(q_{n})}\right) \right\} +o(1) \\
&=&P\left\{ h^{-d/2}\left( \frac{\bar{\theta}_{n}(q_{n})-a_{n}(q_{n})}{%
\sigma _{n}(q_{n})}\right) >h^{-d/2}\left( \frac{\bar{c}_{n}^{\alpha \ast
}(q_{n})-a_{n}(q_{n})}{\sigma _{n}(q_{n})}\right) \right\} +o(1),
\end{eqnarray*}%
where the first equality uses (\ref{cv33}), (\ref{c_ehta}), and the second
equality uses (\ref{cv7}). Since $\sigma _{n}(q_{n})\geq \eta /\Phi
^{-1}(1-\alpha )>0$ for all $P\in \mathcal{P}_{n}(\lambda _{n},q_{n})\cap 
\mathcal{P}_{0}$ by definition, using the same arguments in the proof of
Lemma A4, we obtain that the last probability is equal to%
\begin{equation*}
\alpha +o(1),
\end{equation*}%
uniformly over $P\in \mathcal{P}_{n}(\lambda _{n},q_{n})\cap \mathcal{P}_{0}$%
.
\end{proof}

\begin{proof}[Proof of Theorem \protect\ref{Thm3}]
For any convex nonnegative map $f$ on $\mathbf{R}^{J},$ we have $2f(b/2)\leq
f(a+b)+f(-a)$. Hence we find that%
\begin{eqnarray*}
\hat{\theta} &=&\int \Lambda _{p}\left( \mathbf{\hat{s}}_{\tau }(x)+\mathbf{u%
}_{\tau }(x;\hat{\sigma})\right) dQ(x,\tau ) \\
&\geq &\frac{1}{2^{p-1}}\int \Lambda _{p}\left( \mathbf{u}_{\tau }(x;\hat{%
\sigma})\right) dQ(x,\tau )-\int \Lambda _{p}\left( -\mathbf{\hat{s}}_{\tau
}(x)\right) dQ(x,\tau ).
\end{eqnarray*}%
From Assumption A\ref{assumption-A3}, the last term is $O_{P}(\left( \log
n\right) ^{p/2})$. Using Assumption A3, we bound the leading integral from
below by%
\begin{equation}
\min_{j\in \mathbb{N}_{J}}r_{n,j}^{p}\left( \int \Lambda _{p}\left( \mathbf{%
\tilde{v}}_{n,\tau }(x)\right) dQ(x,\tau )\left\{ \frac{\int \Lambda
_{p}\left( \mathbf{v}_{n,\tau }(x)\right) dQ(x,\tau )}{\int \Lambda
_{p}\left( \mathbf{\tilde{v}}_{n,\tau }(x)\right) dQ(x,\tau )}-1\right\}
+o_{P}(1)\right) ,  \label{min}
\end{equation}%
where $\mathbf{v}_{n,\tau }(x)\equiv \left[ v_{n,\tau ,j}(x)/\sigma _{n,\tau
,j}(x)\right] _{j\in \mathbb{N}_{J}}$ and $\mathbf{\tilde{v}}_{n,\tau
}(x)\equiv \left[ v_{\tau ,j}(x)/\sigma _{n,\tau ,j}(x)\right] _{j\in 
\mathbb{N}_{J}}$. Since 
\begin{equation*}
\text{liminf}_{n\rightarrow \infty }\int \Lambda _{p}\left( \mathbf{\tilde{v}%
}_{n,\tau }(x)\right) dQ(x,\tau )>0,
\end{equation*}%
we use Assumption C1 and apply the Dominated Convergence Theorem to write (%
\ref{min}) as%
\begin{equation*}
\min_{j\in \mathbb{N}_{J}}r_{n,j}^{p}\int \Lambda _{p}\left( \mathbf{\tilde{v%
}}_{n,\tau }(x)\right) dQ(x,\tau )\left( 1+o_{P}(1)\right) .
\end{equation*}%
Since $\min_{j\in \mathbb{N}_{J}}r_{n,j}\rightarrow \infty $ as $%
n\rightarrow \infty $ and liminf$_{n\rightarrow \infty }\int \Lambda
_{p}\left( \mathbf{\tilde{v}}_{n,\tau }(x)\right) dQ(x,\tau )>0$, we have
for any $M>0$,%
\begin{equation*}
P\left\{ \frac{1}{2^{p-1}}\int \Lambda _{p}\left( \mathbf{u}_{\tau }(x;\hat{%
\sigma})\right) dQ(x,\tau )>M\right\} \rightarrow 1,
\end{equation*}%
as $n\rightarrow \infty $. Also since $\sqrt{\log n}/\min_{j\in \mathbb{N}%
_{J}}r_{n,j}\rightarrow 0$ (Assumption A4(i)), Assumption A\ref%
{assumption-A3} implies that for any $M>0$,%
\begin{equation*}
P\left\{ \hat{\theta}>M\right\} \rightarrow 1.
\end{equation*}%
Also, note that by Lemma A2(ii), $h^{-d/2}(c_{\alpha }^{\ast }-a_{n})/\sigma
_{n}=O_{P}(1)$. Hence 
\begin{equation*}
c_{\alpha }^{\ast }=a_{n}+O_{P}(h^{d/2})=O_{P}(1)\text{.}
\end{equation*}%
Given that $c_{\alpha }^{\ast }=O_{P}(1)$ and $\hat{a}^{\ast }=O_{P}(1)$ by
Lemma A1 and Assumption A6(i), we obtain that $P\{\hat{\theta}>c_{\alpha
,\eta }^{\ast }\}\rightarrow 1$, as $n\rightarrow \infty $.
\end{proof}

\begin{LemmaA}
Suppose that the conditions of Theorem 4 or Theorem 5 hold. Then \textit{as }%
$n\rightarrow \infty ,$ the following holds: for any $c_{n,1},c_{n,2}>0$
such that 
\begin{equation*}
\sqrt{\log n}/c_{n,2}\rightarrow 0,
\end{equation*}%
as $n\rightarrow \infty $. Then%
\begin{equation*}
\inf_{P\in \mathcal{P}_{n}^{0}(\lambda _{n})}P\left\{ \int_{\mathcal{S}%
\backslash B_{n}^{0}(c_{n,1},c_{n,2})}\Lambda _{p}\left( \mathbf{\hat{u}}%
_{\tau }(x)\right) dQ(x,\tau )=0\right\} \rightarrow 1\text{\textit{.}}
\end{equation*}%
Furthermore, we have for any $A\in \mathcal{N}_{J},$%
\begin{equation*}
\inf_{P\in \mathcal{P}_{n}^{0}(\lambda _{n})}P\left\{
\int_{B_{n,A}^{0}(c_{n,1},c_{n,2})}\left\{ \Lambda _{p}\left( \mathbf{\hat{u}%
}_{\tau }(x)\right) -\Lambda _{A,p}\left( \mathbf{\hat{u}}_{\tau }(x)\right)
\right\} dQ(x,\tau )=0\right\} \rightarrow 1\text{.}
\end{equation*}
\end{LemmaA}

\begin{proof}[Proof of Lemma A5]
Consider the first statement. Let $\lambda $ be either $d/2$ or $d/4$. We
write%
\begin{eqnarray*}
&&\int_{\mathcal{S}\backslash B_{n}^{0}(c_{n,1},c_{n,2})}\Lambda _{p}\left( 
\mathbf{\hat{u}}_{\tau }(x)\right) dQ(x,\tau ) \\
&=&\int_{\mathcal{S}\backslash B_{n}^{0}(c_{n,1},c_{n,2})}\Lambda _{p}\left( 
\mathbf{\hat{s}}_{\tau }(x)+(\mathbf{u}_{\tau }(x;\hat{\sigma})\right)
dQ(x,\tau ). \\
&=&\int_{\mathcal{S}\backslash B_{n}^{0}(c_{n,1},c_{n,2})}\Lambda _{p}\left( 
\mathbf{\hat{s}}_{\tau }(x)+\mathbf{u}_{\tau }^{0}(x;\hat{\sigma}%
)+h^{\lambda }\delta _{\tau ,\hat{\sigma}}(x)\right) dQ(x,\tau ),
\end{eqnarray*}%
where $\mathbf{u}_{\tau }^{0}(x;\hat{\sigma})\equiv (r_{n,1}v_{n,\tau
,1}^{0}(x)/\hat{\sigma}_{\tau ,1}(x),\cdot \cdot \cdot ,r_{n,J}v_{n,\tau
,J}^{0}(x)/\hat{\sigma}_{\tau ,J}(x))$ and%
\begin{equation}
\delta _{\tau ,\hat{\sigma}}(x)\equiv \left( \frac{\delta _{\tau ,1}(x)}{%
\hat{\sigma}_{\tau ,1}(x)},\cdot \cdot \cdot ,\frac{\delta _{\tau ,J}(x)}{%
\hat{\sigma}_{\tau ,J}(x)}\right) \mathbf{.}  \label{v_sigh_0}
\end{equation}%
Note that $\delta _{\tau ,\hat{\sigma}}(x)$ is bounded with probability
approaching one by Assumption A3. Also note that for each $j\in \mathbb{N}%
_{J}$,%
\begin{align}
\sup_{(x,\tau )\in \mathcal{S}}\left\vert \frac{r_{n,j}\{\hat{v}_{n,\tau
,j}(x)-v_{n,\tau ,j}^{0}(x)\}}{\hat{\sigma}_{\tau ,j}(x)}\right\vert & \leq
\sup_{(x,\tau )\in \mathcal{S}}\left\vert \frac{r_{n,j}\{\hat{v}_{n,\tau
,j}(x)-v_{n,\tau ,j}(x)\}}{\hat{\sigma}_{\tau ,j}(x)}\right\vert +h^{\lambda
}\sup_{(x,\tau )\in \mathcal{S}}\left\vert \frac{\delta _{\tau ,j}(x)}{\hat{%
\sigma}_{\tau ,j}(x)}\right\vert  \label{vrate} \\
& =O_{P}\left( \sqrt{\log n}+h^{\lambda }\right) =O_{P}\left( \sqrt{\log n}%
\right) ,  \notag
\end{align}%
by Assumption A\ref{assumption-A3}. Hence we obtain the desired result,
using the same arguments as in the proof of Lemma 1.

Given that we have (\ref{vrate}), the proof of the second statement can be
proceeded in the same way as the proof of the first statement.
\end{proof}

Recall the definitions of $\bar{\Lambda}_{x,\tau }(\mathbf{v})$ in (\ref%
{lambe}). We define for $\mathbf{v}\in \mathbf{R}^{J},$ $\bar{\Lambda}%
_{x,\tau }^{0}(\mathbf{v})$ to be $\bar{\Lambda}_{x,\tau }(\mathbf{v})$
except that $B_{n,A}(c_{n,1},c_{n,2})$ is replaced by $%
B_{n,A}^{0}(c_{n,1},c_{n,2})$. Define for $\lambda \in \{0,d/4,d/2\},$%
\begin{equation}
\hat{\theta}_{\delta }(c_{n,1},c_{n,2};\lambda )\equiv \int \bar{\Lambda}%
_{x,\tau }^{0}\left( \mathbf{\hat{s}}_{\tau }(x)+h^{\lambda }\delta _{\tau
,\sigma }(x)\right) dQ(x,\tau ).  \label{theta_del_e}
\end{equation}%
Let%
\begin{equation*}
a_{n,\delta }^{R}(c_{n,1},c_{n,2};\lambda )\equiv \int \mathbf{E}\left[ \bar{%
\Lambda}_{x,\tau }^{0}\left( \sqrt{nh^{d}}\mathbf{z}_{N,\tau }(x)+h^{\lambda
}\delta _{\tau ,\sigma }(x)\right) \right] dQ(x,\tau )\text{,}
\end{equation*}%
\begin{equation*}
\hat{\theta}_{\delta }^{\ast }(c_{n,1},c_{n,2};\lambda )\equiv \int \bar{%
\Lambda}_{x,\tau }^{0}\left( \mathbf{\hat{s}}_{\tau }^{\ast }(x)+h^{\lambda
}\delta _{\tau ,\sigma }(x)\right) dQ(x,\tau ),
\end{equation*}%
and%
\begin{equation}
a_{n,\delta }^{R\ast }(c_{n,1},c_{n,2};\lambda )\equiv \int \mathbf{E}^{\ast
}\left[ \bar{\Lambda}_{x,\tau }^{0}\left( \sqrt{nh^{d}}\mathbf{z}_{N,\tau
}^{\ast }(x)+h^{\lambda }\delta _{\tau ,\sigma }(x)\right) \right] dQ(x,\tau
).  \label{atilde2}
\end{equation}%
We also define%
\begin{equation*}
a_{n,\delta }(c_{n,1},c_{n,2};\lambda )\equiv \int \mathbf{E}\left[ \bar{%
\Lambda}_{x,\tau }^{0}(\mathbb{W}_{n,\tau ,\tau }^{(1)}(x,0)+h^{\lambda
}\delta _{\tau ,\sigma }(x))\right] dQ(x,\tau ).
\end{equation*}%
When $c_{n,1}=c_{n,2}=c_{n}$, we simply write $a_{n,\delta
}^{R}(c_{n};\lambda ),$ $a_{n,\delta }^{R\ast }(c_{n};\lambda )$, and $%
a_{n,\delta }(c_{n};\lambda )$, instead of writing $a_{n,\delta
}^{R}(c_{n},c_{n};\lambda ),$ $a_{n,\delta }^{R\ast }(c_{n},c_{n};\lambda )$%
, and $a_{n,\delta }(c_{n},c_{n};\lambda )$.

\begin{LemmaA}
\textit{Suppose that the conditions of Assumptions A\ref{assumption-A6}(i)
and B\ref{assumption-B4} hold. Then for each }$P\in \mathcal{P}$\textit{\
such that the local alternatives in (\ref{la}) hold with }$%
b_{n,j}=r_{n,j}h^{-\lambda }$, $j=1,\cdot \cdot \cdot ,J$, \textit{for some }%
$\lambda \in \{0,d/4,d/2\}$, \textit{and for each nonnegative sequences }$%
c_{n,1},c_{n,2},$%
\begin{eqnarray*}
\left\vert a_{n,\delta }^{R}(c_{n,1},c_{n,2};\lambda )-a_{n,\delta
}(c_{n,1},c_{n,2};\lambda )\right\vert &=&o(h^{d/2}),\text{\textit{\ and}} \\
\left\vert a_{n,\delta }^{R\ast }(c_{n,1},c_{n,2};\lambda )-a_{n,\delta
}(c_{n,1},c_{n,2};\lambda )\right\vert &=&o_{P}(h^{d/2}).
\end{eqnarray*}
\end{LemmaA}

\begin{proof}[Proof of Lemma A6]
The result follows immediately from Lemma C12 in Appendix C.
\end{proof}

\begin{LemmaA}
\textit{Suppose that the conditions of Theorem 4 are satisfied. Then for
each }$\lambda \in \{0,d/4,d/2\}$, \textit{for each }$P\in \mathcal{P}%
_{n}^{0}(\lambda _{n})$\textit{\ such that the local alternatives in (\ref%
{la}) hold, }%
\begin{eqnarray*}
&&h^{-d/2}\left( \frac{\bar{\theta}_{n,\delta }(c_{n,U},c_{n,L};\lambda
)-a_{n,\delta }^{R}(c_{n,U},c_{n,L};\lambda )}{\sigma _{n}(c_{n,U},c_{n,L})}%
\right) \overset{d}{\rightarrow }N(0,1)\text{ and} \\
&&h^{-d/2}\left( \frac{\bar{\theta}_{n,\delta }^{\ast
}(c_{n,U},c_{n,L};\lambda )-a_{n,\delta }^{R\ast }(c_{n,U},c_{n,L};\lambda )%
}{\sigma _{n}(c_{n,U},c_{n,L})}\right) \overset{d^{\ast }}{\rightarrow }%
N(0,1),\text{ }\mathcal{P}_{n}^{0}(\lambda _{n})\text{-uniformly.}
\end{eqnarray*}
\end{LemmaA}

\begin{proof}[Proof of Lemma A7]
Note that by the definition of $\mathcal{P}_{n}^{0}(\lambda _{n})$, we have%
\begin{equation*}
\liminf_{n\rightarrow \infty }\inf_{P\in \mathcal{P}_{n}^{0}(\lambda
_{n})}\sigma _{n}^{2}(c_{n,U},c_{n,L})\geq \frac{\eta }{\Phi ^{-1}(1-\alpha )%
}.
\end{equation*}%
Hence we can follow the proof of Lemma A2 to obtain the desired results.%
\textit{\ }
\end{proof}

\begin{proof}[Proof of Theorem \protect\ref{Thm4}]
Using Lemma A5, we find that%
\begin{equation*}
\hat{\theta}=\sum_{A\in \mathcal{N}_{J}}\int_{B_{n,A}^{0}(c_{n,U},c_{n,L})}%
\Lambda _{A,p}\left( \mathbf{\hat{s}}_{\tau }(x)+\mathbf{u}_{\tau }(x;\hat{%
\sigma})\right) dQ(x,\tau )
\end{equation*}%
with probability approaching one. We write the leading sum as 
\begin{equation*}
\sum_{A\in \mathcal{N}_{J}}\int_{B_{n,A}^{0}(0)}\Lambda _{A,p}\left( \mathbf{%
\hat{s}}_{\tau }(x)+\mathbf{u}_{\tau }(x;\hat{\sigma})\right) dQ(x,\tau
)+R_{n},
\end{equation*}%
where 
\begin{equation*}
R_{n}\equiv \sum_{A\in \mathcal{N}_{J}}\int_{B_{n,A}^{0}(c_{n,U},c_{n,L})%
\backslash B_{n,A}^{0}(0)}\Lambda _{A,p}\left( \mathbf{\hat{s}}_{\tau }(x)+%
\mathbf{u}_{\tau }(x;\hat{\sigma})\right) dQ(x,\tau ).
\end{equation*}%
We write $h^{-d/2}R_{n}$ as 
\begin{eqnarray*}
&&h^{-d/2}\sum_{A\in \mathcal{N}_{J}}\int_{B_{n,A}^{0}(c_{n,U},c_{n,L})%
\backslash B_{n,A}^{0}(0)}\Lambda _{A,p}\left( 
\begin{array}{c}
\mathbf{\hat{s}}_{\tau }(x)+\mathbf{u}_{\tau }^{0}(x;\hat{\sigma}) \\ 
+h^{d/2}\delta _{\tau ,\hat{\sigma}}(x)(1+o(1))%
\end{array}%
\right) dQ(x,\tau ) \\
&\leq &h^{-d/2}\sum_{A\in \mathcal{N}_{J}}\int_{B_{n,A}^{0}(c_{n,U},c_{n,L})%
\backslash B_{n,A}^{0}(0)}\Lambda _{A,p}\left( \mathbf{\hat{s}}_{\tau
}(x)+h^{d/2}\delta _{\tau ,\hat{\sigma}}(x)(1+o(1))\right) dQ(x,\tau ),
\end{eqnarray*}%
by Assumption C2. We bound the last sum as 
\begin{equation*}
Ch^{-d/2}\sum_{A\in \mathcal{N}_{J}}\left( \sup_{(x,\tau )\in \mathcal{S}}||%
\mathbf{\hat{s}}_{\tau }(x)||\right) ^{p}Q\left(
B_{n,A}^{0}(c_{n,U},c_{n,L})\backslash B_{n,A}^{0}(0)\right) =O_{P}\left(
h^{-d/2}\left( \log n\right) ^{p/2}\lambda _{n}\right) =o_{P}(1)
\end{equation*}%
using Assumption A\ref{assumption-A3} and the rate condition in (\ref{lamq}%
). We conclude that%
\begin{eqnarray}
h^{-d/2}\hat{\theta} &=&h^{-d/2}\sum_{A\in \mathcal{N}_{J}}%
\int_{B_{n,A}^{0}(0)}\Lambda _{A,p}\left( \mathbf{\hat{s}}_{\tau }(x)+%
\mathbf{u}_{\tau }(x;\hat{\sigma})\right) dQ(x,\tau )+o_{P}(1)  \label{eqs}
\\
&=&h^{-d/2}\sum_{A\in \mathcal{N}_{J}}\int_{B_{n,A}^{0}(0)}\Lambda
_{A,p}\left( \mathbf{\hat{s}}_{\tau }(x)+h^{d/2}\delta _{\tau ,\hat{\sigma}%
}(x)\right) dQ(x,\tau )+o_{P}(1),  \notag
\end{eqnarray}%
where the second equality follows by Assumption C2 and by the definition of $%
B_{n,A}^{0}(0)$.

Fix small $\kappa >0$ and define%
\begin{eqnarray*}
\delta _{\tau ,\sigma ,\kappa ,j}^{L}(x) &\equiv &\left\{ 
\begin{array}{c}
\frac{\delta _{\tau ,j}(x)}{(1+\kappa )\sigma _{n,\tau ,j}(x)}\text{ if }%
\delta _{\tau ,j}(x)\geq 0 \\ 
\frac{\delta _{\tau ,j}(x)}{(1-\kappa )\sigma _{n,\tau ,j}(x)}\text{ if }%
\delta _{\tau ,j}(x)<0%
\end{array}%
\right. \text{ and} \\
\delta _{\tau ,\sigma ,\kappa ,j}^{U}(x) &\equiv &\left\{ 
\begin{array}{c}
\frac{\delta _{\tau ,j}(x)}{(1-\kappa )\sigma _{n,\tau ,j}(x)}\text{ if }%
\delta _{\tau ,j}(x)\geq 0 \\ 
\frac{\delta _{\tau ,j}(x)}{(1+\kappa )\sigma _{n,\tau ,j}(x)}\text{ if }%
\delta _{\tau ,j}(x)<0%
\end{array}%
\right. .
\end{eqnarray*}%
Define $\delta _{\tau ,\sigma ,\kappa }^{L}(x)$ and $\delta _{\tau ,\sigma
,\kappa }^{U}(x)$ to be $\mathbf{R}^{J}$-valued maps whose $j$-th entries
are given by $\delta _{\tau ,\sigma ,\kappa ,j}^{L}(x)$ and $\delta _{\tau
,\sigma ,\kappa ,j}^{U}(x)$ respectively. By construction, Assumptions A3
and C2(ii), we have%
\begin{equation*}
P\left\{ \delta _{\tau ,\sigma ,\kappa }^{L}(x)\leq \delta _{\tau ,\hat{%
\sigma}}(x)\leq \delta _{\tau ,\sigma ,\kappa }^{U}(x)\right\} \rightarrow 1,
\end{equation*}%
as $n\rightarrow \infty $. Therefore, with probability approaching one,%
\begin{eqnarray}
\hat{\theta}_{\delta ,L}(0;d/2) &\equiv &\sum_{A\in \mathcal{N}%
_{J}}\int_{B_{n,A}^{0}(0)}\Lambda _{A,p}\left( \mathbf{\hat{s}}_{\tau
}(x)+h^{d/2}\delta _{\tau ,\sigma ,\kappa }^{L}(x)\right) dQ(x,\tau )
\label{three_integrals} \\
&\leq &\sum_{A\in \mathcal{N}_{J}}\int_{B_{n,A}^{0}(0)}\Lambda _{A,p}\left( 
\mathbf{\hat{s}}_{\tau }(x)+h^{d/2}\delta _{\tau ,\hat{\sigma}}(x)\right)
dQ(x,\tau )  \notag \\
&\leq &\sum_{A\in \mathcal{N}_{J}}\int_{B_{n,A}^{0}(0)}\Lambda _{A,p}\left( 
\mathbf{\hat{s}}_{\tau }(x)+h^{d/2}\delta _{\tau ,\sigma ,\kappa
}^{U}(x)\right) dQ(x,\tau )\equiv \hat{\theta}_{\delta ,U}(0;d/2).  \notag
\end{eqnarray}%
We conclude from (\ref{eqs}) that%
\begin{equation}
\hat{\theta}_{\delta ,L}(0;d/2)+o_{P}(h^{d/2})\leq \hat{\theta}\leq \hat{%
\theta}_{\delta ,U}(0;d/2)+o_{P}(h^{d/2}).  \label{cv14}
\end{equation}

As for the bootstrap counterpart, note that since $\delta _{\tau ,j}(x)$ is
bounded and $\sigma _{n,\tau ,j}(x)$ is bounded away from zero uniformly
over $(x,\tau )\in \mathcal{S}$ and $n\geq 1$, and hence 
\begin{equation}
\sup_{(x,\tau )\in \mathcal{S}}\left\vert \frac{1}{h^{-d/2}}\frac{\delta
_{\tau ,j}(x)}{\sigma _{n,\tau ,j}(x)}\right\vert \leq Ch^{d/2}\rightarrow 0,
\label{cv6}
\end{equation}%
as $n\rightarrow \infty $. By (\ref{cv6}), the difference between $%
r_{n,j}v_{n,\tau ,j}(x)/\sigma _{n,\tau ,j}(x)$ and $r_{n,j}v_{n,\tau
,j}^{0}(x)/\sigma _{n,\tau ,j}(x)$ vanishes uniformly over $(x,\tau )\in 
\mathcal{S}$. Therefore, combining this with Lemma A3, we find that%
\begin{equation}
P\left\{ \hat{B}_{n}(\hat{c}_{n})\subset B_{n}^{0}(c_{n,U},c_{n,L})\right\}
\rightarrow 1,  \label{cv5}
\end{equation}%
as $n\rightarrow \infty $.

Now with probability approaching one,%
\begin{eqnarray}
\hat{\theta}^{\ast } &=&\sum_{A\in \mathcal{N}_{J}}\int_{\hat{B}_{A}(\hat{c}%
_{n})}\Lambda _{A,p}\left( \mathbf{\hat{s}}_{\tau }^{\ast }(x)\right)
dQ(x,\tau )  \label{decomp} \\
&=&\sum_{A\in \mathcal{N}_{J}}\int_{B_{n,A}^{0}(0)}\Lambda _{A,p}\left( 
\mathbf{\hat{s}}_{\tau }^{\ast }(x)\right) dQ(x,\tau )  \notag \\
&&+\sum_{A\in \mathcal{N}_{J}}\int_{\hat{B}_{A}(\hat{c}_{n})\backslash
B_{n,A}^{0}(0)}\Lambda _{A,p}\left( \mathbf{\hat{s}}_{\tau }^{\ast
}(x)\right) dQ(x,\tau ).  \notag
\end{eqnarray}%
As for the last sum, it is bounded by%
\begin{equation*}
\sum_{A\in \mathcal{N}_{J}}\int_{B_{n,A}^{0}(c_{n,U},c_{n,L})\backslash
B_{n,A}^{0}(0)}\Lambda _{A,p}\left( \mathbf{\hat{s}}_{\tau }^{\ast
}(x)\right) dQ(x,\tau ),
\end{equation*}%
with probability approaching one by (\ref{cv5}). The above sum multiplied by 
$h^{-d/2}$ is bounded by%
\begin{eqnarray*}
&&h^{-d/2}\left( \sup_{(x,\tau )\in \mathcal{S}}||\mathbf{\hat{s}}_{\tau
}^{\ast }(x)||\right) ^{p}\sum_{A\in \mathcal{N}_{J}}Q\left(
B_{n,A}^{0}(c_{n,U},c_{n,L})\backslash B_{n,A}^{0}(0)\right) \\
&=&O_{P^{\ast }}\left( h^{-d/2}(\log n)^{p/2}\lambda _{n}\right) =o_{P^{\ast
}}(1),\ \mathcal{P}\text{-uniformly,}
\end{eqnarray*}%
by Assumption B2 and the rate condition for $\lambda _{n}$. Thus, we
conclude that%
\begin{equation}
\hat{\theta}^{\ast }=\bar{\theta}^{\ast }(0)+o_{P^{\ast }}(h^{d/2}),\text{ }%
\mathcal{P}_{n}^{0}(\lambda _{n})\text{-uniformly,}  \label{cv15}
\end{equation}%
where%
\begin{equation*}
\bar{\theta}^{\ast }(0)\equiv \sum_{A\in \mathcal{N}_{J}}%
\int_{B_{n,A}^{0}(0)}\Lambda _{A,p}\left( \mathbf{\hat{s}}_{\tau }^{\ast
}(x)\right) dQ(x,\tau ).
\end{equation*}%
Let $\bar{c}_{n}^{\alpha \ast }(0)$ be the $(1-\alpha )$-th quantile of the
bootstrap distribution of $\bar{\theta}^{\ast }(0)$ and let $\gamma
_{n}^{\alpha \ast }(0)$ be the $(1-\alpha )$-th quantile of the bootstrap
distribution of 
\begin{equation}
h^{-d/2}\left( \frac{\bar{\theta}^{\ast }(0)-a_{n}^{R\ast }(0)}{\sigma
_{n}(0)}\right) .  \label{c*tilde0}
\end{equation}%
By the definition of $\mathcal{P}_{n}^{0}(\lambda _{n})$, we have $\sigma
_{n}^{2}(0)>\eta /\Phi ^{-1}(1-\alpha )$. Let $a_{\delta ,U}^{R}(0;d/2)$ and 
$a_{\delta ,L}^{R}(0;d/2)$ be $a_{n,\delta }^{R}(0;d/2)$ except that $\delta
_{\tau ,\sigma }$ is replaced by $\delta _{\tau ,\sigma ,\kappa }^{U}$ and $%
\delta _{\tau ,\sigma ,\kappa }^{L}$ respectively. Also, let $a_{\delta
,U}(0;d/2)$ and $a_{\delta ,L}(0;d/2)$ be $a_{n,\delta }(0;d/2)$ except that 
$\delta _{\tau ,\sigma }$ is replaced by $\delta _{\tau ,\sigma ,\kappa
}^{U} $ and $\delta _{\tau ,\sigma ,\kappa }^{L}$ respectively. We bound $P\{%
\hat{\theta}>c_{\alpha ,\eta }^{\ast }\}$ by%
\begin{eqnarray*}
&&P\left\{ h^{-d/2}\left( \frac{\hat{\theta}_{\delta ,U}(0;d/2)-a_{\delta
,U}^{R}(0;d/2)}{\sigma _{n}(0)}\right) >h^{-d/2}\left( \frac{c_{\alpha
}^{\ast }-a_{\delta ,U}^{R}(0;d/2)}{\sigma _{n}(0)}\right) \right\} +o(1) \\
&=&P\left\{ h^{-d/2}\left( \frac{\hat{\theta}_{\delta ,U}(0;d/2)-a_{\delta
,U}^{R}(0;d/2)}{\sigma _{n}(0)}\right) >h^{-d/2}\left( \frac{\bar{c}%
_{n}^{\alpha \ast }(0)-a_{\delta ,U}^{R}(0;d/2)}{\sigma _{n}(0)}\right)
\right\} +o(1),
\end{eqnarray*}%
where the equality uses (\ref{cv15}). Then we observe that%
\begin{eqnarray*}
\frac{\bar{c}_{n}^{\alpha \ast }(0)-a_{\delta ,U}^{R}(0;d/2)}{\sigma _{n}(0)}
&=&\frac{\bar{c}_{n}^{\alpha \ast }(0)-a_{n}^{R\ast }(0)}{\sigma _{n}(0)}+%
\frac{a_{n}^{R\ast }(0)-a_{\delta ,U}^{R}(0;d/2)}{\sigma _{n}(0)} \\
&=&h^{d/2}\gamma _{n}^{\alpha \ast }(0)+\frac{a_{n}^{R\ast }(0)-a_{\delta
,U}^{R}(0;d/2)}{\sigma _{n}(0)}.
\end{eqnarray*}%
As for the last term, we use Lemmas A1 and A6 to deduce that%
\begin{eqnarray*}
a_{n}^{R\ast }(0)-a_{\delta ,U}^{R}(0;d/2) &=&a_{n}^{R}(0)-a_{\delta
,U}^{R}(0;d/2)+o_{P}(h^{d/2}) \\
&=&a_{n}(0)-a_{\delta ,U}(0;d/2)+o_{P}(h^{d/2}).
\end{eqnarray*}%
As for $a_{n}(0)-a_{\delta ,U}(0;d/2)$, we observe that%
\begin{eqnarray}
&&\sigma _{n}(0)^{-1}h^{-d/2}\left\{ \mathbf{E}\left[ \Lambda _{A,p}\left( 
\mathbb{W}_{n,\tau ,\tau }^{(1)}(x,0)+h^{d/2}\delta _{\tau ,\sigma ,\kappa
}^{U}(x)\right) \right] -\mathbf{E}\left[ \Lambda _{A,p}(\mathbb{W}_{n,\tau
,\tau }^{(1)}(x,0))\right] \right\}  \label{eq55} \\
&=&\sigma _{n}(0)^{-1}h^{-d/2}\left\{ \mathbf{E}\left[ \Lambda _{A,p}\left( 
\mathbb{W}_{n,\tau ,\tau }^{(1)}(x,0)+h^{d/2}\delta _{\tau ,\sigma ,\kappa
}^{U}(x)\right) \right] -\mathbf{E}\left[ \Lambda _{A,p}(\mathbb{W}_{n,\tau
,\tau }^{(1)}(x,0))\right] \right\}  \notag \\
&=&\psi _{n,A,\tau }^{(1)}(\mathbf{0};x)^{\top }\delta _{\tau ,\sigma
,\kappa }^{U}(x)+O\left( h^{d/2}\right) ,  \notag
\end{eqnarray}%
so that%
\begin{eqnarray*}
\frac{h^{-d/2}\left( a_{n}(0)-a_{\delta ,U}(0)\right) }{\sigma _{n}(0)}
&=&-\sum_{A\in \mathcal{N}_{J}}\int \psi _{n,A,\tau }^{(1)}(\mathbf{0}%
;x)^{\top }\delta _{\tau ,\sigma ,\kappa }^{U}(x)dQ(x,\tau )+o(1) \\
&=&-\sum_{A\in \mathcal{N}_{J}}\int \psi _{A,\tau }^{(1)}(\mathbf{0}%
;x)^{\top }\delta _{\tau ,\sigma ,\kappa }^{U}(x)dQ(x,\tau )+o(1),
\end{eqnarray*}%
where the last equality follows by the Dominated Convergence Theorem. On the
other hand, by Lemma A7, we have 
\begin{equation*}
h^{-d/2}\left( \frac{\hat{\theta}_{\delta ,U}(0;d/2)-a_{\delta ,U}^{R}(0;d/2)%
}{\sigma _{n}(0)}\right) \overset{d}{\rightarrow }N(0,1).
\end{equation*}%
Since $\gamma _{n}^{\alpha \ast }(0)=\gamma _{\alpha ,\infty }+o_{P}(1)$ by
Lemma A4, we use this result to deduce that%
\begin{eqnarray*}
&&\lim_{n\rightarrow \infty }P\left\{ h^{-d/2}\left( \frac{\hat{\theta}%
_{\delta ,U}(0;d/2)-a_{\delta ,U}^{R}(0;d/2)}{\sigma _{n}(0)}\right)
>h^{-d/2}\left( \frac{\bar{c}_{n}^{\alpha \ast }(0)-a_{\delta ,U}^{R}(0;d/2)%
}{\sigma _{n}(0)}\right) \right\} \\
&=&1-\Phi \left( z_{1-\alpha }-\sum_{A\in \mathcal{N}_{J}}\int \psi _{A,\tau
}^{(1)}(\mathbf{0};x)^{\top }\delta _{\tau ,\sigma ,\kappa }^{U}(x)dQ(x,\tau
)\right) .
\end{eqnarray*}%
Similarly, we also use (\ref{cv14}) to bound $P\left\{ \hat{\theta}%
>c_{\alpha ,\eta }^{\ast }\right\} $ from below by%
\begin{equation*}
P\left\{ h^{-d/2}\left( \frac{\hat{\theta}_{\delta ,L}(0;d/2)-a_{\delta
,L}^{R}(0;d/2)}{\sigma _{n}(0)}\right) >h^{-d/2}\left( \frac{\bar{c}%
_{n}^{\alpha \ast }(0)-a_{\delta ,L}^{R}(0;d/2)}{\sigma _{n}(0)}\right)
\right\} +o(1),
\end{equation*}%
and using similar arguments as before, we obtain that%
\begin{eqnarray*}
&&\lim_{n\rightarrow \infty }P\left\{ h^{-d/2}\left( \frac{\hat{\theta}%
_{\delta ,L}(0;d/2)-a_{\delta ,L}^{R}(0;d/2)}{\sigma _{n}(0)}\right)
>h^{-d/2}\left( \frac{\bar{c}_{n}^{\alpha \ast }(0)-a_{\delta ,L}^{R}(0;d/2)%
}{\sigma _{n}(0)}\right) \right\} \\
&=&1-\Phi \left( z_{1-\alpha }-\sum_{A\in \mathcal{N}_{J}}\int \psi _{A,\tau
}^{(1)}(\mathbf{0};x)^{\top }\delta _{\tau ,\sigma ,\kappa }^{L}(x)dQ(x,\tau
)\right) .
\end{eqnarray*}%
We conclude from this and (\ref{three_integrals}) that for any small $\kappa
>0$, 
\begin{eqnarray*}
&&1-\Phi \left( z_{1-\alpha }-\sum_{A\in \mathcal{N}_{J}}\int \psi _{A,\tau
}^{(1)}(\mathbf{0};x)^{\top }\delta _{\tau ,\sigma ,\kappa }^{L}(x)dQ(x,\tau
)\right) +o(1) \\
&\leq &P\left\{ \hat{\theta}>c_{\alpha ,\eta }^{\ast }\right\} \leq 1-\Phi
\left( z_{1-\alpha }-\sum_{A\in \mathcal{N}_{J}}\int \psi _{A,\tau }^{(1)}(%
\mathbf{0};x)^{\top }\delta _{\tau ,\sigma ,\kappa }^{U}(x)dQ(x,\tau
)\right) +o(1).
\end{eqnarray*}%
Note that $\psi _{A,\tau }^{(1)}(\mathbf{0};x)^{\top }\delta _{\tau ,\sigma
,\kappa }^{U}(x)$ and $\psi _{A,\tau }^{(1)}(\mathbf{0};x)^{\top }\delta
_{\tau ,\sigma ,\kappa }^{L}(x)$ are bounded maps in $(x,\tau )$ by the
assumption of the theorem, and that 
\begin{equation*}
\lim_{\kappa \rightarrow 0}\delta _{\tau ,\sigma ,\kappa
}^{L}(x)=\lim_{\kappa \rightarrow 0}\delta _{\tau ,\sigma ,\kappa
}^{U}(x)=\delta _{\tau ,\sigma }(x),
\end{equation*}%
for each $(x,\tau )\in \mathcal{S}$. Hence by sending $\kappa \rightarrow 0$
and applying the Dominated Convergence Theorem to both the bounds above, we
obtain the desired result.
\end{proof}

\begin{proof}[Proof of Theorem \protect\ref{Thm5}]
First, observe that Lemma A5 continues to hold. This can be seen by
following the proof of Lemma A5 and noting that (\ref{vrate}) becomes here 
\begin{equation*}
\sup_{(x,\tau )\in \mathcal{S}}\left\vert \frac{r_{n,j}\{\hat{v}_{n,\tau
,j}(x)-v_{n,\tau ,j}^{0}(x)\}}{\hat{\sigma}_{\tau ,j}(x)}\right\vert
=O_{P}\left( \sqrt{\log n}+h^{d/4}\right) =O_{P}\left( \sqrt{\log n}\right) ,
\end{equation*}%
yielding the same convergence rate. The rest of the proof is the same.
Similarly, Lemma A6 continues to hold also under the modified local
alternatives of (\ref{la}) with $b_{n,j}=r_{n,j}h^{-d/4}$. We define 
\begin{equation}
\tilde{\delta}_{\tau ,\sigma }(x)\equiv h^{-d/4}\delta _{\tau ,\sigma }(x).
\label{delta_tilde}
\end{equation}%
We follow the proof of Theorem 4 and take up arguments from (\ref{eq55}).
Observe that%
\begin{eqnarray*}
&&\sigma _{n}(0)^{-1}h^{-d/2}\left\{ \mathbf{E}\left[ \Lambda _{A,p}\left( 
\mathbb{W}_{n,\tau ,\tau }^{(1)}(x,0)+h^{d/2}\tilde{\delta}_{\tau ,\sigma
}(x)\right) \right] -\mathbf{E}\left[ \Lambda _{A,p}(\mathbb{W}_{n,\tau
,\tau }^{(1)}(x,0))\right] \right\} \\
&=&\sigma _{n}(0)^{-1}h^{-d/2}\left\{ \mathbf{E}\left[ \Lambda _{A,p}\left( 
\mathbb{W}_{n,\tau ,\tau }^{(1)}(x,0)+h^{d/2}\tilde{\delta}_{\tau ,\sigma
}(x)\right) \right] -\mathbf{E}\left[ \Lambda _{A,p}(\mathbb{W}_{n,\tau
,\tau }^{(1)}(x,0))\right] \right\} \\
&=&\psi _{n,A,\tau }^{(1)}(\mathbf{0};x)^{\top }\tilde{\delta}_{\tau ,\sigma
}(x)+h^{d/2}\tilde{\delta}_{\tau ,\sigma }(x)^{\top }\psi _{n,A,\tau }^{(2)}(%
\mathbf{0};x)\tilde{\delta}_{\tau ,\sigma }(x)/2.
\end{eqnarray*}%
By the Dominated Convergence Theorem, 
\begin{eqnarray*}
\int \psi _{n,A,\tau }^{(1)}(\mathbf{0};x)^{\top }\tilde{\delta}_{\tau
,\sigma }(x)dQ(x,\tau ) &=&\int \psi _{A,\tau }^{(1)}(\mathbf{0};x)^{\top }%
\tilde{\delta}_{\tau ,\sigma }(x)dQ(x,\tau )+o(1)\text{ and} \\
\int \psi _{n,A,\tau }^{(2)}(\mathbf{0};x)^{\top }\tilde{\delta}_{\tau
,\sigma }(x)dQ(x,\tau ) &=&\int \psi _{A,\tau }^{(2)}(\mathbf{0};x)^{\top }%
\tilde{\delta}_{\tau ,\sigma }(x)dQ(x,\tau )+o(1).
\end{eqnarray*}%
Since $\sum_{A\in \mathcal{N}_{J}}\int \psi _{A,\tau }^{(1)}(\mathbf{0}%
;x)^{\top }\tilde{\delta}_{\tau ,\sigma }(x)dQ(x,\tau )=0$, by the condition
for $\delta _{\tau ,\sigma }(x)$ in the theorem,%
\begin{eqnarray*}
&&\sum_{A\in \mathcal{N}_{J}}\int h^{-d/2}\left\{ 
\begin{array}{c}
\mathbf{E}\left[ \Lambda _{A,p}\left( \mathbb{W}_{n,\tau ,\tau
}^{(1)}(x,0)+h^{d/2}\tilde{\delta}_{\tau ,\sigma }(x)\right) \right] \\ 
-\mathbf{E}\left[ \Lambda _{A,p}(\mathbb{W}_{n,\tau ,\tau }^{(1)}(x,0))%
\right]%
\end{array}%
\right\} dQ(x,\tau ) \\
&=&\frac{1}{2}\sum_{A\in \mathcal{N}_{J}}\int \delta _{\tau ,\sigma
}(x)^{\top }\psi _{A,\tau }^{(2)}(\mathbf{0};x)\delta _{\tau ,\sigma
}(x)dQ(x,\tau )+o(1).
\end{eqnarray*}%
Now we can use the above result by replacing $\delta _{\tau ,\sigma }(x)$ by 
$\delta _{\tau ,\sigma ,\kappa }^{U}(x)$ and $\delta _{\tau ,\sigma ,\kappa
}^{L}(x)$ and follow the proof of Theorem 4 to obtain the desired result.
\end{proof}

\section{Proofs of Auxiliary Results for Lemmas A2(i), Lemma A4(i), and
Theorem 1}

\label{appendix:C}

The eventual result in this appendix is Lemma B9 which is used to show the
asymptotic normality of the location-scale normalized representation of $%
\hat{\theta}$ and its bootstrap version, and to establish its asymptotic
behavior in the degenerate case. For this, we first prepare Lemmas B1-B3. To
obtain uniformity that covers the case of degeneracy, this paper uses a
method of regularization, where the covariance matrix of random quantities
is added by a diagonal matrix of small diagonal elements. The regularized
random quantities having this covariance matrix does not suffer from
degeneracy in the limit, even when the original quantities have covariate
matrix that is degenerate in the limit. Thus, for these regularized
quantities, we can obtain uniform asymptotic theory using an appropriate
Berry-Esseen-type bound. Then, we need to deal with the difference between
the regularized covariance matrix and the original one. Lemma B1 is a simple
result of linear algebra that is used to control this discrepancy.

Lemma B2 has two sub-results from which one can deduce a uniform version of
Levy's continuity theorem. We have not seen any such results in the
literature or monographs, so we provide its full proof. The result has two
functions. First, the result enables one to deduce convergence in
distribution in terms of convergence of cumulative distribution functions
and convergence in distribution in terms of convergence of characteristic
functions in a manner that is uniform over a given collection of
probabilities. The original form of convergence in distribution due to the
Poissonization method in \citeasnoun{GMZ} is convergence of characteristic
functions. Certainly pointwise in $P$, this convergence implies convergence
of cumulative distribution functions, but it is not clear under what
conditions this implication is uniform over a given class of probabilities.
Lemma B2 essentially clarifies this issue.

Lemma B3 is an extension of de-Poissonization lemma that appeared in %
\citeasnoun{Beirlant/Mason:95}. The proof is based on the proof of their
same result in \citeasnoun{GMZ}, but involves a substantial modification,
because unlike their results, we need a version that holds uniformly over $%
P\in \mathcal{P}$. This de-Poissonization lemma is used to transform the
asymptotic distribution theory for the Poissonized version of the test
statistic into that for the original test statistic.

Lemmas B4-B5 establish some moment bounds for a normalized sum of
independent quantities. This moment bound is later used to control a
Berry-Esseen-type bound, when we approximate those sums by corresponding
centered normal random vectors.

Lemma B6 obtains an approximate version for the scale normalizer $\sigma
_{n} $. The approximate version involves a functional of a Gaussian random
vector, which stems from approximating a normalized sum of independent
random vectors by a Gaussian random vector through using a Berry-Esseen-type
bound. For this result, we use the regularization method that we mentioned
before. Due to the regularization, we are able to cover the degenerate case
eventually.

Lemma B7 is an auxiliary result that is used to establish Lemma B9 in
combination with the de-Poissonization lemma (Lemma B3). And Lemma B8
establishes asymptotic normality of the Poissonized version of the test
statistics. The asymptotic normality for the Poissonized statistic involves
the discretization of the integrals, thereby, reducing the integral to a sum
of 1-dependent random variables, and then applies the Berry-Esseen-type
bound in \citeasnoun{Shergin:93}. Note that by the moment bound in Lemmas
B4-B5 that is uniform over $P\in \mathcal{P}$, we obtain the asymptotic
approximation that is uniform over $P\in \mathcal{P}$. The lemma also
presents a corresponding result for the degenerate case.

Finally, Lemma B9 combines the asymptotic distribution theory for the
Poissonized test statistic in Lemma B7 with the de-Poissonization lemma
(Lemma B3) to obtain the asymptotic distribution theory for the original
test statistic. The result of Lemma B9 is used to establish the asymptotic
normality result in Lemma A7.

The following lemma provides some inequality of matrix algebra.

\begin{LemmaC}
\label{lem-c1} \textit{For any} $J\times J$ \textit{positive semidefinite
symmetric matrix} $\Sigma $ \textit{and any} $\varepsilon >0,$%
\begin{equation*}
\left\Vert \left( \Sigma +\varepsilon I\right) ^{1/2}-\Sigma
^{1/2}\right\Vert \leq \sqrt{J\varepsilon },
\end{equation*}%
\textit{where} $\left\Vert A\right\Vert =\sqrt{tr(AA^{\prime })}$ \textit{%
for any square matrix} $A$.
\end{LemmaC}

\begin{rem}
The main point of Lemma B1 is that the bound $\sqrt{J\varepsilon }$ is
independent of the matrix $\Sigma $. Such a uniform bound is crucially used
for the derivation of asymptotic validity of the test uniform in $P\in 
\mathcal{P}$.
\end{rem}

\begin{proof}[Proof of Lemma B\protect\ref{lem-c1}]
First observe that%
\begin{eqnarray}
&&tr\{\left( \Sigma +\varepsilon I\right) ^{1/2}-\Sigma ^{1/2}\}^{2}
\label{eq} \\
&=&tr\left( 2\Sigma +\varepsilon I\right) -2tr(\left( \Sigma +\varepsilon
I\right) ^{1/2}\Sigma ^{1/2}).  \notag
\end{eqnarray}%
Since $\Sigma \leq \Sigma +\varepsilon I$, we have $\Sigma ^{1/2}\leq \left(
\Sigma +\varepsilon I\right) ^{1/2}.$ For any positive semidefinite matrices 
$A$ and $B,$ $tr(AB)\geq 0$ (see e.g. \citeasnoun{Abadir:Magnus:05}).
Therefore, $tr(\Sigma )\leq tr(\left( \Sigma +\varepsilon I\right)
^{1/2}\Sigma ^{1/2}).$ From (\ref{eq}), we find that 
\begin{eqnarray*}
&&tr\left( 2\Sigma +\varepsilon I\right) -2tr(\left( \Sigma +\varepsilon
I\right) ^{1/2}\Sigma ^{1/2}) \\
&\leq &tr\left( 2\Sigma +\varepsilon I\right) -2tr(\Sigma )=\varepsilon J.
\end{eqnarray*}
\end{proof}

The following lemma can be used to derive a version of Levy's Continuity
Theorem that is uniform in $P\in \mathcal{P}$.

\begin{LemmaC}
\label{lem-c2} \textit{Suppose that }$V_{n}\in \mathbf{R}^{d}$ \textit{is a
sequence of random vectors and} $V\in \mathbf{R}^{d}$ \textit{is a random
vector}. \textit{We assume without loss of generality that }$V_{n}$\textit{\
and }$V$ \textit{live on the same measure space }$(\Omega ,\mathcal{F})$, 
\textit{and $\mathcal{P}$ is a given collection of probabilities on }$%
(\Omega ,\mathcal{F})$. \textit{Furthermore define}%
\begin{eqnarray*}
\varphi _{n}(t) &\equiv &\mathbf{E}\left[ \exp (it^{\top }V_{n})\right] 
\text{, }\varphi (t)\equiv \mathbf{E}\left[ \exp (it^{\top }V)\right] \text{,%
} \\
F_{n}(t) &\equiv &P\left\{ V_{n}\leq t\right\} \text{\textit{, and} }%
F(t)\equiv P\left\{ V\leq t\right\} .
\end{eqnarray*}%
\noindent (i) \textit{Suppose that the distribution }$P\circ V^{-1}$ \textit{%
is uniformly tight in }$\{P\circ V^{-1}:P\in \mathcal{P}\}$. \textit{Then
for any continuous function }$f$\textit{\ on }$\mathbf{R}^{d}$\textit{\
taking values in }$[-1,1]$ \textit{and for any }$\varepsilon \in (0,1],$%
\textit{\ we have}%
\begin{equation*}
\sup_{P\in \mathcal{P}}\left\vert \mathbf{E}f(V_{n})-\mathbf{E}%
f(V)\right\vert \leq \varepsilon ^{-d}C_{d}\sup_{P\in \mathcal{P}}\sup_{t\in 
\mathbf{R}^{d}}\left\vert F_{n}(t)-F(t)\right\vert +4\varepsilon ,
\end{equation*}%
\textit{where }$C_{d}>0$ \textit{is a constant that depends only on} $d$%
\textit{.}

\noindent (ii) \textit{Suppose that }$\sup_{P\in \mathcal{P}}\mathbf{E||}%
V||^{2}<\infty $. \textit{If }%
\begin{equation*}
\sup_{P\in \mathcal{P}}\sup_{u\in \mathbf{R}^{d}}\left\vert \varphi
_{n}(u)-\varphi (u)\right\vert \rightarrow 0\text{\textit{, as }}%
n\rightarrow \infty \text{\textit{,}}
\end{equation*}%
\textit{then for each }$t\in \mathbf{R}^{d},$%
\begin{equation*}
\sup_{P\in \mathcal{P}}\left\vert F_{n}(t)-F(t)\right\vert \rightarrow 0,\ 
\text{\textit{as }}n\rightarrow \infty \text{.}
\end{equation*}%
\textit{\ On the other hand, if for each }$t\in \mathbf{R}^{d},$\textit{\ }%
\begin{equation*}
\sup_{P\in \mathcal{P}}\left\vert F_{n}(t)-F(t)\right\vert \rightarrow 0,\ 
\text{\textit{as }}n\rightarrow \infty ,
\end{equation*}%
\textit{then for each }$u\in \mathbf{R}^{d}$\textit{, }%
\begin{equation*}
\sup_{P\in \mathcal{P}}\left\vert \varphi _{n}(u)-\varphi (u)\right\vert
\rightarrow 0\text{\textit{, as }}n\rightarrow \infty \text{\textit{.}}
\end{equation*}
\end{LemmaC}

\begin{proof}[Proof of Lemma B\protect\ref{lem-c2}]
(i) The proof uses arguments in the proof of Lemma 2.2 of %
\citeasnoun{VanDerVaart:98}. Take a large compact rectangle $B\subset 
\mathbf{R}^{d}$ such that $P\{V\notin B\}<\varepsilon $. Since the
distribution of $V$ is tight uniformly over $P\in \mathcal{P}$, we can take
such $B$ independently of $P\in \mathcal{P}$. Take a partition $B=\cup
_{j=1}^{J_{\varepsilon }}B_{j}$ and points $x_{j}\in B_{j}$ such that $%
J_{\varepsilon }\leq C_{d,1}\varepsilon ^{-d}$, and $|f(x)-f_{\varepsilon
}(x)|<\varepsilon $ for all $x\in B$, where $C_{d,1}>0$ is a constant that
depends only on $d$, and%
\begin{equation*}
f_{\varepsilon }(x)\equiv \sum_{j=1}^{J_{\varepsilon }}f(x_{j})1\{x\in
B_{j}\}.
\end{equation*}%
Thus we have%
\begin{eqnarray*}
\left\vert \mathbf{E}f(V_{n})-\mathbf{E}f(V)\right\vert &\leq &\left\vert 
\mathbf{E}f(V_{n})-\mathbf{E}f_{\varepsilon }(V_{n})\right\vert +\left\vert 
\mathbf{E}f_{\varepsilon }(V_{n})-\mathbf{E}f_{\varepsilon }(V)\right\vert
+\left\vert \mathbf{E}f_{\varepsilon }(V)-\mathbf{E}f(V)\right\vert \\
&\leq &2\varepsilon +P\{V_{n}\notin B\}+P\{V\notin B\}+\left\vert \mathbf{E}%
f_{\varepsilon }(V_{n})-\mathbf{E}f_{\varepsilon }(V)\right\vert \\
&\leq &4\varepsilon +\left\vert P\{V_{n}\notin B\}-P\{V\notin B\}\right\vert
+\left\vert \mathbf{E}f_{\varepsilon }(V_{n})-\mathbf{E}f_{\varepsilon
}(V)\right\vert \\
&=&4\varepsilon +\left\vert P\{V_{n}\in B\}-P\{V\in B\}\right\vert
+\left\vert \mathbf{E}f_{\varepsilon }(V_{n})-\mathbf{E}f_{\varepsilon
}(V)\right\vert .
\end{eqnarray*}%
The second inequality following by $P\{V\notin B\}<\varepsilon $. As for the
last term, we let%
\begin{equation*}
b_{n}\equiv \sup_{P\in \mathcal{P}}\sup_{t\in \mathbf{R}^{d}}\left\vert
F_{n}(t)-F(t)\right\vert ,
\end{equation*}%
and observe that%
\begin{eqnarray*}
\left\vert \mathbf{E}f_{\varepsilon }(V_{n})-\mathbf{E}f_{\varepsilon
}(V)\right\vert &\leq &\sum_{j=1}^{J_{\varepsilon }}\left\vert P\{V_{n}\in
B_{j}\}-P\{V\in B_{j}\}\right\vert |f(x_{j})| \\
&\leq &\sum_{j=1}^{J_{\varepsilon }}\left\vert P\{V_{n}\in B_{j}\}-P\{V\in
B_{j}\}\right\vert \leq C_{d,2}b_{n}J_{\varepsilon },
\end{eqnarray*}%
where $C_{d,2}>0$ is a constant that depends only on $d$. The last
inequality follows because for any rectangle $B_{j}$, we have $|P\{V_{n}\in
B_{j}\}-P\{V\in B_{j}\}|\leq C_{d,2}b_{n}$ for some $C_{d,2}>0.$ We conclude
that%
\begin{equation*}
\left\vert \mathbf{E}f(V_{n})-\mathbf{E}f(V)\right\vert \leq 4\varepsilon
+C_{d,2}\left( C_{d,1}\varepsilon ^{-d}+1\right) b_{n}\leq 4\varepsilon
+C_{d}\varepsilon ^{-d}b_{n},
\end{equation*}%
where $C_{d}=C_{d,2}\{C_{d,1}+1\}$. The last inequality follows because $%
\varepsilon \leq 1$.

\noindent (ii) We show the first statement. We first show that under the
stated condition, the sequence $\{P\circ V_{n}^{-1}\}_{n=1}^{\infty }$ is
uniformly tight uniformly over $P\in \mathcal{P}$. That is, for any $%
\varepsilon >0$, we show there exists a compact set $B\subset \mathbf{R}^{d}$
such that for all $n\geq 1,$%
\begin{equation*}
\sup_{P\in \mathcal{P}}P\left\{ V_{n}\in \mathbf{R}^{d}\backslash B\right\}
<\varepsilon \text{.}
\end{equation*}%
For this, we assume $d=1$ without loss of generality, let $P_{n}$ denote the
distribution of $V_{n}$ and consider the following: (using arguments in the
proof of Theorem 3.3.6 of \citeasnoun{Durrett:10})%
\begin{eqnarray*}
P\left\{ |V_{n}|>\frac{2}{u}\right\} &\leq &2\int_{|x|>2/u}\left( 1-\frac{1}{%
|ux|}\right) dP_{n}(x) \\
&\leq &2\int \left( 1-\frac{\sin ux}{ux}\right) dP_{n}(x) \\
&=&\frac{1}{u}\int_{-u}^{u}\left( 1-\varphi _{n}(t)\right) dt.
\end{eqnarray*}%
Define $\bar{e}_{n}\equiv \sup_{P\in \mathcal{P}}\sup_{t\in \mathbf{R}%
}\left\vert \varphi _{n}(t)-\varphi (t)\right\vert $. Using Theorem 3.3.8 of %
\citeasnoun{Durrett:10}, we bound the last term by%
\begin{eqnarray*}
2\bar{e}_{n}+\frac{1}{u}\int_{-u}^{u}\left( 1-\varphi (t)\right) dt &\leq &2%
\bar{e}_{n}+\left\vert \frac{1}{u}\int_{-u}^{u}\left( -it\mathbf{E}V+\frac{%
t^{2}\mathbf{E}V^{2}}{2}\right) dt\right\vert \\
&&+2\left\vert \frac{1}{u}\int_{-u}^{u}t^{2}\mathbf{E}V^{2}dt\right\vert .
\end{eqnarray*}%
The supremum of the right hand side terms over $P\in \mathcal{P}$ vanishes
as we send $n\rightarrow \infty $ and then $u\downarrow 0$, by the
assumption that sup$_{P\in \mathcal{P}}\mathbf{E|}V|^{2}<\infty $. Hence the
sequence $\{P\circ V_{n}^{-1}\}_{n=1}^{\infty }$ is uniformly tight
uniformly over $P\in \mathcal{P}$.\newline
Now, for each $t\in \mathbf{R}^{d}$, there exists a subsequence $\{n^{\prime
}\}\subset \{n\}$ and $\{P_{n^{\prime }}\}\subset \mathcal{P}$ such that%
\begin{equation}
\underset{n\rightarrow \infty }{\text{limsup}}\ \sup_{P\in \mathcal{P}%
}\left\vert F_{n}(t)-F(t)\right\vert =\lim_{n^{\prime }\rightarrow \infty
}\left\vert F_{n^{\prime }}(t;P_{n^{\prime }})-F(t;P_{n^{\prime
}})\right\vert ,  \label{eq5}
\end{equation}%
where%
\begin{equation*}
F_{n}(t;P_{n})=P_{n}\left\{ V_{n}\leq t\right\} \text{ and }%
F(t;P_{n})=P_{n}\left\{ V\leq t\right\} .
\end{equation*}%
(Hence, $F(t;P_{n})$ is the cdf of distribution $P_{n}\circ V^{-1}$.)

Since $\{P_{n^{\prime }}\circ V_{n^{\prime }}^{-1}\}_{n^{\prime }=1}^{\infty
}$ is uniformly tight (as shown above), there exists a subsequence $%
\{n_{k}^{\prime }\}\subset \{n^{\prime }\}$ such that 
\begin{equation}
F_{n_{k}^{\prime }}(t;P_{n_{k}^{\prime }})\rightarrow F^{\ast }(t)\text{, as 
}k\rightarrow \infty \text{,}  \label{cv12}
\end{equation}%
for some cdf $F^{\ast }$. Also $\{P_{n^{\prime }}\circ V^{-1}\}_{n^{\prime
}=1}^{\infty }$ is uniformly tight (because $\sup_{P\in \mathcal{P}}\mathbf{E%
}||V||^{2}<\infty $), $\{P_{n_{k}^{\prime }}\circ V^{-1}\}_{k=1}^{\infty }$
is uniformly tight and hence there exists a further subsequence $%
\{n_{k_{j}}^{\prime }\}\subset \{n_{k}^{\prime }\}$ such that 
\begin{equation}
F(t;P_{n_{k_{j}}^{\prime }})\rightarrow F^{\ast \ast }(t)\text{, as }%
j\rightarrow \infty \text{,}  \label{cv25}
\end{equation}%
for some cdf $F^{\ast \ast }$. Since $\{n_{k_{j}}^{\prime }\}\subset
\{n_{k}^{\prime }\}$, we have from (\ref{cv12}),%
\begin{equation}
F_{n_{k_{j}}^{\prime }}(t;P_{n_{k_{j}}^{\prime }})\rightarrow F^{\ast }(t)%
\text{, as }j\rightarrow \infty \text{.}  \label{cv67}
\end{equation}%
By the condition of (ii), we have 
\begin{equation}
\left\vert \varphi _{n_{k_{j}}^{\prime }}(u;P_{n_{k_{j}}^{\prime }})-\varphi
(u;P_{n_{k_{j}}^{\prime }})\right\vert \rightarrow 0\text{, as }j\rightarrow
\infty \text{,}  \label{cv88}
\end{equation}%
where 
\begin{equation*}
\varphi _{n}\left( u;P_{n}\right) =\mathbf{E}_{P_{n}}\left( \exp \left(
iuV_{n}\right) \right) \text{ and }\varphi \left( u;P_{n}\right) =\mathbf{E}%
_{P_{n}}\left( \exp \left( iuV\right) \right) ,
\end{equation*}%
and $\mathbf{E}_{P_{n}}$ represents expectation with respect to the
probability measure $P_{n}$. Furthermore, by (\ref{cv25}) and (\ref{cv67}),
and Levy's Continuity Theorem,%
\begin{equation*}
\lim_{j\rightarrow \infty }\varphi _{n_{k_{j}}^{\prime
}}(u;P_{n_{k_{j}}^{\prime }})\text{ and }\lim_{j\rightarrow \infty }\varphi
(u;P_{n_{k_{j}}^{\prime }})
\end{equation*}%
exist and coincide by (\ref{cv88}). Therefore, for all $t\in \mathbf{R}^{d}$,%
\begin{equation*}
F^{\ast \ast }(t)=F^{\ast }(t).
\end{equation*}%
In other words,%
\begin{equation*}
\lim_{n^{\prime }\rightarrow \infty }\left\vert F_{n^{\prime }}\left(
t;P_{n^{\prime }}\right) -F\left( t;P_{n^{\prime }}\right) \right\vert
=\lim_{n^{\prime }\rightarrow \infty }\left\vert F_{n_{k_{j}}^{\prime
}}\left( t;P_{n_{k_{j}}^{\prime }}\right) -F\left( t;P_{n_{k_{j}}^{\prime
}}\right) \right\vert =0.
\end{equation*}%
Therefore, the first statement of (ii) follows by the last limit applied to (%
\ref{eq5}).\newline
Let us turn to the second statement. Again, we show that $\{P\circ
V_{n}^{-1}\}_{n=1}^{\infty }$ is uniformly tight uniformly in $P\in \mathcal{%
P}$. Note that given a large rectangle $B$,%
\begin{equation*}
P\left\{ V_{n}\in \mathbf{R}^{d}\backslash B\right\} \leq \left\vert
P\left\{ V_{n}\in \mathbf{R}^{d}\backslash B\right\} -P\left\{ V\in \mathbf{R%
}^{d}\backslash B\right\} \right\vert +P\left\{ V\in \mathbf{R}%
^{d}\backslash B\right\} .
\end{equation*}%
There exists $N$ such that for all $n\geq N$, the first difference vanishes
as $n\rightarrow \infty $, uniformly in $P\in \mathcal{P}$, by the condition
of the lemma. As for the second term, we bound it by%
\begin{equation*}
P\left\{ V_{j}>a_{j},\ j=1,\cdot \cdot \cdot ,d\right\} \leq \sum_{j=1}^{d}%
\frac{\mathbf{E}V_{j}^{2}}{a_{j}},
\end{equation*}%
where $V_{j}$ is the $j$-th entry of $V$ and $B=\times
_{j=1}^{d}[a_{j},b_{j}]$, $b_{j}<0<a_{j}$. By taking $a_{j}$'s large enough,
we make the last bound arbitrarily small independently of $P\in \mathcal{P}$%
, because sup$_{P\in \mathcal{P}}\mathbf{E}V_{j}^{2}<\infty $ for each $%
j=1,\cdot \cdot \cdot ,d$. Therefore, $\{P\circ V_{n}^{-1}\}_{n=1}^{\infty }$
is uniformly tight uniformly in $P\in \mathcal{P}$.

Now, we turn to the proof of the second statement of (ii). For each $u\in 
\mathbf{R}^{d}$, there exists a subsequence $\{n^{\prime }\}\subset \{n\}$
and $\{P_{n^{\prime }}\}\subset \mathcal{P}$ such that 
\begin{equation*}
\underset{n\rightarrow \infty }{\text{limsup}}\ \sup_{P\in \mathcal{P}%
}\left\vert \varphi _{n}(u)-\varphi (u)\right\vert =\lim_{n^{\prime
}\rightarrow \infty }\left\vert \varphi _{n^{\prime }}(u;P_{n^{\prime
}})-\varphi (u;P_{n^{\prime }})\right\vert ,
\end{equation*}%
where $\varphi _{n}(u;P_{n})=\mathbf{E}_{P_{n}}\exp (iu^{\top }V_{n})$ and $%
\varphi (u;P_{n})=\mathbf{E}_{P_{n}}\exp (iu^{\top }V)$. By the condition in
the second statement of (ii), for each $t\in \mathbf{R}^{d}$,%
\begin{equation}
\lim_{n^{\prime }\rightarrow \infty }\left\vert F_{n^{\prime }}\left(
t;P_{n^{\prime }}\right) -F\left( t;P_{n^{\prime }}\right) \right\vert =0%
\text{.}  \label{cv54}
\end{equation}

Since $\{P_{n^{\prime }}\circ V_{n^{\prime }}^{-1}\}_{n^{\prime }=1}^{\infty
}$ is uniformly tight (as shown above), there exists a subsequence $%
\{n_{k}^{\prime }\}\subset \{n^{\prime }\}$ such that $F_{n_{k}^{\prime
}}(t;P_{n_{k}^{\prime }})\rightarrow F^{\ast }(t)$, as $k\rightarrow \infty $%
, and hence by Levy's Continuity Theorem, we have $\varphi _{n_{k}^{\prime
}}(u;P_{n_{k}^{\prime }})\rightarrow \varphi ^{\ast }(u)$, as $k\rightarrow
\infty $. Similarly, we also have $\varphi (u;P_{n_{k}^{\prime
}})\rightarrow \varphi ^{\ast \ast }(u)$, as $k\rightarrow \infty $. By (\ref%
{cv54}), we have $F^{\ast }(t)=F^{\ast \ast }(t)$ and $\varphi ^{\ast
}(u)=\varphi ^{\ast \ast }(u)$. Therefore,%
\begin{equation*}
\lim_{n^{\prime }\rightarrow \infty }\left\vert \varphi _{n^{\prime }}\left(
u;P_{n^{\prime }}\right) -\varphi \left( u;P_{n^{\prime }}\right)
\right\vert =\lim_{n^{\prime }\rightarrow \infty }\left\vert \varphi
_{n_{k_{j}}^{\prime }}\left( u;P_{n_{k_{j}}^{\prime }}\right) -\varphi
\left( u;P_{n_{k_{j}}^{\prime }}\right) \right\vert =0.
\end{equation*}%
Thus we arrive at the desired result.
\end{proof}

The following lemma offers a version of the de-Poissonization lemma of %
\citeasnoun{Beirlant/Mason:95} (see Theorem 2.1 on page 5). In contrast to
the result of \citeasnoun{Beirlant/Mason:95}, the version here is uniform in 
$P\in \mathcal{P}$.

\begin{LemmaC}
\label{lem-c3} \textit{Let} $N_{1,n}(\alpha )$ \textit{and} $N_{2,n}(\alpha
) $ \textit{be independent Poisson random variables with} $N_{1,n}(\alpha )$ 
\textit{being Poisson} $(n(1-\alpha ))$ \textit{and} $N_{2,n}(\alpha )$ 
\textit{being Poisson} $(n\alpha )$, \textit{where} $\alpha \in (0,1)$. 
\textit{Denote} $N_{n}(\alpha )=N_{1,n}(\alpha )+N_{2,n}(\alpha )$ \textit{%
and set}%
\begin{equation*}
U_{n}(\alpha )=\frac{N_{1,n}(\alpha )-n(1-\alpha )}{\sqrt{n}}\text{ \textit{%
and} }V_{n}(\alpha )=\frac{N_{2,n}(\alpha )-n\alpha }{\sqrt{n}}.
\end{equation*}%
\textit{Let }$\{S_{n}\}_{n=1}^{\infty }$ \textit{be a sequence of random
variables and} $\mathcal{P}$ \textit{be a given set of probabilities} $P$ 
\textit{on a measure space on which }$(S_{n},U_{n}(\alpha _{P}),V_{n}(\alpha
_{P}))$ \textit{lives, where} $\alpha _{P}\in (0,1)$ \textit{is a quantity
that may depend on} $P\in \mathcal{P}$ \textit{and for some} $\varepsilon >0$%
, 
\begin{equation}
\varepsilon \leq \inf_{P\in \mathcal{P}}\alpha _{P}\leq \sup_{P\in \mathcal{P%
}}\alpha _{P}\leq 1-\varepsilon .  \label{bds}
\end{equation}

\textit{Furthermore, assume that for each} $n\geq 1$, \textit{the random
vector} $(S_{n},U_{n}(\alpha _{P}))$\textit{\ is independent of }$%
V_{n}(\alpha _{P})$ \textit{with respect to each} $P\in \mathcal{P}$. 
\textit{Let for }$t_{1},t_{2}\in \mathbf{R}^{2}$,%
\begin{equation*}
b_{n,P}(t_{1},t_{2};\sigma _{P})\equiv \left\vert P\left\{ S_{n}\leq
t_{1},U_{n}(\alpha _{P})\leq t_{2}\right\} -P\{\sigma _{P}\mathbb{Z}_{1}\leq
t_{1},\sqrt{1-\alpha _{P}}\mathbb{Z}_{2}\leq t_{2}\}\right\vert ,
\end{equation*}%
\textit{where }$\mathbb{Z}_{1}$ \textit{and} $\mathbb{Z}_{2}$\textit{\ are
independent standard normal random variables and }$\sigma _{P}^{2}>0$\textit{%
\ for each }$P\in \mathcal{P}$. \textit{(Note that }$\inf_{P\in \mathcal{P}%
}\sigma _{P}^{2}$ \textit{is allowed to be zero.})\newline
(i) \textit{As} $n\rightarrow \infty ,$%
\begin{eqnarray*}
&&\sup_{P\in \mathcal{P}}\sup_{t\in \mathbf{R}}\left\vert \mathbf{E}[\exp
(itS_{n})|N_{n}(\alpha _{P})=n]-\exp \left( -\frac{\sigma _{P}^{2}t^{2}}{2}%
\right) \right\vert \\
&\leq &2\varepsilon +\left( 4C_{d}\sup_{P\in \mathcal{P}}a_{n,P}(\varepsilon
)\right) \sqrt{\frac{2\pi }{\varepsilon }},
\end{eqnarray*}%
\textit{where }$a_{n,P}(\varepsilon )\equiv \varepsilon
^{-d}b_{n,P}+\varepsilon ,b_{n,P}\equiv \sup_{t_{1},t_{2}\in \mathbf{R}%
}b_{n,P}(t_{1},t_{2};\sigma _{P})$,\textit{\ and }$\varepsilon $\textit{\ is
the constant in }(\ref{bds}).\newline
(ii) \textit{Suppose further that for all }$t_{1},t_{2}\in \mathbf{R}$, 
\textit{as} $n\rightarrow \infty ,$%
\begin{equation*}
\sup_{P\in \mathcal{P}}b_{n,P}(t_{1},t_{2};0)\rightarrow 0.
\end{equation*}%
\textit{Then, for all }$t\in \mathbf{R}$, \textit{we have as }$n\rightarrow
\infty ,$%
\begin{equation*}
\sup_{P\in \mathcal{P}}\left\vert \mathbf{E}[\exp (itS_{n})|N_{n}(\alpha
_{P})=n]-1\right\vert \rightarrow 0.
\end{equation*}
\end{LemmaC}

\begin{rem}
While the proof of Lemma B3 follows that of Lemma 2.4 of \citeasnoun{GMZ},
it is worth noting that in contrast to Lemma 2.4 of \citeasnoun{GMZ} or
Theorem 2.1 of \citeasnoun{Beirlant/Mason:95}, Lemma B3 gives an explicit
bound for the difference between the conditional characteristic function of $%
S_{n}$ given $N_{n}(\alpha _{P})=n$ and the characteristic function of $%
N(0,\sigma _{P}^{2})$. Under the stated conditions, (in particular (\ref{bds}%
)), the explicit bound is shown to depend on $P\in \mathcal{P}$ only through 
$b_{n,P} $. Thus in order to obtain a bound uniform in $P\in \mathcal{P}$,
it suffices to control $\alpha _{P}$ and $b_{n,P}$ uniformly in $P\in 
\mathcal{P}$.
\end{rem}

\begin{proof}[Proof of Lemma B\protect\ref{lem-c3}]
(i) Let $\phi _{n,P}(t,u)=\mathbf{E}[\exp (itS_{n}+iuU_{n}(\alpha _{P}))]$
and%
\begin{equation*}
\phi _{P}(t,u)=\exp (-(\sigma _{P}^{2}t^{2}+(1-\alpha _{P})u^{2})/2).
\end{equation*}%
By the condition of the lemma and Lemma B2(i), we have for any $\varepsilon
>0,$%
\begin{eqnarray}
\left\vert \phi _{n,P}(t,u)-\phi _{P}(t,u)\right\vert &\leq &(\varepsilon
^{-d}C_{d}b_{n,P}+4\varepsilon )  \label{cv44} \\
&\leq &4\varepsilon ^{-d}C_{d}b_{n,P}+4\varepsilon
=4C_{d}a_{n,P}(\varepsilon )\text{.}  \notag
\end{eqnarray}%
Note that $a_{n,P}(\varepsilon )$ depends on $P\in \mathcal{P}$ only through 
$b_{n,P}$.

Following the proof of Lemma 2.4 of \citeasnoun{GMZ}, we have%
\begin{eqnarray*}
\psi _{n,P}(t) &=&\mathbf{E}[\exp (itS_{n})|N_{n}(\alpha _{P})=n] \\
&=&\frac{1}{\sqrt{2\pi }}\left( 1+o(1)\right) \int_{-\pi \sqrt{n}}^{\pi 
\sqrt{n}}\phi _{n,P}(t,v)\mathbf{E}\left[ \exp (ivV_{n}(\alpha _{P}))\right]
dv,
\end{eqnarray*}%
uniformly over $P\in \mathcal{P}$. Note that the equality comes after
applying Sterling's formula to $2\pi P\{N_{n}(\alpha _{P})=n\}$ and change
of variables from $u$ to $v/\sqrt{n}$. (See the proof of Lemma 2.4 of %
\citeasnoun{GMZ}.) The distribution of $N_{n}(\alpha _{P}),$ being Poisson $%
(n)$, does not depend on the particular choice of $\alpha _{P}\in (0,1),$
and hence the $o(1)$ term is $o(1)$ uniformly over $t\in \mathbf{R}$ and
over $P\in \mathcal{P}$. We follow the proof of Theorem 3 of %
\citeasnoun[p.517]{Feller:66} to observe that there exists $n_{0}>0$ such
that uniformly over $\alpha \in \lbrack \varepsilon ,1-\varepsilon ],$%
\begin{equation*}
\left\{ \int_{-\pi \sqrt{n}}^{\pi \sqrt{n}}\left\vert \mathbf{E}\exp
(ivV_{n}(\alpha ))-\exp (-\alpha v^{2}/2)\right\vert dv+\int_{|v|>\pi \sqrt{n%
}}\exp \left( -\alpha v^{2}/2\right) dv\right\} <\varepsilon ,
\end{equation*}%
for all $n>n_{0}$. Note that the distribution of $V_{n}(\alpha _{P})$
depends on $P\in \mathcal{P}$ only through $\alpha _{P}\in \lbrack
\varepsilon ,1-\varepsilon ]$ and $\varepsilon $ does not depend on $P$.
Since there exists $n_{1}$ such that for all $n>n_{1}$, 
\begin{equation*}
\sup_{P\in \mathcal{P}}\int_{|v|>\pi \sqrt{n}}\exp \left( -\alpha
_{P}v^{2}/2\right) dv<\varepsilon ,
\end{equation*}%
the previous inequality implies that for all $n>\max \{n_{0},n_{1}\}$,%
\begin{eqnarray}
&&\sup_{P\in \mathcal{P}}\int_{-\pi \sqrt{n}}^{\pi \sqrt{n}}\left\vert \phi
_{n,P}(t,u)\left( \mathbf{E}\exp (iuV_{n}(\alpha _{P}))-\exp (-\alpha
_{P}u^{2}/2)\right) \right\vert du  \label{cv23} \\
&\leq &\sup_{P\in \mathcal{P}}\int_{-\pi \sqrt{n}}^{\pi \sqrt{n}}\left(
\sup_{P\in \mathcal{P}}|\phi _{n,P}(t,u)|\right) |\mathbf{E}\exp
(iuV_{n}(\alpha _{P}))-\exp (-\alpha _{P}u^{2}/2)|du  \notag \\
&\leq &\sup_{P\in \mathcal{P}}\int_{-\pi \sqrt{n}}^{\pi \sqrt{n}}|\mathbf{E}%
\exp (iuV_{n}(\alpha _{P}))-\exp (-\alpha _{P}u^{2}/2)|du\leq \varepsilon . 
\notag
\end{eqnarray}%
By (\ref{cv44}) and (\ref{cv23}),%
\begin{eqnarray*}
&&\sup_{P\in \mathcal{P}}\left\vert \int_{-\pi \sqrt{n}}^{\pi \sqrt{n}}\phi
_{n,P}(t,u)\mathbf{E}\left[ \exp (iuV_{n}(\alpha _{P}))\right] du-\int_{-\pi 
\sqrt{n}}^{\pi \sqrt{n}}\phi _{P}(t,u)\exp \left( -\alpha _{P}u^{2}/2\right)
du\right\vert \\
&\leq &\sup_{P\in \mathcal{P}}\sup_{\alpha \in \lbrack \varepsilon
,1-\varepsilon ]}\int_{-\pi \sqrt{n}}^{\pi \sqrt{n}}\left\vert \phi
_{n,P}(t,u)\left( \mathbf{E}\exp (iuV_{n}(\alpha ))-\exp (-\alpha
u^{2}/2)\right) \right\vert du \\
&&+\int_{-\pi \sqrt{n}}^{\pi \sqrt{n}}\sup_{P\in \mathcal{P}}\sup_{\alpha
\in \lbrack \varepsilon ,1-\varepsilon ]}\left\vert \phi _{n,P}(t,u)-\phi
_{P}(t,u)\right\vert \exp (-\alpha u^{2}/2)du \\
&\leq &\varepsilon +\left( 4C_{d}\sup_{P\in \mathcal{P}}a_{n,P}(\varepsilon
)\right) \sup_{\alpha \in \lbrack \varepsilon ,1-\varepsilon ]}\int_{-\pi 
\sqrt{n}}^{\pi \sqrt{n}}\exp (-\alpha u^{2}/2)du \\
&\leq &\varepsilon +\left( 4C_{d}\sup_{P\in \mathcal{P}}a_{n,P}(\varepsilon
)\right) \sup_{\alpha \in \lbrack \varepsilon ,1-\varepsilon ]}\sqrt{\frac{%
2\pi }{\alpha }}=\varepsilon +\left( 4C_{d}\sup_{P\in \mathcal{P}%
}a_{n,P}(\varepsilon )\right) \sqrt{\frac{2\pi }{\varepsilon }}
\end{eqnarray*}%
as $n\rightarrow \infty $. Since 
\begin{equation*}
\exp \left( -\frac{\sigma _{P}^{2}t^{2}}{2}\right) =\frac{1}{\sqrt{2\pi }}%
\int_{-\infty }^{\infty }\phi _{P}(t,u)\exp \left( -\frac{\alpha _{P}u^{2}}{2%
}\right) du,
\end{equation*}%
and from some large $n$ on that does not depend on $P\in \mathcal{P}$,%
\begin{eqnarray*}
&&\left\vert \int_{-\infty }^{\infty }\phi _{P}(t,u)\exp \left( -\frac{%
\alpha _{P}u^{2}}{2}\right) du-\int_{-\pi \sqrt{n}}^{\pi \sqrt{n}}\phi
_{P}(t,u)\exp \left( -\frac{\alpha _{P}u^{2}}{2}\right) du\right\vert \\
&=&\exp \left( -\frac{\sigma _{P}^{2}t^{2}}{2}\right) \left\vert
\int_{-\infty }^{\infty }\exp \left( -\frac{u^{2}}{2}\right) du-\int_{-\pi 
\sqrt{n}}^{\pi \sqrt{n}}\exp \left( -\frac{u^{2}}{2}\right) du\right\vert
<\varepsilon ,
\end{eqnarray*}%
we conclude that for each $t\in \mathbf{R}$,%
\begin{equation*}
\left\vert \psi _{n,P}(t)-\exp \left( -\frac{\sigma _{P}^{2}t^{2}}{2}\right)
\right\vert \leq 2\varepsilon +\left( 4C_{d}\sup_{P\in \mathcal{P}%
}a_{n,P}(\varepsilon )\right) \sqrt{\frac{2\pi }{\varepsilon }},
\end{equation*}%
as $n\rightarrow \infty $. Since the right hand side does not depend on $%
t\in \mathbf{R}$ and $P\in \mathcal{P}$, we obtain the desired result.%
\newline
(ii) By the condition of the lemma and Lemma B2(ii), we have for any $t,u\in 
\mathbf{R},$%
\begin{equation*}
\sup_{P\in \mathcal{P}}\left\vert \phi _{n,P}(t,u)-\phi _{P}(0,u)\right\vert
\rightarrow 0\text{,}
\end{equation*}%
as $n\rightarrow \infty $. The rest of the proof is similar to that of (i).
We omit the details.
\end{proof}

Define for $x\in \mathcal{X}$, $\tau _{1},\tau _{2}\in \mathcal{T}$, and $%
j,k\in \mathbb{N}_{J},$%
\begin{equation*}
k_{n,\tau ,j,m}(x)\equiv \frac{1}{h^{d}}\mathbf{E}\left[ \left\vert \beta
_{n,x,\tau ,j}\left( Y_{ij},\frac{X_{i}-x}{h}\right) \right\vert ^{m}\right].
\end{equation*}

\begin{LemmaC}
\label{lem-c4} \textit{Suppose that Assumption A\ref{assumption-A6}(i) holds.%
} \textit{Then for all }$m\in \lbrack 2,M]$, \textit{(with }$M>0$ \textit{%
being the constant that appears in Assumption A\ref{assumption-A6}(i)}), 
\textit{there exists }$C_{1}\in (0,\infty )$ \textit{that does not depend on}
$n$ \textit{such that for each }$j\in \mathbb{N}_{J}$,%
\begin{equation*}
\sup_{\tau \in \mathcal{T},x\in \mathcal{S}_{\tau }(\varepsilon )}\sup_{P\in 
\mathcal{P}}k_{n,\tau ,j,m}(x)\leq C_{1}.
\end{equation*}
\end{LemmaC}

\begin{proof}[Proof of Lemma B\protect\ref{lem-c4}]
The proof can be proceeded by using Assumption A\ref{assumption-A6}(i) and
following the proof of Lemma 4 of \citeasnoun{LSW}.
\end{proof}

Let $N$ be a Poisson random variable with mean $n$ and independent of $%
(Y_{i}^{\top },X_{i}^{\top })_{i=1}^{\infty }$. Also, let $\beta _{n,x,\tau
}(Y_{i},(X_{i}-x)/h)$ be the $J$-dimensional vector whose $j$-th entry is
equal to $\beta _{n,x,\tau ,j}(Y_{ij},(X_{i}-x)/h)$. We define%
\begin{eqnarray*}
\mathbf{z}_{N,\tau }(x) &\equiv &\frac{1}{nh^{d}}\sum_{i=1}^{N}\beta
_{n,x,\tau }\left( Y_{i},\frac{X_{i}-x}{h}\right) -\frac{1}{h^{d}}\mathbf{E}%
\beta _{n,x,\tau }\left( Y_{i},\frac{X_{i}-x}{h}\right) \text{ and} \\
\mathbf{z}_{n,\tau }(x) &\equiv &\frac{1}{nh^{d}}\sum_{i=1}^{n}\beta
_{n,x,\tau }\left( Y_{i},\frac{X_{i}-x}{h}\right) -\frac{1}{h^{d}}\mathbf{E}%
\beta _{n,x,\tau }\left( Y_{i},\frac{X_{i}-x}{h}\right) .
\end{eqnarray*}%
Let $N_{1}$ be a Poisson random variable with mean 1, independent of $%
(Y_{i}^{\top },X_{i}^{\top })_{i=1}^{\infty }$. Define%
\begin{eqnarray*}
q_{n,\tau }(x) &\equiv &\frac{1}{\sqrt{h^{d}}}\sum_{1\leq i\leq
N_{1}}\left\{ \beta _{n,x,\tau }\left( Y_{i},\frac{X_{i}-x}{h}\right) -%
\mathbf{E}\beta _{n,x,\tau }\left( Y_{i},\frac{X_{i}-x}{h}\right) \right\} 
\text{ and} \\
\bar{q}_{n,\tau }(x) &\equiv &\frac{1}{\sqrt{h^{d}}}\left\{ \beta _{n,x,\tau
}\left( Y_{i},\frac{X_{i}-x}{h}\right) -\mathbf{E}\beta _{n,x,\tau }\left(
Y_{i},\frac{X_{i}-x}{h}\right) \right\} .
\end{eqnarray*}

\begin{LemmaC}
\label{lem-c5} \textit{Suppose that Assumption A\ref{assumption-A6}(i)
holds. Then for any }$m\in \lbrack 2,M]$\textit{\ (with }$M>0$\textit{\
being the constant in Assumption A6(i))}%
\begin{eqnarray}
\sup_{(x,\tau )\in \mathcal{S}}\sup_{P\in \mathcal{P}}\mathbf{E}\left[
||q_{n,\tau }(x)||^{m}\right] &\leq &\bar{C}_{1}h^{d(1-(m/2))}\text{\textit{%
\ and}}  \label{bd52} \\
\sup_{(x,\tau )\in \mathcal{S}}\sup_{P\in \mathcal{P}}\mathbf{E}\left[ ||%
\bar{q}_{n,\tau }(x)||^{m}\right] &\leq &\bar{C}_{2}h^{d(1-(m/2))},  \notag
\end{eqnarray}%
\textit{where }$\bar{C}_{1},\bar{C}_{2}>0$ \textit{are constants that depend
only on} $m.$

\textit{If furthermore, } $\limsup_{n\rightarrow \infty
}n^{-(m/2)+1}h^{d(1-(m/2))}<C$ \textit{for some constant} $C>0$, \textit{then%
} 
\begin{eqnarray}
\sup_{(x,\tau )\in \mathcal{S}}\sup_{P\in \mathcal{P}}\mathbf{E}\left[
||n^{1/2}h^{d/2}\mathbf{z}_{N,\tau }(x)||^{m}\right] &\leq &\left( \frac{15m%
}{\log m}\right) ^{m}\max \left\{ \bar{C}_{1},2\bar{C}_{1}C\right\} \text{%
\textit{\ and}}  \label{ineqs} \\
\sup_{(x,\tau )\in \mathcal{S}}\sup_{P\in \mathcal{P}}\mathbf{E}\left[
||n^{1/2}h^{d/2}\mathbf{z}_{n,\tau }(x)||^{m}\right] &\leq &\left( \frac{15m%
}{\log m}\right) ^{m}\max \left\{ \bar{C}_{2},2\bar{C}_{2}C\right\} ,  \notag
\end{eqnarray}%
\textit{where} $\bar{C}_{1},\bar{C}_{2}>0$\textit{\ are the constants that
appear in (\ref{bd52}).}
\end{LemmaC}

\begin{proof}[Proof of Lemma B\protect\ref{lem-c5}]
Let $q_{n,\tau ,j}(x)$ be the $j$-th entry of $q_{n,\tau }(x)$. For the
first statement of the lemma, it suffices to observe that for some positive
constants $C_{1}$ and $\bar{C},$%
\begin{equation}
\sup_{(x,\tau )\in \mathcal{S}}\sup_{P\in \mathcal{P}}\mathbf{E}\left[
|q_{n,\tau ,j}(x)|^{m}\right] \leq \frac{C_{1}h^{d}k_{n,\tau ,j,m}}{h^{dm/2}}%
\leq \bar{C}h^{d(1-(m/2))},  \label{bdd}
\end{equation}%
where the first inequality uses the definition of $k_{n,\tau ,j,m}$, and the
last inequality uses Lemma B4 and the fact that $m\in \lbrack 2,M]$. The
second statement in (\ref{bd52}) follows similarly.

We consider the statements in (\ref{ineqs}). We consider the first
inequality in (\ref{ineqs}). Let $z_{N,\tau ,j}(x)$ be the $j$-th entry of $%
\mathbf{z}_{N,\tau }(x)$. Then using Rosenthal's inequality (e.g. (2.3) of %
\citeasnoun{GMZ}), we find that%
\begin{eqnarray*}
&&\sup_{(x,\tau )\in \mathcal{S}}\sup_{P\in \mathcal{P}}\mathbf{E}[|\sqrt{%
nh^{d}}z_{N,\tau ,j}(x)|^{m}] \\
&\leq &\left( \frac{15m}{\log m}\right) ^{m}\sup_{(x,\tau )\in \mathcal{S}%
}\sup_{P\in \mathcal{P}}\max \left\{ \left( \mathbf{E}q_{n,\tau
,j}^{2}(x)\right) ^{m/2},n^{-m/2+1}\mathbf{E}|q_{n,\tau ,j}(x)|^{m}\right\} .
\end{eqnarray*}%
Since $\mathbf{E}q_{n,\tau ,j}^{2}(x)\leq (\mathbf{E}|q_{n,\tau
,j}(x)|^{m})^{2/m}$, by (\ref{bdd}), the last term is bounded by%
\begin{eqnarray*}
&&\left( \frac{15m}{\log m}\right) ^{m}\max \left\{ \bar{C},\bar{C}%
n^{-(m/2)+1}h^{d(1-(m/2))}\right\} \\
&\leq &\left( \frac{15m}{\log m}\right) ^{m}\max \left\{ \bar{C},2\bar{C}%
C\right\} ,
\end{eqnarray*}%
from some large $n$ on by the condition limsup$_{n\rightarrow \infty
}n^{-(m/2)+1}h^{d(1-(m/2))}<C$.

As for the second inequality in (\ref{ineqs}), for some $C>0$, we use the
second inequality in (\ref{bd52}) and use Rosenthal's inequality in the same
way as before, to obtain the inequality.
\end{proof}

The following lemma offers a characterization of the scale normalizer of our
test statistic. For $A,A^{\prime }\subset \mathbb{N}_{J}$, define $\mathbf{%
\zeta }_{n,\tau }(x)\equiv \sqrt{nh^{d}}\mathbf{z}_{N,\tau }(x)$,%
\begin{eqnarray}
C_{n,\tau ,\tau ^{\prime },A,A^{\prime }}^{R}(x,x^{\prime }) &\equiv
&h^{-d}Cov\left( \Lambda _{A,p}\left( \mathbf{\zeta }_{n,\tau }(x)\right)
,\Lambda _{A^{\prime },p}\left( \mathbf{\zeta }_{n,\tau ^{\prime
}}(x^{\prime })\right) \right) \text{, and}  \label{C} \\
C_{n,\tau ,\tau ^{\prime },A,A^{\prime }}(x,u) &\equiv &Cov\left( \Lambda
_{A,p}\left( \mathbb{W}_{n,\tau ,\tau ^{\prime }}^{(1)}(x,u)\right) ,\Lambda
_{A^{\prime },p}\left( \mathbb{W}_{n,\tau ,\tau ^{\prime
}}^{(2)}(x,u)\right) \right) ,  \notag
\end{eqnarray}%
where we recall that $[\mathbb{W}_{n,\tau _{1},\tau _{2}}^{(1)}(x,u)^{\top },%
\mathbb{W}_{n,\tau _{1},\tau _{2}}^{(2)}(x,u)^{\top }]^{\top }$ is a mean
zero $\mathbf{R}^{2J}$-valued Gaussian random vector whose covariance matrix
is given by (\ref{cov}).

Then for Borel sets $B,B^{\prime }\subset \mathcal{S}$ and $A,A^{\prime
}\subset \mathbb{N}_{J},$ let%
\begin{equation*}
\sigma _{n,A,A^{\prime }}^{R}(B,B^{\prime })\equiv \int_{B^{\prime
}}\int_{B}C_{n,\tau ,\tau ^{\prime },A,A^{\prime }}^{R}(x,x^{\prime
})dQ(x,\tau )dQ(x^{\prime },\tau ^{\prime })
\end{equation*}%
and%
\begin{equation}
\sigma _{n,A,A^{\prime }}(B,B^{\prime })\equiv \int_{\mathcal{T}}\int_{%
\mathcal{T}}\int_{B_{\tau }\cap B_{\tau ^{\prime }}^{\prime }}\int_{\mathcal{%
U}}C_{n,\tau ,\tau ^{\prime },A,A^{\prime }}(x,u)dudxd\tau d\tau ^{\prime },
\label{s2}
\end{equation}%
where $B_{\tau }\equiv \{x\in \mathcal{X}:(x,\tau )\in B\}$ and $B_{\tau
^{\prime }}^{\prime }\equiv \{x\in \mathcal{X}:(x,\tau ^{\prime })\in
B^{\prime }\}$.

The lemma below shows that $\sigma _{n,A,A^{\prime }}^{R}(B,B^{\prime })$
and $\sigma _{n,A,A^{\prime }}(B,B^{\prime })$ are asymptotically equivalent
uniformly in $P\in \mathcal{P}$. We introduce some notation. Recall the
definition of $\Sigma _{n,\tau _{1},\tau _{2}}(x,u)$, which is found below (%
\ref{rho}). Define for $\bar{\varepsilon}>0,$%
\begin{equation*}
\tilde{\Sigma}_{n,\tau _{1},\tau _{2},\bar{\varepsilon}}(x,u)\equiv \left[ 
\begin{array}{c}
\Sigma _{n,\tau _{1},\tau _{1}}(x,0)+\bar{\varepsilon}I_{J} \\ 
\Sigma _{n,\tau _{1},\tau _{2}}(x,u)%
\end{array}%
\begin{array}{c}
\Sigma _{n,\tau _{1},\tau _{2}}(x,u) \\ 
\Sigma _{n,\tau _{2},\tau _{2}}(x+uh,0)+\bar{\varepsilon}I_{J}%
\end{array}%
\right] ,
\end{equation*}%
where $I_{J}$ is the $J$ dimensional identity matrix. Certainly $\tilde{%
\Sigma}_{n,\tau _{1},\tau _{2},\bar{\varepsilon}}(x,u)$ is positive
definite. We define%
\begin{equation*}
\mathbf{\xi }_{N,\tau _{1},\tau _{2}}(x,u;\eta _{1},\eta _{2})\equiv \sqrt{%
nh^{d}}\tilde{\Sigma}_{n,\tau _{1},\tau _{2},\bar{\varepsilon}}^{-1/2}(x,u)%
\left[ 
\begin{array}{c}
\mathbf{z}_{N,\tau _{1}}(x;\eta _{1}) \\ 
\mathbf{z}_{N,\tau _{2}}(x+uh;\eta _{2})%
\end{array}%
\right] ,
\end{equation*}%
where $\eta _{1}\in \mathbf{R}^{J}$ and $\eta _{2}\in \mathbf{R}^{J}$ are
random vectors that are independent, and independent of $(Y_{i}^{\top
},X_{i}^{\top })_{i=1}^{\infty }$, each following $N(0,\bar{\varepsilon}%
I_{J})$, and $\mathbf{z}_{N,\tau }(x;\eta _{1})\equiv \mathbf{z}_{N,\tau
}(x)+\eta _{1}/\sqrt{nh^{d}}.\mathbf{\ }$We are prepared to state the lemma.

\begin{LemmaC}
\label{lem-c6} \textit{Suppose that\ Assumption A\ref{assumption-A6}(i)
holds and that }$nh^{d}\rightarrow \infty ,$\textit{\ as }$n\rightarrow
\infty $, \textit{and }%
\begin{equation*}
\underset{n\rightarrow \infty }{\text{limsup}}\ n^{-(m/2)+1}h^{d(1-(m/2))}<C,
\end{equation*}%
\textit{for some constant }$C>0$\textit{\ and some }$m\in \lbrack 2(p+1),M]$.

\textit{Then for any sequences of Borel sets} $B_{n},$ $B_{n}^{\prime
}\subset \mathcal{S}$ \textit{and for any} $A,A^{\prime }\subset \mathbb{N}%
_{J},$%
\begin{equation*}
\sigma _{n,A,A^{\prime }}^{R}(B_{n},B_{n}^{\prime })=\sigma _{n,A,A^{\prime
}}(B_{n},B_{n}^{\prime })+o(1),
\end{equation*}%
\textit{where }$o(1)$ \textit{vanishes uniformly in} $P\in \mathcal{P}$%
\textit{\ as }$n\rightarrow \infty $.
\end{LemmaC}

\begin{rem}
The main innovative element of Lemma B\ref{lem-c6} is that the result does
not require that $\sigma _{n,A,A^{\prime }}(B_{n},B_{n}^{\prime })$ be
positive for each finite $n$ or positive in the limit. Hence the result can
be applied to the case where the scale normalizer $\sigma _{n,A,A^{\prime
}}^{R}(B_{n},B_{n}^{\prime })$ is degenerate (either in finite samples or
asymptotically).
\end{rem}

\begin{proof}[Proof of Lemma B\protect\ref{lem-c6}]
Define $B_{n,\tau }\equiv \{x\in \mathcal{X}:(x,\tau )\in B_{n}\}$, $w_{\tau
,B_{n}}(x)\equiv 1_{B_{n,\tau }}(x).$ For a given $\bar{\varepsilon}>0,$ let 
\begin{eqnarray*}
g_{1n,\tau _{1},\tau _{2},\bar{\varepsilon}}(x,u) &\equiv &h^{-d}Cov(\Lambda
_{A,p}(\sqrt{nh^{d}}\mathbf{z}_{N,\tau _{1}}(x;\eta _{1})),\Lambda
_{A^{\prime },p}(\sqrt{nh^{d}}\mathbf{z}_{N,\tau _{2}}(x+uh;\eta _{2}))), \\
g_{2n,\tau _{1},\tau _{2},\bar{\varepsilon}}(x,u) &\equiv &Cov(\Lambda
_{A,p}(\mathbb{Z}_{n,\tau _{1},\tau _{2},\bar{\varepsilon}}(x)),\Lambda
_{A^{\prime },p}(\mathbb{Z}_{n,\tau _{1},\tau _{2},\bar{\varepsilon}%
}(x+uh))),
\end{eqnarray*}%
and $\left( \mathbb{Z}_{n,\tau _{1},\tau _{2},\bar{\varepsilon}}^{\top }(x),%
\mathbb{Z}_{n,\tau _{1},\tau _{2},\bar{\varepsilon}}^{\top }(v)\right)
^{\top }$ is a centered normal $\mathbf{R}^{2J}$-valued random vector with
the same covariance matrix as that of $[\sqrt{nh^{d}}\mathbf{z}_{N,\tau
_{1}}^{\top }(x;\eta _{1}),\sqrt{nh^{d}}\mathbf{z}_{N,\tau _{2}}^{\top
}(v;\eta _{2})]^{\top }$. Then we define%
\begin{equation*}
\sigma _{n,A,A^{\prime },\bar{\varepsilon}}^{R}(B_{n},B_{n}^{\prime })\equiv
\int_{\mathcal{T}}\int_{\mathcal{T}}\int_{B_{n,\tau _{1}}}\int_{\mathcal{U}%
}g_{1n,\tau _{1},\tau _{2},\bar{\varepsilon}}(x,u)w_{\tau
_{1},B_{n}}(x)w_{\tau _{2},B_{n}^{\prime }}(x+uh)dudxd\tau _{1}d\tau _{2},
\end{equation*}%
and%
\begin{equation*}
\sigma _{n,A,A^{\prime },\bar{\varepsilon}}(B_{n},B_{n}^{\prime })\equiv
\int_{\mathcal{T}}\int_{\mathcal{T}}\int_{B_{n,\tau _{1}}\cap B_{n,\tau
_{2}}^{\prime }}\int_{\mathcal{U}}C_{n,\tau _{1},\tau _{2},A,A^{\prime },%
\bar{\varepsilon}}(x,u)dudxd\tau _{1}d\tau _{2},
\end{equation*}%
where%
\begin{equation}
C_{n,\tau _{1},\tau _{2},A,A^{\prime },\bar{\varepsilon}}(x,u)\equiv
Cov\left( \Lambda _{A,p}(\mathbb{W}_{n,\tau _{1},\tau _{2},\bar{\varepsilon}%
}^{(1)}(x,u)),\Lambda _{A^{\prime },p}(\mathbb{W}_{n,\tau _{1},\tau _{2},%
\bar{\varepsilon}}^{(2)}(x,u))\right) ,  \label{Ce}
\end{equation}%
and, with $\mathbb{Z}\sim N(0,I_{2J})$, 
\begin{equation}
\left[ 
\begin{array}{c}
\mathbb{W}_{n,\tau _{1},\tau _{2},\bar{\varepsilon}}^{(1)}(x,u) \\ 
\mathbb{W}_{n,\tau _{1},\tau _{2},\bar{\varepsilon}}^{(2)}(x,u)%
\end{array}%
\right] \equiv \tilde{\Sigma}_{n,\tau _{1},\tau _{2},\bar{\varepsilon}%
}^{1/2}(x,u)\mathbb{Z}.  \label{Ws}
\end{equation}%
Thus, $\sigma _{n,A,A^{\prime },\bar{\varepsilon}}^{R}(B_{n},B_{n}^{\prime
}) $ and $\sigma _{n,A,A^{\prime },\bar{\varepsilon}}(B_{n},B_{n}^{\prime })$
are \textquotedblleft regularized\textquotedblright\ versions of $\sigma
_{n,A,A^{\prime }}^{R}(B_{n},B_{n}^{\prime })$ and $\sigma _{n,A,A^{\prime
}}(B_{n},B_{n}^{\prime })$. We also define%
\begin{equation*}
\tau _{n,A,A^{\prime },\bar{\varepsilon}}(B_{n},B_{n}^{\prime })\equiv \int_{%
\mathcal{T}}\int_{\mathcal{T}}\int_{B_{n,\tau _{1}}}\int_{\mathcal{U}%
}g_{2n,\tau _{1},\tau _{2},\bar{\varepsilon}}(x,u)w_{\tau
_{1},B_{n}}(x)w_{\tau _{2},B_{n}^{\prime }}(x+uh)dudxd\tau _{1}d\tau _{2}.
\end{equation*}%
Then it suffices for the lemma to show the following two statements. \newline
\textbf{Step 1: }As $n\rightarrow \infty ,$%
\begin{eqnarray*}
\sup_{P\in \mathcal{P}}\left\vert \sigma _{n,A,A^{\prime },\bar{\varepsilon}%
}^{R}(B_{n},B_{n}^{\prime })-\tau _{n,A,A^{\prime },\bar{\varepsilon}%
}(B_{n},B_{n}^{\prime })\right\vert &\rightarrow &0,\text{ and} \\
\sup_{P\in \mathcal{P}}\left\vert \tau _{n,A,A^{\prime },\bar{\varepsilon}%
}(B_{n},B_{n}^{\prime })-\sigma _{n,A,A^{\prime },\bar{\varepsilon}%
}(B_{n},B_{n}^{\prime })\right\vert &\rightarrow &0.
\end{eqnarray*}%
\newline
\textbf{Step 2: }For some $C>0$ that does not depend on $\bar{\varepsilon}\ $%
or $n,$ 
\begin{eqnarray*}
\sup_{P\in \mathcal{P}}|\sigma _{n,A,A^{\prime },\bar{\varepsilon}%
}^{R}(B_{n},B_{n}^{\prime })-\sigma _{n,A,A^{\prime
}}^{R}(B_{n},B_{n}^{\prime })| &\leq &C\sqrt{\bar{\varepsilon}},\text{ and}
\\
\sup_{P\in \mathcal{P}}\left\vert \sigma _{n,A,A^{\prime },\bar{\varepsilon}%
}(B_{n},B_{n}^{\prime })-\sigma _{n,A,A^{\prime }}(B_{n},B_{n}^{\prime
})\right\vert &\leq &C\sqrt{\bar{\varepsilon}}.
\end{eqnarray*}

Then the desired result follows by sending $n\rightarrow \infty $ and then $%
\bar{\varepsilon}\downarrow 0$, while chaining Steps 1 and 2. \newline
\textbf{Proof of Step 1:} We first focus on the first statement. For any
vector $\mathbf{v}=[\mathbf{v}_{1}^{\top },\mathbf{v}_{2}^{\top }]^{\top
}\in \mathbf{R}^{2J}$, we define%
\begin{eqnarray*}
\tilde{\Lambda}_{A,p,1}\left( \mathbf{v}\right) &\equiv &\Lambda
_{A,p}\left( \left[ \tilde{\Sigma}_{n,\tau _{1},\tau _{2},\bar{\varepsilon}%
}^{1/2}(x,u)\mathbf{v}\right] _{1}\right) , \\
\tilde{\Lambda}_{A^{\prime },p,2}\left( \mathbf{v}\right) &\equiv &\Lambda
_{A^{\prime },p}\left( \left[ \tilde{\Sigma}_{n,\tau _{1},\tau _{2},\bar{%
\varepsilon}}^{1/2}(x,u)\mathbf{v}\right] _{2}\right) ,
\end{eqnarray*}%
and%
\begin{equation}
C_{n,p}(\mathbf{v})\equiv \tilde{\Lambda}_{A,p,1}\left( \mathbf{v}\right) 
\tilde{\Lambda}_{A^{\prime },p,2}\left( \mathbf{v}\right) ,  \label{lambda_t}
\end{equation}%
where $[a]_{1}$ of a vector $a\in \mathbf{R}^{2J}$ indicates the vector of
the first $J$ entries of $a$, and $[a]_{2}$ the vector of the remaining $J$
entries of $a.$ By Theorem 9 of \citeasnoun[p. 208]{Magnus:Neudecker:01},%
\begin{eqnarray}
\lambda _{\min }\left( \tilde{\Sigma}_{n,\tau _{1},\tau _{2},\bar{\varepsilon%
}}(x,u)\right) &\geq &\lambda _{\min }\left( \left[ 
\begin{array}{c}
\Sigma _{n,\tau _{1},\tau _{2}}(x,0) \\ 
\Sigma _{n,\tau _{1},\tau _{2}}^{\top }(x,u)%
\end{array}%
\begin{array}{c}
\Sigma _{n,\tau _{1},\tau _{2}}(x,u) \\ 
\Sigma _{n,\tau _{2},\tau _{2}}(x+uh,0)%
\end{array}%
\right] \right)  \label{lb2} \\
&&+\lambda _{\min }\left( \left[ 
\begin{array}{c}
\bar{\varepsilon}I_{J} \\ 
0%
\end{array}%
\begin{array}{c}
0 \\ 
\bar{\varepsilon}I_{J}%
\end{array}%
\right] \right)  \notag \\
&\geq &\lambda _{\min }\left( \left[ 
\begin{array}{c}
\bar{\varepsilon}I_{J} \\ 
0%
\end{array}%
\begin{array}{c}
0 \\ 
\bar{\varepsilon}I_{J}%
\end{array}%
\right] \right) =\bar{\varepsilon}.  \notag
\end{eqnarray}%
Let $q_{n,\tau ,j}(x;\eta _{1j})\equiv p_{n,\tau ,j}(x)+\eta _{1j},$ where 
\begin{equation*}
p_{n,\tau ,j}(x)\equiv \frac{1}{\sqrt{h^{d}}}\sum_{1\leq i\leq N_{1}}\left\{
\beta _{n,x,\tau ,j}\left( Y_{ij},\frac{X_{i}-x}{h}\right) -\mathbf{E}\left[
\beta _{n,x,\tau ,j}\left( Y_{ij},\frac{X_{i}-x}{h}\right) \right] \right\} ,
\end{equation*}%
$\eta _{1j}$ is the $j$-th entry of $\eta _{1}$, and $N_{1}$ is a Poisson
random variable with mean $1$ and $((\eta _{1j})_{j\in \mathbb{N}%
_{J}},N_{1}) $ is independent of $\{(Y_{i}^{\top },X_{i}^{\top
})\}_{i=1}^{\infty }.$ Let $p_{n,\tau }(x)$ be the column vector of entries $%
p_{n,\tau ,j}(x)$ with $j$ running in the set $\mathbb{N}_{J}$. Let $%
[p_{n,\tau _{1}}^{(i)}(x),p_{n,\tau _{2}}^{(i)}(x+uh)]$ be i.i.d. copies of $%
[p_{n,\tau _{1}}(x),p_{n,\tau _{2}}(x+uh)]$ and $\eta _{1}^{(i)}$ and $\eta
_{2}^{(i)}$ be also i.i.d. copies of $\eta _{1}$ and $\eta _{2}$. Define 
\begin{equation*}
q_{n,\tau ,1}^{(i)}(x)\equiv p_{n,\tau }^{(i)}(x)+\eta _{1}^{(i)}\text{ and }%
q_{n,\tau ,2}^{(i)}(x+uh)\equiv p_{n,\tau }^{(i)}(x+uh)+\eta _{2}^{(i)}.
\end{equation*}%
Note that%
\begin{equation*}
\frac{1}{\sqrt{n}}\sum_{i=1}^{n}\left[ 
\begin{array}{c}
q_{n,\tau _{1},1}^{(i)}(x) \\ 
q_{n,\tau _{2},2}^{(i)}(x+uh)%
\end{array}%
\right] =\frac{1}{\sqrt{n}}\sum_{i=1}^{n}\left[ 
\begin{array}{c}
p_{n,\tau _{1}}^{(i)}(x) \\ 
p_{n,\tau _{2}}^{(i)}(x+uh)%
\end{array}%
\right] +\frac{1}{\sqrt{n}}\sum_{i=1}^{n}\left[ 
\begin{array}{c}
\eta _{1}^{(i)} \\ 
\eta _{2}^{(i)}%
\end{array}%
\right] .
\end{equation*}%
The last sum has the same distribution as $[\eta _{1}^{\top },\eta
_{2}^{\top }]^{\top }$ and the leading sum on the right-hand side has the
same distribution as that of $[\mathbf{z}_{N,\tau _{1}}^{\top }(x),\mathbf{z}%
_{N,\tau _{2}}^{\top }(x+uh)]^{\top }$. Therefore, we conclude that%
\begin{equation*}
\mathbf{\xi }_{N,\tau _{1},\tau _{2}}(x,u;\eta _{1},\eta _{2})\overset{d}{=}%
\frac{1}{\sqrt{n}}\sum_{i=1}^{n}\tilde{W}_{n,\tau _{1},\tau _{2}}^{(i)}(x,u),
\end{equation*}%
where%
\begin{equation*}
\tilde{W}_{n,\tau _{1},\tau _{2}}^{(i)}(x,u)\equiv \tilde{\Sigma}_{n,\tau
_{1},\tau _{2},\bar{\varepsilon}}^{-1/2}(x,u)\left[ 
\begin{array}{c}
q_{n,\tau _{1},1}^{(i)}(x) \\ 
q_{n,\tau _{2},2}^{(i)}(x+uh)%
\end{array}%
\right] .
\end{equation*}%
Now we invoke the Berry-Esseen-type bound of 
\citeasnoun[Theorem
1]{Sweeting:77} to prove Step 1. By Lemma B5, we deduce that%
\begin{equation}
\sup_{(x,\tau )\in \mathcal{S}}\sup_{P\in \mathcal{P}}\mathbf{E}||q_{n,\tau
,1}^{(i)}(x)||^{3}\leq Ch^{-d/2},  \label{bd35}
\end{equation}%
for some $C>0$. Also, recall the definition of $\rho _{n,\tau _{1},\tau
_{1},j,j}(x,0)$ in (\ref{rho}) and note that%
\begin{eqnarray}
&&\sup_{\tau \in \mathcal{T}}\sup_{(x,u)\in \mathcal{S}_{\tau }(\varepsilon
)\times \mathcal{U}}\sup_{P\in \mathcal{P}}tr\left( \tilde{\Sigma}_{n,\tau
_{1},\tau _{2},\bar{\varepsilon}}(x,u)\right)  \label{trace} \\
&\leq &\sup_{\tau \in \mathcal{T},x\in \mathcal{S}_{\tau }(\varepsilon
)}\sup_{P\in \mathcal{P}}\sum_{j\in J}\left( \rho _{n,\tau _{1},\tau
_{1},j,j}(x,0)+\rho _{n,\tau _{2},\tau _{2},j,j}(x,0)+2\bar{\varepsilon}%
\right) \leq C,  \notag
\end{eqnarray}%
for some $C>0$ that depends only on $J$ and $\bar{\varepsilon}$ by Lemma B4.
Observe that by the definition of $C_{n,p}$ in (\ref{lambda_t}), and (\ref%
{trace}),%
\begin{equation*}
\sup_{\mathbf{v}\in \mathbf{R}^{2J}}\frac{\left\vert C_{n,p}(\mathbf{v}%
)-C_{n,p}(0)\right\vert }{1+||\mathbf{v}||^{2p+2}\min \left\{ ||\mathbf{v}%
||,1\right\} }\leq C.
\end{equation*}%
We find that for each $u\in \mathcal{U},$ $||\tilde{W}_{n,\tau _{1},\tau
_{2}}^{(i)}(x,u)||^{2}$ is equal to%
\begin{eqnarray}
&&tr\left( \tilde{\Sigma}_{n,\tau _{1},\tau _{2},\bar{\varepsilon}%
}^{-1/2}(x,u)\left[ 
\begin{array}{c}
q_{n,\tau _{1},1}^{(i)}(x) \\ 
q_{n,\tau _{1},2}^{(i)}(x+uh)%
\end{array}%
\right] \left[ 
\begin{array}{c}
q_{n,\tau _{2},1}^{(i)}(x) \\ 
q_{n,\tau _{2},2}^{(i)}(x+uh)%
\end{array}%
\right] ^{\top }\tilde{\Sigma}_{n,\tau _{1},\tau _{2},\bar{\varepsilon}%
}^{-1/2}(x,u)\right)  \label{der5} \\
&\leq &\lambda _{\max }\left( \tilde{\Sigma}_{n,\tau _{1},\tau _{2},\bar{%
\varepsilon}}^{-1}(x,u)\right) tr\left( \left[ 
\begin{array}{c}
q_{n,\tau _{1},1}^{(i)}(x) \\ 
q_{n,\tau _{1},2}^{(i)}(x+uh)%
\end{array}%
\right] \left[ 
\begin{array}{c}
q_{n,\tau _{2},1}^{(i)}(x) \\ 
q_{n,\tau _{2},2}^{(i)}(x+uh)%
\end{array}%
\right] ^{\top }\right) .  \notag
\end{eqnarray}%
Therefore, $\mathbf{E}||\tilde{W}_{n,\tau _{1},\tau _{2}}^{(i)}(x,u)||^{3}$
is bounded by%
\begin{equation*}
\lambda _{\max }^{3/2}\left( \tilde{\Sigma}_{n,\tau _{1},\tau _{2},\bar{%
\varepsilon}}^{-1}(x,u)\right) \mathbf{E}\left\Vert \left[ 
\begin{array}{c}
q_{n,\tau _{1},1}^{(i)}(x) \\ 
q_{n,\tau _{2},2}^{(i)}(x+uh)%
\end{array}%
\right] \right\Vert ^{3}.
\end{equation*}%
From (\ref{lb2}),%
\begin{equation*}
\lambda _{\max }^{3/2}(\tilde{\Sigma}_{n,\tau _{1},\tau _{2},\bar{\varepsilon%
}}^{-1}(x,u))=\lambda _{\min }^{-3/2}(\tilde{\Sigma}_{n,\tau _{1},\tau _{2},%
\bar{\varepsilon}}(x,u))\leq \bar{\varepsilon}^{-3/2}.
\end{equation*}%
Therefore, we conclude that%
\begin{eqnarray*}
&&\sup_{\tau \in \mathcal{T}}\sup_{(x,u)\in \mathcal{S}_{\tau }(\varepsilon
)\times \mathcal{U}}\sup_{P\in \mathcal{P}}\mathbf{E}||\tilde{W}_{n,\tau
_{1},\tau _{2}}^{(i)}(x,u)||^{3} \\
&\leq &C_{1}\bar{\varepsilon}^{-3/2}\cdot \sup_{\tau \in \mathcal{T},x\in 
\mathcal{S}_{\tau }(\varepsilon )}\sup_{P\in \mathcal{P}}\mathbf{E}%
||q_{n,\tau _{1},1}^{(i)}(x)||^{3} \\
&&+C_{1}\bar{\varepsilon}^{-3/2}\cdot \sup_{\tau \in \mathcal{T}%
}\sup_{(x,u)\in \mathcal{S}_{\tau }(\varepsilon )\times \mathcal{U}%
}\sup_{P\in \mathcal{P}}\mathbf{E}||q_{n,\tau _{2},2}^{(i)}(x+uh)||^{3}\leq
C_{2}\bar{\varepsilon}^{-3/2}/\sqrt{h^{d}},
\end{eqnarray*}%
where $C_{1}>0$ and $C_{2}>0$ are constants depending only on $J$, and the
last bound follows by (\ref{bd35}). Therefore, by Theorem 1 of %
\citeasnoun{Sweeting:77}, we find that with $\bar{\varepsilon}>0$ fixed and $%
n\rightarrow \infty $,%
\begin{eqnarray}
&&\sup_{\tau \in \mathcal{T}}\sup_{(x,u)\in \mathcal{S}_{\tau }(\varepsilon
)\times \mathcal{U}}\sup_{P\in \mathcal{P}}\left\vert \mathbf{E}%
C_{n,p}\left( \frac{1}{\sqrt{n}}\sum_{i=1}^{n}\tilde{W}_{n,\tau _{1},\tau
_{2}}^{(i)}(x,u)\right) -\mathbf{E}C_{n,p}\left( \mathbb{\tilde{Z}}_{n,\tau
_{1},\tau _{2}}(x,u)\right) \right\vert  \label{cv4} \\
&=&O\left( n^{-1/2}h^{-d/2}\right) =o(1),  \notag
\end{eqnarray}%
where $\mathbb{\tilde{Z}}_{n,\tau _{1},\tau _{2}}(x,u)=[\mathbb{Z}_{n,\tau
_{1},\tau _{2},\bar{\varepsilon}}(x)^{\top },\mathbb{Z}_{n,\tau _{1},\tau
_{2},\bar{\varepsilon}}(x+uh)^{\top }]^{\top }.$

Using similar arguments, we also deduce that for $j=1,2,$ and $A\subset 
\mathbb{N}_{J}$,%
\begin{equation*}
\sup_{\tau \in \mathcal{T}}\sup_{(x,u)\in \mathcal{S}_{\tau }(\varepsilon
)\times \mathcal{U}}\sup_{P\in \mathcal{P}}\left\vert \mathbf{E}\tilde{%
\Lambda}_{A,p,j}\left( \frac{1}{\sqrt{n}}\sum_{i=1}^{n}\tilde{W}_{n,\tau
_{1},\tau _{2}}^{(i)}(x,u)\right) -\mathbf{E}\tilde{\Lambda}_{A,p,j}\left( 
\mathbb{\tilde{Z}}_{n,\tau _{1},\tau _{2}}(x,u)\right) \right\vert =o(1).
\end{equation*}

For some $C>0,$%
\begin{eqnarray*}
&&\sup_{\tau \in \mathcal{T}}\sup_{(x,u)\in \mathcal{S}_{\tau }(\varepsilon
)\times \mathcal{U}}\sup_{P\in \mathcal{P}}Cov\left( \Lambda _{p}(\mathbb{Z}%
_{n,\tau _{1},\tau _{2},\bar{\varepsilon}}(x)),\Lambda _{p}(\mathbb{Z}%
_{n,\tau _{1},\tau _{2},\bar{\varepsilon}}(x+uh))\right) \\
&\leq &\sup_{\tau \in \mathcal{T}}\sup_{(x,u)\in \mathcal{S}_{\tau
}(\varepsilon )\times \mathcal{U}}\sup_{P\in \mathcal{P}}\sqrt{\mathbf{E}%
\left\Vert \mathbb{Z}_{n,\tau _{1},\tau _{2},\bar{\varepsilon}%
}(x)\right\Vert ^{2p}}\sqrt{\mathbf{E}\left\Vert \mathbb{Z}_{n,\tau
_{1},\tau _{2},\bar{\varepsilon}}(x+uh)\right\Vert ^{2p}}<C.
\end{eqnarray*}%
The last inequality follows because $\mathbb{Z}_{n,\tau _{1},\tau _{2},\bar{%
\varepsilon}}(x)$ and $\mathbb{Z}_{n,\tau _{1},\tau _{2},\bar{\varepsilon}%
}(x+uh)$ are centered normal random vectors with a covariance matrix that
has a finite Euclidean norm by Lemma B4. Hence we apply the Dominated
Convergence Theorem to deduce the first statement of Step 1 from (\ref{cv4}).%
\newline
We turn to the second statement of Step 1. The statement immediately follows
because for each $u\in \mathcal{U},$ the covariance matrix of $\tilde{\Sigma}%
_{n,\tau _{1},\tau _{2},\bar{\varepsilon}}^{-1/2}(x,u)\mathbf{\xi }_{n,\tau
_{1},\tau _{2},\bar{\varepsilon}}(x,u)$ is equal to the covariance matrix of 
$[\mathbb{W}_{n,\tau _{1},\tau _{2},\bar{\varepsilon}}^{(1)\top }(x,u),%
\mathbb{W}_{n,\tau _{1},\tau _{2},\bar{\varepsilon}}^{(2)\top }(x,u)]^{\top
}and$%
\begin{equation*}
\left\vert w_{\tau _{1},B_{n}}(x)w_{\tau _{2},B_{n}^{\prime }}(x+uh)-w_{\tau
_{1},B_{n}}(x)w_{\tau _{2},B_{n}^{\prime }}(x)\right\vert \rightarrow 0,
\end{equation*}%
as $n\rightarrow \infty $, for each $u\in \mathcal{U}$, and for almost every 
$x\in \mathcal{X}$ (with respect to Lebesgue measure.)\newline
\textbf{Proof of Step 2: }We consider the first statement. First, we write%
\begin{eqnarray}
&&\left\vert \left( \sigma _{n,A,A^{\prime },\bar{\varepsilon}%
}^{R}(B_{n},B_{n}^{\prime })\right) ^{2}-\left( \sigma _{n,A,A^{\prime
}}^{R}(B_{n},B_{n}^{\prime })\right) ^{2}\right\vert  \label{ineq44} \\
&\leq &\int_{\mathcal{T}}\int_{\mathcal{T}}\int_{B_{n}}\int_{\mathcal{U}%
}\left\vert \Delta _{n,\tau _{1},\tau _{2},1}^{\eta }(x,u)\right\vert
w_{\tau _{1},B_{n}}(x)w_{\tau _{2},B_{n}^{\prime }}(x+uh)dudxd\tau _{1}d\tau
_{2}  \notag \\
&&+\int_{\mathcal{T}}\int_{\mathcal{T}}\int_{B_{n}}\int_{\mathcal{U}%
}\left\vert \Delta _{n,\tau _{1},\tau _{2},2}^{\eta }(x,u)\right\vert
w_{\tau _{1},B_{n}}(x)w_{\tau _{2},B_{n}^{\prime }}(x+uh)dudxd\tau _{1}d\tau
_{2},  \notag
\end{eqnarray}%
where%
\begin{eqnarray*}
\Delta _{n,\tau _{1},\tau _{2},1}^{\eta }(x,u) &=&\mathbf{E}\Lambda _{A,p}(%
\sqrt{nh^{d}}\mathbf{z}_{N,\tau _{1}}(x))\mathbf{E}\Lambda _{A^{\prime },p}(%
\sqrt{nh^{d}}\mathbf{z}_{N,\tau _{2}}(x+uh)) \\
&&-\mathbf{E}\Lambda _{A,p}(\sqrt{nh^{d}}\mathbf{z}_{N,\tau _{1}}(x;\eta
_{1}))\mathbf{E}\Lambda _{A^{\prime },p}(\sqrt{nh^{d}}\mathbf{z}_{N,\tau
_{2}}(x+uh;\eta _{2})),
\end{eqnarray*}%
and%
\begin{eqnarray*}
\Delta _{n,\tau _{1},\tau _{2},2}^{\eta }(x,u) &=&\mathbf{E}\Lambda _{A,p}(%
\sqrt{nh^{d}}\mathbf{z}_{N,\tau _{1}}(x))\Lambda _{A^{\prime },p}(\sqrt{%
nh^{d}}\mathbf{z}_{N,\tau _{2}}(x+uh)) \\
&&-\mathbf{E}\Lambda _{A,p}(\sqrt{nh^{d}}\mathbf{z}_{N,\tau _{1}}(x;\eta
_{1}))\Lambda _{A^{\prime },p}(\sqrt{nh^{d}}\mathbf{z}_{N,\tau
_{2}}(x+uh;\eta _{2})).
\end{eqnarray*}%
By H\"{o}lder inequality, for $C>0$ that depends only on $P$,%
\begin{equation*}
\left\vert \Delta _{n,\tau _{1},\tau _{2},2}^{\eta }(x,u)\right\vert \leq
CA_{1n}(x,u)+CA_{2n}(x,u)\text{,}
\end{equation*}%
where, if $p=1$ then we set $s=2$, and $q=1$, and if $p>1$, we set $%
s=(p+1)/(p-1)$ and $q=(1-1/s)^{-1},$%
\begin{eqnarray*}
A_{1n}(x,u) &=&(nh^{d})^{p}\left\{ \mathbf{E}\left\Vert \mathbf{z}_{N,\tau
_{1}}(x)-\mathbf{z}_{N,\tau _{1}}(x;\eta _{1})\right\Vert ^{2q}\right\} ^{%
\frac{1}{2q}} \\
&&\times \left( \left\{ \mathbf{E}\left\Vert \mathbf{z}_{N,\tau
_{1}}(x)\right\Vert ^{2s(p-1)}\right\} ^{\frac{1}{2s}}+\left\{ \mathbf{E}%
\left\Vert \mathbf{z}_{N,\tau _{1}}(x;\eta _{1})\right\Vert
^{2s(p-1)}\right\} ^{\frac{1}{2s}}\right) \\
&&\times \sqrt{\mathbf{E}\left( \left\Vert \mathbf{z}_{N,\tau
_{2}}(x+uh)\right\Vert ^{2p}\right) },
\end{eqnarray*}%
and%
\begin{eqnarray*}
A_{2n}(x,u) &=&(nh^{d})^{p}\left\{ \mathbf{E}\left\Vert \mathbf{z}_{N,\tau
_{2}}(x+uh)-\mathbf{z}_{N,\tau _{2}}(x+uh;\eta _{2})\right\Vert
^{2q}\right\} ^{\frac{1}{2q}} \\
&&\times \left( \left\{ \mathbf{E}\left\Vert \mathbf{z}_{N,\tau
_{2}}(x+uh)\right\Vert ^{2s(p-1)}\right\} ^{\frac{1}{2s}}+\left\{ \mathbf{E}%
\left\Vert \mathbf{z}_{N,\tau _{2}}(x+uh;\eta _{2})\right\Vert
^{2s(p-1)}\right\} ^{\frac{1}{2s}}\right) \\
&&\times \sqrt{\mathbf{E}\left( \left\Vert \mathbf{z}_{N,\tau _{1}}(x;\eta
_{1})\right\Vert ^{2p}\right) }.
\end{eqnarray*}%
Now,%
\begin{equation*}
\sup_{(x,\tau )\in \mathcal{S}}\sup_{P\in \mathcal{P}}\mathbf{E}\left\Vert 
\sqrt{nh^{d}}\{\mathbf{z}_{N,\tau }(x)-\mathbf{z}_{N,\tau }(x;\eta
_{1})\}\right\Vert ^{2q}=\mathbf{E}\left\Vert \sqrt{\bar{\varepsilon}}%
\mathbb{Z}\right\Vert ^{2q}=C\bar{\varepsilon}^{q}\text{,}
\end{equation*}%
where $\mathbb{Z}\in \mathbf{R}^{J}$ is a centered normal random vector with
identity covariance matrix $I_{J}$. Also, we deduce that for some $C>0,$%
\begin{equation*}
\sup_{(x,\tau )\in \mathcal{S}}\sup_{P\in \mathcal{P}}\mathbf{E}\left\Vert 
\sqrt{nh^{d}}\mathbf{z}_{N,\tau }(x)\right\Vert ^{2s(p-1)}\leq C,
\end{equation*}%
by (\ref{ineqs}) of Lemma B5 and by the fact that $2s(p-1)=2(p+1)\leq M$.
Similarly, from some large $n$ on,%
\begin{eqnarray*}
&&\sup_{(x,\tau )\in \mathcal{S}}\sup_{P\in \mathcal{P}}\mathbf{E}\left(
\left\Vert \sqrt{nh^{d}}\mathbf{z}_{N,\tau }(x+uh;\eta _{2})\right\Vert
^{2p}\right) \\
&\leq &\sup_{\tau \in \mathcal{T},x\in \mathcal{S}_{\tau }(\varepsilon
)}\sup_{P\in \mathcal{P}}\mathbf{E}\left( \left\Vert \sqrt{nh^{d}}\mathbf{z}%
_{N,\tau }(x;\eta _{2})\right\Vert ^{2p}\right) <C,
\end{eqnarray*}%
for some $C>0$. Thus we conclude that for some $C>0$,%
\begin{equation*}
\sup_{(\tau _{1},\tau _{2})\in \mathcal{T}\times \mathcal{T}}\sup_{(x,u)\in (%
\mathcal{S}_{\tau _{1}}(\varepsilon )\cup \mathcal{S}_{\tau
_{2}}(\varepsilon ))\times \mathcal{U}}\sup_{P\in \mathcal{P}}\left(
A_{1n}(x,u)+A_{2n}(x,u)\right) \leq C\sqrt{\bar{\varepsilon}},
\end{equation*}%
and that for some $C>0,$%
\begin{equation*}
\sup_{(\tau _{1},\tau _{2})\in \mathcal{T}\times \mathcal{T}}\sup_{(x,u)\in (%
\mathcal{S}_{\tau _{1}}(\varepsilon )\cup \mathcal{S}_{\tau
_{2}}(\varepsilon ))\times \mathcal{U}}\sup_{P\in \mathcal{P}}\left\vert
\Delta _{n,\tau _{1},\tau _{2},2}^{\eta }(x,u)\right\vert \leq C\sqrt{\bar{%
\varepsilon}}.
\end{equation*}%
Using similar arguments, we also find that for some $C>0,$%
\begin{equation*}
\sup_{(\tau _{1},\tau _{2})\in \mathcal{T}\times \mathcal{T}}\sup_{(x,u)\in (%
\mathcal{S}_{\tau _{1}}(\varepsilon )\cup \mathcal{S}_{\tau
_{2}}(\varepsilon ))\times \mathcal{U}}\sup_{P\in \mathcal{P}}\left\vert
\Delta _{n,\tau _{1},\tau _{2},1}^{\eta }(x,u)\right\vert \leq C\sqrt{\bar{%
\varepsilon}}.
\end{equation*}%
Therefore, there exist $C_{1}>0$ and $C_{2}>0$ such that from some large $n$
on,%
\begin{eqnarray*}
&&\sup_{P\in \mathcal{P}}\left\vert \sigma _{n,A,A^{\prime },\bar{\varepsilon%
}}^{2}(B_{n},B_{n}^{\prime })-\sigma _{n,A,A^{\prime
}}^{2}(B_{n},B_{n}^{\prime })\right\vert \\
&\leq &C_{1}\sqrt{\bar{\varepsilon}}\int_{\mathcal{T}}\int_{\mathcal{T}%
}\int_{B_{n}}\int_{\mathcal{U}}w_{\tau _{1},B_{n}}(x)w_{\tau
_{2},B_{n}^{\prime }}(x+uh)dudxd\tau _{1}d\tau _{2}.
\end{eqnarray*}%
Since the last multiple integral is finite, we obtain the first statement of
Step 2.

We turn to the second statement of Step 2. Similarly as before, we write%
\begin{eqnarray*}
&&\left\vert \sigma _{n,A,A^{\prime },\bar{\varepsilon}}^{2}(B_{n},B_{n}^{%
\prime })-\sigma _{n,A,A^{\prime }}^{2}(B_{n},B_{n}^{\prime })\right\vert \\
&\leq &\int_{\mathcal{T}}\int_{\mathcal{T}}\int_{B_{n}}\int_{\mathcal{U}%
}\left\vert \Delta _{1,\tau _{1},\tau _{2}}^{\eta }(x,u)\right\vert w_{\tau
_{1},B_{n}}(x)w_{\tau _{2},B_{n}^{\prime }}(x+uh)dudxd\tau _{1}d\tau _{2} \\
&&+\int_{\mathcal{T}}\int_{\mathcal{T}}\int_{B_{n}}\int_{\mathcal{U}%
}\left\vert \Delta _{2,\tau _{1},\tau _{2}}^{\eta }(x,u)\right\vert w_{\tau
_{1},B_{n}}(x)w_{\tau _{2},B_{n}^{\prime }}(x+uh)dudxd\tau _{1}d\tau _{2},
\end{eqnarray*}%
where%
\begin{eqnarray*}
\Delta _{1,\tau _{1},\tau _{2}}^{\eta }(x,u) &=&\mathbf{E}\Lambda _{A,p}(%
\mathbb{W}_{n,\tau _{1},\tau _{2}}^{(1)}(x,u))\mathbf{E}\Lambda _{A^{\prime
},p}(\mathbb{W}_{n,\tau _{1},\tau _{2}}^{(2)}(x,u)) \\
&&-\mathbf{E}\Lambda _{A,p}(\mathbb{W}_{n,\tau _{1},\tau _{2},\bar{%
\varepsilon}}^{(1)}(x,u))\mathbf{E}\Lambda _{A^{\prime },p}(\mathbb{W}%
_{n,\tau _{1},\tau _{2},\bar{\varepsilon}}^{(2)}(x,u)),
\end{eqnarray*}%
and%
\begin{eqnarray*}
\Delta _{2,\tau _{1},\tau _{2}}^{\eta }(x,u) &=&\mathbf{E}\Lambda _{A,p}(%
\mathbb{W}_{n,\tau _{1},\tau _{2}}^{(1)}(x,u))\Lambda _{A^{\prime },p}(%
\mathbb{W}_{n,\tau _{1},\tau _{2}}^{(2)}(x,u)) \\
&&-\mathbf{E}\Lambda _{A,p}(\mathbb{W}_{n,\tau _{1},\tau _{2},\bar{%
\varepsilon}}^{(1)}(x,u))\Lambda _{A^{\prime },p}(\mathbb{W}_{n,\tau
_{1},\tau _{2},\bar{\varepsilon}}^{(2)}(x,u)).
\end{eqnarray*}%
Now, observe that for $C>0$ that does not depend on $\bar{\varepsilon},$ we
have by Lemma B1(i),%
\begin{equation*}
\sup_{(x,u)\in (\mathcal{S}_{\tau _{1}}(\varepsilon )\cup \mathcal{S}_{\tau
_{2}}(\varepsilon ))\times \mathcal{U}}\sup_{P\in \mathcal{P}}\left\Vert 
\tilde{\Sigma}_{n,\tau _{1},\tau _{2},\bar{\varepsilon}}^{1/2}(x,u)-\left[ 
\begin{array}{c}
\Sigma _{n,\tau _{1}}(x,0) \\ 
\Sigma _{n,\tau _{1},\tau _{2}}(x,u)%
\end{array}%
\begin{array}{c}
\Sigma _{n,\tau _{1},\tau _{2}}(x,u) \\ 
\Sigma _{n,\tau _{2}}(x+uh)%
\end{array}%
\right] ^{1/2}\right\Vert \leq C\sqrt{\bar{\varepsilon}}\text{.}
\end{equation*}%
Using this, recalling the definitions of $\mathbb{W}_{n,\tau _{1},\tau
_{2}}^{(1)}(x,u)$ and $\mathbb{W}_{n,\tau _{1},\tau _{2}}^{(2)}(x,u)$ in (%
\ref{Ws}), and following the previous arguments, we obtain the second
statement of Step 2.
\end{proof}

\begin{LemmaC}
\label{lem-c7} \textit{Suppose that for some small }$\nu _{1}>0$, $%
n^{-1/2}h^{-d-\nu _{1}}\rightarrow 0$\textit{, as }$n\rightarrow \infty $%
\textit{\ and the conditions of Lemma B6 hold. Then there exists} $C>0$%
\textit{\ such that for any sequence of Borel sets }$B_{n}\subset \mathcal{S}%
,$ \textit{and }$A\subset \mathbb{N}_{J},$ \textit{from some large }$n$%
\textit{\ on,}%
\begin{eqnarray*}
&&\sup_{P\in \mathcal{P}}\mathbf{E}\left[ \left\vert
h^{-d/2}\int_{B_{n}}\left\{ \Lambda _{A,p}(\sqrt{nh^{d}}\mathbf{z}_{n,\tau
}(x))-\mathbf{E}\left[ \Lambda _{A,p}(\sqrt{nh^{d}}\mathbf{z}_{N,\tau }(x))%
\right] \right\} dQ(x,\tau )\right\vert \right] \\
&\leq &C\sqrt{Q(B_{n})}.
\end{eqnarray*}
\end{LemmaC}

\begin{rem}
The result is in the same spirit as Lemma 6.2 of Gin\'{e}, Mason, and
Zaitsev (2003). (Also see Lemma A8 of Lee, Song and Whang (2013).) However,
unlike these results, the location normalization here involves\ $\mathbf{E}%
[\Lambda _{A,p}(\sqrt{nh^{d}}\mathbf{z}_{N,\tau }(x))]$\ instead of $\mathbf{%
E}[\Lambda _{A,p}(\sqrt{nh^{d}}\mathbf{z}_{n,\tau }(x))]$. We can obtain the
same result with $\mathbf{E}[\Lambda _{A,p}(\sqrt{nh^{d}}\mathbf{z}_{N,\tau
}(x))]$ replaced by $\mathbf{E}[\Lambda _{A,p}(\sqrt{nh^{d}}\mathbf{z}%
_{n,\tau }(x))]$, but with a stronger bandwidth condition.
\end{rem}

Like Lemma B6, the result of Lemma B7 does not require that the quantities $%
\sqrt{nh^{d}}\mathbf{z}_{n,\tau }(x)$ and $\sqrt{nh^{d}}\mathbf{z}_{N,\tau
}(x)$ have a (pointwise in $x$) nondegenerate limit distribution.

\begin{proof}[Proof of Lemma B\protect\ref{lem-c7}]
As in the proof of Lemma A8 of \citeasnoun{LSW}, it suffices to show that
there exists $C>0$ such that $C$ does not depend on $n$ and for any Borel
set $B\subset \mathbf{R}$,\newline
\textbf{Step 1:}%
\begin{equation*}
\sup_{P\in \mathcal{P}}\mathbf{E}\left[ \left\vert
h^{-d/2}\int_{B_{n}}\left\{ \Lambda _{A,p}(\sqrt{nh^{d}}\mathbf{z}_{n,\tau
}(x))-\Lambda _{A,p}(\sqrt{nh^{d}}\mathbf{z}_{N,\tau }(x))\right\} dQ(x,\tau
)\right\vert \right] \leq CQ(B_{n}),\text{ and}
\end{equation*}%
\textbf{Step 2:}%
\begin{equation*}
\sup_{P\in \mathcal{P}}\mathbf{E}\left[ \left\vert
h^{-d/2}\int_{B_{n}}\left\{ \Lambda _{A,p}(\sqrt{nh^{d}}\mathbf{z}_{N,\tau
}(x))-\mathbf{E}\left[ \Lambda _{A,p}(\sqrt{nh^{d}}\mathbf{z}_{N,\tau }(x))%
\right] \right\} dQ(x,\tau )\right\vert \right] \leq C\sqrt{Q(B_{n})}.
\end{equation*}

By chaining Steps 1 and 2, we obtain the desired result.\newline
\textbf{Proof of Step 1: }Similarly as in (2.13) of \citeasnoun{Horvath:91},
we first write%
\begin{equation}
\mathbf{z}_{n,\tau }(x)=\mathbf{z}_{N,\tau }(x)+\mathbf{v}_{n,\tau }(x)+%
\mathbf{s}_{n,\tau }(x),  \label{dec55}
\end{equation}%
where, for $\beta _{n,x,\tau }(Y_{i},(X_{i}-x)/h)$ defined prior to Lemma B5,%
\begin{eqnarray*}
\mathbf{v}_{n,\tau }(x) &\equiv &\left( \frac{n-N}{n}\right) \cdot \frac{1}{%
h^{d}}\mathbf{E}\left[ \beta _{n,x,\tau }\left( Y_{i},\frac{X_{i}-x}{h}%
\right) \right] \text{ and} \\
\mathbf{s}_{n,\tau }(x) &\equiv &\frac{1}{nh^{d}}\sum_{i=N+1}^{n}\left\{
\beta _{n,x,\tau }\left( Y_{i},\frac{X_{i}-x}{h}\right) -\mathbf{E}\left[
\beta _{n,x,\tau }\left( Y_{i},\frac{X_{i}-x}{h}\right) \right] \right\} ,
\end{eqnarray*}%
and we write $N=n,\ \sum_{i=N+1}^{n}=0,$ and if $N>n$, $\sum_{i=N+1}^{n}=-%
\sum_{i=n+1}^{N}$.

Using (\ref{dec55}), we deduce that for some $C_{1},C_{2}>0$ that depend
only on $P$,%
\begin{eqnarray}
&&\int_{B_{n}}\left\vert \Lambda _{A,p}\left( \mathbf{z}_{n,\tau }(x)\right)
-\Lambda _{A,p}\left( \mathbf{z}_{N,\tau }(x)\right) \right\vert dQ(x,\tau )
\label{ineq3} \\
&\leq &C_{1}\int_{B_{n}}\left\Vert \mathbf{v}_{n,\tau }(x)\right\Vert \left(
\left\Vert \mathbf{z}_{n,\tau }(x)\right\Vert ^{p-1}+\left\Vert \mathbf{z}%
_{N,\tau }(x)\right\Vert ^{p-1}\right) dQ(x,\tau )  \notag \\
&&+C_{2}\int_{B_{n}}\left\Vert \mathbf{s}_{n,\tau }(x)\right\Vert \left(
\left\Vert \mathbf{z}_{n,\tau }(x)\right\Vert ^{p-1}+\left\Vert \mathbf{z}%
_{N,\tau }(x)\right\Vert ^{p-1}\right) dQ(x,\tau )  \notag \\
&\equiv &D_{1n}+D_{2n}\text{, say.}  \notag
\end{eqnarray}%
To deal with $D_{1n}$ and $D_{2n}$, we first show the following:\medskip

\noindent \textsc{Claim 1:} $\sup_{(x,\tau )\in \mathcal{S}}\sup_{P\in 
\mathcal{P}}\mathbf{E}[||\mathbf{v}_{n,\tau }(x)||^{2}]=O(n^{-1})$,
and\medskip

\noindent \textsc{Claim 2:} $\sup_{(x,\tau )\in \mathcal{S}}\sup_{P\in 
\mathcal{P}}\mathbf{E}[||\mathbf{s}_{n,\tau }(x)||^{2}]=O(n^{-3/2}h^{-d})$%
.\medskip

\noindent \textsc{Proof of Claim 1:} First, note that%
\begin{equation*}
\sup_{(x,\tau )\in \mathcal{S}}\mathbf{E}\left[ ||\mathbf{v}_{n,\tau
}(x)||^{2}\right] \leq \mathbf{E}\left\vert \frac{n-N}{n}\right\vert
^{2}\cdot \sup_{(x,\tau )\in \mathcal{S}}\left\Vert \frac{1}{h^{d}}\mathbf{E}%
\left[ \beta _{n,x,\tau }\left( Y_{i},\frac{X_{i}-x}{h}\right) \right]
\right\Vert ^{2}.
\end{equation*}

Since $\mathbf{E}|n^{-1/2}(n-N)|^{2}$ does not depend on the joint
distribution of $(Y_{i},X_{i})$, $\mathbf{E}|n^{-1/2}(n-N)|^{2}\leq O(1)$
uniformly over $P\in \mathcal{P}$. Combining this with the second statement
of (\ref{ineqs}), the product on the right hand side becomes $O(n^{-1})$
uniformly over $P\in \mathcal{P}$.

\noindent \textsc{Proof of Claim 2:} Let $\eta _{1}\in \mathbf{R}^{J}$ be
the random vector defined prior to Lemma B6, and define%
\begin{equation*}
\mathbf{s}_{n,\tau }(x;\eta _{1})\equiv \mathbf{s}_{n,\tau }(x)+\frac{%
(N-n)\eta _{1}}{n^{3/2}h^{d/2}}.
\end{equation*}%
Note that%
\begin{equation}
\mathbf{E}\left\Vert \mathbf{s}_{n,\tau }(x)\right\Vert ^{2}\leq 2\mathbf{E}%
\left\Vert \mathbf{s}_{n,\tau }(x;\eta _{1})\right\Vert ^{2}+\frac{2}{%
n^{2}h^{d}}\mathbf{E}\left\Vert \frac{(N-n)\eta _{1}}{\sqrt{n}}\right\Vert
^{2}.  \label{dec22}
\end{equation}%
As for the last term, since $N$ and $\eta _{1}$ are independent, it is
bounded by%
\begin{equation*}
\frac{1}{n^{2}h^{d}}\left( \mathbf{E}\left\vert \frac{N-n}{\sqrt{n}}%
\right\vert ^{2}\right) \cdot \mathbf{E}\left\Vert \eta _{1}\right\Vert
^{2}\leq \frac{C\bar{\varepsilon}}{n^{2}h^{d}}=O(n^{-2}h^{-d-\nu _{1}}),
\end{equation*}%
from some large $n$ on.

As for the leading expectation on the right hand side of (\ref{dec22}), we
write%
\begin{eqnarray*}
\mathbf{E}\left\Vert \sqrt{nh^{d}}\mathbf{s}_{n,\tau }(x;\eta
_{1})\right\Vert ^{2} &=&\mathbf{E}\left\Vert \frac{1}{\sqrt{n}}%
\sum_{i=N+1}^{n}q_{n,\tau ,1}^{(i)}(x)\right\Vert ^{2} \\
&=&\frac{1}{n}\sum_{j=1}^{J}\bar{\sigma}_{n,\tau ,j}^{2}(x)\mathbf{E}\left(
\sum_{i=N+1}^{n}\frac{q_{n,\tau ,1,j}^{(i)}(x)}{\bar{\sigma}_{n,\tau ,j}(x)}%
\right) ^{2},
\end{eqnarray*}%
where $q_{n,\tau ,1}^{(i)}(x)$'s ($i=1,2,\cdot \cdot \cdot \ $) are i.i.d.
copies of $q_{n,\tau }(x)+\eta _{1}$ and $q_{n,\tau ,1,j}^{(i)}(x)$ is the $%
j $-th entry of $q_{n,\tau ,1}^{(i)}(x),$ and $\bar{\sigma}_{n,\tau
,j}^{2}(x)\equiv Var(q_{n,\tau ,1,j}^{(i)}(x)).$ Recall that $q_{n,\tau }(x)$
was defined prior to Lemma B5. Now we apply Lemma 1(i) of %
\citeasnoun{Horvath:91} to deduce that%
\begin{eqnarray*}
&&\sup_{(x,\tau )\in \mathcal{S}}\sup_{P\in \mathcal{P}}\mathbf{E}\left(
\sum_{i=N+1}^{n}\frac{q_{n,\tau ,1,j}^{(i)}(x)}{\bar{\sigma}_{n,\tau ,j}(x)}%
\right) ^{2} \\
&\leq &\mathbf{E}|N-n|\cdot \mathbf{E}|\mathbb{Z}_{1}|^{2}+C\mathbf{E}%
\left\vert N-n\right\vert ^{1/2}\cdot \sup_{(x,\tau )\in \mathcal{S}%
}\sup_{P\in \mathcal{P}}\mathbf{E}\left\vert \frac{q_{n,\tau ,1,j}^{(i)}(x)}{%
\bar{\sigma}_{n,\tau ,j}(x)}\right\vert ^{3} \\
&&+C\sup_{(x,\tau )\in \mathcal{S}}\sup_{P\in \mathcal{P}}\mathbf{E}%
\left\vert \frac{q_{n,\tau ,1,j}^{(i)}(x)}{\bar{\sigma}_{n,\tau ,j}(x)}%
\right\vert ^{4},
\end{eqnarray*}%
for some $C>0$, where $\mathbb{Z}_{1}\sim N(0,1)$.

First, observe that $\sup_{(x,\tau )\in \mathcal{S}}\sup_{P\in \mathcal{P}}%
\bar{\sigma}_{n,\tau ,j}(x)<\infty $ by Lemma B5, and 
\begin{equation}
\inf_{(x,\tau )\in \mathcal{S}}\inf_{P\in \mathcal{P}}\bar{\sigma}_{n,\tau
,j}(x)>\bar{\varepsilon}>0,  \label{lb}
\end{equation}%
due to the additive term $\eta _{1}$ in $q_{n,\tau }(x)+\eta _{1}$. Let $%
\eta _{1j}$ be the $j$-th entry of $\eta _{1}$. We apply Lemma B5 to deduce
that for some $C>0,$ from some large $n$ on,%
\begin{eqnarray}
\sup_{(x,\tau )\in \mathcal{S}}\sup_{P\in \mathcal{P}}\mathbf{E}|(q_{n,\tau
,j}(x)+\eta _{1j})/\bar{\sigma}_{n,\tau ,j}(x)|^{3} &\leq &Ch^{-(d/2)-(\nu
_{1}/2)}\text{ and}  \label{arg1} \\
\sup_{(x,\tau )\in \mathcal{S}}\sup_{P\in \mathcal{P}}\mathbf{E}|(q_{n,\tau
,j}(x)+\eta _{1j})/\bar{\sigma}_{n,\tau ,j}(x)|^{4} &\leq &Ch^{-d-\nu _{1}}.
\notag
\end{eqnarray}%
Since $\mathbf{E}|N-n|=O(n^{1/2})$ and $\mathbf{E}\left\vert N-n\right\vert
^{1/2}=O(n^{1/4})$ (e.g. (2.21) and (2.22) of \citeasnoun{Horvath:91}),
there exists $C>0$ such that%
\begin{equation}
\sup_{(x,\tau )\in \mathcal{S}}\sup_{P\in \mathcal{P}}\mathbf{E}\left(
\sum_{i=N+1}^{n}\frac{q_{n,\tau ,1,j}^{(i)}(x)}{\bar{\sigma}_{n,\tau ,j}(x)}%
\right) ^{2}\leq \frac{C}{\bar{\varepsilon}^{4}}\left\{
n^{1/2}+n^{1/4}h^{-(d/2)-(\nu _{1}/2)}+h^{-d-\nu _{1}}\right\} .
\label{arg2}
\end{equation}%
This implies that for some $C>0,$ (with $\bar{\varepsilon}>0$ fixed while $%
n\rightarrow \infty $)%
\begin{eqnarray}
&&\sup_{(x,\tau )\in \mathcal{S}}\sup_{P\in \mathcal{P}}\mathbf{E}\left\Vert 
\sqrt{nh^{d}}\mathbf{s}_{n,\tau }(x)\right\Vert ^{2}  \label{arg3} \\
&\leq &O\left( n^{-1}h^{-\nu _{1}}\right) +O\left(
n^{-1/2}+n^{-3/4}h^{-(d/2)-(\nu _{1}/2)}+n^{-1}h^{-d-\nu _{1}}\right)  \notag
\\
&=&O\left( n^{-1}h^{-\nu _{1}}\right) +O(n^{-1/2})=O(n^{-1/2}),  \notag
\end{eqnarray}%
since $n^{-1/2}h^{-d-\nu _{1}}\rightarrow 0$. Hence, we obtain Claim 2.

Using Claim 1 and the second statement of Lemma B5, we deduce that%
\begin{eqnarray*}
\sup_{P\in \mathcal{P}}\mathbf{E}\left[ n^{p/2}h^{d(p-1)/2}D_{1n}\right]
&\leq &C_{1}Q(B_{n})\sup_{(x,\tau )\in \mathcal{S}}\sup_{P\in \mathcal{P}}%
\sqrt{\mathbf{E}\left\Vert \sqrt{n}\mathbf{v}_{n,\tau }(x)\right\Vert ^{2}}
\\
&&\times \sqrt{\mathbf{E}\left\Vert \sqrt{nh^{d}}\mathbf{z}_{n,\tau
}(x)\right\Vert ^{2p-2}+\mathbf{E}\left\Vert \sqrt{nh^{d}}\mathbf{z}_{N,\tau
}(x)\right\Vert ^{2p-2}} \\
&\leq &C_{2}Q(B_{n}),
\end{eqnarray*}%
for $C_{1},C_{2}>0$. Similarly, we can see that 
\begin{equation*}
\sup_{P\in \mathcal{P}}\mathbf{E}\left[ n^{p/2}h^{d(p-1)/2}D_{2n}\right]
=O(n^{-1/2}h^{-d})=o(1),
\end{equation*}%
using Claim 2 and the second statement of Lemma B5. Thus, we obtain Step 1.%
\newline
\textbf{Proof of Step 2:} We can follow the proof of Lemma B6 to show that%
\begin{eqnarray*}
&&\mathbf{E}\left[ h^{-d/2}\int_{B_{n}}\left( \Lambda _{A,p}(\sqrt{nh^{d}}%
\mathbf{z}_{N,\tau }(x))-\mathbf{E}\left[ \Lambda _{A,p}(\sqrt{nh^{d}}%
\mathbf{z}_{N,\tau }(x))\right] \right) dQ(x,\tau )\right] ^{2} \\
&=&\int_{\mathcal{T}}\int_{\mathcal{T}}\int_{B_{n,\tau _{1}}\cap B_{n,\tau
_{2}}}\int_{\mathcal{U}}C_{n,\tau _{1},\tau _{2},A,A^{\prime
}}(x,u)dudxd\tau _{1}d\tau _{2}+o(1),
\end{eqnarray*}%
where $C_{n,\tau _{1},\tau _{2},A,A^{\prime }}(x,u)$ is defined in (\ref{C})
and $o(1)$ is uniform over $P\in \mathcal{P}$. Now, observe that%
\begin{eqnarray*}
&&\sup_{(\tau _{1},\tau _{2})\in \mathcal{T}\times \mathcal{T}}\sup_{u\in 
\mathcal{U}}\sup_{x\in \mathcal{X}}\sup_{P\in \mathcal{P}}\left\vert
C_{n,\tau _{1},\tau _{2},A,A^{\prime }}(x,u)\right\vert \\
&\leq &\sup_{(\tau _{1},\tau _{2})\in \mathcal{T}\times \mathcal{T}%
}\sup_{u\in \mathcal{U}}\sup_{x\in \mathcal{X}}\sup_{P\in \mathcal{P}}\sqrt{%
\mathbf{E||}\mathbb{W}_{n,\tau _{1},\tau _{2}}^{(1)}(x,u)||^{2p}\mathbf{E||}%
\mathbb{W}_{n,\tau _{1},\tau _{2}}^{(2)}(x,u)||^{2p}}<\infty .
\end{eqnarray*}%
Therefore,%
\begin{eqnarray*}
&&\mathbf{E}\left[ \left\vert h^{-d/2}\int_{B_{n}}\left( \Lambda _{A,p}(%
\sqrt{nh^{d}}\mathbf{z}_{N,\tau }(x))-\mathbf{E}\left[ \Lambda _{A,p}(\sqrt{%
nh^{d}}\mathbf{z}_{N,\tau }(x))\right] \right) dQ(x,\tau )\right\vert \right]
\\
&\leq &\sqrt{\int_{\mathcal{T}}\int_{\mathcal{T}}\int_{\mathcal{U}%
}\int_{B_{n,\tau _{1}}\cap B_{n,\tau _{2}}}C_{n,\tau _{1},\tau
_{2},A,A^{\prime }}(x,u)dxdud\tau _{1}d\tau _{2}}+o(1) \\
&\leq &C\sqrt{\int_{\mathcal{T}}\int_{\mathcal{T}}\int_{\mathcal{U}%
}\int_{B_{n,\tau _{1}}\cap B_{n,\tau _{2}}}dxdud\tau _{1}d\tau _{2}}+o(1),
\end{eqnarray*}%
for some $C>0.$ Now, observe that%
\begin{equation*}
\int_{\mathcal{T}}\int_{\mathcal{T}}\int_{B_{n,\tau _{1}}\cap B_{n,\tau
_{2}}}dxd\tau _{1}d\tau _{2}\leq \int_{\mathcal{T}}d\tau _{2}\cdot \left(
\int_{\mathcal{T}}\int_{B_{n,\tau _{1}}}dxd\tau _{1}\right) \leq CQ(B_{n}),
\end{equation*}%
because $\mathcal{T}$ is a bounded set. Thus the proof of Step 2 is
completed.
\end{proof}

The next lemma shows the joint asymptotic normality of a Poissonized version
of a normalized test statistic and a Poisson random variable. Using this
result, we can apply the de-Poissonization lemma in Lemma B3. To define a
Poissonized version of a normalized test statistic, we introduce some
notation.

Let $\mathcal{C}\subset \mathbf{R}^{d}$ be a compact set such that $\mathcal{%
C}$ does not depend on $P\in \mathcal{P}$ and $\alpha _{P}\equiv P\{X\in 
\mathbf{R}^{d}\backslash \mathcal{C}\}$ satisfies that $0<\inf_{P\in 
\mathcal{P}}\alpha _{P}\leq \sup_{P\in \mathcal{P}}\alpha _{P}<1$. Existence
of such $\mathcal{C}$ is assumed in Assumption A6(ii). For $c_{n}\rightarrow
\infty $, we let $B_{n,A}(c_{n};\mathcal{C})\equiv B_{n,A}(c_{n})\cap (%
\mathcal{C}\times \mathcal{T})$, where we recall the definition of $%
B_{n,A}(c_{n})=B_{n,A}(c_{n},c_{n})$. (See the definition of $%
B_{n,A}(c_{n,1},c_{n,2})$ before Lemma 1.) Define%
\begin{eqnarray*}
\zeta _{n,A} &\equiv &\int_{B_{n,A}(c_{n};\mathcal{C})}\Lambda _{A,p}(\sqrt{%
nh^{d}}\mathbf{z}_{n,\tau }(x))dQ(x,\tau ),\text{ and} \\
\zeta _{N,A} &\equiv &\int_{B_{n,A}(c_{n};\mathcal{C})}\Lambda _{A,p}(\sqrt{%
nh^{d}}\mathbf{z}_{N,\tau }(x))dQ(x,\tau ).
\end{eqnarray*}%
Let $\mu _{A}$'s be real numbers indexed by $A\in \mathcal{N}_{J}$, and
define%
\begin{equation*}
\sigma _{n}^{2}(\mathcal{C})\equiv \sum_{A\in \mathcal{N}_{J}}\sum_{A^{%
\prime }\in \mathcal{N}_{J}}\mu _{A}\mu _{A^{\prime }}\sigma _{n,A,A^{\prime
}}(B_{n,A}(c_{n};\mathcal{C}),B_{n,A^{\prime }}(c_{n};\mathcal{C}))\text{%
\textit{,}}
\end{equation*}%
\textit{\ }where we recall the definition of $\sigma _{n,A,A^{\prime
}}(\cdot ,\cdot )$ prior to Lemma B6. Define%
\begin{equation*}
S_{n}\equiv h^{-d/2}\sum_{A\in \mathcal{N}_{J}}\mu _{A}\left\{ \zeta _{N,A}-%
\mathbf{E}\zeta _{N,A}\right\} .
\end{equation*}%
Also define%
\begin{eqnarray*}
U_{n} &\equiv &\frac{1}{\sqrt{n}}\left\{ \sum_{i=1}^{N}1\{X_{i}\in \mathcal{C%
}\}-nP\left\{ X_{i}\in \mathcal{C}\right\} \right\} ,\text{ and} \\
V_{n} &\equiv &\frac{1}{\sqrt{n}}\left\{ \sum_{i=1}^{N}1\{X_{i}\in \mathbf{R}%
^{d}\backslash \mathcal{C}\}-nP\left\{ X_{i}\in \mathbf{R}^{d}\backslash 
\mathcal{C}\right\} \right\} .
\end{eqnarray*}%
Let 
\begin{equation*}
H_{n}\equiv \left[ \frac{S_{n}}{\sigma _{n}(\mathcal{C})},\frac{U_{n}}{\sqrt{%
1-\alpha _{P}}}\right] ^{\top }.
\end{equation*}%
The following lemma establishes the joint convergence of $H_{n}$. In doing
so, we need to be careful in dealing with uniformity in $P\in \mathcal{P}$,
and potential degeneracy of the normalized test statistic $S_{n}$.

\begin{LemmaC}
\label{lem-c8} \textit{Suppose that the conditions of Lemma B6 hold and that 
}$c_{n}\rightarrow \infty $ \textit{as }$n\rightarrow \infty .$

\noindent (i) If $\liminf_{n\rightarrow \infty }\inf_{P\in \mathcal{P}%
}\sigma _{n}^{2}(\mathcal{C})>0$, \textit{then}%
\begin{equation*}
\sup_{P\in \mathcal{P}}\sup_{t\in \mathbf{R}^{2}}\left\vert P\left\{
H_{n}\leq t\right\} -P\left\{ \mathbb{Z}\leq t\right\} \right\vert
\rightarrow 0,
\end{equation*}%
\textit{where }$\mathbb{Z}\sim N(0,I_{2})$.

\noindent (ii) If $\limsup_{n\rightarrow \infty }\sigma _{n}^{2}(\mathcal{C}%
)=0$,\textit{\ then for each }$(t_{1},t_{2})\in \mathbf{R}^{2},$%
\begin{equation*}
\left\vert P\left\{ S_{n}\leq t_{1}\text{ and }\frac{U_{n}}{\sqrt{1-\alpha
_{P}}}\leq t_{2}\right\} -1\{0\leq t_{1}\}P\left\{ \mathbb{Z}_{1}\leq
t_{2}\right\} \right\vert \rightarrow 0,
\end{equation*}%
\textit{where }$\mathbb{Z}_{1}\sim N(0,1)$.
\end{LemmaC}

\begin{rem}
The joint convergence result is divided into two separate results. The first
case is a situation where $S_{n}$ is asymptotically nondegenerate uniformly
in $P\in \mathcal{P}$. The second case deals with a situation where $S_{n}$
is asymptotically degenerate for some $P\in \mathcal{P}$.
\end{rem}

\begin{proof}[Proof of Lemma B\protect\ref{lem-c8}]
(i) Define $\bar{\varepsilon}>0$ and let%
\begin{equation*}
H_{n,\bar{\varepsilon}}\equiv \left[ \frac{S_{n,\bar{\varepsilon}}}{\sigma
_{n,\bar{\varepsilon}}(\mathcal{C})},\frac{U_{n}}{\sqrt{1-\alpha _{P}}}%
\right] ^{\top },
\end{equation*}%
where $S_{n,\bar{\varepsilon}}$ is equal to $S_{n}$, except that $\zeta
_{N,A}$ is replaced by%
\begin{equation*}
\zeta _{N,A,\bar{\varepsilon}}\equiv \int_{B_{n,A}(c_{n};\mathcal{C}%
)}\Lambda _{A,p}(\sqrt{nh^{d}}\mathbf{z}_{N,\tau }(x;\eta _{1}))dQ(x,\tau ),
\end{equation*}%
and $\mathbf{z}_{N,\tau }(x;\eta _{1})$ is as defined prior to Lemma B6, and 
$\sigma _{n,\bar{\varepsilon}}(\mathcal{C})$ is $\sigma _{n}(\mathcal{C})$
except that $\tilde{\Sigma}_{n,\tau _{1},\tau _{2}}(x,u)$ is replaced by $%
\tilde{\Sigma}_{n,\tau _{1},\tau _{2},\bar{\varepsilon}}(x,u)$. Also let 
\begin{equation*}
C_{n}\equiv \mathbf{E}H_{n}H_{n}^{\top }\text{ and }C_{n,\bar{\varepsilon}%
}\equiv \mathbf{E}H_{n,\bar{\varepsilon}}H_{n,\bar{\varepsilon}}^{\top }%
\text{.}
\end{equation*}%
First, we show the following statements.\medskip \newline
\textbf{Step 1:} For some $C>0$, $\sup_{P\in \mathcal{P}}|Cov(S_{n,\bar{%
\varepsilon}}-S_{n},U_{n})|\leq C\sqrt{\bar{\varepsilon}}$, for each fixed $%
\bar{\varepsilon}>0$.\medskip \newline
\textbf{Step 2: }$\sup_{P\in \mathcal{P}}\left\vert Cov(S_{n,\bar{\varepsilon%
}},U_{n})\right\vert =o(h^{d/2}),$ as $n\rightarrow \infty .$\medskip 
\newline
\textbf{Step 3: }There exists $c>0$ such that from some large $n$ on,%
\begin{equation*}
\inf_{P\in \mathcal{P}}\lambda _{\min }(C_{n})>c.
\end{equation*}%
\textbf{Step 4: }As $n\rightarrow \infty ,$%
\begin{equation*}
\sup_{P\in \mathcal{P}}\sup_{t\in \mathbf{R}^{2}}\left\vert P\left\{
C_{n}^{-1/2}H_{n}\leq t\right\} \rightarrow P\left\{ \mathbb{Z}\leq
t\right\} \right\vert \rightarrow 0\text{.}
\end{equation*}%
\newline
From Steps 1-3, we find that $\sup_{P\in \mathcal{P}}\left\Vert
C_{n}-I_{2}\right\Vert \rightarrow 0$, as $n\rightarrow \infty $ and as $%
\bar{\varepsilon}\rightarrow 0$. By Step 4, we obtain (i) of Lemma
B8.\medskip \newline
\textbf{Proof of Step 1: }Observe that from an inequality similar to (\ref%
{ineq3}) in the proof of Lemma B7, 
\begin{equation*}
\left\vert \zeta _{N,A,\bar{\varepsilon}}-\zeta _{N,A}\right\vert \leq
C||\eta _{1}||\int_{B_{n,A}(c_{n};\mathcal{C})}\left\Vert \sqrt{nh^{d}}%
\mathbf{z}_{N,\tau }(x)\right\Vert ^{p-1}dQ(x,\tau ).
\end{equation*}%
Using the fact that $\mathcal{S}$ is compact and does not depend on $P\in 
\mathcal{P}$, for some constants $C_{1},C_{2},C_{3}>0$ that do not depend on 
$P\in \mathcal{P}$, 
\begin{eqnarray*}
\mathbf{E}\left\vert \zeta _{N,A,\bar{\varepsilon}}-\zeta _{N,A}\right\vert
^{2} &\leq &C_{1}\mathbf{E}\left[ ||\eta _{1}||^{2}\right] \cdot
\int_{B_{n,A}(c_{n};\mathcal{C})}\mathbf{E}\left\Vert \sqrt{nh^{d}}\mathbf{z}%
_{N,\tau }(x)\right\Vert ^{2p-2}dQ(x,\tau ) \\
&\leq &C_{2}\bar{\varepsilon}\cdot \int_{B_{n,A}(c_{n};\mathcal{C})}\mathbf{E%
}\left\Vert \sqrt{nh^{d}}\mathbf{z}_{N,\tau }(x)\right\Vert ^{2p-2}dQ(x,\tau
)\leq C_{3}\bar{\varepsilon},
\end{eqnarray*}%
by the independence between $\eta _{1}$ and $\{\mathbf{z}_{N,\tau
}(x):(x,\tau )\in \mathcal{S}\}$, and by the second statement of Lemma B5.
From the fact that 
\begin{equation*}
\sup_{P\in \mathcal{P}}\mathbf{E}U_{n}^{2}\leq \sup_{P\in \mathcal{P}%
}(1-\alpha _{P})\leq 1,
\end{equation*}%
we obtain the desired result.\medskip \newline
\textbf{Proof of Step 2: }Let $\Sigma _{2n,\tau ,\bar{\varepsilon}}$ be the
covariance matrix of $[(q_{n,\tau }(x)+\eta _{1})^{\top },\tilde{U}%
_{n}]^{\top }$, where $\tilde{U}_{n}=U_{n}/\sqrt{P\{X\in \mathcal{C}\}}.$ We
can write $\Sigma _{2n,\tau ,\bar{\varepsilon}}$ as%
\begin{eqnarray*}
&&\left[ 
\begin{array}{c}
\Sigma _{n,\tau ,\tau }(x,0)+\bar{\varepsilon}I_{J} \\ 
\mathbf{E}[(q_{n,\tau }(x)+\eta _{1})^{\top }\tilde{U}_{n}]%
\end{array}%
\begin{array}{c}
\mathbf{E}[(q_{n,\tau }(x)+\eta _{1})\tilde{U}_{n}] \\ 
1%
\end{array}%
\right] \\
&=&\left[ 
\begin{array}{c}
\Sigma _{n,\tau ,\tau }(x,0) \\ 
\sqrt{1-\bar{\varepsilon}}\mathbf{E}[q_{n,\tau }^{\top }(x)\tilde{U}_{n}]%
\end{array}%
\begin{array}{c}
\sqrt{1-\bar{\varepsilon}}\mathbf{E}[q_{n,\tau }(x)\tilde{U}_{n}] \\ 
1-\bar{\varepsilon}%
\end{array}%
\right] +\left[ 
\begin{array}{c}
\bar{\varepsilon}I_{J} \\ 
\mathbf{0}^{\top }%
\end{array}%
\begin{array}{c}
\mathbf{0} \\ 
\bar{\varepsilon}%
\end{array}%
\right] +A_{n,\tau }(x),
\end{eqnarray*}%
where%
\begin{equation*}
A_{n,\tau }(x)\equiv \left[ 
\begin{array}{c}
\mathbf{0} \\ 
\left( 1-\sqrt{1-\bar{\varepsilon}}\right) \mathbf{E}[q_{n,\tau }^{\top }(x)%
\tilde{U}_{n}]%
\end{array}%
\begin{array}{c}
\left( 1-\sqrt{1-\bar{\varepsilon}}\right) \mathbf{E}[q_{n,\tau }(x)\tilde{U}%
_{n}] \\ 
\mathbf{0}%
\end{array}%
\right] .
\end{equation*}%
The first matrix on the right hand side is certainly positive semidefinite.
Note that 
\begin{equation*}
\left( q_{n,\tau ,j}(x),\tilde{U}_{n}\right) \overset{d}{=}\left( \frac{1}{%
\sqrt{n}}\sum_{k=1}^{n}q_{n,\tau ,j}^{(k)}(x),\frac{1}{\sqrt{n}}%
\sum_{k=1}^{n}\tilde{U}_{n}^{(k)}\right) ,
\end{equation*}%
where $(q_{n,\tau ,j}^{(k)}(x),\tilde{U}_{n}^{(k)})$'s with $k=1,\cdot \cdot
\cdot ,n$ are i.i.d. copies of $(q_{n,\tau ,j}(x),\bar{U}_{n})$, where 
\begin{equation*}
\bar{U}_{n}\equiv \frac{1}{\sqrt{P\{X\in \mathcal{C}\}}}\left\{ \sum_{1\leq
i\leq N_{1}}1\{X_{i}\in \mathcal{C}\}-P\left\{ X_{i}\in \mathcal{C}\right\}
\right\} ,
\end{equation*}%
where $N_{1}$ is the Poisson random variable with mean $1$ that is involved
in the definition of $q_{n,\tau ,j}(x)$. Hence as for $A_{n,\tau }(x)$, note
that for $C_{1},C_{2}>0,$%
\begin{eqnarray}
\sup_{(x,\tau )\in \mathcal{S}}\sup_{P\in \mathcal{P}}\left\vert \mathbf{E}%
\left[ q_{n,\tau ,j}(x)\tilde{U}_{n}\right] \right\vert &\leq &\sup_{(x,\tau
)\in \mathcal{S}}\sup_{P\in \mathcal{P}}\left\vert \mathbf{E}\left[
q_{n,\tau ,j}^{(k)}(x)\tilde{U}_{n}^{(k)}\right] \right\vert  \label{dev7} \\
&\leq &\sup_{(x,\tau )\in \mathcal{S}}\sup_{P\in \mathcal{P}}\frac{\mathbf{E}%
\left[ |q_{n,\tau ,j}(x)|\right] }{\sqrt{P\{X_{i}\in \mathcal{C}\}}}  \notag
\\
&\leq &\frac{C_{1}h^{d}\sup_{(x,\tau )\in \mathcal{S}}\sup_{P\in \mathcal{P}%
}k_{n,\tau ,j,1}}{h^{d/2}\left( 1-\alpha _{P}\right) }\leq C_{2}h^{d/2}. 
\notag
\end{eqnarray}%
We conclude that%
\begin{equation*}
\sup_{(x,\tau )\in \mathcal{S}}\sup_{P\in \mathcal{P}}||A_{n,\tau
}(x)||=O(h^{d/2}).
\end{equation*}%
Therefore, from some large $n$ on,%
\begin{equation}
\inf_{(x,\tau )\in \mathcal{S}}\inf_{P\in \mathcal{P}}\lambda _{\min }\left(
\Sigma _{2n,\tau ,\bar{\varepsilon}}\right) \geq \bar{\varepsilon}/2.
\label{LB2}
\end{equation}

Let 
\begin{equation*}
W_{n,\tau }(x;\eta _{1})\equiv \Sigma _{2n,\tau ,\bar{\varepsilon}}^{-1/2}%
\left[ 
\begin{array}{c}
q_{n,\tau }(x)+\eta _{1} \\ 
\tilde{U}_{n}%
\end{array}%
\right] .
\end{equation*}%
Similarly as in (\ref{der5}), we find that for some $C>0$, from some large $%
n $ on,%
\begin{eqnarray*}
&&\sup_{(x,\tau )\in \mathcal{S}}\sup_{P\in \mathcal{P}}\mathbf{E}\left\Vert
W_{n,\tau }(x;\eta _{1})\right\Vert ^{3} \\
&\leq &C\sup_{(x,\tau )\in \mathcal{S}}\sup_{P\in \mathcal{P}}\lambda _{\max
}^{3/2}\left( \Sigma _{2n,\tau ,\bar{\varepsilon}}^{-1}\right) \sup_{(x,\tau
)\in \mathcal{S}}\sup_{P\in \mathcal{P}}\left\{ \mathbf{E}\left[ ||q_{n,\tau
}(x)+\eta _{1}||^{3}\right] +\mathbf{E}\left[ |\tilde{U}_{n}|^{3}\right]
\right\} \\
&\leq &C\left( \frac{\bar{\varepsilon}}{2}\right) ^{-3/2}\sup_{(x,\tau )\in 
\mathcal{S}}\sup_{P\in \mathcal{P}}\left\{ \mathbf{E}\left[ ||q_{n,\tau
}(x)+\eta _{1}||^{3}\right] +\mathbf{E}\left[ |\tilde{U}_{n}|^{3}\right]
\right\} ,
\end{eqnarray*}%
where the last inequality uses (\ref{LB2}). As for the last expectation,
note that by Rosenthal's inequality, we have 
\begin{equation*}
\sup_{(x,\tau )\in \mathcal{S}}\sup_{P\in \mathcal{P}}\mathbf{E}\left[ |%
\tilde{U}_{n}|^{3}\right] \leq C
\end{equation*}%
for some $C>0$. We apply the first statement of Lemma B5 to conclude that%
\begin{equation*}
\sup_{(x,\tau )\in \mathcal{S}}\sup_{P\in \mathcal{P}}\mathbf{E}\left\Vert
W_{n,\tau }(x;\eta _{1})\right\Vert ^{3}\leq C\bar{\varepsilon}%
^{-3/2}h^{-d/2},
\end{equation*}%
for some $C>0$. For any vector $\mathbf{v}=[\mathbf{v}_{1}^{\top
},v_{2}]^{\top }\in \mathbf{R}^{J+1}$, we define%
\begin{equation*}
D_{n,\tau ,p}(\mathbf{v})\equiv \Lambda _{A,p}\left( \left[ \Sigma _{2n,\tau
,\bar{\varepsilon}}^{1/2}\mathbf{v}\right] _{1}\right) \left[ \Sigma
_{2n,\tau ,\bar{\varepsilon}}^{1/2}\mathbf{v}\right] _{2},
\end{equation*}%
where $[a]_{1}$ of a vector $a\in \mathbf{R}^{J+1}$ indicates the vector of
the first $J$ entries of $a$, and $[a]_{2}$ the last entry of $a.$ By
Theorem 1 of \citeasnoun{Sweeting:77}, we find that (with $\bar{\varepsilon}%
>0$ fixed)%
\begin{equation*}
\mathbf{E}\left[ D_{n,\tau ,p}\left( \frac{1}{\sqrt{n}}\sum_{i=1}^{n}W_{n,%
\tau }^{(i)}(x;\eta _{1})\right) \right] =\mathbf{E}\left[ D_{n,\tau
,p}\left( \mathbb{Z}_{J+1}\right) \right] +O(n^{-1/2}h^{-d/2}),
\end{equation*}%
where $\mathbb{Z}_{J+1}\sim N(0,I_{J+1})$ and $W_{n,\tau }^{(i)}(x;\eta
_{1}) $'s are i.i.d. copies of $W_{n,\tau }(x;\eta _{1})$. Since $%
O(n^{-1/2}h^{-d/2})=o(h^{d/2})\ $(by the condition that $n^{-1/2}h^{-d-\nu
}\rightarrow 0$, as $n\rightarrow \infty $),%
\begin{equation*}
Cov\left( \Lambda _{A,p}\left( \sqrt{nh^{d}}\mathbf{z}_{N,\tau }(x;\eta
_{1})\right) ,U_{n}\right) =\mathbf{E}\left[ D_{n,\tau ,p}\left( \frac{1}{%
\sqrt{n}}\sum_{i=1}^{n}W_{n,\tau }^{(i)}(x;\eta _{1})\right) \right]
+o(h^{d/2}).
\end{equation*}%
Noting that $\mathbf{E}[D_{n,\tau ,p}\left( \mathbb{Z}_{J+1}\right) ]=0$, we
conclude that 
\begin{equation*}
\sup_{(x,\tau )\in \mathcal{S}}\sup_{P\in \mathcal{P}}\left\vert Cov\left(
\Lambda _{A,p}\left( \sqrt{nh^{d}}\mathbf{z}_{N,\tau }(x;\eta _{1})\right)
,U_{n}\right) \right\vert =o(h^{d/2}).
\end{equation*}%
By applying the Dominated Convergence Theorem, we obtain Step 2.\medskip 
\newline
\textbf{Proof of Step 3: }First, we show that 
\begin{equation}
Var\left( S_{n}\right) =\sigma _{n}^{2}(\mathcal{C})+o(1),  \label{var}
\end{equation}%
where $o(1)$ is an asymptotically negligible term uniformly over $P\in 
\mathcal{P}$. Note that%
\begin{equation*}
Var\left( S_{n}\right) =\sum_{A\in \mathcal{N}_{J}}\sum_{A^{\prime }\in 
\mathcal{N}_{J}}\mu _{A}\mu _{A^{\prime }}Cov(\psi _{n,A},\psi _{n,A^{\prime
}}),
\end{equation*}%
where $\psi _{n,A}\equiv h^{-d/2}(\zeta _{N,A}-\mathbf{E}\zeta _{N,A})$. By
Lemma B6, we find that for $A,A^{\prime }\in \mathcal{N}_{J},$ 
\begin{equation*}
Cov(\psi _{n,A},\psi _{n,A^{\prime }})=\sigma _{n,A,A^{\prime
}}(B_{n,A}(c_{n};\mathcal{C}),B_{n,A^{\prime }}(c_{n};\mathcal{C}))+o(1),
\end{equation*}%
uniformly in $P\in \mathcal{P}$, yielding the desired result.

Combining Steps 1 and 2, we deduce that\textbf{\ }%
\begin{equation}
\sup_{P\in \mathcal{P}}\left\vert Cov(S_{n},U_{n})\right\vert \leq C\sqrt{%
\bar{\varepsilon}}+o(h^{d/2}).  \label{bd45}
\end{equation}%
Let $\bar{\sigma}_{1}^{2}\equiv \inf_{P\in \mathcal{P}}\sigma _{n}^{2}(%
\mathcal{C})$ and $\bar{\sigma}_{2}^{2}\equiv \inf_{P\in \mathcal{P}%
}(1-\alpha _{P})$. Note that for some $C_{1}>0$,%
\begin{equation}
\inf_{P\in \mathcal{P}}\bar{\sigma}_{1}^{2}\bar{\sigma}_{2}^{2}>C_{1},
\label{bd56}
\end{equation}%
by the condition of the lemma. A simple calculation gives us%
\begin{eqnarray}
\lambda _{\min }(C_{n}) &=&\frac{\bar{\sigma}_{1}^{2}+\bar{\sigma}_{2}^{2}}{2%
}-\frac{1}{2}\left( \sqrt{\left( \bar{\sigma}_{1}^{2}+\bar{\sigma}%
_{2}^{2}\right) ^{2}-4\left\{ \bar{\sigma}_{1}^{2}\bar{\sigma}%
_{2}^{2}-Cov(S_{n},U_{n})^{2}\right\} }\right)  \label{dev67} \\
&\geq &\frac{1}{2}\left\{ \sqrt{\left( \bar{\sigma}_{1}^{2}+\bar{\sigma}%
_{2}^{2}\right) ^{2}}-\left( \sqrt{\left( \bar{\sigma}_{1}^{2}+\bar{\sigma}%
_{2}^{2}\right) ^{2}-4\bar{\sigma}_{1}^{2}\bar{\sigma}_{2}^{2}}\right)
\right\} -\left\vert Cov(S_{n},U_{n})\right\vert  \notag \\
&\geq &\bar{\sigma}_{1}^{2}\bar{\sigma}_{2}^{2}-\left\vert
Cov(S_{n},U_{n})\right\vert \geq C_{1}-C\sqrt{\bar{\varepsilon}}+o(h^{d/2}),
\notag
\end{eqnarray}%
where the last inequality follows by (\ref{bd45}) and (\ref{bd56}). Taking $%
\bar{\varepsilon}$ small enough, we obtain the desired result.\medskip 
\newline
\textbf{Proof of Step 4:} Suppose that\textit{\ }liminf$_{n\rightarrow
\infty }\inf_{P\in \mathcal{P}}\sigma _{n}^{2}(\mathcal{C})>0$. Let $\kappa $
be the diameter of the compact set $\mathcal{K}_{0}$ introduced in
Assumption A2. Let $\mathcal{C}$ be given as in the lemma. Let $\mathbb{Z}%
^{d}$ be the set of $d$-tuples of integers, and let $\{R_{n,\mathbf{i}}:%
\mathbf{i}\in \mathbb{Z}^{d}\}$ be the collection of rectangles in $\mathbf{R%
}^{d}$ such that $R_{n,\mathbf{i}}=[a_{n,\mathbf{i}_{1}},b_{n,\mathbf{i}%
_{1}}]\times \cdot \cdot \cdot \cdot \times \lbrack a_{n,\mathbf{i}%
_{d}},b_{n,\mathbf{i}_{d}}]$, where $\mathbf{i}_{j}$ is the $j$-th entry of $%
\mathbf{i}$, and $h\kappa \leq b_{n,\mathbf{i}_{j}}-a_{n,\mathbf{i}_{j}}\leq
2h\kappa $, for all $j=1,\cdot \cdot \cdot ,d$, and two different rectangles 
$R_{n,\mathbf{i}}$ and $R_{n,\mathbf{j}}$ do not have intersection with
nonempty interior, and the union of the rectangles $R_{n,\mathbf{i}}$, $%
\mathbf{i}\in \mathbb{Z}_{n}^{d}$, cover $\mathcal{X}$, from some
sufficiently large $n$ on, where $\mathbb{Z}_{n}^{d}$ be the set of $d$%
-tuples of integers whose absolute values less than or equal to $n$.

We let%
\begin{eqnarray*}
B_{n,A,x}(c_{n}) &\equiv &\left\{ \tau \in \mathcal{T}:(x,\tau )\in
B_{A}(c_{n})\right\} , \\
B_{n,\mathbf{i}} &\equiv &R_{n,\mathbf{i}}\cap \mathcal{C},
\end{eqnarray*}%
and $\mathcal{I}_{n}\equiv \{\mathbf{i}\in \mathbb{Z}_{n}^{d}:B_{n,\mathbf{i}%
}\neq \varnothing \}$. Then $B_{n,\mathbf{i}}$ has Lebesgue measure $m(B_{n,%
\mathbf{i}})$ bounded by $C_{1}h^{d}$ and the cardinality of the set $%
\mathcal{I}_{n}$ is bounded by $C_{2}h^{-d}$ for some positive constants $%
C_{1}$ and $C_{2}$. Now let us define%
\begin{equation*}
\Delta _{n,A,\mathbf{i}}\equiv h^{-d/2}\int_{B_{n,\mathbf{i}%
}}\int_{B_{n,A,x}(c_{n})}\left\{ \Lambda _{A,p}(\sqrt{nh^{d}}\mathbf{z}%
_{N,\tau }(x))-\mathbf{E}\left[ \Lambda _{A,p}(\sqrt{nh^{d}}\mathbf{z}%
_{N,\tau }(x))\right] \right\} d\tau dx.
\end{equation*}%
And also define $B_{n,A,\mathbf{i}}(c_{n})\equiv (B_{n,\mathbf{i}}\times 
\mathcal{T})\cap B_{n,A}(c_{n}),$%
\begin{eqnarray*}
\alpha _{n,\mathbf{i}} &\equiv &\frac{\sum_{A\in \mathcal{N}_{J}}\mu
_{A}\Delta _{n,A,\mathbf{i}}}{\sigma _{n}(\mathcal{C})}\mathbf{\ }\text{and}
\\
u_{n,\mathbf{i}} &\equiv &\frac{1}{\sqrt{n}}\left\{ \sum_{i=1}^{N}1\left\{
X_{i}\in B_{n,\mathbf{i}}\right\} -nP\{X_{i}\in B_{n,\mathbf{i}}\}\right\} .
\end{eqnarray*}%
Then, we can write%
\begin{equation*}
\frac{S_{n}}{\sigma _{n}(\mathcal{C})}=\sum_{\mathbf{i}\in \mathcal{I}%
_{n}}\alpha _{n,\mathbf{i}}\text{ and }U_{n}=\sum_{\mathbf{i}\in \mathcal{I}%
_{n}}u_{n,\mathbf{i}}.
\end{equation*}%
By the definition of $\mathcal{K}_{0}$ in Assumption A2, by the definition
of $R_{n,\mathbf{i}}$ and by the properties of Poisson processes, one can
see that the array $\{(\alpha _{n,\mathbf{i}},u_{n,\mathbf{i}})\}_{\mathbf{i}%
\in \mathcal{I}_{n}}$ is an array of $1$-dependent random field. (See %
\citeasnoun{Mason/Polonik:09} for details.) For any $q_{1},q_{2}\in \mathbf{R%
}$, let $y_{n,\mathbf{i}}\equiv q_{1}\alpha _{n,\mathbf{i}}+q_{2}u_{n,%
\mathbf{i}}$. The focus is on the convergence in distribution of $\sum_{%
\mathbf{i}\in \mathcal{I}_{n}}y_{n,\mathbf{i}}$ uniform over $P\in \mathcal{P%
}$. Without loss of generality, we choose $q_{1},q_{2}\in \mathbf{R}%
\backslash \{0\}$. Define%
\begin{equation*}
Var_{P}\left( \sum_{\mathbf{i}\in \mathcal{I}_{n}}y_{n,\mathbf{i}}\right)
=q_{1}^{2}+q_{2}^{2}(1-\alpha _{P})+2q_{1}q_{2}c_{n,P},
\end{equation*}%
uniformly over $P\in \mathcal{P}$, where $c_{n,P}=Cov(S_{n},U_{n})$. On the
other hand, using Lemma B4 and following the proof of Lemma A8 of %
\citeasnoun{LSW}, we deduce that%
\begin{equation}
\sup_{P\in \mathcal{P}}\sum_{\mathbf{i}\in \mathcal{I}_{n}}\mathbf{E}|y_{n,%
\mathbf{i}}|^{r}=o(1)  \label{conv_y}
\end{equation}%
as $n\rightarrow \infty $, for any $r\in (2,(2p+2)/p]$. By Theorem 1 of %
\citeasnoun{Shergin:93}, we have%
\begin{eqnarray*}
&&\sup_{P\in \mathcal{P}}\sup_{t\in \mathbf{R}}\left\vert P\left\{ \frac{1}{%
\sqrt{q_{1}^{2}+q_{2}^{2}(1-\alpha _{P})+2q_{1}q_{2}c_{n,P}}}\sum_{\mathbf{i}%
\in \mathcal{I}_{n}}y_{n,\mathbf{i}}\leq t\right\} -\Phi \left( t\right)
\right\vert \\
&\leq &\sup_{P\in \mathcal{P}}\frac{C}{\left\{ q_{1}^{2}+q_{2}^{2}(1-\alpha
_{P})+2q_{1}q_{2}c_{n,P}\right\} ^{r/2}}\left\{ \sum_{\mathbf{i}\in \mathcal{%
I}_{n}}\mathbf{E}|y_{n,\mathbf{i}}|^{r}\right\} ^{1/2}=o(1),
\end{eqnarray*}%
for some $C>0$, by (\ref{conv_y}). Therefore, by Lemma B2(i), we have for
each $t\in \mathbf{R},$ and each $q\in \mathbf{R}^{2}\backslash \{0\}$, as $%
n\rightarrow \infty ,$%
\begin{equation*}
\sup_{P\in \mathcal{P}}\left\vert \mathbf{E}\left[ \exp \left( it\frac{%
q^{\top }H_{n}}{\sqrt{q^{\top }C_{n}q}}\right) \right] -\exp \left( -\frac{%
t^{2}}{2}\right) \right\vert \rightarrow 0.
\end{equation*}%
Thus by Lemma B2(ii), for each $t\in \mathbf{R}^{2},$ we have%
\begin{equation*}
\sup_{P\in \mathcal{P}}\left\vert P\left\{ C_{n}^{-1/2}H_{n}\leq t\right\}
-P\left\{ \mathbb{Z}\leq t\right\} \right\vert \rightarrow 0.
\end{equation*}%
Since the limit distribution of $C_{n}^{-1/2}H_{n}$ is continuous, the
convergence above is uniform in $t\in \mathbf{R}^{2}$.\medskip

\noindent (ii) We fix $P\in \mathcal{P}$ such that limsup$_{n\rightarrow
\infty }\sigma _{n}^{2}(\mathcal{C})=0.$ Then by (\ref{var}) above, 
\begin{equation*}
Var\left( S_{n}\right) =\sigma _{n}^{2}(\mathcal{C})+o(1)=o(1).
\end{equation*}%
Hence, we find that $S_{n}=o_{P}(1)$. The desired result follows by applying
Theorem 1 of \citeasnoun{Shergin:93} to the sum $U_{n}=\sum_{\mathbf{i}\in 
\mathcal{I}_{n}}u_{n,\mathbf{i}}$, and then applying Lemma B2(ii).
\end{proof}

\begin{LemmaC}
\label{lem-c9} \textit{Let }$\mathcal{C}$\textit{\ be the Borel set in Lemma
B8}.

\noindent \textit{(i) Suppose that the conditions of Lemma B8(i) are
satisfied. Then for each }$t\in \mathbf{R}$, \textit{as }$n\rightarrow
\infty ,$%
\begin{equation*}
\sup_{P\in \mathcal{P}}\sup_{t\in \mathbf{R}}\left\vert P\left\{ \frac{%
h^{-d/2}\sum_{A\in \mathcal{N}_{J}}\mu _{A}\left\{ \zeta _{n,A}-\mathbf{E}%
\zeta _{N,A}\right\} }{\sigma _{n}(\mathcal{C})}\leq t\right\} -\Phi
(t)\right\vert \rightarrow 0.
\end{equation*}%
\noindent \textit{(ii) Suppose that the conditions of Lemma B8(ii) are
satisfied. Then as }$n\rightarrow \infty ,$%
\begin{equation*}
h^{-d/2}\sum_{A\in \mathcal{N}_{J}}\mu _{A}\left\{ \zeta _{n,A}-\mathbf{E}%
\zeta _{N,A}\right\} \overset{p}{\rightarrow }0.
\end{equation*}
\end{LemmaC}

Note that in both statements, the location normalization has $\mathbf{E}%
\zeta _{N,A}$ instead of $\mathbf{E}\zeta _{n,A}$.

\begin{proof}[Proof of Lemma B\protect\ref{lem-c9}]
(i) The conditional distribution of $S_{n}/\sigma _{n}(\mathcal{C})$ given $%
N=n$ is equal to that of 
\begin{equation*}
\frac{\sum_{A\in \mathcal{N}_{J}}\mu _{A}\int_{B_{n,A}(c_{n};\mathcal{C}%
)\cap \mathcal{C}}\left\{ \Lambda _{A,p}(\sqrt{nh^{d}}\mathbf{z}_{n,\tau
}(x))-\mathbf{E}\Lambda _{A,p}(\sqrt{nh^{d}}\mathbf{z}_{N,\tau }(x))\right\}
dQ(x,\tau )}{h^{d/2}\sigma _{n}(\mathcal{C})}.
\end{equation*}%
Using Lemmas B3(i) and B8(i), we find that%
\begin{equation*}
\frac{h^{-d/2}\sum_{A\in \mathcal{N}_{J}}\mu _{A}\left\{ \zeta _{n,A}-%
\mathbf{E}\zeta _{N,A}\right\} }{\sigma _{n}(\mathcal{C})}\overset{d}{%
\rightarrow }N(0,1).
\end{equation*}%
Since the limit distribution $N(0,1)$ is continuous and the convergence is
uniform in $P\in \mathcal{P}$, we obtain the desired result.

\noindent (ii) Similarly as before, the result follows from Lemmas B3(ii),
B2(ii), and B8(ii).
\end{proof}

\section{Proofs of Auxiliary Results for Lemmas A2(ii), Lemma A4(ii), and
Theorem 1}

\label{appendix:D}

The auxiliary results in this section are mostly bootstrap versions of the
results in Appendix B. To facilitate comparison, we name the first lemma to
be Lemma C3, which is used to control the discrepancy between the sample
version of the scale normalizer $\sigma _{n}$, and its population version.
Then we proceed to prove Lemmas C4-C9 which run in parallel with Lemmas
B4-B9 as their bootstrap counterparts. We finish this subsection with Lemmas
C10-C12 which are crucial for dealing with the bootstrap test statistic's
location normalization. More specifically, Lemmas C10 and C11 are auxiliary
moment bound results that are used for proving Lemma C12. Lemma C12
essentially delivers the result of Lemma A1 in Appendix A. This lemma is
used to deal with the discrepancy between the population location normalizer
and the sample location normalizer. Controlling this discrepancy to the rate 
$o_{P}(h^{d/2})$ is crucial for our purpose, because our bootstrap test
statistic does not involve the sample version of the location normalizer $%
a_{n}$ for computational reasons. Lemmas C10 and C11 provide necessary
moment bounds to achieve this convergence rate.

The random variables $N$ and $N_{1}$ represent Poisson random variables with
mean $n$ and $1$ respectively. These random variables are independent of $%
\left( (Y_{i}^{\ast \top },X_{i}^{\ast \top })_{i=1}^{\infty },(Y_{i}^{\top
},X_{i}^{\top })_{i=1}^{\infty }\right) $. Let $\eta _{1}$ and $\eta _{2}$
be centered normal random vectors that are independent of each other and
independent of 
\begin{equation*}
\left( (Y_{i}^{\ast \top },X_{i}^{\ast \top })_{i=1}^{\infty },(Y_{i}^{\top
},X_{i}^{\top })_{i=1}^{\infty },N,N_{1}\right) .
\end{equation*}%
We will specify their covariance matrices in the proofs below. Throughout
the proofs, the bootstrap distribution $P^{\ast }$ and expectations $\mathbf{%
E}^{\ast }$ are viewed as the distribution of 
\begin{equation*}
\left( (Y_{i}^{\ast },X_{i}^{\ast })_{i=1}^{n},N,N_{1},\eta _{1},\eta
_{2}\right) ,
\end{equation*}%
conditional on $(Y_{i},X_{i})_{i=1}^{n}$.

Define 
\begin{eqnarray*}
\tilde{\rho}_{n,\tau _{1},\tau _{2},j,k}(x,u) &\equiv &\frac{1}{h^{d}}%
\mathbf{E}^{\ast }\left[ \beta _{n,x,\tau _{1},j}\left( Y_{ij}^{\ast },\frac{%
X_{i}^{\ast }-x}{h}\right) \beta _{n,x,\tau _{2},k}\left( Y_{ik}^{\ast },%
\frac{X_{i}^{\ast }-x}{h}+u\right) \right] \text{ and } \\
\tilde{k}_{n,\tau ,j,m}(x) &\equiv &\frac{1}{h^{d}}\mathbf{E}^{\ast }\left[
\left\vert \beta _{n,x,\tau ,j}\left( Y_{ij}^{\ast },\frac{X_{i}^{\ast }-x}{h%
}\right) \right\vert ^{m}\right] .
\end{eqnarray*}%
Note that $\tilde{\rho}_{n,\tau _{1},\tau _{2},j,k}(x,u)$ and $\tilde{k}%
_{n,\tau ,j,m}(x)$ are bootstrap versions of $\rho _{n,\tau _{1},\tau
_{2},j,k}(x,u)$ and $\tilde{k}_{n,\tau ,j,m}(x)$. The lemma below
establishes that the bootstrap version $\tilde{\rho}_{n,\tau _{1},\tau
_{2},j,k}(x,u)$ is consistent for $\rho _{n,\tau _{1},\tau _{2},j,k}(x,u)$.

\setcounter{LemmaD}{2}

\begin{LemmaD}
\label{lem-d3} \textit{Suppose that Assumption A\ref{assumption-A6}(i) holds
and that }$n^{-1/2}h^{-d/2}\rightarrow 0$\textit{, as }$n\rightarrow \infty $%
.\textit{\ Then for each }$\varepsilon \in (0,\varepsilon _{1})$\textit{,
with }$\varepsilon _{1}>0$\textit{\ as in Assumption A6(i), as }$%
n\rightarrow \infty $,%
\begin{equation*}
\sup_{(\tau _{1},\tau _{2})\in \mathcal{T}\times \mathcal{T}}\sup_{(x,u)\in (%
\mathcal{S}_{\tau _{1}}(\varepsilon )\cup \mathcal{S}_{\tau
_{2}}(\varepsilon ))\times \mathcal{U}}\sup_{P\in \mathcal{P}}\mathbf{E}%
\left( \left\vert \tilde{\rho}_{n,\tau _{1},\tau _{2},j,k}(x,u)-\rho
_{n,\tau _{1},\tau _{2},j,k}(x,u)\right\vert ^{2}\right) \rightarrow 0.
\end{equation*}
\end{LemmaD}

\begin{proof}[Proof of Lemma C\protect\ref{lem-d3}]
Define $\pi _{n,x,u,\tau _{1},\tau _{2}}(y,z)=\beta _{n,x,\tau
_{1},j}(y_{j},(z-x)/h)\beta _{n,x,\tau _{2},k}(y_{k},(z-x)/h+u)$ for $%
y=(y_{1},\cdot \cdot \cdot ,y_{J})^{\top }\in \mathbf{R}^{J}$, and write%
\begin{equation*}
\tilde{\rho}_{n,\tau _{1},\tau _{2},j,k}(x,u)-\rho _{n,\tau _{1},\tau
_{2},j,k}(x,u)=\frac{1}{nh^{d}}\sum_{i=1}^{n}\left\{ \pi _{n,x,u,\tau
_{1},\tau _{2}}(Y_{i},X_{i})-\mathbf{E}\left[ \pi _{n,x,u,\tau _{1},\tau
_{2}}(Y_{i},X_{i})\right] \right\} .
\end{equation*}%
First, we note that%
\begin{equation*}
\mathbf{E}\left( \frac{1}{\sqrt{n}}\sum_{i=1}^{n}\left\{ \pi _{n,x,u,\tau
_{1},\tau _{2}}(Y_{i},X_{i})-\mathbf{E}\left[ \pi _{n,x,u,\tau _{1},\tau
_{2}}(Y_{i},X_{i})\right] \right\} \right) ^{2}\leq \mathbf{E}\left[ \pi
_{n,x,u,\tau _{1},\tau _{2}}^{2}(Y_{i},X_{i})\right] .
\end{equation*}%
By change of variables and Assumption A\ref{assumption-A6}(i), we have $%
\mathbf{E}\left[ \pi _{n,x,u,\tau _{1},\tau _{2}}^{2}(Y_{i},X_{i})\right]
=O(h^{d})$ uniformly over $(\tau _{1},\tau _{2})\in \mathcal{T}\times 
\mathcal{T}$, $(x,u)\in (\mathcal{S}_{\tau _{1}}(\varepsilon )\cup \mathcal{S%
}_{\tau _{2}}(\varepsilon ))\times \mathcal{U}$ and over $P\in \mathcal{P}$.
Hence%
\begin{equation*}
\mathbf{E}\left( \left\vert \tilde{\rho}_{n,\tau _{1},\tau
_{2},j,k}(x,u)-\rho _{n,\tau _{1},\tau _{2},j,k}(x,u)\right\vert ^{2}\right)
=O\left( n^{-1}h^{-d}\right) ,
\end{equation*}%
uniformly over $(\tau _{1},\tau _{2})\in \mathcal{T}\times \mathcal{T}$, $%
(x,u)\in (\mathcal{S}_{\tau _{1}}(\varepsilon )\cup \mathcal{S}_{\tau
_{2}}(\varepsilon ))\times \mathcal{U}$ and over $P\in \mathcal{P}$. Since
we have assumed that$\ n^{-1/2}h^{-d/2}\rightarrow 0$ as $n\rightarrow
\infty $, we obtain the desired result.
\end{proof}

\begin{LemmaD}
\label{lem-d4} \textit{Suppose that Assumption A\ref{assumption-A6}(i) holds
and that for some }$C>0,$%
\begin{equation*}
\text{limsup}_{n\rightarrow \infty }n^{-1/2}h^{-d/2}\leq C\text{.}
\end{equation*}%
\textit{\ Then for all }$m\in \lbrack 2,M]$\textit{\ and all }$\varepsilon
\in (0,\varepsilon _{1})$\textit{, with }$M>0$\textit{\ and }$\varepsilon
_{1}>0$\textit{\ being the constants that appear in Assumption A\ref%
{assumption-A6}(i)), there exists }$C_{1}\in (0,\infty )$\textit{\ that does
not depend on }$n$\textit{\ such that for each }$j\in \mathbb{N}_{J}$\textit{%
,}%
\begin{equation*}
\sup_{\tau \in \mathcal{T},x\in \mathcal{S}_{\tau }(\varepsilon )}\sup_{P\in 
\mathcal{P}}\mathbf{E}\left[ \tilde{k}_{n,\tau ,j,m}^{2}(x)\right] \leq
C_{1}.
\end{equation*}
\end{LemmaD}

\begin{proof}[Proof of Lemma C\protect\ref{lem-d4}]
Since $\mathbf{E}^{\ast }[|\beta _{n,x,\tau ,j}(Y_{ij}^{\ast },(X_{i}^{\ast
}-x)/h)|^{m}]=\frac{1}{n}\sum_{i=1}^{n}|\beta _{n,x,\tau
,j}(Y_{ij},(X_{i}-x)/h)|^{m}$, we find that%
\begin{equation*}
\tilde{k}_{n,\tau ,j,m}^{2}(x)\leq 2k_{n,\tau ,j,m}^{2}(x)+2e_{n,\tau
,j,m}^{2}(x),
\end{equation*}%
where%
\begin{equation*}
e_{n,\tau ,j,m}(x)\equiv \left\vert \frac{1}{nh^{d}}\sum_{i=1}^{n}\left\vert
\beta _{n,x,\tau ,j}\left( Y_{ij},\frac{X_{i}-x}{h}\right) \right\vert ^{m}-%
\frac{1}{h^{d}}\mathbf{E}\left( \left\vert \beta _{n,x,\tau ,j}\left( Y_{ij},%
\frac{X_{i}-x}{h}\right) \right\vert ^{m}\right) \right\vert .
\end{equation*}%
Similarly as in the proof of Lemma C3, we note that%
\begin{eqnarray*}
&&\sup_{\tau \in \mathcal{T},x\in \mathcal{S}_{\tau }(\varepsilon
)}\sup_{P\in \mathcal{P}}\mathbf{E}\left[ \left\vert e_{n,\tau
,j,m}^{2}(x)\right\vert \right] \\
&\leq &\sup_{\tau \in \mathcal{T},x\in \mathcal{S}_{\tau }(\varepsilon
)}\sup_{P\in \mathcal{P}}\frac{1}{nh^{2d}}\mathbf{E}\left[ \left\vert \beta
_{n,x,\tau ,j}\left( Y_{ij},\frac{X_{i}-x}{h}\right) \right\vert ^{2m}\right]
=O(n^{-1}h^{-d})=o(1),\text{ as }n\rightarrow \infty \text{.}
\end{eqnarray*}%
Hence the desired statement follows from Lemma B4.
\end{proof}

Let%
\begin{eqnarray*}
\mathbf{z}_{n,\tau }^{\ast }(x) &\equiv &\frac{1}{nh^{d}}\sum_{i=1}^{n}\beta
_{n,x,\tau }\left( Y_{i}^{\ast },\frac{X_{i}^{\ast }-x}{h}\right) -\frac{1}{%
h^{d}}\mathbf{E}^{\ast }\left[ \beta _{n,x,\tau }\left( Y_{i}^{\ast },\frac{%
X_{i}^{\ast }-x}{h}\right) \right] \text{, and} \\
\mathbf{z}_{N,\tau }^{\ast }(x) &\equiv &\frac{1}{nh^{d}}\sum_{i=1}^{N}\beta
_{n,x,\tau }\left( Y_{i}^{\ast },\frac{X_{i}^{\ast }-x}{h}\right) -\frac{1}{%
h^{d}}\mathbf{E}^{\ast }\left[ \beta _{n,x,\tau }\left( Y_{i}^{\ast },\frac{%
X_{i}^{\ast }-x}{h}\right) \right] .
\end{eqnarray*}%
We also let%
\begin{eqnarray*}
q_{n,\tau }^{\ast }(x) &\equiv &\frac{1}{\sqrt{h^{d}}}\sum_{i\leq
N_{1}}\left\{ \beta _{n,x,\tau }(Y_{i}^{\ast },(X_{i}^{\ast }-x)/h)-\mathbf{E%
}^{\ast }\beta _{n,x,\tau }(Y_{i}^{\ast },(X_{i}^{\ast }-x)/h)\right\} \ 
\text{and} \\
\bar{q}_{n,\tau }^{\ast }(x) &\equiv &\frac{1}{\sqrt{h^{d}}}\left\{ \beta
_{n,x,\tau }(Y_{i}^{\ast },(X_{i}^{\ast }-x)/h)-\mathbf{E}^{\ast }\beta
_{n,x,\tau }(Y_{i}^{\ast },(X_{i}^{\ast }-x)/h)\right\} .
\end{eqnarray*}

\begin{LemmaD}
\label{lem-d5} \textit{Suppose that Assumption A\ref{assumption-A6}(i) holds
and that for some }$C>0,$%
\begin{equation*}
\text{limsup}_{n\rightarrow \infty }n^{-1/2}h^{-d/2}\leq C\text{.}
\end{equation*}%
\textit{Then for any }$m\in \lbrack 2,M]$ \textit{(with }$M$ \textit{being
the constant }$M$\textit{\ in Assumption A6(i)), }%
\begin{eqnarray}
\sup_{(x,\tau )\in \mathcal{S}}\sup_{P\in \mathcal{P}}\sqrt{\mathbf{E}\left[
\left( \mathbf{E}^{\ast }\left[ ||q_{n,\tau }^{\ast }(x)||^{m}\right]
\right) ^{2}\right] } &\leq &\bar{C}_{1}h^{d(1-(m/2))},\text{\textit{\ and}}
\label{bd11} \\
\sup_{(x,\tau )\in \mathcal{S}}\sup_{P\in \mathcal{P}}\sqrt{\mathbf{E}\left[
\left( \mathbf{E}^{\ast }\left[ ||\bar{q}_{n,\tau }^{\ast }(x)||^{m}\right]
\right) ^{2}\right] } &\leq &\bar{C}_{2}h^{d(1-(m/2))},  \notag
\end{eqnarray}%
\textit{where }$\bar{C}_{1},\bar{C}_{2}>0$\textit{\ are constants that
depend only on }$m.$ \textit{If furthermore, }%
\begin{equation*}
\limsup_{n\rightarrow \infty }n^{-(m/2)+1}h^{d(1-(m/2))}<C,
\end{equation*}%
\textit{for some constant} $C>0$, \textit{then} 
\begin{eqnarray}
\sup_{(x,\tau )\in \mathcal{S}}\sup_{P\in \mathcal{P}}\mathbf{E}\left[ 
\mathbf{E}^{\ast }\left[ ||\sqrt{nh^{d}}\mathbf{z}_{N,\tau }^{\ast }(x)||^{m}%
\right] \right] &\leq &\left( \frac{15m}{\log m}\right) ^{m}\max \left\{ 
\bar{C}_{1},2\bar{C}_{1}C\right\} \text{,\textit{\ and}}  \label{bds22} \\
\sup_{(x,\tau )\in \mathcal{X}^{\varepsilon /2}\times \mathcal{T}}\sup_{P\in 
\mathcal{P}}\mathbf{E}\left[ \mathbf{E}^{\ast }\left[ ||\sqrt{nh^{d}}\mathbf{%
z}_{n,\tau }^{\ast }(x)||^{m}\right] \right] &\leq &\left( \frac{15m}{\log m}%
\right) ^{m}\max \left\{ \bar{C}_{2},2\bar{C}_{2}C\right\} ,  \notag
\end{eqnarray}%
\textit{where} $\bar{C}_{1},\bar{C}_{2}>0$\textit{\ are the constants that
appear in (\ref{bd11}).}
\end{LemmaD}

\begin{proof}[Proof of Lemma C\protect\ref{lem-d5}]
Let $q_{n,\tau ,j}^{\ast }(x)$ be the $j$-th entry of $q_{n,\tau }^{\ast
}(x) $. For the first statement of the lemma, it suffices to observe that
for each $\varepsilon \in (0,\varepsilon _{1})$, there exist $C_{1}>0$ and $%
\bar{C}_{1}>0$ such that%
\begin{equation*}
\sup_{\tau \in \mathcal{T},x\in \mathcal{S}_{\tau }(\varepsilon )}\mathbf{E}%
\left[ \left( \mathbf{E}^{\ast }\left[ |q_{n,\tau ,j}^{\ast }(x)|^{m}\right]
\right) ^{2}\right] \leq \frac{C_{1}h^{2d}\sum_{j=1}^{J}\sup_{\tau \in 
\mathcal{T},x\in \mathcal{S}_{\tau }(\varepsilon )}\sup_{P\in \mathcal{P}}%
\mathbf{E}\left[ \tilde{k}_{n,\tau ,j,m}^{2}(x)\right] }{h^{dm}}\leq \bar{C}%
_{1}h^{2d(1-(m/2))},
\end{equation*}%
where the last inequality uses Lemma C4. The second inequality in (\ref{bd11}%
) follows similarly.

Let us consider (\ref{bds22}). Let $z_{N,\tau ,j}^{\ast }(x)$ be the $j$-th
entry of $\mathbf{z}_{N,\tau }^{\ast }(x)$. Then using Rosenthal's
inequality (e.g. (2.3) of \citeasnoun{GMZ}), for some constant $C_{1}>0,$%
\begin{eqnarray*}
&&\sup_{\tau \in \mathcal{T},x\in \mathcal{S}_{\tau }(\varepsilon
)}\sup_{P\in \mathcal{P}}\mathbf{E}\left[ \mathbf{E}^{\ast }[|\sqrt{nh^{d}}%
z_{N,\tau ,j}^{\ast }(x)|^{m}]\right] \\
&\leq &\left( \frac{15m}{\log m}\right) ^{2m}\sup_{\tau \in \mathcal{T},x\in 
\mathcal{S}_{\tau }(\varepsilon )}\sup_{P\in \mathcal{P}}\left\{ \left( 
\mathbf{E}\left[ \mathbf{E}^{\ast }\left( q_{n,\tau ,j}^{\ast 2}(x)\right) %
\right] \right) ^{m/2}+\mathbf{E}\left[ n^{-(m/2)+1}\mathbf{E}^{\ast
}|q_{n,\tau ,j}^{\ast }(x)|^{m}\right] \right\} .
\end{eqnarray*}%
The first expectation is bounded by $\bar{C}_{1}$ by (\ref{bd11}).

The second expectation is bounded by $\bar{C}_{1}n^{-(m/2)+1}h^{d(1-(m/2))}$%
. This gives the first bound in (\ref{bds22}). The second bound in (\ref%
{bds22}) can be obtained similarly.
\end{proof}

For any Borel sets $B,B^{\prime }\subset \mathcal{S}$ and $A,A^{\prime
}\subset \mathbb{N}_{J},$ let%
\begin{equation*}
\tilde{\sigma}_{n,A,A^{\prime }}^{R}(B,B^{\prime })\equiv \int_{\mathcal{T}%
}\int_{\mathcal{T}}\int_{B_{\tau _{2}}^{\prime }}\int_{B_{\tau
_{1}}}C_{n,\tau _{1},\tau _{2},A,A^{\prime }}^{\ast }(x,v)dxdvd\tau
_{1}d\tau _{2},
\end{equation*}%
where $B_{\tau }\equiv \{x\in \mathcal{X}:(x,\tau )\in B\},$ 
\begin{equation}
C_{n,\tau _{1},\tau _{2},A,A^{\prime }}^{\ast }(x,v)\equiv h^{-d}Cov^{\ast
}\left( \Lambda _{A,p}(\sqrt{nh^{d}}\mathbf{z}_{N,\tau _{1}}^{\ast
}(x)),\Lambda _{A^{\prime },p}(\sqrt{nh^{d}}\mathbf{z}_{N,\tau _{2}}^{\ast
}(v))\right) ,  \label{C*}
\end{equation}%
and $Cov^{\ast }$ represents covariance under $P^{\ast }$. We also define 
\begin{equation}
\tilde{\sigma}_{n,A}^{R}(B)\equiv \tilde{\sigma}_{n,A,A}^{R}(B,B),
\label{sigma_tilde}
\end{equation}%
for brevity. Also, let $\Sigma _{n,\tau _{1},\tau _{2}}^{\ast }(x,u)$ be a $%
J\times J$ matrix whose $(j,k)$-th entry is given by $\tilde{\rho}_{n,\tau
_{1},\tau _{2},j,k}(x,u)$. Fix $\bar{\varepsilon}>0$ and define%
\begin{equation*}
\tilde{\Sigma}_{n,\tau _{1},\tau _{2},\bar{\varepsilon}}^{\ast }(x,u)\equiv %
\left[ 
\begin{array}{c}
\Sigma _{n,\tau _{1},\tau _{1}}^{\ast }(x,0)+\bar{\varepsilon}I_{J} \\ 
\Sigma _{n,\tau _{1},\tau _{2}}^{\ast }(x,u)%
\end{array}%
\begin{array}{c}
\Sigma _{n,\tau _{1},\tau _{2}}^{\ast }(x,u) \\ 
\Sigma _{n,\tau _{2},\tau _{2}}^{\ast }(x,0)+\bar{\varepsilon}I_{J}%
\end{array}%
\right] .
\end{equation*}%
We also define%
\begin{equation*}
\mathbf{\xi }_{N,\tau _{1},\tau _{2}}^{\ast }(x,u;\eta _{1},\eta _{2})\equiv 
\sqrt{nh^{d}}\Sigma _{n,\tau _{1},\tau _{2},\bar{\varepsilon}}^{\ast
-1/2}(x,u)\left[ 
\begin{array}{c}
\mathbf{z}_{N,\tau _{1}}^{\ast }(x;\eta _{1}) \\ 
\mathbf{z}_{N,\tau _{2}}^{\ast }(x+uh;\eta _{2})%
\end{array}%
\right] ,
\end{equation*}%
where $\eta _{1}\in \mathbf{R}^{J}$ and $\eta _{2}\in \mathbf{R}^{J}$ are
random vectors that are independent, and independent of $((Y_{i}^{\ast
},X_{i}^{\ast })_{i=1}^{\infty },(Y_{i},X_{i})_{i=1}^{\infty },N,N_{1})$,
each following $N(0,\bar{\varepsilon}I_{J})$, and define $\mathbf{z}_{N,\tau
}^{\ast }(x;\eta _{1})\equiv \mathbf{z}_{N,\tau }^{\ast }(x)+\eta _{1}/\sqrt{%
nh^{d}}.$

\begin{LemmaD}
\label{lem-d6} \textit{Suppose that\ Assumption A\ref{assumption-A6}(i)
holds and that }$nh^{d}\rightarrow \infty $\textit{, and }%
\begin{equation*}
\text{limsup}_{n\rightarrow \infty }n^{-(m/2)+1}h^{d(1-(m/2))}<C,
\end{equation*}%
\textit{for some }$C>0$\textit{\ and some }$m\in \lbrack 2(p+1),M]$.

\textit{Then for any sequences of Borel sets }$B_{n},B_{n}^{\prime }\subset 
\mathcal{S}$\textit{\ and for any }$A,A^{\prime }\subset \mathbb{N}_{J},$%
\begin{equation*}
\sup_{P\in \mathcal{P}}\mathbf{E}\left( \left\vert \left( \tilde{\sigma}%
_{n,A,A^{\prime }}^{R}(B_{n},B_{n}^{\prime })\right) ^{2}-\sigma
_{n,A,A^{\prime }}^{2}(B_{n},B_{n}^{\prime })\right\vert \right) \rightarrow
0,
\end{equation*}%
\textit{where }$\sigma _{n,A,A^{\prime }}^{2}(B_{n},B_{n}^{\prime })$\textit{%
\ is as defined in (\ref{s2})}.
\end{LemmaD}

\begin{proof}[Proof of Lemma C\protect\ref{lem-d6}]
The proof is very similar to that of Lemma B6. For brevity, we sketch the
proof here. Define for $\bar{\varepsilon}>0,$%
\begin{eqnarray*}
\tilde{\sigma}_{n,A,A^{\prime },\bar{\varepsilon}}^{R}(B_{n},B_{n}^{\prime
}) &\equiv &\int_{\mathcal{T}}\int_{\mathcal{T}}\int_{B_{n,\tau _{1}}}\int_{%
\mathcal{U}}\tilde{g}_{1n,\tau _{1},\tau _{2},\bar{\varepsilon}}(x,u)w_{\tau
_{1},B_{n}}(x)w_{\tau _{2},B_{n}^{\prime }}(x+uh)dudxd\tau _{1}d\tau _{2}, \\
\tilde{\tau}_{n,A,A^{\prime },\bar{\varepsilon}}(B_{n},B_{n}^{\prime })
&\equiv &\int_{\mathcal{T}}\int_{\mathcal{T}}\int_{B_{n,\tau _{1}}}\int_{%
\mathcal{U}}\tilde{g}_{2n,\tau _{1},\tau _{2},\bar{\varepsilon}}(x,u)w_{\tau
_{1},B_{n}}(x)w_{\tau _{2},B_{n}^{\prime }}(x+uh)dudxd\tau _{1}d\tau _{2},
\end{eqnarray*}%
where%
\begin{eqnarray*}
\tilde{g}_{1n,\tau _{1},\tau _{2},\bar{\varepsilon}}(x,u) &\equiv
&h^{-d}Cov^{\ast }(\Lambda _{A,p}(\sqrt{nh^{d}}\mathbf{z}_{N,\tau
_{1}}^{\ast }(x;\eta _{1})),\Lambda _{A^{\prime },p}(\sqrt{nh^{d}}\mathbf{z}%
_{N,\tau _{2}}^{\ast }(x+uh;\eta _{2}))),\text{ and} \\
\tilde{g}_{2n,\tau _{1},\tau _{2},\bar{\varepsilon}}(x,u) &\equiv &Cov^{\ast
}(\Lambda _{A,p}(\mathbb{\tilde{Z}}_{n,\tau _{1},\tau _{2},\bar{\varepsilon}%
}(x)),\Lambda _{A^{\prime },p}(\mathbb{\tilde{Z}}_{n,\tau _{1},\tau _{2},%
\bar{\varepsilon}}(x+uh))),
\end{eqnarray*}%
and $[\mathbb{\tilde{Z}}_{n,\tau _{1},\tau _{2},\bar{\varepsilon}}^{\top
}(x),\mathbb{\tilde{Z}}_{n,\tau _{1},\tau _{2},\bar{\varepsilon}}^{\top
}(z)]^{\top }$ is a centered normal $\mathbf{R}^{2J}$-valued random vector
with the same covariance matrix as the covariance matrix of $[\sqrt{nh^{d}}%
\mathbf{z}_{N,\tau _{1}}^{\ast \top }(x;\eta _{1}),\sqrt{nh^{d}}\mathbf{z}%
_{N,\tau _{2}}^{\ast \top }(z;\eta _{2})]^{\top }$ under the product measure
of the bootstrap distribution $P^{\ast }$ and the distribution of $(\eta
_{1}^{\top },\eta _{2}^{\top })^{\top }$. As in the proof of Lemma B6, it
suffices for the lemma to show the following two statements.

\noindent (\textit{Step 1}): As $n\rightarrow \infty ,$%
\begin{eqnarray*}
\sup_{P\in \mathcal{P}}\mathbf{E}\left( \left\vert \tilde{\sigma}%
_{n,A,A^{\prime },\bar{\varepsilon}}^{R}(B_{n},B_{n}^{\prime })-\tilde{\tau}%
_{n,A,A^{\prime },\bar{\varepsilon}}(B_{n},B_{n}^{\prime })\right\vert
\right) &\rightarrow &0,\text{ and} \\
\sup_{P\in \mathcal{P}}\mathbf{E}\left( \left\vert \tilde{\tau}%
_{n,A,A^{\prime },\bar{\varepsilon}}(B_{n},B_{n}^{\prime })-\sigma
_{n,A,A^{\prime },\bar{\varepsilon}}(B_{n},B_{n}^{\prime })\right\vert
\right) &\rightarrow &0.
\end{eqnarray*}

\noindent (\textit{Step 2}): For some $C>0$ that does not depend on $\bar{%
\varepsilon}\ $or $n,$ 
\begin{equation*}
\sup_{P\in \mathcal{P}}|\tilde{\sigma}_{n,A,A^{\prime },\bar{\varepsilon}%
}^{R}(B_{n},B_{n}^{\prime })-\tilde{\sigma}_{n,A,A^{\prime
}}^{R}(B_{n},B_{n}^{\prime })|\leq C\sqrt{\bar{\varepsilon}}.
\end{equation*}%
Then the desired result follows by sending $n\rightarrow \infty $ and $\bar{%
\varepsilon}\downarrow 0$, while chaining Steps 1 and 2 and the second
convergence in Step 2 in the proof of Lemma B6.

We first focus on the first statement of (Step 1). For any vector $\mathbf{v}%
=[\mathbf{v}_{1}^{\top },\mathbf{v}_{2}^{\top }]^{\top }\in \mathbf{R}^{2J}$%
, we define%
\begin{equation}
\tilde{C}_{n,p}(\mathbf{v})\equiv \Lambda _{A,p}\left( \left[ \tilde{\Sigma}%
_{n,\tau _{1},\tau _{2},\bar{\varepsilon}}^{\ast 1/2}(x,u)\mathbf{v}\right]
_{1}\right) \Lambda _{A^{\prime },p}\left( \left[ \tilde{\Sigma}_{n,\tau
_{1},\tau _{2},\bar{\varepsilon}}^{\ast 1/2}(x,u)\mathbf{v}\right]
_{2}\right) ,  \label{lambda_b}
\end{equation}%
where $[a]_{1}$ of a vector $a\in \mathbf{R}^{2J}$ indicates the vector of
the first $J$ entries of $a$, and $[a]_{2}$ the vector of the remaining $J$
entries of $a.$ Also, similarly as in (\ref{lb2}),%
\begin{equation}
\lambda _{\min }\left( \tilde{\Sigma}_{n,\tau _{1},\tau _{2},\bar{\varepsilon%
}}^{\ast }(x,u)\right) \geq \bar{\varepsilon}.  \label{lambda_min_b}
\end{equation}

Let $\bar{q}_{n,\tau }^{\ast }(x;\eta _{1})$ be the column vector of entries 
$\bar{q}_{n,\tau ,j}^{\ast }(x;\eta _{1j})$ with $j$ running in the set $%
\mathbb{N}_{J}$, and with%
\begin{equation*}
\bar{q}_{n,\tau ,j}^{\ast }(x;\eta _{1j})\equiv p_{n,\tau ,j}^{\ast
}(x)+\eta _{1j},
\end{equation*}%
where 
\begin{equation*}
p_{n,\tau ,j}^{\ast }(x)=\frac{1}{\sqrt{h^{d}}}\sum_{1\leq i\leq
N_{1}}\left\{ \beta _{n,x,\tau ,j}\left( Y_{ij}^{\ast },\frac{X_{i}^{\ast }-x%
}{h}\right) -\mathbf{E}\left[ \beta _{n,x,\tau ,j}\left( Y_{ij}^{\ast },%
\frac{X_{i}^{\ast }-x}{h}\right) \right] \right\} ,
\end{equation*}%
$\eta _{1j}$ is the $j$-th entry of $\eta _{1}$, and $N_{1}$ is a Poisson
random variable with mean $1$ and $((\eta _{1j})_{j\in A},N_{1})$ is
independent of $\{(Y_{i}^{\top },X_{i}^{\top },Y_{i}^{\ast \top
},X_{i}^{\ast \top })\}_{i=1}^{n}$. Let $[p_{n,\tau _{1}}^{\ast
(i)}(x),p_{n,\tau _{2}}^{\ast (i)}(x+uh)]$ be the i.i.d. copies of $%
[p_{n,\tau _{1}}^{\ast }(x),p_{n,\tau _{2}}^{\ast }(x+uh)]$ conditional on
the observations $\{(Y_{i},X_{i})\}_{i=1}^{n}$, and $\eta _{1}^{(i)}$ and $%
\eta _{2}^{(i)}$ be i.i.d. copies of $\eta _{1}$ and $\eta _{2}$. Define 
\begin{equation*}
q_{n,\tau _{1}}^{\ast (i)}(x;\eta _{1}^{(i)})=p_{n,\tau _{1}}^{\ast
(i)}(x)+\eta _{1}^{(i)}\text{ and }q_{n,\tau _{2}}^{\ast (i)}(x+uh;\eta
_{2}^{(i)})=p_{n,\tau _{2}}^{\ast (i)}(x+uh)+\eta _{2}^{(i)}.
\end{equation*}%
Note that%
\begin{equation*}
\frac{1}{\sqrt{n}}\sum_{i=1}^{n}\left[ 
\begin{array}{c}
q_{n,\tau _{1}}^{\ast (i)}(x;\eta _{1}^{(i)}) \\ 
q_{n,\tau _{2}}^{\ast (i)}(x+uh;\eta _{2}^{(i)})%
\end{array}%
\right] =\frac{1}{\sqrt{n}}\sum_{i=1}^{n}\left[ 
\begin{array}{c}
p_{n,\tau _{1}}^{\ast (i)}(x) \\ 
p_{n,\tau _{2}}^{\ast (i)}(x+uh)%
\end{array}%
\right] +\frac{1}{\sqrt{n}}\sum_{i=1}^{n}\left[ 
\begin{array}{c}
\eta _{1}^{(i)} \\ 
\eta _{2}^{(i)}%
\end{array}%
\right] .
\end{equation*}%
The last sum has the same distribution as $[\eta _{1}^{\top },\eta
_{2}^{\top }]^{\top }$ and the leading sum on the right-hand side has the
same bootstrap distribution as that of $[\mathbf{z}_{N,\tau _{1}}^{\ast \top
}(x),\mathbf{z}_{N,\tau _{2}}^{\ast \top }(x+uh)]^{\top }$, $P$-a.e.
Therefore, we conclude that%
\begin{equation*}
\mathbf{\xi }_{N,\tau _{1},\tau _{2}}^{\ast }(x,u;\eta _{1}^{(i)},\eta
_{2}^{(i)})\overset{d^{\ast }}{=}\frac{1}{\sqrt{n}}\sum_{i=1}^{n}\tilde{W}%
_{n,\tau _{1},\tau _{2}}^{(i)}(x,u;\eta _{1}^{(i)},\eta _{2}^{(i)}),
\end{equation*}%
where $\overset{d^{\ast }}{=}$ indicates the distributional equivalence with
respect to the product measure of the bootstrap distribution $P^{\ast }$ and
the joint distribution of $(\eta _{1}^{(i)},\eta _{2}^{(i)})$, $P$-a.e, and%
\begin{equation*}
\tilde{W}_{n,\tau _{1},\tau _{2}}^{(i)}(x,u;\eta _{1}^{(i)},\eta
_{2}^{(i)})\equiv \tilde{\Sigma}_{n,\tau _{1},\tau _{2},\bar{\varepsilon}%
}^{\ast -1/2}(x,u)\left[ 
\begin{array}{c}
q_{n}^{(i)}(x;\eta _{1}^{(i)}) \\ 
q_{n}^{(i)}(x+uh;\eta _{2}^{(i)})%
\end{array}%
\right] .
\end{equation*}%
Following the arguments in the proof of Lemma B6, we find that for each $%
u\in \mathcal{U},$ and for $\varepsilon \in (0,\varepsilon _{1})$ with $%
\varepsilon _{1}$ as in Assumption A6(i),%
\begin{eqnarray*}
&&\sup_{(x,u)\in (\mathcal{S}_{\tau _{1}}\cup \mathcal{S}_{\tau _{2}})\times 
\mathcal{U}}\sup_{P\in \mathcal{P}}\mathbf{E}\left[ \mathbf{E}^{\ast }||%
\tilde{W}_{n,\tau _{1},\tau _{2}}^{(i)}(x,u;\eta _{1}^{(i)},\eta
_{2}^{(i)})||^{3}\right] \\
&\leq &C_{1}\sup_{(x,u)\in (\mathcal{S}_{\tau _{1}}(\varepsilon )\cup 
\mathcal{S}_{\tau _{2}}(\varepsilon ))\times \mathcal{U}}\sup_{P\in \mathcal{%
P}}\mathbf{E}\left[ \lambda _{\min }^{3}\left( \tilde{\Sigma}_{n,\tau
_{1},\tau _{2},\bar{\varepsilon}}^{\ast -1/2}(x,u)\right) \mathbf{E}^{\ast
}||q_{n,\tau _{1}}^{\ast (i)}(x;\eta _{1}^{(i)})||^{3}\right] \\
&&+C_{1}\sup_{(x,u)\in (\mathcal{S}_{\tau _{1}}(\varepsilon )\cup \mathcal{S}%
_{\tau _{2}}(\varepsilon ))\times \mathcal{U}}\sup_{P\in \mathcal{P}}\mathbf{%
E}\left[ \lambda _{\min }^{3}\left( \tilde{\Sigma}_{n,\tau _{1},\tau _{2},%
\bar{\varepsilon}}^{\ast -1/2}(x,u)\right) \mathbf{E}^{\ast }||q_{n,\tau
_{2}}^{\ast (i)}(x+uh;\eta _{2}^{(i)})||^{3}\right] ,
\end{eqnarray*}%
for some $C_{1}>0$. As for the leading term,%
\begin{eqnarray*}
&&\sup_{(x,u)\in (\mathcal{S}_{\tau _{1}}(\varepsilon )\cup \mathcal{S}%
_{\tau _{2}}(\varepsilon ))\times \mathcal{U}}\sup_{P\in \mathcal{P}}\mathbf{%
E}\left[ \lambda _{\min }^{3}\left( \tilde{\Sigma}_{n,\tau _{1},\tau _{2},%
\bar{\varepsilon}}^{\ast -1/2}(x,u)\right) \mathbf{E}^{\ast }||q_{n,\tau
_{1}}^{\ast (i)}(x;\eta _{1}^{(i)})||^{3}\right] \\
&\leq &\sup_{(x,u)\in (\mathcal{S}_{\tau _{1}}(\varepsilon )\cup \mathcal{S}%
_{\tau _{2}}(\varepsilon ))\times \mathcal{U}}\sup_{P\in \mathcal{P}}\sqrt{%
\mathbf{E}\left[ \left( \mathbf{E}^{\ast }||q_{n,\tau _{1}}^{\ast
(i)}(x;\eta _{1}^{(i)})||^{3}\right) ^{2}\right] } \\
&&\times \sup_{(x,u)\in (\mathcal{S}_{\tau _{1}}(\varepsilon )\cup \mathcal{S%
}_{\tau _{2}}(\varepsilon ))\times \mathcal{U}}\sup_{P\in \mathcal{P}}\sqrt{%
\mathbf{E}\left[ \lambda _{\min }^{6}\left( \tilde{\Sigma}_{n,\tau _{1},\tau
_{2},\bar{\varepsilon}}^{\ast -1/2}(x,u)\right) \right] }\leq \frac{C_{2}%
\bar{\varepsilon}^{-3}}{\sqrt{h^{d}}},
\end{eqnarray*}%
by Lemma C5 and (\ref{lambda_min_b}). Similarly, we observe that 
\begin{equation*}
\sup_{(x,u)\in (\mathcal{S}_{\tau _{1}}(\varepsilon )\cup \mathcal{S}_{\tau
_{2}}(\varepsilon ))\times \mathcal{U}}\sup_{P\in \mathcal{P}}\mathbf{E}%
\left[ \lambda _{\min }^{3}\left( \tilde{\Sigma}_{n,\tau _{1},\tau _{2},\bar{%
\varepsilon}}^{\ast -1/2}(x,u)\right) \mathbf{E}^{\ast }||q_{n,\tau
_{2}}^{\ast (i)}(x+uh;\eta _{2}^{(i)})||^{3}\right] \leq \frac{C_{2}\bar{%
\varepsilon}^{-3}}{\sqrt{h^{d}}}.
\end{equation*}

Define 
\begin{equation*}
c_{n,\tau _{1},\tau _{2}}(x,u)=\tilde{C}_{n,p}\left( \frac{1}{\sqrt{n}}%
\sum_{i=1}^{n}\tilde{W}_{n,\tau _{1},\tau _{2}}^{(i)}(x,u;\eta
_{1}^{(i)},\eta _{2}^{(i)})\right) .
\end{equation*}%
Let $\Phi _{n,\tau _{1},\tau _{2}}(\cdot ;x,u)$ be the joint CDF of the
random vector $(\mathbb{\tilde{Z}}_{n,\tau _{1},\tau _{2},\bar{\varepsilon}%
}^{\top }(x),\mathbb{\tilde{Z}}_{n,\tau _{1},\tau _{2},\bar{\varepsilon}%
}^{\top }(x+uh))^{\top }$. By Theorem 1 of \citeasnoun{Sweeting:77},%
\begin{eqnarray}
&&\sup_{(x,u)\in (\mathcal{S}_{\tau _{1}}(\varepsilon )\cup \mathcal{S}%
_{\tau _{2}}(\varepsilon ))\times \mathcal{U}}\sup_{P\in \mathcal{P}}\mathbf{%
E}\left[ \left\vert c_{n,\tau _{1},\tau _{2}}(x,u)-\int \tilde{C}%
_{n,p}(\zeta )d\Phi _{n,\tau _{1},\tau _{2}}(\zeta ;x,u)\right\vert \right]
\label{bd0} \\
&\leq &\frac{C_{1}}{\sqrt{n}}\sup_{(x,u)\in (\mathcal{S}_{\tau
_{1}}(\varepsilon )\cup \mathcal{S}_{\tau _{2}}(\varepsilon ))\times 
\mathcal{U}}\sup_{P\in \mathcal{P}}\mathbf{E}\left[ \mathbf{E}^{\ast }||%
\tilde{W}_{n,\tau _{1},\tau _{2}}^{(i)}(x,u;\eta _{1}^{(i)},\eta
_{2}^{(i)})||^{3}\right] \leq \frac{C_{2}\bar{\varepsilon}^{-3}}{\sqrt{nh^{d}%
}}.  \notag
\end{eqnarray}%
Hence%
\begin{eqnarray*}
&&\mathbf{E}\left[ \left\vert \int_{B_{\tau _{1}}}\int_{\mathcal{U}}\left\{ 
\tilde{g}_{1n,\tau _{1},\tau _{2},\bar{\varepsilon}}(x,u)-\tilde{g}_{2n,\tau
_{1},\tau _{2},\bar{\varepsilon}}(x,u)\right\} w_{\tau _{1},B}(x)w_{\tau
_{2},B^{\prime }}(x+uh)dudx\right\vert \right] \\
&\leq &\int_{B_{\tau _{1}}}\int_{\mathcal{U}}\mathbf{E}\left\vert \tilde{g}%
_{1n,\tau _{1},\tau _{2},\bar{\varepsilon}}(x,u)-\tilde{g}_{2n,\tau
_{1},\tau _{2},\bar{\varepsilon}}(x,u)\right\vert w_{\tau _{1},B}(x)w_{\tau
_{2},B^{\prime }}(x+uh)dudx \\
&\leq &\int_{B_{\tau _{1}}}w_{\tau _{1},B}(x)w_{\tau _{2},B^{\prime }}(x)dx
\\
&&\times \sup_{(x,u)\in (\mathcal{S}_{\tau _{1}}(\varepsilon )\cup \mathcal{S%
}_{\tau _{2}}(\varepsilon ))\times \mathcal{U}}\sup_{P\in \mathcal{P}}%
\mathbf{E}\left\vert \tilde{g}_{1n,\tau _{1},\tau _{2},\bar{\varepsilon}%
}(x,u)-\tilde{g}_{2n,\tau _{1},\tau _{2},\bar{\varepsilon}}(x,u)\right\vert
\\
&\rightarrow &0,
\end{eqnarray*}%
as $n\rightarrow \infty $. The last convergence is due to (\ref{bd0}) and
hence uniform over $(\tau _{1},\tau _{2})\in \mathcal{T}\times \mathcal{T}$.
The proof of (Step 1) is thus complete.

We turn to the second statement of (Step 1). Similarly as in the proof of
Step 1 in the proof of Lemma B6, the second statement of Step 1 follows by
Lemma C4.

Now we turn to (Step 2). In view of the proof of Step 2 in the proof of
Lemma B6, it suffices to show that with $s=(p+1)/(p-1)$ if $p>1$ and $s=2$
if $p=1,$%
\begin{eqnarray}
\sup_{\tau \in \mathcal{T},x\in \mathcal{S}_{\tau }(\varepsilon )}\sup_{P\in 
\mathcal{P}}\mathbf{E}\left[ \mathbf{E}^{\ast }\left\Vert \sqrt{nh^{d}}%
\mathbf{z}_{N,\tau }^{\ast }(x)\right\Vert ^{2s(p-1)}\right] &<&C\text{ and}
\label{bdd3} \\
\sup_{\tau \in \mathcal{T},x\in \mathcal{S}_{\tau }(\varepsilon )}\sup_{P\in 
\mathcal{P}}\mathbf{E}\left[ \mathbf{E}^{\ast }\left\Vert \sqrt{nh^{d}}%
\mathbf{z}_{N,\tau }^{\ast }(x;\eta _{1})\right\Vert ^{2s(p-1)}\right] &<&C%
\text{,}  \notag
\end{eqnarray}%
for some $C>0$. First note that for any $q>0$,%
\begin{eqnarray*}
&&\sup_{\tau \in \mathcal{T},x\in \mathcal{S}_{\tau }(\varepsilon
)}\sup_{P\in \mathcal{P}}\mathbf{E}\left[ \mathbf{E}^{\ast }\left\Vert \sqrt{%
nh^{d}}\{\mathbf{z}_{N,\tau }^{\ast }(x)-\mathbf{z}_{N,\tau }^{\ast }(x;\eta
_{1})\}\right\Vert ^{2q}\right] \\
&=&\mathbf{E}\left\Vert \sqrt{\bar{\varepsilon}}\mathbb{Z}\right\Vert ^{2q}=C%
\bar{\varepsilon}^{q}\text{,}
\end{eqnarray*}%
where $\mathbb{Z}\in \mathbf{R}^{J}$ is a centered normal random vector with
covariance matrix $I_{J}$. Also, we deduce that for some constants $%
C_{1},C_{2}>0$,%
\begin{eqnarray*}
&&\sup_{\tau \in \mathcal{T},x\in \mathcal{S}_{\tau }(\varepsilon
)}\sup_{P\in \mathcal{P}}\mathbf{E}\left[ \mathbf{E}^{\ast }\left\Vert \sqrt{%
nh^{d}}\mathbf{z}_{N,\tau }^{\ast }(x)\right\Vert ^{2s(p-1)}\right] \\
&\leq &\sup_{\tau \in \mathcal{T},x\in \mathcal{S}_{\tau }(\varepsilon
)}\sup_{P\in \mathcal{P}}\mathbf{E}\left[ \mathbf{E}^{\ast }\left\Vert \sqrt{%
nh^{d}}\mathbf{z}_{N,\tau }^{\ast }(x;\eta _{1})\right\Vert ^{2s(p-1)}\right]
+C_{1}\bar{\varepsilon}^{s(p-1)}\leq C_{1}+C_{2}\bar{\varepsilon}^{s(p-1)},
\end{eqnarray*}%
by the third statement of Lemma C5. This leads to the first and second
statements of (\ref{bdd3}). Thus the proof of the lemma is complete.
\end{proof}

\begin{LemmaD}
\label{lem-d7} \textit{Suppose that for some small }$\nu _{1}>0$, $%
n^{-1/2}h^{-d-\nu _{1}}\rightarrow \infty $\textit{, as }$n\rightarrow
\infty $\textit{\ and the conditions of Lemma B6 hold. Then there exists} $%
C>0$\textit{\ such that for any sequence of Borel sets }$B_{n}\subset 
\mathcal{S},$ \textit{and }$A\subset \mathbb{N}_{J},$\textit{\ from some
large }$n$\textit{\ on,}%
\begin{align*}
& \sup_{P\in \mathcal{P}}\mathbf{E}\left( \mathbf{E}^{\ast }\left[
\left\vert h^{-d/2}\int_{B_{n}}\left\{ \Lambda _{A,p}(\sqrt{nh^{d}}\mathbf{z}%
_{n,\tau }^{\ast }(x))-\mathbf{E}^{\ast }\left[ \Lambda _{A,p}(\sqrt{nh^{d}}%
\mathbf{z}_{N,\tau }^{\ast }(x))\right] \right\} dQ(x,\tau )\right\vert %
\right] \right) \\
& \leq C\sqrt{Q(B_{n})}.
\end{align*}
\end{LemmaD}

\begin{proof}[Proof of Lemma C\protect\ref{lem-d7}]
We follow the proof of Lemma B7 and show that for some $C>0$, we have the
following:\newline
\textbf{Step 1:} $\sup_{P\in \mathcal{P}}\mathbf{E}\left( \mathbf{E}^{\ast }%
\left[ \left\vert h^{-d/2}\int_{B_{n}}\left\{ \Lambda _{A,p}(\sqrt{nh^{d}}%
\mathbf{z}_{n,\tau }^{\ast }(x))-\Lambda _{A,p}(\sqrt{nh^{d}}\mathbf{z}%
_{N,\tau }^{\ast }(x))\right\} dQ(x,\tau )\right\vert \right] \right) \leq
CQ(B_{n}),$\newline
\textbf{Step 2:}%
\begin{eqnarray*}
&&\sup_{P\in \mathcal{P}}\mathbf{E}\left( \mathbf{E}^{\ast }\left[
\left\vert h^{-d/2}\int_{B_{n}}\left\{ \Lambda _{A,p}(\sqrt{nh^{d}}\mathbf{z}%
_{n,\tau }^{\ast }(x))-\Lambda _{A,p}(\sqrt{nh^{d}}\mathbf{z}_{N,\tau
}^{\ast }(x))\right\} dQ(x,\tau )\right\vert \right] \right) \\
&\leq &C\sqrt{Q(B_{n})}.
\end{eqnarray*}

\noindent \textbf{Proof of Step 1: }Similarly as in the proof of Step 1 in
the proof of Lemma B7, we first write%
\begin{equation*}
\mathbf{z}_{n,\tau }^{\ast }(x)=\mathbf{z}_{N,\tau }^{\ast }(x)+\mathbf{v}%
_{n,\tau }^{\ast }(x)+\mathbf{s}_{n,\tau }^{\ast }(x),
\end{equation*}%
where%
\begin{eqnarray*}
\mathbf{v}_{n,\tau }^{\ast }(x) &\equiv &\left( \frac{n-N}{n}\right) \cdot 
\frac{1}{h^{d}}\mathbf{E}^{\ast }\left[ \beta _{n,x,\tau }\left( Y_{i}^{\ast
},\frac{X_{i}^{\ast }-x}{h}\right) \right] \text{ and} \\
\mathbf{s}_{n,\tau }^{\ast }(x) &\equiv &\frac{1}{nh^{d}}\sum_{i=N+1}^{n}%
\left\{ \beta _{n,x,\tau }\left( Y_{i}^{\ast },\frac{X_{i}^{\ast }-x}{h}%
\right) -\mathbf{E}^{\ast }\left[ \beta _{n,x,\tau }\left( Y_{i}^{\ast },%
\frac{X_{i}^{\ast }-x}{h}\right) \right] \right\} .
\end{eqnarray*}

Similarly as in the proof of Lemma B7, we deduce that for some $%
C_{1},C_{2}>0,$%
\begin{eqnarray*}
\lefteqn{\left\vert \int_{B_{n}}\left\{ \Lambda _{A,p}\left( \mathbf{z}%
_{n,\tau }^{\ast }(x)\right) -\Lambda _{A,p}\left( \mathbf{z}_{N,\tau
}^{\ast }(x)\right) \right\} dQ(x,\tau )\right\vert } \\
&\leq &C_{1}\int_{B_{n}}\left\Vert \mathbf{v}_{n,\tau }^{\ast
}(x)\right\Vert \left( \left\Vert \mathbf{z}_{n,\tau }^{\ast }(x)\right\Vert
^{p-1}+\left\Vert \mathbf{z}_{N,\tau }^{\ast }(x)\right\Vert ^{p-1}\right)
dQ(x,\tau ) \\
&&+C_{2}\int_{B_{n}}\left\Vert \mathbf{s}_{n,\tau }^{\ast }(x)\right\Vert
\left( \left\Vert \mathbf{z}_{n,\tau }^{\ast }(x)\right\Vert
^{p-1}+\left\Vert \mathbf{z}_{N,\tau }^{\ast }(x)\right\Vert ^{p-1}\right)
dQ(x,\tau ) \\
&=&D_{1n}^{\ast }+D_{2n}^{\ast },\text{ say.}
\end{eqnarray*}%
To deal with $D_{1n}^{\ast }$ and $D_{2n}^{\ast }$, we first show the
following:\medskip

\noindent \textsc{Claim 1:} $\sup_{(x,\tau )\in \mathcal{S}}\sup_{P\in 
\mathcal{P}}\mathbf{E}\left( \mathbf{E}^{\ast }[||\mathbf{v}_{n,\tau }^{\ast
}(x)||^{2}]\right) =O(n^{-1}).$\medskip

\noindent \textsc{Claim 2:} $\sup_{(x,\tau )\in \mathcal{S}}\sup_{P\in 
\mathcal{P}}\mathbf{E}\left( \mathbf{E}^{\ast }[||\mathbf{s}_{n,\tau }^{\ast
}(x)||^{2}]\right) =O(n^{-3/2}h^{-d})$.\medskip

\noindent \textsc{Proof of Claim 1:} Similarly as in the proof of Lemma B7,
we note that%
\begin{equation*}
\mathbf{E}\left( \mathbf{E}^{\ast }[||\mathbf{v}_{n,\tau }^{\ast
}(x)||^{2}]\right) \leq \mathbf{E}\left\vert \left( \frac{n-N}{n}\right)
\right\vert ^{2}\mathbf{E}\left[ \left\Vert \frac{1}{h^{d}}\mathbf{E}^{\ast }%
\left[ \beta _{n,x,\tau }\left( Y_{i}^{\ast },\frac{X_{i}^{\ast }-x}{h}%
\right) \right] \right\Vert ^{2}\right] .
\end{equation*}%
By the first statement of Lemma C5, we have%
\begin{equation*}
\sup_{(x,\tau )\in \mathcal{S}}\sup_{P\in \mathcal{P}}\mathbf{E}\left[
\left\Vert \frac{1}{h^{d}}\mathbf{E}^{\ast }\left[ \beta _{n,x,\tau }\left(
Y_{i}^{\ast },\frac{X_{i}^{\ast }-x}{h}\right) \right] \right\Vert ^{2}%
\right] =O(1).
\end{equation*}%
Since $\mathbf{E}\left\vert (n-N)/n\right\vert ^{2}=O(n^{-1})$, we obtain
Claim 1.\medskip

\noindent \textsc{Proof of Claim 2:} Let 
\begin{equation*}
\mathbf{s}_{n,\tau }^{\ast }(x;\eta _{1})=\mathbf{s}_{n,\tau }^{\ast }(x)+%
\frac{(N-n)\eta _{1}}{n^{3/2}h^{d/2}},
\end{equation*}%
where $\eta _{1}$ is a random vector independent of $((Y_{i}^{\ast
},X_{i}^{\ast })_{i=1}^{n},(Y_{i},X_{i})_{i=1}^{n},N)$ and follows $N(0,\bar{%
\varepsilon}I_{J})$. Note that%
\begin{eqnarray*}
&&\sup_{(x,\tau )\in \mathcal{S}}\sup_{P\in \mathcal{P}}\mathbf{E}\left( 
\mathbf{E}^{\ast }\left\Vert \sqrt{nh^{d}}\mathbf{s}_{n,\tau }^{\ast
}(x)\right\Vert ^{2}\right) \\
&\leq &2\sup_{(x,\tau )\in \mathcal{S}}\sup_{P\in \mathcal{P}}\mathbf{E}%
\left( \mathbf{E}^{\ast }\left\Vert \sqrt{nh^{d}}\mathbf{s}_{n,\tau }^{\ast
}(x;\eta _{1})\right\Vert ^{2}\right) +\frac{2}{n}\mathbf{E}\left\Vert \frac{%
(N-n)\eta _{1}}{\sqrt{n}}\right\Vert ^{2} \\
&\leq &2\sup_{(x,\tau )\in \mathcal{S}}\sup_{P\in \mathcal{P}}\mathbf{E}%
\left( \mathbf{E}^{\ast }\left\Vert \sqrt{nh^{d}}\mathbf{s}_{n,\tau }^{\ast
}(x;\eta _{1})\right\Vert ^{2}\right) +\frac{C\bar{\varepsilon}^{2}}{n},
\end{eqnarray*}%
as in the proof of Lemma B7. As for the leading expectation on the right
hand side of (\ref{dec22}), we let $C_{1}>0$ be as in Lemma C4 and note that%
\begin{eqnarray*}
\mathbf{E}\left( \mathbf{E}^{\ast }\left\Vert \sqrt{nh^{d}}\mathbf{s}%
_{n,\tau }^{\ast }(x;\eta _{1})\right\Vert ^{2}\right) &=&\sum_{j\in \mathbb{%
N}_{J}}\mathbf{E}\left( \mathbf{E}^{\ast }\left( \frac{1}{\sqrt{n}}%
\sum_{i=N+1}^{n}q_{n,\tau ,j}^{\ast (i)}(x;\eta _{1j}^{(i)})\right)
^{2}\right) \\
&=&\frac{1}{n}\sum_{j\in \mathbb{N}_{J}}\mathbf{E}\left( \tilde{\sigma}%
_{n,\tau ,j}^{2}(x)\mathbf{E}^{\ast }\left( \sum_{i=N+1}^{n}\frac{q_{n,\tau
,j}^{\ast (i)}(x;\eta _{1j}^{(i)})}{\tilde{\sigma}_{n,\tau ,j}(x)}\right)
^{2}\right) ,
\end{eqnarray*}%
where $q_{n,\tau }^{\ast (i)}(x;\eta _{1}^{(i)})$'s ($i=1,2,\cdot \cdot
\cdot \ $) are as defined in the proof of Lemma C6 and $q_{n,\tau ,j}^{\ast
(i)}(x;\eta _{1j}^{(i)})$ is the $j$-th entry of $q_{n,\tau }^{\ast
(i)}(x;\eta _{1}^{(i)})$ and $\tilde{\sigma}_{n,\tau ,j}^{2}(x)=Var^{\ast
}(q_{n,\tau ,j}^{\ast (i)}(x;\eta _{1j}^{(i)}))>0$ and $Var^{\ast }$ denotes
the variance with respect to the joint distribution of $((Y_{i}^{\ast
},X_{i}^{\ast })_{i=1}^{n},\eta _{1j}^{(i)})$ conditional on $%
(Y_{i},X_{i})_{i=1}^{n}.$ We apply Lemma 1(i) of \citeasnoun{Horvath:91} to
deduce that%
\begin{eqnarray}
\mathbf{E}^{\ast }\left( \sum_{i=N+1}^{n}\frac{q_{n,\tau ,j}^{\ast
(i)}(x;\eta _{1j}^{(i)})}{\tilde{\sigma}_{n,\tau ,j}(x)}\right) ^{2} &\leq &C%
\sqrt{n}+C\mathbf{E}^{\ast }\left( \left\vert \frac{q_{n,\tau ,j}^{\ast
(i)}(x;\eta _{1j}^{(i)})}{\tilde{\sigma}_{n,\tau ,j}(x)}\right\vert
^{3}\right)  \label{dev4} \\
&&+C\mathbf{E}^{\ast }\left( \left\vert \frac{q_{n,\tau ,j}^{\ast
(i)}(x;\eta _{1j}^{(i)})}{\tilde{\sigma}_{n,\tau ,j}(x)}\right\vert
^{4}\right) ,  \notag
\end{eqnarray}%
for some $C>0$. Using this, Lemma C5, and following arguments similarly as
in (\ref{arg1}), (\ref{arg2}) and (\ref{arg3}), we conclude that 
\begin{eqnarray*}
\sup_{(x,\tau )\in \mathcal{S}}\sup_{P\in \mathcal{P}}\mathbf{E}\left( 
\mathbf{E}^{\ast }\left\Vert \sqrt{nh^{d}}\mathbf{s}_{n,\tau }^{\ast
}(x)\right\Vert ^{2}\right) &\leq &O\left( n^{-1}h^{-\nu _{1}}\right)
+O\left( n^{-1/2}+n^{-3/4}h^{-d/2-\nu _{1}}+n^{-1}h^{-d-\nu _{1}}\right) \\
&=&O\left( n^{-1}h^{-\nu _{1}}\right) +O\left( n^{-1/2}\right) ,
\end{eqnarray*}%
since $n^{-1/2}h^{-d-\nu _{1}}\rightarrow 0$. This delivers Claim 2.

Using Claims 1 and 2, and following the arguments in the proof of Lemma B7,
we obtain (Step 1).\medskip

\noindent \textbf{Proof of Step 2:} We can follow the proof of Lemma B6 to
show that%
\begin{eqnarray*}
&&\mathbf{E}\left[ \mathbf{E}^{\ast }\left[ h^{-d/2}\int_{B_{n}}\left(
\Lambda _{A,p}(\sqrt{nh^{d}}\mathbf{z}_{N,\tau }^{\ast }(x))-\mathbf{E}%
^{\ast }\left[ \Lambda _{A,p}(\sqrt{nh^{d}}\mathbf{z}_{N,\tau }^{\ast }(x))%
\right] \right) dQ(x,\tau )\right] ^{2}\right] \\
&=&\mathbf{E}\left[ \int_{\mathcal{T}}\int_{\mathcal{T}}\int_{B_{n,\tau
_{1}}\cap B_{n,\tau _{2}}}\int_{\mathcal{U}}C_{n,\tau _{1},\tau
_{2},A,A^{\prime }}^{\ast }(x,u)dudxd\tau _{1}d\tau _{2}\right] +o(1) \\
&\leq &C\int_{\mathcal{T}}\int_{\mathcal{T}}\int_{B_{n,\tau _{1}}\cap
B_{n,\tau _{2}}}dxd\tau _{1}d\tau _{2}+o(1)\leq CQ(B_{n}),
\end{eqnarray*}%
where $C_{n,\tau _{1},\tau _{2},A,A^{\prime }}^{\ast }(x,v)$ is as defined
in (\ref{C*}). We obtain the desired result of Step 2.
\end{proof}

Let $\mathcal{C}\subset \mathbf{R}^{d},$ $\alpha _{P}\equiv P\{X\in \mathbf{R%
}^{d}\backslash \mathcal{C}\}$ and $B_{n,A}(c_{n};\mathcal{C})\ $be as
introduced prior to Lemma B8. Define%
\begin{eqnarray*}
\zeta _{n,A}^{\ast } &\equiv &\int_{B_{n,A}(c_{n};\mathcal{C})}\Lambda
_{A,p}(\sqrt{nh^{d}}\mathbf{z}_{n,\tau }^{\ast }(x))dQ(x,\tau ),\text{ and}
\\
\zeta _{N,A}^{\ast } &\equiv &\int_{B_{n,A}(c_{n};\mathcal{C})}\Lambda
_{A,p}(\sqrt{nh^{d}}\mathbf{z}_{N,\tau }^{\ast }(x))dQ(x,\tau ).
\end{eqnarray*}%
Let $\mu _{A}$'s be real numbers indexed by $A\subset \mathbb{N}_{J}$. We
also define $B_{n,A}(c_{n};\mathcal{C})$ as prior to Lemma B8 and let 
\begin{eqnarray*}
S_{n}^{\ast } &\equiv &h^{-d/2}\sum_{A\in \mathcal{N}_{J}}\mu _{A}\left\{
\zeta _{N,A}^{\ast }-\mathbf{E}^{\ast }\zeta _{N,A}^{\ast }\right\} , \\
U_{n}^{\ast } &\equiv &\frac{1}{\sqrt{n}}\left\{
\sum_{i=1}^{N}1\{X_{i}^{\ast }\in \mathcal{C}\}-nP^{\ast }\left\{
X_{i}^{\ast }\in \mathcal{C}\right\} \right\} ,\text{ and} \\
V_{n}^{\ast } &\equiv &\frac{1}{\sqrt{n}}\left\{
\sum_{i=1}^{N}1\{X_{i}^{\ast }\in \mathbf{R}^{d}\backslash \mathcal{C}%
\}-nP^{\ast }\left\{ X_{i}^{\ast }\in \mathbf{R}^{d}\backslash \mathcal{C}%
\right\} \right\} .
\end{eqnarray*}%
We let 
\begin{equation*}
H_{n}^{\ast }\equiv \left[ \frac{S_{n}^{\ast }}{\sigma _{n}(\mathcal{C})},%
\frac{U_{n}^{\ast }}{\sqrt{1-\alpha _{P}}}\right] .
\end{equation*}%
The following lemma is a bootstrap counterpart of Lemma B8.

\begin{LemmaD}
\label{lem-d8} \textit{Suppose that the conditions of Lemma C6 hold and that 
}$c_{n}\rightarrow \infty $, \textit{as }$n\rightarrow \infty . $

\noindent (i)\textit{\ If } $\liminf_{n\rightarrow \infty }\inf_{P\in 
\mathcal{P}}\sigma _{n}^{2}(\mathcal{C})>0$, \textit{then for all }$a>0,$ 
\begin{equation*}
\sup_{P\in \mathcal{P}}P\left\{ \sup_{t\in \mathbf{R}^{2}}\left\vert P^{\ast
}\left\{ H_{n}^{\ast }\leq t\right\} -P\left\{ \mathbb{Z}\leq t\right\}
\right\vert >a\right\} \rightarrow 0.
\end{equation*}

\noindent (ii) \textit{If } $\limsup_{n\rightarrow \infty }\sigma _{n}^{2}(%
\mathcal{C})=0$, \textit{then, for each }$(t_{1},t_{2})\in \mathbf{R}^{2}$ 
\textit{and }$a>0,$ 
\begin{equation*}
\sup_{P\in \mathcal{P}}P\left\{ \left\vert P^{\ast }\left\{ S_{n}^{\ast
}\leq t_{1}\text{ and }\frac{U_{n}^{\ast }}{\sqrt{1-\alpha _{P}}}\leq
t_{2}\right\} -1\left\{ 0\leq t_{1}\right\} P\left\{ \mathbb{Z}_{1}\leq
t_{2}\right\} \right\vert >a\right\} \rightarrow 0.
\end{equation*}
\end{LemmaD}

\begin{proof}[Proof of Lemma C\protect\ref{lem-d8}]
Similarly as in the proof of Lemma C8, we fix $\bar{\varepsilon}>0$ and let%
\begin{equation*}
H_{n,\bar{\varepsilon}}^{\ast }\equiv \left[ \frac{S_{n,\bar{\varepsilon}%
}^{\ast }}{\sigma _{n,\bar{\varepsilon}}(\mathcal{C})},\frac{U_{n}^{\ast }}{%
\sqrt{1-\alpha _{P}}}\right] ^{\top },
\end{equation*}%
where $S_{n,\bar{\varepsilon}}^{\ast }$ is equal to $S_{n}^{\ast }$, except
that $\zeta _{N,A}^{\ast }$ is replaced by%
\begin{equation*}
\zeta _{N,A,\bar{\varepsilon}}^{\ast }\equiv \int_{B_{n,A}(c_{n};\mathcal{C}%
)}\Lambda _{A,p}(\sqrt{nh^{d}}\mathbf{z}_{N,\tau }^{\ast }(x;\eta
_{1}))dQ(x,\tau ),
\end{equation*}%
and $\mathbf{z}_{N,\tau }^{\ast }(x;\eta _{1})$ is as defined prior to Lemma
C6. Also let 
\begin{equation*}
\tilde{C}_{n}\equiv \mathbf{E}^{\ast }H_{n}^{\ast }H_{n}^{\ast \top }\text{
and }\tilde{C}_{n,\bar{\varepsilon}}\equiv \mathbf{E}^{\ast }H_{n,\bar{%
\varepsilon}}^{\ast }H_{n,\bar{\varepsilon}}^{\ast \top }\text{.}
\end{equation*}%
First, we show the following statements.\medskip

\noindent \textbf{Step 1: }sup$_{P\in \mathcal{P}}P\left\{ |Cov^{\ast }(S_{n,%
\bar{\varepsilon}}^{\ast }-S_{n}^{\ast },U_{n}^{\ast })|>M\sqrt{\bar{%
\varepsilon}}\right\} \rightarrow 0,$ as $n\rightarrow \infty $ and $%
M\rightarrow \infty $.\medskip

\noindent \textbf{Step 2: }For any $a>0$, sup$_{P\in \mathcal{P}}P\left\{
\left\vert Cov(S_{n,\bar{\varepsilon}}^{\ast },U_{n}^{\ast })\right\vert
>ah^{d/2}\right\} \rightarrow 0,$ as $n\rightarrow \infty .$\medskip

\noindent \textbf{Step 3:\ }There exists $c>0$ such that from some large $n$
on,%
\begin{equation*}
\inf_{P\in \mathcal{P}}\lambda _{\min }(\tilde{C}_{n})>c.
\end{equation*}

\noindent \textbf{Step 4: }For any $a>0$, as $n\rightarrow \infty ,$%
\begin{equation*}
\sup_{P\in \mathcal{P}}P\left\{ \sup_{t\in \mathbf{R}^{2}}\left\vert P^{\ast
}\left\{ \tilde{C}_{n}^{-1/2}H_{n}^{\ast }\leq t\right\} \rightarrow
P\left\{ \mathbb{Z}\leq t\right\} \right\vert >a\right\} \rightarrow 0\text{.%
}
\end{equation*}

Combining Steps 1-4, we obtain (i) of Lemma B8.\medskip

\noindent \textbf{Proof of Step 1:} Observe that 
\begin{equation*}
\left\vert \zeta _{N,A,\bar{\varepsilon}}^{\ast }-\zeta _{N,A}^{\ast
}\right\vert \leq C||\eta _{1}||\int_{B_{n,A}(c_{n};\mathcal{C})}\left\Vert 
\sqrt{nh^{d}}\mathbf{z}_{N,\tau }^{\ast }(x)\right\Vert ^{p-1}dQ(x,\tau ).
\end{equation*}%
As in the proof of Step 1 in the proof of Lemma B8, we deduce that 
\begin{equation*}
\mathbf{E}^{\ast }\left[ \left\vert \zeta _{N,A,\bar{\varepsilon}}^{\ast
}-\zeta _{N,A}^{\ast }\right\vert ^{2}\right] \leq C\bar{\varepsilon}%
\int_{B_{n,A}(c_{n};\mathcal{C})}\mathbf{E}^{\ast }\left\Vert \sqrt{nh^{d}}%
\mathbf{z}_{N,\tau }^{\ast }(x)\right\Vert ^{2p-2}dQ(x,\tau ).
\end{equation*}%
Hence for some $C_{1},C_{2}>0,$%
\begin{eqnarray}
&&\mathbf{E}\left( \mathbf{E}^{\ast }\left[ \left\vert \zeta _{N,A,\bar{%
\varepsilon}}^{\ast }-\zeta _{N,A}^{\ast }\right\vert ^{2}\right] \right)
\label{dev6} \\
&\leq &C\bar{\varepsilon}\int_{B_{n,A}(c_{n};\mathcal{C})}\mathbf{E}\left( 
\mathbf{E}^{\ast }\left\Vert \sqrt{nh^{d}}\mathbf{z}_{N,\tau }^{\ast
}(x)\right\Vert ^{2p-2}\right) dQ(x,\tau )\leq C_{2}\bar{\varepsilon}  \notag
\end{eqnarray}%
by the second statement of Lemma C5.

On the other hand, observe that $\mathbf{E}^{\ast }U_{n}^{\ast 2}\leq 1.$
Hence%
\begin{equation*}
P\left\{ |Cov^{\ast }(S_{n,\bar{\varepsilon}}^{\ast }-S_{n}^{\ast
},U_{n}^{\ast })|>M\sqrt{\bar{\varepsilon}}\right\} \leq |\mathcal{N}%
_{J}|\cdot P\left\{ \max_{A\in \mathcal{N}_{J}}\mathbf{E}^{\ast }\left[
\left\vert \zeta _{N,A,\bar{\varepsilon}}^{\ast }-\zeta _{N,A}^{\ast
}\right\vert ^{2}\right] >M^{2}\bar{\varepsilon}\right\} .
\end{equation*}%
By Chebychev's inequality, the last probability is bounded by (for some $C>0$
that does not depend on $P\in \mathcal{P}$)%
\begin{equation*}
M^{-2}\bar{\varepsilon}^{-1}\sum_{A\in \mathcal{N}_{J}}\mathbf{E}\left( 
\mathbf{E}^{\ast }\left[ \left\vert \zeta _{N,A,\bar{\varepsilon}}^{\ast
}-\zeta _{N,A}^{\ast }\right\vert ^{2}\right] \right) \leq CM^{-2},
\end{equation*}%
by (\ref{dev6}). Hence we obtain the desired result.\medskip

\noindent \textbf{Proof of Step 2:} Let $\tilde{\Sigma}_{2n,\tau ,\bar{%
\varepsilon}}^{\ast }$ be the covariance matrix of $[(q_{n,\tau }^{\ast
}(x)+\eta _{1})^{\top },\tilde{U}_{n}^{\ast }]^{\top }$ under $P^{\ast }$,
where $\tilde{U}_{n}^{\ast }=U_{n}^{\ast }/\sqrt{P\{X\in \mathcal{C}\}}.$
Using Lemma C4 and following the same arguments in (\ref{dev7}), we find that%
\begin{equation*}
\sup_{(x,\tau )\in \mathcal{S}}\sup_{P\in \mathcal{P}}\mathbf{E}\left[ 
\mathbf{E}^{\ast }\left[ q_{n,\tau ,j}^{\ast }(x)\tilde{U}_{n}^{\ast }\right]
\right] \leq C_{2}h^{d/2},
\end{equation*}%
for some $C_{2}>0$. Therefore, using this result and following the proof of
Step 3 in the proof of Lemma B8, we deduce that (everywhere)%
\begin{equation}
\lambda _{\min }\left( \tilde{\Sigma}_{2n,\tau ,\bar{\varepsilon}}^{\ast
}\right) \geq \bar{\varepsilon}-\left\Vert A_{n,\tau }^{\ast }(x)\right\Vert
,  \label{ineq37}
\end{equation}%
for some random matrix $A_{n,\tau }^{\ast }(x)$ such that 
\begin{equation*}
\sup_{(x,\tau )\in \mathcal{S}}\sup_{P\in \mathcal{P}}\mathbf{E}\left[
\left\Vert A_{n,\tau }^{\ast }(x)\right\Vert \right] =O(h^{d/2}).
\end{equation*}%
Hence by (\ref{ineq37}), 
\begin{eqnarray}
&&\inf_{(x,\tau )\in \mathcal{S}}\inf_{P\in \mathcal{P}}P\left\{ \lambda
_{\min }\left( \tilde{\Sigma}_{2n,\tau ,\bar{\varepsilon}}^{\ast }\right)
\geq \bar{\varepsilon}/2\right\}  \label{lb3} \\
&\geq &\inf_{(x,\tau )\in \mathcal{S}}\inf_{P\in \mathcal{P}}P\left\{
\left\Vert A_{n,\tau }^{\ast }(x)\right\Vert \leq \bar{\varepsilon}/2\right\}
\notag \\
&\geq &1-\frac{2}{\bar{\varepsilon}}\sup_{(x,\tau )\in \mathcal{S}%
}\sup_{P\in \mathcal{P}}\mathbf{E}\left[ \left\Vert A_{n,\tau }^{\ast
}(x)\right\Vert \right] \rightarrow 1,  \notag
\end{eqnarray}%
as $n\rightarrow \infty $.

Now note that 
\begin{equation*}
\left( q_{n,\tau ,j}^{\ast }(x),\tilde{U}_{n}^{\ast }\right) \overset{%
d^{\ast }}{=}\left( \frac{1}{\sqrt{n}}\sum_{k=1}^{n}q_{n,\tau ,j}^{(k)\ast
}(x),\frac{1}{\sqrt{n}}\sum_{k=1}^{n}\tilde{U}_{n}^{(k)\ast }\right) ,
\end{equation*}%
where $(q_{n,\tau ,j}^{(k)\ast }(x),\tilde{U}_{n}^{(k)\ast })$'s with $%
k=1,\cdot \cdot \cdot ,n$ are i.i.d. copies of $(q_{n,\tau ,j}^{\ast }(x),%
\bar{U}_{n}^{\ast })$, and 
\begin{equation*}
\bar{U}_{n}^{\ast }\equiv \frac{1}{\sqrt{nP\{X\in \mathcal{C}\}}}\left\{
\sum_{1\leq i\leq N_{1}}1\{X_{i}^{\ast }\in \mathcal{C}\}-P^{\ast }\left\{
X_{i}^{\ast }\in \mathcal{C}\right\} \right\} .
\end{equation*}%
Note also that by Rosenthal's inequality, 
\begin{equation*}
\text{limsup}_{n\rightarrow \infty }\sup_{P\in \mathcal{P}}P\left\{ \mathbf{E%
}^{\ast }\left[ |\tilde{U}_{n}^{(k)\ast }|^{3}\right] >M\right\} \rightarrow
0,
\end{equation*}%
as $M\rightarrow \infty $. Define%
\begin{equation*}
W_{n,\tau }^{\ast }(x;\eta _{1})\equiv \tilde{\Sigma}_{2n,\tau ,\bar{%
\varepsilon}}^{\ast -1/2}\left[ 
\begin{array}{c}
q_{n,\tau }^{\ast }(x)+\eta _{1} \\ 
\tilde{U}_{n}^{\ast }%
\end{array}%
\right] .
\end{equation*}%
Using (\ref{lb3}) and Lemma C5, and following the same arguments in the
proof of Step 2 in the proof of Lemma B8, we deduce that%
\begin{equation*}
\text{limsup}_{n\rightarrow \infty }\sup_{(x,\tau )\in \mathcal{S}%
}\sup_{P\in \mathcal{P}}P\left\{ \mathbf{E}^{\ast }\left\Vert W_{n,\tau
}^{\ast }(x;\eta _{1})\right\Vert ^{3}>M\bar{\varepsilon}^{-3/2}h^{-d/2}%
\right\} \rightarrow 0,
\end{equation*}%
as $M\rightarrow \infty $. For any vector $\mathbf{v}=[\mathbf{v}_{1}^{\top
},v_{2}]^{\top }\in \mathbf{R}^{J+1}$, we define%
\begin{equation*}
\tilde{D}_{n,\tau ,p}(\mathbf{v})\equiv \Lambda _{p}\left( \left[ \tilde{%
\Sigma}_{2n,\tau ,\bar{\varepsilon}}^{\ast 1/2}\mathbf{v}\right] _{1}\right) %
\left[ \tilde{\Sigma}_{2n,\tau ,\bar{\varepsilon}}^{\ast 1/2}\mathbf{v}%
\right] _{2},
\end{equation*}%
where $[a]_{1}$ of a vector $a\in \mathbf{R}^{J+1}$ indicates the vector of
the first $J$ entries of $a$, and $[a]_{2}$ the last entry of $a.$ By
Theorem 1 of \citeasnoun{Sweeting:77}, we find that (with $\bar{\varepsilon}%
>0$ fixed)%
\begin{equation*}
\mathbf{E}^{\ast }\left[ \tilde{D}_{n,\tau ,p}\left( \frac{1}{\sqrt{n}}%
\sum_{i=1}^{n}W_{n,\tau }^{(i)\ast }(x;\eta _{1})\right) \right] =\mathbf{E}%
\left[ \tilde{D}_{n,\tau ,p}\left( \mathbb{Z}_{J+1}\right) \right]
+O_{P}(n^{-1/2}h^{-d/2}),\ \mathcal{P}\text{-uniformly,}
\end{equation*}%
where $\mathbb{Z}_{J+1}\sim N(0,I_{J+1})$ and $W_{n,\tau }^{(i)\ast }(x;\eta
_{1})$'s are i.i.d. copies of $W_{n,\tau }^{\ast }(x;\eta _{1})$ under $%
P^{\ast }$. Since $O(n^{-1/2}h^{-d/2})=o(h^{d/2})$,%
\begin{equation*}
Cov^{\ast }\left( \Lambda _{A,p}\left( \sqrt{nh^{d}}\mathbf{z}_{N,\tau
}^{\ast }(x;\eta _{1})\right) ,U_{n}^{\ast }\right) =\mathbf{E}^{\ast }\left[
\tilde{D}_{n,\tau ,p}\left( \frac{1}{\sqrt{n}}\sum_{i=1}^{n}W_{n,\tau
}^{(i)\ast }(x)\right) \right] +o_{P}(h^{d/2}),
\end{equation*}%
uniformly in $P\in \mathcal{P}$, and that $\mathbf{E}^{\ast }[\tilde{D}%
_{n,\tau ,p}\left( \mathbb{Z}_{J+1}\right) ]=0$, we conclude that 
\begin{equation}
\sup_{(x,\tau )\in \mathcal{S}}\left\vert Cov^{\ast }\left( \Lambda
_{A,p}\left( \sqrt{nh^{d}}\mathbf{z}_{N,\tau }^{\ast }(x;\eta _{1})\right)
,U_{n}^{\ast }\right) \right\vert =o_{P}(h^{d/2}),  \label{cv11}
\end{equation}%
uniformly in $P\in \mathcal{P}$.

Now for some $C>0,$%
\begin{equation*}
P\left\{ \left\vert Cov(S_{n,\bar{\varepsilon}}^{\ast },U_{n}^{\ast
})\right\vert >ah^{d/2}\right\} \leq P\left\{ C\sup_{(x,\tau )\in \mathcal{S}%
}\left\vert Cov^{\ast }\left( \Lambda _{A,p}\left( \sqrt{nh^{d}}\mathbf{z}%
_{N,\tau }^{\ast }(x;\eta _{1})\right) ,U_{n}^{\ast }\right) \right\vert
>ah^{d/2}\right\} .
\end{equation*}%
The last probability vanishes uniformly in $P\in \mathcal{P}$ by (\ref{cv11}%
). By applying the Dominated Convergence Theorem, we obtain (Step 2).\medskip

\noindent \textbf{Proof of Step 3:} First, we show that 
\begin{equation}
Var^{\ast }\left( S_{n}^{\ast }\right) =\sigma _{n}^{2}(\mathcal{C}%
)+o_{P}(1),  \label{varvar}
\end{equation}%
where $o_{P}(1)$ is uniform over $P\in \mathcal{P}$. Note that%
\begin{equation*}
Var^{\ast }\left( S_{n}^{\ast }\right) =\sum_{A\in \mathcal{N}%
_{J}}\sum_{A^{\prime }\in \mathcal{N}_{J}}\mu _{A}\mu _{A^{\prime
}}Cov^{\ast }(\psi _{n,A}^{\ast },\psi _{n,A^{\prime }}^{\ast }),
\end{equation*}%
where $\psi _{n,A}^{\ast }\equiv h^{-d/2}(\zeta _{N,A}^{\ast }-\mathbf{E}%
^{\ast }\zeta _{N,A}^{\ast })$. By Lemma C6, we find that for $A,A\in 
\mathcal{N}_{J},$ 
\begin{equation*}
Cov^{\ast }(\psi _{n,A}^{\ast },\psi _{n,A^{\prime }}^{\ast })=\sigma
_{n,A,A^{\prime }}(B_{n,A}(c_{n};\mathcal{C}),B_{n,A^{\prime }}(c_{n};%
\mathcal{C}))+o_{P}(1),
\end{equation*}%
uniformly in $P\in \mathcal{P}$, yielding the desired result of (\ref{varvar}%
).

Combining Steps 1 and 2, we deduce that for some $C>0$,\textbf{\ }%
\begin{equation*}
\sup_{P\in \mathcal{P}}\left\vert Cov^{\ast }(S_{n}^{\ast },U_{n}^{\ast
})\right\vert \leq \sqrt{\bar{\varepsilon}}\cdot O_{P}(1)+o_{P}(h^{d/2}).
\end{equation*}%
Let $\tilde{\sigma}_{1}^{2}\equiv Var^{\ast }(S_{n}^{\ast })$ and $\tilde{%
\sigma}_{2}^{2}\equiv 1-\tilde{\alpha}_{P}$, where $\tilde{\alpha}_{P}\equiv
P^{\ast }\left\{ X_{i}^{\ast }\in \mathbf{R}^{d}\backslash \mathcal{C}%
\right\} $. Observe that%
\begin{equation*}
\tilde{\sigma}_{1}^{2}=\sigma _{n}(\mathcal{C})+o_{P}(1)>C_{1}+o_{P}(1),\ 
\mathcal{P}\text{-uniformly,}
\end{equation*}%
for some $C_{1}>0$ that does not depend on $n$ or $P$ by the assumption of
the lemma. Also note that 
\begin{equation*}
\tilde{\alpha}_{P}=\alpha _{P}+o_{P}(1)<1-C_{2}+o_{P}(1),\ \mathcal{P}\text{%
-uniformly,}
\end{equation*}%
for some $C_{2}>0$. Therefore, following the same arguments as in (\ref%
{dev67}), we obtain the desired result.\medskip

\noindent \textbf{Proof of Step 4:} We take $\{R_{n,\mathbf{i}}:\mathbf{i}%
\in \mathbb{Z}^{d}\}$, and define%
\begin{eqnarray*}
B_{A,x}(c_{n}) &\equiv &\left\{ \tau \in \mathcal{T}:(x,\tau )\in
B_{A}(c_{n})\right\} , \\
B_{n,\mathbf{i}} &\equiv &R_{n,\mathbf{i}}\cap \mathcal{C}\text{,} \\
B_{n,A,\mathbf{i}}(c_{n}) &\equiv &(B_{n,\mathbf{i}}\times \mathcal{T})\cap
B_{A}(c_{n}),
\end{eqnarray*}%
and $\mathcal{I}_{n}\equiv \{\mathbf{i}\in \mathbb{Z}_{n}^{d}:B_{n,\mathbf{i}%
}\neq \varnothing \}$ as in the proof of Step 4 in the proof of Lemma B8.
Also, define%
\begin{equation*}
\Delta _{n,A,\mathbf{i}}^{\ast }\equiv h^{-d/2}\int_{B_{n,\mathbf{i}%
}}\int_{B_{A,x}(c_{n})}\left\{ \Lambda _{A,p}(\mathbf{z}_{N,\tau }^{\ast
}(x))-\mathbf{E}^{\ast }\left[ \Lambda _{A,p}(\mathbf{z}_{N,\tau }^{\ast
}(x))\right] \right\} d\tau dx.
\end{equation*}%
Also, define%
\begin{eqnarray*}
\alpha _{n,\mathbf{i}}^{\ast } &\equiv &\frac{\sum_{A\in \mathcal{N}_{J}}\mu
_{A}\Delta _{n,A,\mathbf{i}}^{\ast }}{\sqrt{Var^{\ast }\left( S_{n}^{\ast
}\right) }}\mathbf{\ }\text{and} \\
u_{n,\mathbf{i}}^{\ast } &\equiv &\frac{1}{\sqrt{n}}\left\{
\sum_{i=1}^{N}1\left\{ X_{i}^{\ast }\in B_{n,\mathbf{i}}\right\} -nP^{\ast
}\{X_{i}^{\ast }\in B_{n,\mathbf{i}}\}\right\}
\end{eqnarray*}%
and write%
\begin{equation*}
\frac{S_{n}^{\ast }}{\sqrt{Var^{\ast }\left( S_{n}^{\ast }\right) }}=\sum_{%
\mathbf{i}\in \mathcal{I}_{n}}\alpha _{n,\mathbf{i}}^{\ast }\text{ and }%
U_{n}^{\ast }=\sum_{\mathbf{i}\in \mathcal{I}_{n}}u_{n,\mathbf{i}}^{\ast }.
\end{equation*}%
By the properties of Poisson processes, one can see that the array $%
\{(\alpha _{n,\mathbf{i}}^{\ast },u_{n,\mathbf{i}}^{\ast })\}_{\mathbf{i}\in 
\mathcal{I}_{n}}$ is an array of $1$-dependent random field under $P^{\ast }$%
. For any $q=(q_{1},q_{2})\in \mathbf{R}^{2}\backslash \{0\}$, let $y_{n,%
\mathbf{i}}^{\ast }\equiv q_{1}\alpha _{n,\mathbf{i}}^{\ast }+q_{2}u_{n,%
\mathbf{i}}^{\ast }$ and write%
\begin{equation*}
Var^{\ast }\left( \sum_{\mathbf{i}\in \mathcal{I}_{n}}y_{n,\mathbf{i}}^{\ast
}\right) =q_{1}^{2}+q_{2}^{2}(1-\tilde{\alpha}_{P})+2q_{1}q_{2}\tilde{c}%
_{n,P},
\end{equation*}%
uniformly over $P\in \mathcal{P}$, where $\tilde{c}_{n,P}=Cov^{\ast
}(S_{n}^{\ast },U_{n}^{\ast })$. On the other hand, following the proof of
Lemma A8 of \citeasnoun{LSW} using Lemma C4, we deduce that%
\begin{equation}
\sum_{\mathbf{i}\in \mathcal{I}_{n}}\mathbf{E}^{\ast }|y_{n,\mathbf{i}%
}^{\ast }|^{r}=o_{P}(1),\ \mathcal{P}\text{-uniformly,}  \label{conv_y2}
\end{equation}%
as $n\rightarrow \infty $, for any $r\in (2,(2p+2)/p]$, uniformly over $P\in 
\mathcal{P}$. By Theorem 1 of \citeasnoun{Shergin:93}, we have%
\begin{eqnarray*}
&&\sup_{t\in \mathbf{R}}\left\vert P^{\ast }\left\{ \frac{1}{\sqrt{%
q_{1}^{2}+q_{2}^{2}(1-\tilde{\alpha}_{P})+2q_{1}q_{2}\tilde{c}_{n,P}}}\sum_{%
\mathbf{i}\in \mathcal{I}_{n}}y_{n,\mathbf{i}}^{\ast }\leq t\right\} -\Phi
^{\ast }\left( t\right) \right\vert \\
&\leq &\frac{C}{\left\{ q_{1}^{2}+q_{2}^{2}(1-\tilde{\alpha}_{P})+2q_{1}q_{2}%
\tilde{c}_{n,P}\right\} ^{r/2}}\left\{ \sum_{\mathbf{i}\in \mathcal{I}_{n}}%
\mathbf{E}^{\ast }|y_{n,\mathbf{i}}^{\ast }|^{r}\right\} ^{1/2}=o_{P}(1),
\end{eqnarray*}%
for some $C>0$ uniformly in $P\in \mathcal{P}$, by (\ref{conv_y2}). By Lemma
B2(i), we have for each $t\in \mathbf{R}$ and $q\in \mathbf{R}^{2}\backslash
\{\mathbf{0}\}$ as $n\rightarrow \infty ,$%
\begin{equation*}
\left\vert \mathbf{E}^{\ast }\left[ \exp \left( it\frac{q^{\top }H_{n}^{\ast
}}{\sqrt{q^{\top }\tilde{C}_{n}q}}\right) \right] -\exp \left( -\frac{t^{2}}{%
2}\right) \right\vert =o_{P}(1),
\end{equation*}%
uniformly in $P\in \mathcal{P}.$ Thus by Lemma B2(ii), for each $t\in 
\mathbf{R}^{2},$ we have%
\begin{equation*}
\left\vert P^{\ast }\left\{ \tilde{C}_{n}^{-1/2}H_{n}^{\ast }\leq t\right\}
-P\left\{ \mathbb{Z}\leq t\right\} \right\vert =o_{P}(1).
\end{equation*}%
Since the limit distribution of $\tilde{C}_{n}^{-1/2}H_{n}^{\ast }$ is
continuous, the convergence above is uniform in $t\in \mathbf{R}^{2}$.

\noindent (ii) We fix $P\in \mathcal{P}$ such that limsup$_{n\rightarrow
\infty }\sigma _{n}^{2}(\mathcal{C})=0.$ Then by (\ref{varvar}) above and
Lemma C6,%
\begin{equation*}
Var^{\ast }\left( S_{n}^{\ast }\right) =\sigma _{n}^{2}(\mathcal{C}%
)+o_{P}(1)=o_{P}(1).
\end{equation*}%
Hence, we find that $S_{n}^{\ast }=o_{P^{\ast }}(1)$ in $P$. The desired
result follows by applying Theorem 1 of \citeasnoun{Shergin:93} to the sum $%
U_{n}^{\ast }=\sum_{\mathbf{i}\in \mathcal{I}_{n}}u_{n,\mathbf{i}}^{\ast }$,
and then applying Lemma B2.
\end{proof}

\begin{LemmaD}
\label{lem-d9} \textit{Let }$\mathcal{C}$\textit{\ be the Borel set in Lemma
C8}.

\noindent \textit{(i) Suppose that the conditions of Lemma C8(i) are
satisfied. Then for each }$a>0$, \textit{as }$n\rightarrow \infty ,$%
\begin{equation*}
\sup_{P\in \mathcal{P}}P\left\{ \sup_{t\in \mathbf{R}}\left\vert P\left\{ 
\frac{h^{-d/2}\sum_{A\in \mathcal{N}_{J}}\mu _{A}\left\{ \zeta _{n,A}^{\ast
}-\mathbf{E}^{\ast }\zeta _{N,A}^{\ast }\right\} }{\sigma _{n}(\mathcal{C})}%
\leq t\right\} -\Phi (t)\right\vert >a\right\} \rightarrow 0.
\end{equation*}%
\noindent \textit{(ii) Suppose that the conditions of Lemma C8(ii) are
satisfied. Then for each }$a>0,$ \textit{as }$n\rightarrow \infty ,$%
\begin{equation*}
\sup_{P\in \mathcal{P}}P\left\{ \left\vert h^{-d/2}\sum_{A\in \mathcal{N}%
_{J}}\mu _{A}\left\{ \zeta _{n,A}^{\ast }-\mathbf{E}^{\ast }\zeta
_{N,A}^{\ast }\right\} \right\vert >a\right\} \rightarrow 0.
\end{equation*}
\end{LemmaD}

\begin{proof}[Proof of Lemma C\protect\ref{lem-d9}]
The proofs are precisely the same as those of Lemma B9, except that we use
Lemma C8 instead of Lemma B8 here.
\end{proof}

\begin{LemmaD}
\label{lem-d10} \textit{Suppose that the conditions of Lemma B5 hold. Then
for any small }$\nu >0,$\textit{\ there exists a positive sequence }$%
\varepsilon _{n}=o(h^{d})$\textit{\ such that for all }$r\in \lbrack 2,M/2]$%
\textit{\ (with }$M>0$\textit{\ being as in Assumption A6(i)),}%
\begin{equation*}
\sup_{(x,\tau )\in \mathcal{S}}\sup_{P\in \mathcal{P}}\mathbf{E}||\Sigma
_{n,\tau ,\varepsilon _{n}}^{-1/2}(x)q_{n,\tau }(x;\eta _{n})||^{r}=O\left(
h^{-(r-2)\left( \frac{M-1}{M-2}\right) d-\nu }\right) ,
\end{equation*}%
\textit{where }$\eta _{n}\in \mathbf{R}^{J}$ \textit{is distributed as} $%
N(0,\varepsilon _{n}I_{J})$\textit{\ and independent of }$((Y_{i}^{\top
},X_{i}^{\top })_{i=1}^{\infty },N)$\textit{\ in the definition of }$%
q_{n,\tau }(x)$,\textit{\ and} 
\begin{equation}
\Sigma _{n,\tau ,\varepsilon _{n}}(x)\equiv \Sigma _{n,\tau ,\tau
}(x,0)+\varepsilon _{n}I_{J}\text{ and }q_{n,\tau }(x;\eta _{n})\equiv
q_{n,\tau }(x)+\eta _{n}\text{.}  \label{SIGMA}
\end{equation}%
\textit{Suppose furthermore that }$\lambda _{\min }(\Sigma _{n,\tau ,\tau
}(x,0))>c>0$\textit{\ for some }$c>0$\textit{\ that does not depend on }$n$%
\textit{\ or }$P\in \mathcal{P}$.\textit{\ Then}%
\begin{equation*}
\sup_{(x,\tau )\in \mathcal{S}}\sup_{P\in \mathcal{P}}\mathbf{E}||\Sigma
_{n,\tau ,\varepsilon _{n}}^{-1/2}(x)q_{n,\tau }(x;\eta _{n})||^{r}=O\left(
h^{-(r-2)d/2}\right) .
\end{equation*}
\end{LemmaD}

\begin{proof}[Proof of Lemma C\protect\ref{lem-d10}]
We first establish the following fact.

\noindent \textbf{Fact:} Suppose that $W$ is a random vector such that $%
\mathbf{E}||W||^{2}\leq c_{W}$ for some constant $c_{W}>0$. Then, for any $%
r\geq 2$ and a positive integer $m\geq 1,$ 
\begin{equation*}
\mathbf{E}\left[ ||W||^{r}\right] \leq C_{m}\left( \mathbf{E}\left[
||W||^{a_{m}(r)}\right] \right) ^{1/(2^{m})},
\end{equation*}%
where $a_{m}(r)=2^{m}(r-2)+2$, and $C_{m}>0$ is a constant that depends only
on $m$ and $c_{W}$.\medskip \newline
\textbf{Proof of Fact:}\ The result follows by repeated application of
Cauchy-Schwarz inequality:%
\begin{equation*}
\mathbf{E}||W||^{r}\leq \left( \mathbf{E}||W||^{2(r-1)}\right) ^{1/2}\left( 
\mathbf{E}||W||^{2}\right) ^{1/2}\leq c_{W}^{1/2}\left( \mathbf{E}%
||W||^{2(r-1)}\right) ^{1/2},
\end{equation*}%
where we replace $r$ on the left hand side by $2(r-1)$, and repeat the
procedure to obtain Fact.\medskip \newline
Let us consider the first statement of the lemma. Using Fact, we take a
small $\nu _{1}>0$ and $\varepsilon _{n}=h^{d+\nu _{1}}$, and choose a
largest integer $m\geq 1$ such that $a_{m}(r)\leq M$. Such an $m$ exists
because $2\leq r\leq M/2$. We bound%
\begin{equation*}
\mathbf{E}||\Sigma _{n,\tau ,\varepsilon _{n}}^{-1/2}(x)q_{n,\tau }(x;\eta
_{n})||^{r}\leq C_{m}\left( \mathbf{E}||\Sigma _{n,\tau ,\varepsilon
_{n}}^{-1/2}(x)q_{n,\tau }(x;\eta _{n})||^{a_{m}(r)}\right) ^{1/(2^{m})}.
\end{equation*}%
By Lemma B5, we find that 
\begin{eqnarray}
&&\sup_{(x,\tau )\in \mathcal{S}}\sup_{P\in \mathcal{P}}\mathbf{E}||\Sigma
_{n,\tau ,\varepsilon _{n}}^{-1/2}(x)q_{n,\tau }(x;\eta _{n})||^{a_{m}(r)}
\label{ineqs4} \\
&\leq &\sup_{(x,\tau )\in \mathcal{S}}\sup_{P\in \mathcal{P}}\lambda _{\max
}^{a_{m}(r)/2}\left( \Sigma _{n,\tau ,\varepsilon _{n}}^{-1}(x)\right) 
\mathbf{E}||q_{n,\tau }(x;\eta _{n})||^{a_{m}(r)}  \notag \\
&\leq &\lambda _{\min }^{-a_{m}(r)/2}\left( \varepsilon _{n}I_{J}\right)
h^{(1-(a_{m}(r)/2))d}.  \notag
\end{eqnarray}%
By the definition of $\varepsilon _{n}=h^{d+\nu _{1}}$,%
\begin{equation*}
\varepsilon
_{n}^{-a_{m}(r)/2}h^{(1-(a_{m}(r)/2))d}=h^{(1-a_{m}(r))d-a_{m}(r)\nu _{1}/2}.
\end{equation*}%
We conclude that 
\begin{eqnarray*}
\mathbf{E}||\Sigma _{n,\tau ,\varepsilon _{n}}^{-1/2}(x)q_{n,\tau }(x;\eta
_{n})||^{r} &\leq &C_{m}\left( h^{(1-a_{m}(r))d-a_{m}(r)\nu _{1}/2}\right)
^{1/2^{m}} \\
&=&C_{m}\left( h^{(-1-2^{m}(r-2))d-(2^{m}(r-2)+2)\nu _{1}/2}\right)
^{1/2^{m}} \\
&=&C_{m}h^{(-2^{-m}-(r-2))d-((r-2)+2^{-m+1})\nu _{1}/2}.
\end{eqnarray*}%
Since $a_{m}(r)\leq M$, or $2^{-m}\geq (r-2)/(M-2)$, the last term is
bounded by%
\begin{equation*}
C_{m}h^{-(r-2)\left( \frac{M-1}{M-2}\right) d-\left( (r-2)+\frac{2(r-2)}{M-2}%
\right) \nu _{1}/2}.
\end{equation*}%
By taking $\nu _{1}$ small enough, we obtain the desired result.

Now, let us turn to the second statement of the lemma. Since, under the
additional condition, 
\begin{equation*}
\lambda _{\max }^{a_{m}(r)/2}\left( \Sigma _{n,\tau ,\varepsilon
_{n}}^{-1}(x)\right) <c^{-a_{m}(r)/2},
\end{equation*}%
the last bound in (\ref{ineqs4}) turns out to be 
\begin{equation*}
c^{-a_{m}(r)/2}h^{(1-(a_{m}(r)/2))d}.
\end{equation*}%
Therefore, we conclude that 
\begin{eqnarray*}
\mathbf{E}||\Sigma _{n,\tau ,\varepsilon _{n}}^{-1/2}(x)q_{n,\tau }(x;\eta
_{n})||^{r} &\leq &C_{m}\left( c^{-a_{m}(r)/2}h^{(1-(a_{m}(r)/2))d}\right)
^{1/2^{m}} \\
&=&C_{m}c^{-\{(r-2)+2^{1-m}\}/2}h^{(2^{-m}-\{(r-2)+2^{1-m}\}/2)d} \\
&=&C_{m}c^{-\{(r-2)+2^{1-m}\}/2}h^{-(r-2)d/2}.
\end{eqnarray*}%
Again, using the inequality $2^{-m}\geq (r-2)/(M-2)$, we obtain the desired
result.
\end{proof}

\begin{LemmaD}
\label{lem-d11} \textit{Suppose that the conditions of Lemma C5 hold. Then
for any small }$\nu >0,$\textit{\ there exists a positive sequence }$%
\varepsilon _{n}=o(h^{d})$\textit{\ such that for all }$r\in \lbrack 2,M/2]$%
\textit{\ (with }$M>0$\textit{\ being as in Assumption A6(i)),}%
\begin{equation*}
\sup_{(x,\tau )\in \mathcal{S}}\mathbf{E}^{\ast }||\tilde{\Sigma}_{n,\tau
,\varepsilon _{n}}^{-1/2}(x)q_{n,\tau }^{\ast }(x;\eta
_{n})||^{r}=O_{P}\left( h^{-(r-2)\left( \frac{M-1}{M-2}\right) d-\nu
}\right) ,\text{\textit{\ uniformly in }}P\in \mathcal{P,}
\end{equation*}%
\textit{where }$\eta _{n}\in \mathbf{R}^{J}$ \textit{is distributed as} $%
N(0,\varepsilon _{n}I_{J})$\textit{\ and independent of }$((Y_{i}^{\ast \top
},X_{i}^{\ast \top })_{i=1}^{n},(Y_{i}^{\top },X_{i}^{\top })_{i=1}^{n},N)$%
\textit{\ in the definition of }$q_{n,\tau }^{\ast }(x)$,\textit{\ and}%
\begin{equation*}
\tilde{\Sigma}_{n,\tau ,\varepsilon _{n}}(x)\equiv \tilde{\Sigma}_{n,\tau
,\tau }(x,0)+\varepsilon _{n}I_{J}.
\end{equation*}%
\textit{Suppose furthermore that }%
\begin{equation*}
\sup_{(x,\tau )\in \mathcal{S}}\sup_{P\in \mathcal{P}}P\left\{ \lambda
_{\min }(\tilde{\Sigma}_{n,\tau ,\tau }(x,0))>c\right\} \rightarrow 0,
\end{equation*}%
\textit{for some }$c>0$\textit{\ that does not depend on }$n$\textit{\ or }$%
P\in \mathcal{P}$.\textit{\ Then}%
\begin{equation*}
\sup_{(x,\tau )\in \mathcal{S}}\mathbf{E}^{\ast }||\tilde{\Sigma}_{n,\tau
,\varepsilon _{n}}^{-1/2}(x)q_{n,\tau }^{\ast }(x;\eta
_{n})||^{r}=O_{P}\left( h^{-(r-2)d/2}\right) ,\text{\textit{\ uniformly in }}%
P\in \mathcal{P.}
\end{equation*}
\end{LemmaD}

\begin{proof}[Proof of Lemma C\protect\ref{lem-d11}]
The proof is precisely the same as that of Lemma C10, where we use Lemma C5
instead of Lemma B5.
\end{proof}

We let for a sequence of Borel sets $B_{n}$ in $\mathcal{S}$ and $\lambda
\in \{0,d/4,d/2\}$, $A\subset \mathbb{N}_{J}$, and a fixed bounded function $%
\delta $ on $\mathcal{S}$, 
\begin{eqnarray*}
a_{n}^{R}(B_{n}) &\equiv &\int_{B_{n}}\mathbf{E}\left[ \Lambda _{A,p}(\sqrt{%
nh^{d}}\mathbf{z}_{N,\tau }(x)+h^{\lambda }\delta (x,\tau ))\right]
dQ(x,\tau )\text{ } \\
a_{n}^{R\ast }(B_{n}) &\equiv &\int_{B_{n}}\mathbf{E}^{\ast }\left[ \Lambda
_{A,p}(\sqrt{nh^{d}}\mathbf{z}_{N,\tau }^{\ast }(x)+h^{\lambda }\delta
(x,\tau ))\right] dQ(x,\tau ),\text{ and} \\
a_{n}(B_{n}) &\equiv &\int_{B_{n}}\mathbf{E}\left[ \Lambda _{A,p}(\mathbb{W}%
_{n,\tau ,\tau }^{(1)}(x,0)+h^{\lambda }\delta (x,\tau ))\right] dQ(x,\tau ),
\end{eqnarray*}%
where $\mathbf{z}_{N,\tau }^{\ast }(x)$ is a random vector whose $j$-th
entry is given by%
\begin{equation*}
z_{N,\tau ,j}^{\ast }(x)\equiv \frac{1}{nh^{d}}\sum_{i=1}^{N}\beta
_{n,x,\tau ,j}(Y_{ij}^{\ast },(X_{i}^{\ast }-x)/h)-\frac{1}{h^{d}}\mathbf{E}%
^{\ast }\left[ \beta _{n,x,\tau ,j}(Y_{ij}^{\ast },(X_{i}^{\ast }-x)/h)%
\right] \text{.}
\end{equation*}

\begin{LemmaD}
\label{lem-d12} \textit{Suppose that the conditions of Lemmas C10 and C11
hold and that }%
\begin{equation*}
n^{-1/2}h^{-\left( \frac{3M-4}{2M-4}\right) d-\nu }\rightarrow 0,
\end{equation*}%
\textit{\ as }$n\rightarrow \infty $, \textit{for some small }$\nu >0$. 
\textit{Then for any sequence of Borel sets }$B_{n}$ \textit{in} $\mathcal{S}
$\textit{,}%
\begin{eqnarray*}
\sup_{P\in \mathcal{P}}\left\vert a_{n}^{R}(B_{n})-a_{n}(B_{n})\right\vert
&=&o(h^{d/2})\text{ \textit{and}} \\
\sup_{P\in \mathcal{P}}P\left\{ \left\vert a_{n}^{R\ast
}(B_{n})-a_{n}(B_{n})\right\vert >ah^{d/2}\right\} &=&o(1).
\end{eqnarray*}
\end{LemmaD}

\begin{proof}[Proof of Lemma C\protect\ref{lem-d12}]
For the statement, it suffices to show that uniformly in $P\in \mathcal{P}$,%
\begin{eqnarray}
\sup_{(x,\tau )\in \mathcal{S}}\left\vert 
\begin{array}{c}
\mathbf{E}\Lambda _{A,p}(\sqrt{nh^{d}}\mathbf{z}_{N,\tau }(x)+h^{\lambda
}\delta (x,\tau )) \\ 
-\mathbf{E}\Lambda _{A,p}(\mathbb{W}_{n,\tau ,\tau }^{(1)}(x,0)+h^{\lambda
}\delta (x,\tau ))%
\end{array}%
\right\vert &=&o(h^{d/2})\text{,}  \label{convs2} \\
\sup_{(x,\tau )\in \mathcal{S}}\left\vert 
\begin{array}{c}
\mathbf{E}^{\ast }\Lambda _{A,p}(\sqrt{nh^{d}}\mathbf{z}_{N,\tau }^{\ast
}(x)+h^{\lambda }\delta (x,\tau )) \\ 
-\mathbf{E}\Lambda _{A,p}(\mathbb{W}_{n,\tau ,\tau }^{(1)}(x,0)+h^{\lambda
}\delta (x,\tau ))%
\end{array}%
\right\vert &=&o_{P}(h^{d/2}),  \notag
\end{eqnarray}%
uniformly in $P\in \mathcal{P}$. We prove the first statement of (\ref%
{convs2}). The proof of the second statement of (\ref{convs2}) can be done
in a similar way.

Take small $\nu >0.$ We apply Lemma C10 by choosing a positive sequence $%
\varepsilon _{n}=o(h^{d})$ such that for any $r\in \lbrack 2,M/2],$%
\begin{equation}
\sup_{(x,\tau )\in \mathcal{S}}\sup_{P\in \mathcal{P}}\mathbf{E}||\Sigma
_{n,\tau ,\varepsilon _{n}}^{-1/2}(x)q_{n,\tau }(x;\eta _{n})||^{r}=O\left(
h^{-(r-2)\left( \frac{M-1}{M-2}\right) d-\nu }\right) ,  \label{rate}
\end{equation}%
where $q_{n,\tau }(x;\eta _{n})$ and $\Sigma _{n,\tau ,\varepsilon _{n}}(x)$
are as in Lemma C10. We follow the arguments in the proof of Step 2 in Lemma
B6 to bound the left-hand side in the first supremum in (\ref{convs2}) by 
\begin{equation*}
\sup_{(x,\tau )\in \mathcal{S}}\sup_{P\in \mathcal{P}}\left\vert \mathbf{E}%
\Lambda _{A,p}(\sqrt{nh^{d}}\mathbf{z}_{N,\tau }(x;\eta _{n})+h^{\lambda
}\delta (x,\tau ))-\mathbf{E}\Lambda _{A,p}(\mathbb{W}_{n,\tau ,\tau
,\varepsilon _{n}}^{(1)}(x,0)+h^{\lambda }\delta (x,\tau ))\right\vert +C%
\sqrt{\varepsilon _{n}},
\end{equation*}%
for some $C>0$, where\ 
\begin{equation*}
\mathbf{z}_{N,\tau }(x;\eta _{n})\equiv \mathbf{z}_{N,\tau }(x)+\eta _{n}/%
\sqrt{nh^{d}},
\end{equation*}%
and $\mathbb{W}_{n,\tau ,\tau ,\varepsilon _{n}}^{(1)}(x,0)$ is as defined
in (\ref{Ws}). Let%
\begin{eqnarray*}
\mathbf{\xi }_{N,\tau }(x;\eta _{n}) &\equiv &\sqrt{nh^{d}}\Sigma _{n,\tau
,\varepsilon _{n}}^{-1/2}(x)\cdot \mathbf{z}_{N,\tau }(x;\eta _{n})\text{ and%
} \\
\mathbb{Z}_{n,\tau ,\tau ,\varepsilon _{n}}^{(1)}(x,0) &\equiv &\Sigma
_{n,\tau ,\varepsilon _{n}}^{-1/2}(x)\cdot \mathbb{W}_{n,\tau ,\tau
,\varepsilon _{n}}^{(1)}(x,0).
\end{eqnarray*}%
We rewrite the previous absolute value as%
\begin{equation}
\sup_{(x,\tau )\in \mathcal{S}}\sup_{P\in \mathcal{P}}\left\vert \mathbf{E}%
\Lambda _{A,n,p}^{\Sigma }(\sqrt{nh^{d}}\mathbf{\xi }_{N,\tau }(x;\eta
_{n}))-\mathbf{E}\Lambda _{n,p}^{\Sigma }(\mathbb{Z}_{n,\tau ,\tau
,\varepsilon _{n}}^{(1)}(x,0))\right\vert ,  \label{bdd4}
\end{equation}%
where $\Lambda _{A,n,p}^{\Sigma }(\mathbf{v})\equiv \Lambda _{A,p}(\Sigma
_{n,\tau ,\varepsilon _{n}}^{1/2}(x)\mathbf{v}+h^{\lambda }\delta (x,\tau ))$%
. Note that the condition for $M$ in Assumption A6(i) that $M\geq 2(p+2)$,
we can choose $r=\max \{p,3\}$. Then $r\in \lbrack 2,M/2]$ as required.
Using Theorem 1 of \citeasnoun{Sweeting:77}, we bound the above supremum by
(with $r=\max \{p,3\}$)%
\begin{eqnarray*}
&&\frac{C_{1}}{\sqrt{n}}\sup_{(x,\tau )\in \mathcal{S}}\sup_{P\in \mathcal{P}%
}\mathbf{E}||\Sigma _{n,\tau ,\varepsilon _{n}}^{-1/2}(x)q_{n,\tau }(x;\eta
_{n})||^{3} \\
&&+\frac{C_{2}}{\sqrt{n^{r-2}}}\sup_{(x,\tau )\in \mathcal{S}}\sup_{P\in 
\mathcal{P}}\mathbf{E}||\Sigma _{n,\tau ,\varepsilon
_{n}}^{-1/2}(x)q_{n,\tau }(x;\eta _{n})||^{r} \\
&&+C_{3}\sup_{(x,\tau )\in \mathcal{S}}\sup_{P\in \mathcal{P}}\mathbf{E}%
\omega _{n,p}\left( \mathbb{Z}_{n,\tau ,\tau ,\varepsilon _{n}}^{(1)}(x,0);%
\frac{C_{4}}{\sqrt{n}}\mathbf{E}||\Sigma _{n,\tau ,\varepsilon
_{n}}^{-1/2}(x)q_{n,\tau }(x;\eta _{n})||^{3}\right) ,
\end{eqnarray*}%
for some positive constants $C_{1},C_{2},C_{3},$ and $C_{4}$, where%
\begin{equation*}
\omega _{n,p}\left( \mathbf{v};c\right) \equiv \sup \left\{ |\Lambda
_{A,n,p}^{\Sigma }(\mathbf{v})-\Lambda _{A,n,p}^{\Sigma }(\mathbf{y})|:%
\mathbf{y}\in \mathbf{R}^{|A|},||\mathbf{v}-\mathbf{y}||\leq c\right\} .
\end{equation*}%
The proof is complete by (\ref{rate}) and by the condition $%
n^{-1/2}h^{-\left( \frac{3M-4}{2M-4}\right) d-\nu }\rightarrow 0$.
\end{proof}

%\clearpage

\section{Proof of Theorem AUC\ref{thm:AUC1}}\label{sec:appendix-D}

The conclusion of Theorem AUC\ref{thm:AUC1} follows immediately from 
Theorem 1, provided that all the regularity conditions in Theorem 1 are satisfied.
The following lemma shows that Assumptions AUC1-AUC4 are sufficient conditions for that purpose. 
One key condition to check regularity condition of Theorem 1 is to establish 
asymptotic linear representations in
Assumptions A1 and B1. We borrow the results from  \citeasnoun{LSW-quantile:15}.

\begin{LemmaAUC}
Suppose that Assumptions AUC1-AUC4 hold. Then Assumptions A1-A6
and B1-B4 hold with the following definitions: $%
J=2,\ r_{n,j}\equiv \sqrt{nh^{d}},$%
\begin{eqnarray*}
v_{n,\tau ,1}(x) &\equiv &\mathbf{e}_{1}^{\top }\{\gamma _{\tau
,2}(x)-\gamma _{\tau ,3}(x)\}, \\
v_{n,\tau ,2}(x) &\equiv &\underline{b}-\mathbf{e}_{1}^{\top }\{2\gamma
_{\tau ,2}(x)-\gamma _{\tau ,3}(x)\}, \\
\beta _{n,x,\tau ,1}(Y_{i},z) &\equiv &\alpha _{n,x,\tau ,2}(Y_{i},z)-\alpha
_{n,x,\tau ,3}(Y_{i},z),\text{ \textit{and}} \\
\beta _{n,x,\tau ,2}(Y_{i},z) &\equiv &-2\alpha _{n,x,\tau
,2}(Y_{i},z)+\alpha _{n,x,\tau ,3}(Y_{i},z)\text{\textit{,}}
\end{eqnarray*}%
\textit{where }$\tilde{l}_{\tau }(u)\equiv \tau -1\{u\leq 0\},$ $%
Y_{i}=\{(B_{\ell i},L_{i}):\ell =1,\ldots ,L_{i}\}$\textit{, and}%
\begin{equation*}
\alpha _{n,x,\tau ,k}(Y_{i},z)\equiv -1\left\{ L_{i}=k\right\} \sum_{l=1}^{k}%
\tilde{l}_{\tau }\left( B_{\ell i}-\gamma _{\tau ,k}^{\top }(x)\cdot H\cdot
c\left( z\right) \right) \mathbf{e}_{1}^{\top }M_{n,\tau ,k}^{-1}(x)c\left(
z\right) K\left( z\right) \text{.}
\end{equation*}
\end{LemmaAUC}

\begin{proof}[Proof of Lemma AUC1]
First, let us turn to Assumption A1. By Assumptions AUC2 and AUC3, it
suffices to consider $\hat{v}_{\tau ,2}(x)$ that uses $\underline{b}$
instead of $\underline{\hat{b}}$. The asymptotic linear representation in
Assumption A1 follows from Theorem 1 of \citeasnoun{LSW-quantile:15}. The error rate $o_{P}(\sqrt{%
h^{d}})$ in Assumption A1 is satisfied, because 
\begin{equation}
h^{-d/2}\left( \frac{\log ^{1/2}n}{n^{1/4}h^{d/4}}\right)
=n^{-1/4}h^{-3d/4}\log ^{1/2}n\rightarrow 0,  \label{error}
\end{equation}%
by Assumption AUC2(ii) and the condition $r>3d/2-1$. Assumption A2 follows
because both $\beta _{n,x,\tau ,1}(Y_{i},z)$ and $\beta _{n,x,\tau
,2}(Y_{i},z)$ have a multiplicative component of $K(z)$ which has a compact
support by Assumption AUC2(i). As for Assumption A3, we use Lemma 2. First define 
\begin{eqnarray*}
e_{x,\tau ,k,li} &\equiv &1\left\{ L_{i}=k\right\} \tilde{l}_{\tau }\left(
B_{li}-\gamma _{\tau ,k}^{\top }(x)\cdot H\cdot c\left( \frac{X_{i}-x}{h}%
\right) \right) \text{ and} \\
\xi _{x,\tau ,k,i} &\equiv &\mathbf{e}_{1}^{\top }M_{n,\tau
,k}^{-1}(x)c\left( \frac{X_{i}-x}{h}\right) K\left( \frac{X_{i}-x}{h}\right)
\end{eqnarray*}%
First observe that for each fixed $x_{2}\in \mathbf{R}^{d},\tau _{2}\in 
\mathcal{T}$, and $\lambda >0$,%
\begin{eqnarray}
&&\mathbf{E}\left[ \sup_{||x_{2}-x_{3}||+||\tau _{2}-\tau _{3}||\leq \lambda
}\left( \alpha _{n,x_{2},\tau _{2},2}\left( Y_{i,}\frac{X_{i}-x_{2}}{h}%
\right) -\alpha _{n,x_{3},\tau _{3},2}\left( Y_{i,}\frac{X_{i}-x_{3}}{h}%
\right) \right) ^{2}\right]  \label{develop} \\
&\leq &2\sum_{l=1}^{k}\mathbf{E}\left[ \mathbf{E}\left[
\sup_{||x_{2}-x_{3}||+||\tau _{2}-\tau _{3}||\leq \lambda }\left(
e_{x_{2},\tau _{2},k,li}-e_{x_{3},\tau _{3},k,li}\right) ^{2}|X_{i}\right]
\xi _{x_{2},\tau _{2},k,i}^{2}\right]  \notag \\
&&+2\sum_{l=1}^{k}\mathbf{E}\left[ \sup_{||x_{2}-x_{3}||+||\tau _{2}-\tau
_{3}||\leq \lambda }\left( \xi _{x_{2},\tau _{2},k,i}-\xi _{x_{3},\tau
_{3},k,i}\right) ^{2}\right] .  \notag
\end{eqnarray}%
Using Lipschitz continuity of the conditional density of $B_{li}$ given $%
L_{i}=k$ and $X_{i}=x$ in $(x,\tau )$ and Lipschitz continuity of $\gamma
_{\tau ,k}(x)$ in $(x,\tau )$ (Assumption AUC1), we find that the first term
is bounded by $Ch^{-s_{1}}\lambda $ for some $C>0$ and $s_{1}>0$. Since 
\begin{equation*}
M_{n,\tau ,k}(x)=kP\left\{ L_{i}=k|X_{i}=x\right\} f_{\tau ,k}(0|x)f(x)\int
K(t)c(t)c(t)^{\top }dt+o(1),
\end{equation*}%
we find that $M_{n,\tau ,k}^{-1}(x)$ is Lipschitz continuous in $(x,\tau )$
by Assumptions AUC1. Hence the last term in (\ref{develop}) is also bounded
by $Ch^{-s_{2}}\lambda ^{2}$ for some $C>0$ and $s_{2}>0$. Therefore, if we
take%
\begin{equation*}
b_{n,ij}(x,\tau )=\alpha _{n,x,\tau ,2}\left( Y_{i,}\frac{X_{i}-x}{h}\right)
,
\end{equation*}%
this function satisfies the condition in Lemma 2. Also, observe that%
\begin{equation*}
\mathbf{E}\left[ \left\vert \alpha _{n,x,\tau ,2}\left( Y_{i,}\frac{X_{i}-x}{%
h}\right) \right\vert ^{4}\right] \leq C,
\end{equation*}%
because $\alpha _{n,x,\tau ,2}(\cdot ,\cdot )$ is uniformly bounded. We also
obtain the same result for $\alpha _{n,x,\tau ,3}(\cdot ,\cdot )$. Thus the
conditions of Lemma 2 are satisfied with $b_{n,ij}(x,\tau )$ taken to
be $\beta _{n,x,\tau ,1}(Y_{i},(X_{i}-x)/h)$ or $\beta _{n,x,\tau
,2}(Y_{i},(X_{i}-x)/h)$. Now Assumption A3 follows from Lemma 2(i). The rate
condition in Assumption A4(i) is satisfied by Assumption AUC2(ii).
Assumption A4(ii) is imposed directly by Assumption AUC4(i). Since we are taking $\hat{\sigma}_{\tau ,j}(x)=%
\hat{\sigma}_{\tau ,j}^{\ast }(x)=1$, it suffices to take $\sigma _{n,\tau
,j}(x)=1$ in Assumption A5 and Assumption B3.$\ $Assumption A6(i) is
satisfied because $\beta _{n,x,\tau ,j}$ is bounded. 
Assumption A6(ii) is imposed directly by Assumption AUC4(ii).
Assumption B1 follows
by Lemma QR2  of \citeasnoun{LSW-quantile:15}. Assumption B2 follows from Lemma 2(ii). Assumption B4
follows from the rate condition in Assumption AUC2(ii). In fact, when $\beta
_{n,x,\tau ,j}$ is bounded, the rate condition in Assumption B4 is reduced
to $n^{-1/2}h^{-3d/2-\nu }\rightarrow 0$, as $n\rightarrow \infty $, for
some small number $\nu >0$.
\end{proof}

\bibliographystyle{econometrica}
\bibliography{LSW_30Aug2014}

\end{document}